\documentclass{amsart}
\usepackage[utf8]{inputenc}
\usepackage[english]{babel}

\usepackage[dvips]{graphics,epsfig}
\usepackage{times}
\usepackage{amsmath,amsthm,amssymb}
\usepackage{array}
\usepackage{bbm}
\usepackage{xcolor}
\usepackage{verbatim}
\usepackage{upgreek}
\usepackage[foot]{amsaddr}
\usepackage{caption}
\usepackage{circuitikz}

\usepackage[most]{tcolorbox}
\usepackage{lipsum}

\usepackage{tikz-cd}

\usepackage{geometry}

\newtheorem{theorem}{Theorem}[section]

\newtheorem{lemma}[theorem]{Lemma}
\newtheorem{question}[theorem]{Question}
\newtheorem{cor}[theorem]{Corollary}
\newtheorem{prop}[theorem]{Proposition}
\newtheorem{observation}[theorem]{Observation}
\newtheorem{fact}[theorem]{Fact}
\newtheorem{claim}[theorem]{Claim}

\theoremstyle{definition}
\newtheorem{hypothesis}[theorem]{Hypothesis}
\newtheorem{definition}[theorem]{Definition}

\theoremstyle{remark}
\newtheorem{remark}[theorem]{Remark}

\def\Ind{\setbox0=\hbox{$x$}\kern\wd0\hbox to 0pt{\hss$\mid$\hss} \lower.9\ht0\hbox to 0pt{\hss$\smile$\hss}\kern\wd0}
\def\Notind{\setbox0=\hbox{$x$}\kern\wd0\hbox to 0pt{\mathchardef \nn=12854\hss$\nn$\kern1.4\wd0\hss}\hbox to 0pt{\hss$\mid$\hss}\lower.9\ht0 \hbox to 0pt{\hss$\smile$\hss}\kern\wd0}
\def\ind{\mathop{\mathpalette\Ind{}}}

\numberwithin{equation}{section}

\title{Some applications of the real strict order property hierarchy}
\author{Scott Mutchnik}

\thanks{This project was supported by the NSF under Grant No. DMS-2303034.}

\begin{document}

\maketitle

\begin{abstract}

This will be part two of a two-part paper, beginning with \cite{ApproxOrder}.

We give external motivation for the properties $\mathrm{NSOP}_{r}$ for non-integer values of $r$, which we introduced and gave internal motivation for in \cite{ApproxOrder}. Specifically, we demonstrate applications of these real-valued properties to the main questions on the properties $\mathrm{NSOP}_{n}$ for integer values of $n$.

We show that the properties $\mathrm{NSOP}_{r}$, previously defined for real values $r \geq 3$, are even well-defined for real values $r \geq 2$, showing that $\mathrm{NSOP}_{2} \subseteq \mathrm{NSOP}_{r}$ for our original definition of $\mathrm{NSOP}_{r}$ even when $2 < r < 3$. Our proof uses the theorem, proven in \cite{NSOP2}, that $\mathrm{NSOP}_{2} = \mathrm{NSOP}_{1}$. For $2 < r < 3$, it will then be the case that $\mathrm{NSOP}_{2} \subseteq \mathrm{NSOP}_{r} \subseteq \mathrm{NSOP}_{3}$, so newness of all of the well-defined properties $\mathrm{NSOP}_{r}$ for non-integer $r$ would negatively resolve the problem of whether $\mathrm{NSOP}_{2}$ is equal to $\mathrm{NSOP}_{3}$.

We then prove an approximate alternative between two possibilities: (1) that in extending Shelah's original $\mathrm{NSOP}_{n}$ hierarchy for integers $n \geq 3$ to the $\mathrm{NSOP}_{r}$ hierarchy for reals $r > 2$, we really did introduce new classification-theoretic properties, and (2) that $\mathrm{NSOP}_{n+1} \cap \mathrm{NTP}_{2} = \mathrm{NSOP}_{n} \cap \mathrm{NTP}_{2}$ for integers $n \geq 3$, which would resolve a central open problem in classification theory. More precisely, we give a rigorous sense in which (1) can fail on particularly general grounds, and then show that if (1) fails for these general reasons, (2) must be true. 

In the longest part of this paper, we apply cycle-removal techniques from the theory of the properties $\mathrm{NSOP}_{r}$ for real-values of $r$ to make progress on the question of whether $\mathrm{NSOP}_{2}$ is equal to $\mathrm{NSOP}_{3}$. We (a) show that if $\mathcal{H}$ is a hereditary class of structures defined by finitely many forbidden weakly embedded substructures, if every theory whose models have age $\mathcal{H}$ has $\mathrm{SOP}_{2}$, then every theory whose models have age $\mathcal{H}$ has $\mathrm{SOP}_{3}$, and (b) observe that this is not the case replacing ``$\mathrm{SOP}_{2}$" with ``the tree property." By recent work of Bodor, Bodirsky and Marimon in \cite{BBM25}, this general result reduces to what initially superficially appeared to be a mere verification for a specific family of structures.

In sidebars, we give two additional applications of our techniques arising from the real-valued $\mathrm{NSOP}_{r}$ hierarchy, one to graphs definable in $\mathrm{NTP}_{2}$ theories, and another to deriving approximate implications between properties from classification theory.

\end{abstract}

\setcounter{tocdepth}{2}
\tableofcontents

\section{Introduction}

One of the main problems of model theory is whether $\mathrm{NSOP}_{2}$ is equal to $\mathrm{NSOP}_{3}$. A related problem is whether the larger $\mathrm{NSOP}_{n}$ hierarchy for integers $n \geq 2$ is strict within $\mathrm{NTP}_{2}$, or equivalently whether $\mathrm{NSOP}_{n} \cap \mathrm{NTP}_{2}$ is equal to $\mathrm{NSOP}_{n+1} \cap \mathrm{NTP}_{2}$ for integers $n \geq 2$.

Both of these problems, which we will describe in more detail later in this introduction and paper, belong to the model-theoretic field of \textit{classification theory}, specifically the classification of unstable theories. Classification theory has its ultimate origins\footnote{Deeper within classification theory's intellectual ancestry, one might look to various classification programs within mathematics in general, including the categorization of rings according to different algebraic invariants such as being a Noetherian, Artinian or Dedekind ring, or the classification of compact surfaces beginning with work of Euler. More boldly, we might add to this other forms of classification such as the Linnean taxonomy into species, genus, family, and so on, the Aristotelian categories, etc. A comprehensive intellectual history of the concept of classification would be beyond the scope of this introduction. We will instead start with the more direct motivating intuition for classification theory as a field.} in observations as simple as the fact that the complex numbers $\mathbb{C}$ form the only algebraically closed field, up to isomorphism of fields, of cardinality that of the continuum. More proximately, classification theory begins with Morley's results in \cite{M65} generalizing this observation. According to Morley's categoricity theorem, if a first-order theory has exactly one model up to isomorphism of size $\kappa$ for some uncountable cardinal $\kappa$, that theory must have exactly one model up to isomorphism of size $\kappa'$ for \textit{every} uncountable cardinal $\kappa'$.

\: Shelah develops Morley's initial work into a far more extensive model-counting program, in his monograph establishing classification theory as one of the main areas of research in mathematical logic (\cite{Sh90}). There, he gives a strikingly rich classification of first-order theories according to the spectrum $I(T, \kappa)$ of the numbers of non-isomorphic models of $T$ of cardinality $\kappa$, a classification that is completed for uncountable $\kappa$ in \cite{HHL00}. The classification of first-order theories in terms of cardinality spectra $I(T, \kappa)$ of uncountable models takes place entirely within the class of \textit{stable theories}, which can be defined either in terms of type-counting or in terms of the \textit{order property} for formulas. If $T$ is unstable, then for uncountable values of $\kappa$, one just obtains the maximum number of models $I(T, \kappa)=2^{\kappa}$, so unstable theories are less amenable to Shelah's model-counting program. Thus, Shelah's original classification program prompts the question of how to classify unstable theories, which the model-counting, type-counting and order property conditions suggest must be more complex than stable theories.

Shelah initiates this project of classifying unstable theories by introducing what, together with additional properties added by Džamonja and Shelah, is now considered to be the classical hierarchy of unstable theories.  This classical hierarchy grounds a system of ``dividing lines" for measuring the logical complexity of a first-order theory. As with stability, all of the main properties of the classical hierarchy can be defined in terms of combinatorial configurations in individual formulas. However, some of the properties in this hierarchy can be described in surprising ways, via analogues of the original techniques used to count the number of models $I(T, \kappa)$ in stable theories. (The most important of these techniques is to develop a theory of independence. See, e.g., \cite{Kim98} and \cite{KR17} for some of the foundational work on independence in unstable theories.) As suggested by the question of whether $\mathrm{NSOP}_{2}$ is equal to $\mathrm{NSOP}_{3}$, as well as the question of whether $\mathrm{NSOP}_{n} \cap \mathrm{NTP}_{2}$ is equal to $\mathrm{NSOP}_{n+1} \cap \mathrm{NTP}_{2}$ for integers $n \geq 2$, some of the most conspicuous gaps in our understanding of the unstable hierarchy involve the family of properties $\mathrm{NSOP}_{n}$ for larger values of $n$. By this, we mean the negation of the \textit{$n$-strict order property}, particularly when $n$ is greater than $1$, or greater than $2$.

An especially intuitive property motivating the properties $\mathrm{NSOP}_{n}$ is Shelah's property $\mathrm{NSOP}$: the negation of the \textit{strict order property}, full stop. A theory has the strict order property if it interprets a partially ordered set with an infinite (ascending or descending) chain for the order relation. For $<$ this order relation, the infinite chain gives a sequence $\{a_i\}_{i < \omega}$ with $\models a_{i}  < a_{j}$ when $i < j$. Moreover, the fact that $<$ is an order implies that the relation $<$ has no directed $n$-cycles for any $n$. These facts about orders motivate the following properties approximating the strict order property, which Shelah introduces in \cite{She95}:

\begin{definition}\label{integer-valued hierarchy}(Shelah, \cite{She95}.)
Let $n \geq 3$ be an integer. A theory $T$ is $\mathrm{NSOP}_{n}$ (that is, does not have the \emph{n-strict order property}) if there is no definable relation $R(x, y)$ such that

\begin{itemize}
    \item there exists a sequence $\{a_{i}\}_{i < \omega}$ such that $\models R(a_{i}, a_{j})$ for $i <j$, but

    \item there does not exist a directed $n$-cycle for $R(x, y)$; i.e., there are no $\{a_{i}\}^{n-1}_{i = 0}$ such that $\models R(a_{i}, a_{(i+1)\: \mathrm{mod}\: n}) $ for $0 \leq i \leq n-1$.
\end{itemize}
 Otherwise, it has $\mathrm{SOP}_{n}$.
\end{definition}

Džamonja and Shelah add to this family by introducing the properties $\mathrm{NSOP}_{1}$ and $\mathrm{NSOP}_{2}$.  (Džamonja and Shelah define these properties in \cite{DS04}; Shelah states this definition earlier, in \cite{Sh99}, referencing then-ongoing joint work where these properties were defined.) These have different definitions--informally speaking, they are both defined in terms of trees, with $\mathrm{SOP}_{2}$ representing an especially symmetrical tree and $\mathrm{SOP}_{1}$ representing a more asymmetrical tree--and we will discuss them in greater detail later on. (They will be defined in Definition \ref{nsop1 definition and nsop2 definition} below.) A longstanding open problem originally posed by Džamonja and Shelah (\cite{Sh99}, \cite{DS04}) asked whether $\mathrm{NSOP}_{1}$ is equal to $\mathrm{NSOP}_{2}$. This was resolved by the author, who showed in \cite{NSOP2} that $\mathrm{NSOP}_{1}$ is equal to $\mathrm{NSOP}_{2}$. However, this still leaves open the additional longstanding problem mentioned at the beginning of the introduction:

\begin{question}\label{is nsop2 equal to nsop3}(\cite{DS04}, \cite{Sh99})

Is $\mathrm{NSOP}_{2}$ equal to $\mathrm{NSOP}_{3}$?

\end{question}

Even questions such as this one exist within a restricted understanding of $\mathrm{NSOP}_{n}$, under which $n$ is assumed to be a (positive) integer. But the legacy integer-valued view does not capture the full extent of the challenges surrounding the higher $\mathrm{NSOP}_{n}$ hierarchy. Until recently, the properties $\mathrm{NSOP}_{n}$ were only defined for integer values of $n$, and it is not completely transparent from Shelah's original definition (for $n\geq 3$) that it can be extended to non-integer values of $n$. This began to change with some initial interest in extensions to non-integer values by Hanson (\cite{Hanson2020}), who defined a real-valued family of properties which he showed coincided with the original integer-valued $\mathrm{NSOP}_n$ hierarchy.\footnote{See the appendix for a discussion of this historical precedent for interest in the possibility of extending this hierarchy to non-integer values of $n$. In this thesis appendix of Hanson (\cite{Hanson2020}), a family of properties parametrized by real values $r$ is introduced, which Hanson describes as ``roughly equivalent to an instance of `$\mathrm{SOP}_{\frac{1}{r}}$.'" These properties are defined in terms of quasimetrics. In contrast to the apparent difficulty of the problem of whether the real-valued $\mathrm{NSOP}_{r}$ hierarchy is distinct from the integer-valued $\mathrm{NSOP}_{n}$ hierarchy (Question \ref{disinctness question}), it is known, per Hanson, that his quasimetric properties coincide with the integer-valued $\mathrm{NSOP}_{n}$ hierarchy. However, this is in discrete logic. In continuous logic, the possibility is left open that the family of real-valued quasimetric properties differs from the family of properties $\mathrm{SOP}_{n}$ from $n$ an integer, but only infinitesimally. Because we are restricting our attention to discrete first-order logic with its classical two-valued semantics (and in continuous logic, the newness of Hanson's properties would imply that they do not in an exact sense give an extension of the integer-valued $\mathrm{SOP}_n$ hierarchy) we will not take Hanson's properties as the definition of $\mathrm{SOP}_{r}$ for $r$ real, but will give an exposition of Hanson's equivalence argument in the appendix. The question of whether the quasimetric properties do differ infinitesimally from Shelah's hierarchy in continuous logic remains of interest for further investigation.} Then in \cite{ApproxOrder}, the author observed that Shelah's definition can be restated to make sense for non-integer values of $n$, so as to yield a hierarchy \textit{not} known to coincide with the hierarchy for integer values of $n$. From this observation, we obtain the following properties $\mathrm{NSOP}_{r}$ for real values of $r$. Note that when $r = n$ is an integer, these coincide with Shelah's original properties $\mathrm{NSOP}_{n}$ for $n \geq 3$.

\begin{definition}\label{real-valued hierarchy}(\cite{ApproxOrder})

Let $r \geq 3$ be a real number. A theory $T$ is $\mathrm{NSOP}_{r}$ if there is no definable relation $R(x, y)$ such that:

\begin{itemize}
    \item there exists a sequence $\{a_{i}\}_{i < \omega}$ such that $\models R(a_{i}, a_{j})$ for $i <j$, but

    \item there does not exist a set $\{a_{\theta}\}_{\theta \in S^{1}}$, indexed by the unit circle $S^{1}$, such that $\models R(a_{\theta}, a_{\psi})$ for all $\theta, \psi \in S^{1}$ with $\psi$ lying at most $\frac{2\pi}{r}$ radians counterclockwise from $\theta$.
\end{itemize}

\end{definition}

We will return soon to the assumption that $r \geq 3$. Having successfully extended the $\mathrm{NSOP}_{n}$ hierarchy for $n$ an integer to the family of properties $\mathrm{NSOP}_{r}$ for $r$ a real number, we are then left with the open problem of whether we actually introduced new classification-theoretic properties:

\begin{question}\label{disinctness question}(\cite{ApproxOrder})

Is the $\mathrm{NSOP}_{r}$ hierarchy for $r \geq 3$ a real number distinct from the $\mathrm{NSOP}_{n}$ hierarchy for $n \geq 3$ an integer? That is, is there some real number $r \geq 3$ such that the class of $\mathrm{NSOP}_{r}$ theories is distinct from the class of $\mathrm{NSOP}_{n}$ theories for all $n \geq 3$?

\end{question}

If this answer is yes, we will have demonstrated a troubling sense in which the classical classification-theoretic hierarchy is incomplete.

The status of this question appears intricate even at the purely combinatorial level. We can straightforwardly translate the properties $\mathrm{SOP}_{r}$ for $r$ real into the langauge of hereditary classes of structures, or classes of structures closed under induced substructures. This will give us a new real-valued combinatorial quantity, $\mathfrak{o}(\mathcal{H})$, associated with any hereditary class $\mathcal{H}$. In \cite{ApproxOrder}, the author proves an integrality theorem for $\mathfrak{o}(\mathcal{H})$ that is of independent interest from the perspective of combinatorics. Specifically, the author shows that the real-valued quantity $\mathfrak{o}(\mathcal{H})$ is in fact an integer, when $\mathcal{H}$ is a hereditary class defined by finitely many forbidden weakly embedded substructures. On the other hand, if $\mathfrak{o}(\mathcal{H})$ can be shown not to be an integer in general, this would establish one of the main preconditions for the real-valued $\mathrm{NSOP}_{r}$ hierarchy for first-order theories to be distinct from the integer-valued $\mathrm{NSOP}_{n}$ hierarchy. Informally, the author shows in \cite{ApproxOrder} that this claim that $\mathfrak{o}(\mathcal{H})$ is not an integer
in general cannot be refuted using \textit{interval helix maps}. These are a class of helix-shaped covering maps powerful enough to develop much of the theory of the real-valued $\mathrm{NSOP}_{r}$ hierarchy.

We then see that there is already enough to motivate the properties $\mathrm{NSOP}_{r}$ for non-integer values of $r$ from an \textit{intrinsic} point of view: the subtle restatement of Shelah's original definition, the disturbing potential for incompleteness within the original classification-theoretic hierarchy, and the behavior of the real-valued $\mathrm{NSOP}_{r}$ as properties in combinatorics. By contrast, this paper will concentrate on \textit{applications} of the properties $\mathrm{NSOP}_{r}$ for non-integer values of $r$. These will include applications of the question of whether the real-valued properties $\mathrm{NSOP}_{r}$ are distinct from the integer-valued properties $\mathrm{NSOP}_{n}$, and the techniques arising out of investigating this distinctness question. We will be especially interested in applying the properties $\mathrm{NSOP}_{r}$ for \textit{non-integer} values of $r$ to questions about the properties $\mathrm{NSOP}_{n}$ for \textit{integer} values of $n$. These will include the question of whether $\mathrm{NSOP}_{2}$ is equal to $\mathrm{NSOP}_{3}$, as well as the question of whether $\mathrm{NSOP}_{n} \cap \mathrm{NTP}_{2}$ is equal to $\mathrm{NSOP}_{n+1} \cap \mathrm{NTP}_{2}$ for integers $n \geq 3$. By generalizing the proof of one of the main results on the real-valued distinctness problem, we will make progress on the question of whether $\mathrm{NSOP}_{2}$ is equal to $\mathrm{NSOP}_{3}$, obtaining the most difficult leap yet since the author proved that $\mathrm{NSOP}_{1}$ is equal to $\mathrm{NSOP}_{2}$. 

\:

\textbf{The $2$-strict order property, $3$-strict order property, and $r$-strict order property for $2 < r < 3$:} Our first application of the properties $\mathrm{NSOP}_{r}$ for non-integer values of $r$, to the problem of whether $\mathrm{NSOP}_{2}$ is equal to $\mathrm{NSOP}_{3}$, will be to develop the intermediate fine structure in between $\mathrm{NSOP}_{2}$ and $\mathrm{NSOP}_{3}$. Because $\mathrm{NSOP}_{2} \subseteq \mathrm{NSOP}_{3}$, the question of whether $\mathrm{NSOP}_{2}$ is equal to $\mathrm{NSOP}_{3}$ is the same as the question of whether this inclusion is strict. This observation motivates the informal question of whether the $\mathrm{NSOP}_{r}$ hierarchy, for real values of $r$, yields a family of intermediate properties $\mathrm{NSOP}_{2} \subseteq \mathrm{NSOP}_{r} \subseteq \mathrm{NSOP}_{3}$ for real values $r$ with $2 < r < 3$.

However, we initially introduced the properties $\mathrm{NSOP}_{r}$ only for real values $r$ with $r \geq 3$. The special difficulty with extending to $2 < r < 3$ is not in formulating the definition of the property $\mathrm{NSOP}_{r}$ for all reals $r > 2$. And the special difficulty for $2 < r < 3$ \textit{ought} not to be in formulating the definitions. Our schema for the properties $\mathrm{NSOP}_{r}$, for real values $r \geq 3$ (in Definition \ref{real-valued hierarchy}), at least makes sense at face value for real values $r > 2$. So if $\mathrm{NSOP}_{r}$ is well-defined for $2 < r < 3$ in any informally reasonable sense, it ought to be defined by that same schema.

Instead, as we noted in \cite{ApproxOrder}, the potential difficulty in defining $\mathrm{NSOP}_{r}$ for $2 < r < 3$ is that the original, integer-valued $\mathrm{NSOP}_{n}$ hierarchy forms an ascending chain beginning $\mathrm{NSOP}_{1} \subseteq \mathrm{NSOP}_{2} \subseteq \mathrm{NSOP}_{3} \subseteq \ldots$. Any extension of the original integer-valued $\mathrm{NSOP}_{n}$ hierarchy, to include properties $\mathrm{NSOP}_{r}$ for some non-integer values of $r$, ought to preserve the fact that the hierarchy forms an ascending chain. For the hierarchy to form an ascending chain will require that $\mathrm{NSOP}_{2} \subseteq \mathrm{NSOP}_{r}$ whenever $\mathrm{NSOP}_{r}$ is well-defined for $r > 2$. And, if we were to define $\mathrm{NSOP}_{r}$ for $2 < r < 3$ using the same schema as $\mathrm{NSOP}_{r}$ for $r \geq 3$, it is not obvious from what is already known that $\mathrm{NSOP}_{2} \subseteq \mathrm{NSOP}_{r}$.

The main result of Section \ref{between nsop2 and nsop3} will be to remove this potential obstruction to $\mathrm{NSOP}_{r}$ being well-defined for $2 < r < 3$. We will show:
\begin{theorem}\label{main theorem 1}
    Let $\mathrm{NSOP}_{r}$ for real values $r > 2$ be defined just as $\mathrm{NSOP}_{r}$ for $r \geq 3$. Then for real values $r > 2$, $\mathrm{NSOP}_{2} \subseteq \mathrm{NSOP}_{r}$.
\end{theorem}
Therefore $\mathrm{NSOP}_{r}$, which we previously only saw to be well-defined for reals $r \geq 3$, is actually well-defined even for $2 < r < 3$ (using the same schema). We conclude that the question of whether the properties $\mathrm{NSOP}_{r}$ for non-integer $r$ actually give us new classification-theoretic properties is directly connected to the question of whether $\mathrm{NSOP}_{2}$ is equal to $\mathrm{NSOP}_{3}$. Specifically, since 

$$\mathrm{NSOP}_{2} \subseteq \mathrm{NSOP}_{r} \subseteq \mathrm{NSOP}_{3} \mathrm{\: \: for \: \:} 2< r< 3,$$ 

if every well-defined property $\mathrm{NSOP}_{r}$ for non-integer $r$ is a new property distinct from all of the integer-values properties $\mathrm{NSOP}_{n}$, it must follow that $\mathrm{NSOP}_{2}$ is not equal to $\mathrm{NSOP}_{3}$.

Our proof of Theorem \ref{main theorem 1} appears to require the theorem, proven in \cite{NSOP2}, that $\mathrm{NSOP}_{1} = \mathrm{NSOP}_{2}$. After applying this theorem, we must use analogues of the original stability-theoretic tools used by Shelah to analyze the model-counting function $I(T, \kappa)$. Namely, we apply the theory of \textit{Kim-independence} in $\mathrm{NSOP}_{1}$ theories, introduced by Kaplan and Ramsey in \cite{KR17}. This allows us to prove a property that implies $\mathrm{NSOP}_{r}$ for real values $r > 2$: if $R(x, y)$ is a definable relation with a sequence $\{a_{i}\}_{i< \omega}$ such that $\models R(a_{i}, a_{j})$ for $i < j$, then the relation defined by $R(x, y)$ embeds every directed graph (possibly with additional edges). This property of obtaining every directed graph from an infinite chain, which we call \textit{o-maximality}, will offer one supplementary motivation for the next section.

\:

\textbf{General non-distinctness of real-valued $\mathrm{NSOP}_{r}$ and integer-valued $\mathrm{NSOP}_{n}$:}  In our next application of the real-valued $\mathrm{NSOP}_{r}$ distinctness question, we will investigate most cases of the second open problem from the beginning of this introduction. This question will involve \textit{$\mathrm{NTP}_{2}$ theories}, or theories with the negation of the \textit{$2$-tree property} defined in Definition \ref{ntp_2 definition} below.

\begin{question}\label{ntp_2 vs nsop_n}

Let $n \geq 3$ be an integer. Is $\mathrm{NSOP}_{n} \cap \mathrm{NTP}_{2}$ equal to $\mathrm{NSOP}_{n+1} \cap \mathrm{NTP}_{2}$?

\end{question}

This question was first explicitly asked in \cite{Che14}, but it originates as early as Shelah's work in \cite{Sh90}, where he gives a claimed counterexample to the statement that every $\mathrm{NTP}_{2}$ theory without the strict order property is simple.

Our goal in Section \ref{the approximate alternative} is to demonstrate a connection between two possibilities with very different motivations:

(1) The real-valued $\mathrm{NSOP}_{r}$ hierarchy is distinct from the integer-valued $\mathrm{NSOP}_{n}$ hierarchy.

(2) $\mathrm{NSOP}_{n} \cap \mathrm{NTP}_{2}$ is equal to $\mathrm{NSOP}_{n+1} \cap \mathrm{NTP}_{2}$ for $n \geq 3$.

If (1) is the case, that will demonstrate that extending the original $\mathrm{NSOP}_{n}$ hierarchy to our real-valued $\mathrm{NSOP}_{r}$ hierarchy really does introduce new classification-theoretic properties, and therefore that the original hierarchy is incomplete in the sense we described. On the other hand, if (2) is the case, that will resolve the established open question we just mentioned, Question \ref{ntp_2 vs nsop_n} above. These two disparate motivations collectively motivate any result making progress toward the claim that either (1) \textit{or} (2) must be true.

Informally speaking, in our main result of Section \ref{the approximate alternative}, we will show that it is \textit{approximately} the case that either (1) is true or (2) is true (even though we will not prove that this is \textit{exactly} the case.) Specifically, we will give a rigorous sense in which (1) might fail on sufficiently general grounds. And we will show that if (1) fails on sufficiently general grounds in the sense we describe, (2) must be true.

The specific sense in which the real-valued $\mathrm{NSOP}_{r}$ hierarchy and the integer-valued $\mathrm{NSOP}_{n}$ hierarchy can be hypothesized to be non-distinct on sufficiently general grounds will be Hypothesis \ref{nondistinctness on sufficiently general grounds} below. Explicitly put, the original open claim that the real-valued $\mathrm{NSOP}_{r}$ hierarchy is non-distinct from integer-valued $\mathrm{NSOP}_{n}$ hierarchy says: \textit{for any integer $n \geq 3$, if a definable relation in an $\mathrm{NSOP}_{n}$ theory has an infinite chain, it must have the unit circle configuration for $n-1< r< n$.} Now informally, let us consider the possibility that this is the case, but for reasons that only involve the bare minimum of the specific features of the unit circle configurations. That is, perhaps the reason why we can get the unit circle configurations from an infinite chain is just that the unit circle configurations for $n-1 < r < n$ have no (directed) $n-1$-cycles. This gives us the more general statement that \textit{for any integer $n \geq 3$, if a definable relation in an $\mathrm{NSOP}_{n}$ theory has an infinite chain, it must have every configuration given by an $n-1$-cycle-free directed graph.} This level of generality will suffice for our main result, but we will go one step farther.

That is to say, we may even make this more general assertion at the purely combinatorial level. This gives us the even more general statement that \textit{for any integer $n \geq 3$, in any directed graph which is $\mathrm{NSOP}_{n}$ in the quantifier-free graph sense, if the edge relation has an infinite chain, it must have a copy of every $n - 1$-cycle-free graph.} This entirely combinatorial statement, to which we have no known counterexample, will be what we mean for the real-valued $\mathrm{NSOP}_{r}$ and integer-valued $\mathrm{NSOP}_{n}$ hierarchies to be non-distinct \textit{on sufficiently general grounds.}

Our motivation for this general non-distinctness hypothesis will have some additional robustness. First, one of our main existing conclusions about non-distinctess of the real and integer hierarchies, for hereditary classes defined by finitely many forbidden weak substructures, actually yields non-distinctness on sufficiently general grounds (Remark \ref{comparison with existing results}). Second, any $\mathrm{NSOP}_{2}$ theory will satisfy the conclusion of our general non-distinctness hypothesis, which is how we show that any $\mathrm{NSOP}_{2}$ theory satisfies $\mathrm{NSOP}_{r}$ for reals $r > 2$ (Corollary \ref{nsop2 implies o-maximality}, shown within the proof of the above Theorem \ref{main theorem 1}). Applying our general non-distinctness hypothesis, we show:

\begin{theorem}\label{main theorem 2}

Suppose the real-valued $\mathrm{NSOP}_{r}$ and integer-valued $\mathrm{NSOP}_{n}$ hierarchies are non-distinct on sufficiently general grounds. Then for integers $n \geq 3$, $\mathrm{NSOP}_{n} \cap \mathrm{NTP}_{2}$ is equal to $\mathrm{NSOP}_{n+1} \cap \mathrm{NTP}_{2}$.

\end{theorem}

So to prove that the real-valued $\mathrm{NSOP}_{r}$ and integer-valued $\mathrm{NSOP}_{n}$ hierarchies are non-distinct on sufficiently general grounds would give us a new, combinatorial method of resolving Question \ref{ntp_2 vs nsop_n}.

The motivation for our proof of Theorem \ref{main theorem 2} comes from a specific example: the standard proof that the theory of the generic $\leq n$-cycle-free directed graph, introduced by Shelah in \cite{She95}, has $\mathrm{TP}_{2}$. We will observe that this proof gives us more general information, assuming our general non-distinctness hypothesis. 

\:

\textbf{Combinatorial analysis of the $\mathrm{NSOP}_{2}$-$\mathrm{NSOP}_{3}$ problem:} Since the author's result in \cite{NSOP2} that $\mathrm{NSOP}_{1}$ is equal to $\mathrm{NSOP}_{2}$, there has been interest in extending this to answer the question of whether $\mathrm{NSOP}_{2}$ is equal to $\mathrm{NSOP}_{3}$. Up to now, the main results have required assumptions about structural properties of models of an $\mathrm{NSOP}_{3}$ theory. These structural assumptions involved the theory of independence generalizing Shelah's methods in \cite{Sh90}, and Kaplan and Ramsey's methods in \cite{KR17}. (See \cite{GFA}, \cite{INDNSOP3} for the main applications of the theory of independence to the question of whether $\mathrm{NSOP}_{2}$ is equal to $\mathrm{NSOP}_{3}$.)

In section \ref{nsop2 vs nsop3}, we will make progress of a much different character on the question of whether $\mathrm{NSOP}_{2}$ is equal to $\mathrm{NSOP}_{3}$. Instead of requiring structural assumptions about models, our final result can be explicitly stated in terms of the axioms within a first-order theory. Specifically, we will reveal a fundamental combinatorial tension between the properties $\mathrm{SOP}_{2}$ and $\mathrm{NSOP}_{3}$. Our background assumption will be a basic condition of interest within hereditary class theory. To demonstrate this tension between $\mathrm{SOP}_2$ and $\mathrm{NSOP}_3$, we will apply the techniques and setting of our results on the real-valued $\mathrm{NSOP}_{r}$ hierarchy, though we will have to be careful in adapting the techniques.

Our results in \cite{ApproxOrder} about the integrality of $\mathfrak{o}(\mathcal{H})$--$\mathfrak{o}(\mathcal{H})$ being the supremum of those real values $r$ for which $\mathcal{H}$ has $\mathrm{SOP}_{r}$-- involved a hereditary class $\mathcal{H}$ defined by finitely many weakly embedded substructures. A weak embedding is just an embedding that preserves all positive instances of relation symbols: for example, an embedding of graphs that preserves all edges, but not necessarily all non-edges. So we are interested in classes $\mathcal{H}$ that are defined, for some finite family $\mathcal{F}$, as the class of finite structures that do not weakly embed any member of $\mathcal{F}$: for example, the class of directed graphs without directed triangles. Classes of structures defined by forbidden substructures are a rich subject of interest in combinatorics; see, e.g., \cite{LZ17} for a discussion of the analogous property with induced substructures.

In our main result of Section \ref{nsop2 vs nsop3}, we will be interested in the collection of all finite substructures of the models of a complete theory. For a  theory $T$, the \textit{age} of a model $M \models T$ will be the set of all finite substructures of $M$. Often, a hereditary class $\mathcal{H}$ will have the property that all theories whose models have age $\mathcal{H}$ will have a certain classification-theoretic configuration. For example, if $\mathcal{H}$ is the hereditary class of directed graphs without directed triangles, every theory whose models have age $\mathcal{H}$ will have $\mathrm{SOP}_{3}$ (but not necessarily $\mathrm{SOP}_{4}$). If $\mathcal{H}$ is the hereditary class of bipartite graphs without the complete bipartite graph $K_{2, 2}$, then every theory whose models have age $\mathcal{H}$ will be non-simple (but will not necessarily have $\mathrm{SOP}_{1}$; see \cite{CoK19}). Generalizing these examples, we will prove the following dichotomy between $\mathrm{NSOP}_{2}$ and $\mathrm{SOP}_{3}$:

\begin{theorem}\label{main theorem 3}

Let $\mathcal{H}$ be a hereditary class defined by a finite family of forbidden weakly embedded substructures. Then if every theory whose models have age $\mathcal{H}$ has $\mathrm{SOP}_{2}$, every theory whose models have age $\mathcal{H}$ has $\mathrm{SOP}_{3}$.

\end{theorem}

Crucially, the assumptions of this theorem are not too powerful: this dichotomy between $\mathrm{NSOP}_{2}$ and $\mathrm{SOP}_{3}$ is sharp. More precisely, the dichotomy does not extend to a false dichotomy between \textit{simplicity} and $\mathrm{SOP}_{3}$.

\begin{observation}\label{main observation}

Let $\mathcal{H}$ be a hereditary class defined by a finite family of forbidden weak substructures. Then it is \emph{not} the case that if every theory whose models have age $\mathcal{H}$ is non-simple, every theory whose models have age $\mathcal{H}$ has $\mathrm{SOP}_{3}$.

\end{observation}

While it remains open whether $\mathrm{NSOP}_{2}$ is equal to $\mathrm{NSOP}_{3}$, it is well-established that simplicity is not equal to $\mathrm{NSOP}_{3}$. So this observation is necessary to ensure that our main theorem of Section \ref{nsop2 vs nsop3} really does constitute progress on this open problem, and does not just extend to a dichotomy that we known to be false in general.

As suggested above, we can even avoid talking about models at all in the statement  of our theorem. Instead, we can state it completely syntactically, in terms of the axioms within a theory:

\:

\textit{Alternative statement of Theorem \ref{main theorem 3}, Observation \ref{main observation}}: For $T$ a complete theory in a language $\mathcal{L}$, let $\mathcal{H}_{T} =: \{A: ``\exists X \: \mathrm{Diag}_{A}(X)" \in T\}$, where $A$ ranges over all finite $\mathcal{L}$-structures, and the formula $\mathrm{Diag}_{A}(X)$ asserts that $X$ satisfies the diagram of $A$. Then for $\mathcal{H}$  a hereditary class defined by a finite family of forbidden weakly embedded substructures, if every theory $T$ with $\mathcal{H}_{T}=\mathcal{H}$ has $\mathrm{SOP}_{2}$, every theory $T$ with $\mathcal{H}_{T}=\mathcal{H}$ has $\mathrm{SOP}_{3}$. But it is \textit{not} the case that, for $\mathcal{H}$  a hereditary class defined by a finite family of forbidden weak substructures, if every theory $T$ with $\mathcal{H}_{T}=\mathcal{H}$ is non-simple, every theory $T$ with $\mathcal{H}_{T}=\mathcal{H}$ has $\mathrm{SOP}_{3}$.

\:

Theorem \ref{main theorem 3} will in some sense be hiding in plain sight. Suppose one were attempting to find a counterexample to $\mathrm{NSOP}_{2}$ and $\mathrm{NSOP}_{3}$ being equivalent. One natural place to look would be within the family of theories of generic structures defined by finitely many forbidden weak substructures. (These are the structures defined by Cherlin, Shelah and Shi in \cite{CSS99}). If one found such a counterexample, this would resolve the open problem of whether $\mathrm{NSOP}_{2}$ is equal to $\mathrm{NSOP}_{3}$. However, if one showed that none of these theories of generic structures are counterexamples, that would \textit{at first glance} appear uninteresting, because one would only have proven a concrete verification for a specific family of theories.

This all changed with very recent work of Bodirsky, Bodor and Marimon in \cite{BBM25}. There, they show that the properties $\mathrm{NSOP}_{n}$ are preserved under taking model companions. We will observe that preservation of $\mathrm{NSOP}_{n}$ under model companions has the following implicit consequence: Theorem \ref{main theorem 3} reduces to verifying that none of the above theories of generic structures are strictly $\mathrm{NSOP}_{3}$. (We will explain how our observation fits into a general schema of Bodirsky, Bodor and Marimon in Remark \ref{Bodirsky, Bodor and Marimon schema} below.) So what initially superficially appeared to be a mere concrete verification is just the disguise for a more general dichotomy between $\mathrm{NSOP}_{2}$ and $\mathrm{SOP}_{3}$.

Then, we must show that none of the theories of generic structures defined by finitely many forbidden weak substructures are strictly $\mathrm{NSOP}_{3}$. The main idea of our proof will be the same used to prove integrality of $\mathfrak{o}(\mathcal{H})$ for $\mathcal{H}$ defined by finitely many forbidden weak substructures, within the setting of the real-valued $\mathrm{NSOP}_{r}$ hierarchy. As implicitly shown in \cite{Mal10b}, the properties $\mathrm{NSOP}_{n}$ can be characterized in terms of helix-shaped covering maps, or helix maps; our integrality argument within the real-valued $\mathrm{NSOP}_{r}$ hierarchy involved repeatedly applying the abstract cycle-removal properties of these maps. This cycle-removal technique is especially suited to the real-valued $\mathrm{NSOP}_{r}$ setting for two reasons. First, it allows us to prove entirely combinatorial results about hereditary classes. And second, in the real-valued $\mathrm{NSOP}_{r}$ setting, the cycle-removal argument is performed with respect to directed cycles, which are especially well-behaved within the category of directed graphs with graph homomorphisms. We will apply this cycle-removal technique in the proof of Theorem \ref{main theorem 3}, giving a much more intricate version of the same argument.

\:

The organization of this paper is as follows. Section \ref{preliminaries} will provide basic definitions and background, including definitions from classification theory ($\mathrm{NSOP}_{1}$ and $\mathrm{NSOP}_{2}$, Definition \ref{nsop1 definition and nsop2 definition}; and $\mathrm{NTP}_{2}$ and simplicity, Definition \ref{ntp_2 definition} and \ref{simplicity definition}) and hereditary class theory (being defined by a finite family of forbidden weakly embedded substructures, Definition \ref{defined by finitely many forbidden weakly embedded substructures}). In Section \ref{between nsop2 and nsop3}, we will show that the properties $\mathrm{NSOP}_{r}$ for real values $r \geq 3$ extend to the case $2 < r < 3$, proving that the ascending chain is preserved with $\mathrm{NSOP}_{2} \subseteq \mathrm{NSOP}_{r}$ (Theorem \ref{main theorem 1, restated}). This will give us our intermediate family $\mathrm{NSOP}_{2} \subseteq \mathrm{NSOP}_{r} \subseteq \mathrm{NSOP}_{3}$ for $2 < r <3$, connecting the real-valued $\mathrm{NSOP}_{r}$ distinctness problem to the problem of whether $\mathrm{NSOP}_{2}$ is equal to $\mathrm{NSOP}_{3}$. In Section \ref{the approximate alternative}, we will introduce and motivate the hypothesis of the real-valued $\mathrm{NSOP}_{r}$ hierarchy and integer-valued $\mathrm{NSOP}_{n}$ hierarchy being non-distinct on sufficiently general grounds (Hypothesis \ref{nondistinctness on sufficiently general grounds}). We will then show that this hypothesis implies that $\mathrm{NSOP}_{n+1} \cap \mathrm{NTP}_{2}$ is equal to $\mathrm{NSOP}_{n} \cap \mathrm{NTP}_{2}$ for integers $n \geq 3$ (Theorem \ref{main theorem 2, restated}), which would resolve Question \ref{ntp_2 vs nsop_n}. Finally, in the longest section of this paper, Section \ref{nsop2 vs nsop3}, we will apply techniques from the real-valued $\mathrm{NSOP}_{r}$ hierarchy to prove our sharp dichotomy between $\mathrm{SOP}_{2}$ and $\mathrm{NSOP}_{3}$, under elementary assumptions from hereditary class theory. Before proving the main dichotomy (Theorem \ref{main theorem 3, restated}), we will use an example from Kruckman and Ramsey (\cite{KR17}) to demonstrate that this dichotomy must be sharp (Observation \ref{main observation, restated}). To begin the argument for the main dichotomy, we will review the constructions from Cherlin, Shelah and Shi (\cite{CSS99}) of theories of generic structures defined by finitely many forbidden weak substructures (Fact \ref{Cherlin Shelah and Shi theorem}), as well as Bodirsky, Bodor and Marimon's results (\cite{BBM25}) on preservation of simplicity or $\mathrm{NSOP}_{n}$ under model companions. These will allow us to reduce the main results of Section \ref{nsop2 vs nsop3} to a verification for a concrete family of theories (Proposition \ref{reduction to generic structures}). We will then enter into the intricacies of the proof of this verification: we will first translate it into more tractable combinatorial language (Lemma \ref{standard hereditary class}). Then, applying a carefully chosen series of variations of the cycle-removal argument from the theory of the real-valued $\mathrm{NSOP}_{r}$ hierarchy, we will complete the proof of the main dichotomy.

In optional sidebars to Sections \ref{the approximate alternative} and \ref{nsop2 vs nsop3}, we will give further applications of independent interest.  In Sidebar 1 to Section \ref{the approximate alternative}, we will relativize the main construction of that section, based on the configuration exhibiting $\mathrm{TP}_{2}$ within the generic triangle-free directed graph. This will allow us to generalize the cycle-removal argument arising from the real-valued $\mathrm{NSOP}_{r}$ hierarchy to any $\mathrm{NTP}_{2}$ theory. As a consequence, we will learn something new about $\mathrm{NTP}_{2}$ theories themselves: any directed graph definable in an $\mathrm{NTP}_{2}$ theory omitting weakly embedded copies of \textit{some} graph must, for arbitrarily large $N$, omit weakly embedded copies of some graph without $\leq N$-cycles (Theorem \ref{ntp2 graph theory}). In Sidebar 2 to Section \ref{nsop2 vs nsop3}, we will extract general information from part of the proof (or a potential simplified proof, see Remark \ref{potential simplified proof}) of Theorem \ref{main theorem 3, restated}, our new hereditary-class dichotomy between $\mathrm{NSOP}_{2}$ and $\mathrm{SOP}_{3}$. We will introduce a sense in which a classification-theoretic property $\mathrm{NP}$, the negation of the configuration $\mathrm{P}$, can \textit{approximately} imply the classification-theoretic property $\mathrm{NQ}$, the negation of $\mathrm{Q}$. Roughly, $\mathrm{NP} \leadsto \mathrm{NQ}$ if any formula $\varphi(x, y)$ in any theory with $\mathrm{NP}$ satisfies a constraint that looks arbitrarily close up like an obstruction to $\varphi(x, y)$ exhibiting $\mathrm{Q}$. Our arguments will result in the approximate implications $\mathrm{NSOP}_{3} \leadsto \mathrm{NSOP}_{2}$ and $\mathrm{NSOP}_{3} \leadsto \mathrm{NTP}_{2}$ (Theorem \ref{the approximate implications}); in particular, we believe the approximate implication $\mathrm{NSOP}_{1} \leadsto \mathrm{NTP}_{2}$ to be a new, nontrivial result on $\mathrm{NSOP}_{1}$ theories.

Finally, in an appendix, we will give a note demonstrating the historical precedent for interest in extending the integer-valued $\mathrm{NSOP}_{n}$ hierarchy to non-integer values of $n$. These come from the quasimetric properties introduced in a thesis appendix of Hanson (\cite{Hanson2020}), which Hanson describes as ``roughly equivalent to an instance of `$\mathrm{SOP}_{\frac{1}{r}}$'" for real $r$; however, in classical logic (as opposed to continuous logic), these are known not to be new properties. We will give an exposition of Hanson's proof that these coincide with the original integer-valued $\mathrm{NSOP}_{n}$ hierarchy, adapted to classical logic.

\section{Preliminaries}\label{preliminaries}

We use standard model-theoretic notation: all sets are assumed to be small subsets of a sufficiently saturated ambient model $\mathbb{M} \models T$ of a theory $T$, which will be left implicit when clear from context, and all models small elementary submodels $M \prec \mathbb{M}$. Unless otherwise stated, by a \textit{theory} we will mean a complete theory. Later in this paper, we will refer to \textit{coheirs} or \textit{finitely satisfiable types} over a model $M$, which will be a special class of \textit{invariant types} over a model $M$: a coheir over $M$ will be a complete global type over $M$ every formula of which is satisfied in $M$, while an invariant type over $M$ will be a complete global type $p(x)$ over $M$ for which $\varphi(x, b) \in p(x)$ $b\equiv_{M} b'$ implies $\varphi(x, b') \in p(x)$. (A non-global type over a set containing $M$ will be said to be finitely satisfiable, or a coheir, under the same conditions. But unless otherwise clear from context, by a coheir we will mean a global coheir.) The most important feature of invariant types for our purposes will be that, if $\mathrm{tp}(a/Mb)$ extends to an invariant type, $\mathrm{acl}(aM) \cap \mathrm{acl}(Mb) = M$. Given a coheir $p(x)$ over $M$, a  \textit{coheir Morley sequence over $M$} in $p(x)$ will be a sequence $\{a_{i}\}_{i \in I}$, say for $I$ a linearly ordered set with no upper bound (such as $I =\omega$), such that for $i \in I $, $a_{i} \models p(x)|_{M \{a_{j}\}_{j < i}}$. An \textit{invariant Morley sequence} will be similar, for an invariant type $p(x)$. For  $\{a_{i}\}_{i \in I}$ a coheir Morley sequence over $M$ and sets $S_{1}, S_{2} \subset I$ with every element of $S_{1}$ below every element of $S_{2}$, $\mathrm{tp}(\{a_{i}\}_{i \in S_{1}}/M\{a_{i}\}_{i \in S_{2}})$ will extend to a coheir over $M$. We will also use the following terminology for binary relations: a binary relation $R$ on a set $S$ will be \textit{antisymmetric} if there are no distinct $a, b \in S$ with $R(a, b)$ and $R(b,a)$, and \textit{irreflexive} if there is no $a \in S$ with $R(a, a)$.

We introduced Shelah's original properties $\mathrm{NSOP}_{n}$ for integers $n \geq 3$, as well as our newer properties $\mathrm{NSOP}_{r}$ for reals $r \geq 3$, in the introduction (Definitions \ref{integer-valued hierarchy} and \ref{real-valued hierarchy}). In fact, most of the main classification-theoretic properties will have reason to appear throughout this paper in at least some fashion. Because our main applications of the real-valued $\mathrm{NSOP}_{r}$ hierarchy will be to Questions \ref{is nsop2 equal to nsop3} and \ref{ntp_2 vs nsop_n} from the introduction--i.e., to whether $\mathrm{NSOP}_{2}$ is equal to $\mathrm{NSOP}_{3}$ and to whether $\mathrm{NSOP}_{n+1} \cap \mathrm{NTP}_{2} $ is equal to $\mathrm{NSOP}_{n} \cap \mathrm{NTP}_{2} $ for integers $n \geq 3$--we draw special attention to the properties $\mathrm{NSOP}_{1}, \mathrm{NSOP}_{2}$ and $\mathrm{NTP}_{2}$, defined below.

\textbf{The $\mathrm{NSOP}_{n}$ hierarchy, $n \leq 2$:} The properties $\mathrm{NSOP}_{n}$ for $n > 2$ are defined according to a common schema. By contrast, $\mathrm{NSOP}_{1}$ and $\mathrm{NSOP}_{2}$ are defined differently:

\begin{definition}\label{nsop1 definition and nsop2 definition}(Džamonja and Shelah, (\cite{Sh99},  \cite{DS04}))

\begin{itemize}
    \item A theory is $\mathrm{NSOP}_{1}$ if there is no formula $\varphi(x, y)$ and tuples $\{b_{\eta}\}_{\eta \in 2^{<\omega}}$ such that $\{\varphi(x, b_{\sigma \upharpoonright n})\}_{n < \omega}$ is consistent for any $\sigma \in 2^{\omega}$, but for any $\eta_{2} \unrhd \eta_{1} \smallfrown \langle 0\rangle$, $\{\varphi(x, b_{\eta_{2}}), \varphi(x, b_{\eta_{1} \smallfrown \langle 1\rangle})\}$ is inconsistent. Otherwise it has $\mathrm{SOP}_{1}$.

    \item A theory is $\mathrm{NSOP}_{2}$ if there is no formula $\varphi(x, y)$ and tuples $\{b_{\eta}\}_{\eta \in 2^{<\omega}}$ so that $\{\varphi(x, b_{\sigma \upharpoonright n})\}_{n < \omega}$ is consistent for any $\sigma \in 2^{\omega}$, but for incomparable $\eta_{1}$ and $\eta_{2}$, $\{\varphi(x, b_{\eta_{1}}), \varphi(x, b_{\eta_{2}})\}$ is inconsistent. Otherwise it has $\mathrm{SOP}_{2}$.
\end{itemize}

\end{definition}

We will introduce some of the more powerful results on $\mathrm{NSOP}_{1}$ and $\mathrm{NSOP}_{2}$, including on Kim-independence (Fact \ref{symmetry and independence theorem}) and the equality of $\mathrm{NSOP}_{1}$ and $\mathrm{NSOP}_{2}$ (Fact \ref{Application of nsop1 = nsop2}) later in this paper. One of the more basic building blocks of the theory of Kim-independence in $\mathrm{NSOP}_{1}$ theories is a characterization of $\mathrm{NSOP}_{1}$ theories based on Kim's lemma (\cite{KR17}). This gives us the following presentation of $\mathrm{SOP}_{1}$ in terms of coheir Morley sequences, which will follow from the proof of Proposition 3.14 of \cite{KR17}. 

\begin{fact}\label{sop1 coheirs}(\cite{KR17})

Let $T$ have $\mathrm{SOP}_{1}$. Then there is a formula $\varphi(x, b)$, and coheir Morley sequences $\{b_{i}\}_{i< \omega}$, $\{b^{i}\}_{i< \omega}$ over a model $M$ 
with $b_{0} = b^{0}= b$, such that $\{\varphi(x, b_i)\}_{i < \omega}$ is consistent, but $\{\varphi(x, b^i)\}_{i < \omega}$ is $2$-inconsistent (i.e., any two distinct formulas in $\{\varphi(x, b^i)\}_{i < \omega}$ are inconsistent.)

\end{fact}

\textbf{Non-strict order classification-theoretic properties:}

The property $\mathrm{NTP}_{2}$, the negation of the \textit{$2$-tree property}, will be involved in the open problem on the intersections $\mathrm{NTP}_{2} \cap \mathrm{NSOP}_{n}$ for integers $n \geq 3$ that we will discuss in Section \ref{the approximate alternative} below.

\begin{definition}\label{ntp_2 definition}(Shelah, \cite{Sh90})

A theory is $\mathrm{NTP}_{2}$ if there is no formula $\varphi(x, y)$ and $\{b_{ij}\}_{i, j < \omega}$ such that for all $i< \omega$ $\{\varphi(x, b_{ij}))\}_{j < \omega}$ is $2$-inconsistent, but for $\sigma: \omega \to \omega$, $\{\varphi(x, b_{i\sigma(i)})\}_{i < \omega}$ is consistent.
    
\end{definition}

We will also refer to \textit{simplicity}, the negation of the \textit{tree property} $\mathrm{TP}$. The tree property will play an important formal role in the proof of our dichotomy between $\mathrm{NSOP}_{2}$ and $\mathrm{SOP}_{3}$ in Theorem \ref{main theorem 3, restated} below. Additionally, to justify this dichotomy, we will have to ensure that it is not really a dichotomy between simplicity and $\mathrm{SOP}_{3}$ (Observation \ref{main observation, restated}).

\begin{definition}\label{simplicity definition} (Shelah, \cite{Sh90})

A theory is \textit{simple} if there is no formula $\varphi(x, y)$ and $\{b_{\eta}\}_{\eta \in \omega^{< \omega}}$ such that for each $\sigma \in\omega^{\omega}$, the set of formulas $\{\varphi(x,b_{\sigma|_n})\}_{n < \omega}$ is consistent, but for each $\eta\in\omega^{<\omega}$, $\{\varphi(x,b_{\eta^\smallfrown \langle i \rangle})\}_{i < \omega}$ is $2$-inconsistent. A non-simple theory is said to have the \textit{tree property}, or $\mathrm{TP}$.
    
\end{definition}

In passing, we will also mention the most basic of the classification-theoretic properties, \textit{stability}. (This will essentially be the degenerate case of the $\mathrm{NSOP}_{n}$ hierarchy, as in Remark \ref{lower limit} below.) A theory is \textit{stable} if there is no formula $\varphi(x, y)$ and sequence $\{a_{i}\}_{i < \omega}$ such that $\models \varphi(a_{i}, a_{j})$ if and only if $i < j$.

\:

\textbf{The $r$-strict order property for real-valued $r$:} We defined the properties $\mathrm{NSOP}_{r}$ for real values $r \geq 3$ in the introduction, in Definition \ref{real-valued hierarchy}. A more finitary statement of this definition may sometimes be useful. We restate the more detailed definition of $\mathrm{NSOP}_{r}$ from \cite{ApproxOrder}, Definition 2.6 there, below:

\begin{definition}\label{real-valued hierarchy, detailed definition}
    Let $r\geq 3$ be a real number. A theory $T$ is $\mathrm{NSOP}_{r}$ if the following equivalent conditions hold: 

    (1) For every definable relation $R(x_{1}, x_{2})$ with a sequence $\{a_{i}\}_{i < \omega}$ such that $\models R(a_{i}, a_{j})$ for all $i < j$, for all integers $N < \omega$ there are $b_{0}, \ldots, b_{N -1} $ such that $\models R(b_{i}, b_{i+j \mathrm{\: mod \:} N})$ for $0 \leq i <N$, $1 \leq j \leq \frac{N}{r}$.

    (1$'$) For every definable relation $R(x_{1}, x_{2})$ with a sequence $\{a_{i}\}_{i < \omega}$ such that $\models R(a_{i}, a_{j})$ for all $i < j$, for all \textit{sufficiently large} integers $N < \omega$ there are $b_{0}, \ldots, b_{N -1} $ such that $\models R(b_{i}, b_{i+j \mathrm{\: mod \:} N})$ for $0 \leq i <N$, $1 \leq j \leq \frac{N}{r}$.

    (2) For every definable relation $R(x_{1}, x_{2})$ with a sequence $\{a_{i}\}_{i < \omega}$ such that $\models R(a_{i}, a_{j})$ for all $i < j$, there exists a set $\{a_{\theta}\}_{\theta \in S^{1}}$, indexed by the unit circle $S^{1}$, such that $\models R(a_{\theta}, a_{\psi})$ for all $\theta, \psi \in S^{1}$ with $\psi$ lying at most $\frac{2\pi}{r}$ radians counterclockwise from $\theta$.

 Otherwise it has $\mathrm{SOP}_{r}$.
\end{definition}

Relatedly, we will sometimes use the following informal terminology: for a binary relation $R$ on a set $S$ (such as a definable relation between tuples in a model, or the edge relation of a graph), a sequence $\{a_{i}\}_{i < \omega}$ with $a_{i} \in S$ such that $ R(a_{i}, a_{j})$ for $i < j$ will be said to be an \textit{infinite chain} for $R$.

\textbf{Hereditary class theory:} Recall that a \textit{hereditary class} of structures in a finite relational language $\mathcal{L}$ is a class of finite $\mathcal{L}$-structures $\mathcal{H}$ that is closed under induced substructures: if $A \in \mathcal{H}$, and $B \subset A$ is an induced substructure of $\mathcal{H}$, then $B \in \mathcal{H}$. When discussing hereditary classes, we apply standard model-theoretic notation to $\mathcal{L}$-structures and their subsets: see the discussion after Definition 2.10 in \cite{ApproxOrder} for a more explicit statement of what this means. Most importantly, we use the following notation when referring to abstract $\mathcal{L}$-structures (as opposed to subsets of an ambient model of a theory): $A \equiv_{C} B$, where $AC$, $BC$ are enumerated $\mathcal{L}$-structures with $C$ as a common induced substructure, means that the ordered tuples $AC$ and $BC$ have the same quantifier-free $\mathcal{L}$-type.

As in \cite{ApproxOrder} we use the following abuse of notation: if $A$ is an \textit{infinite} $\mathcal{L}$-structure, $A \in \mathcal{H}$ if $A_{0} \in \mathcal{H}$ for each finite $A_{0} \subset A$.

Our key definition from hereditary class theory, particularly within Theorem \ref{main theorem 3, restated} and Observation \ref{main observation, restated}, will be that of a hereditary class \textit{defined by a finite family of forbidden weakly embedded substructures}.

\begin{definition}\label{defined by finitely many forbidden weakly embedded substructures}

   Let $\mathcal{L}$ be a finite relational language.

\begin{itemize}
    \item A one-to-one map $\iota: A \hookrightarrow B$ between $\mathcal{L}$-structures is a \emph{weak embedding} if, for $R$ any relation symbol of $\mathcal{L}$,

    $$(a_{1}, \ldots, a_{n}) \in R(A) \Rightarrow (\iota(a_{1}), \ldots, \iota(a_{n})) \in R(B)$$.

    \item For $\mathcal{F}$ a family of finite $\mathcal{L}$-structures, $\mathcal{H}(\mathcal{F})$ is the hereditary class of $\mathcal{L}$-structures $A$ such that there is no $B \in \mathcal{F}$ with a weak embedding $\iota: B \hookrightarrow A$.

    \item A hereditary class $\mathcal{H}$ of $\mathcal{L}$-structures is \textit{defined by a finite family of forbidden weakly embedded substructures}\footnote{This phrasing is based on the phrasing from \cite{ApproxOrder}; in more informal contexts, we may sometimes using the phrasing \textit{defined by finitely many forbidden weakly embedded substructures} or \textit{defined by finitely many forbidden weak substructures}.} if $\mathcal{H} = \mathcal{H}(\mathcal{F})$ for some finite family $\mathcal{F}$ of finite $\mathcal{L}$-structures.
\end{itemize}
    
\end{definition}

\section{Between $\mathrm{NSOP}_{2}$ and $\mathrm{NSOP}_{3}$}\label{between nsop2 and nsop3}

As described in the introduction, in extending the original family of properties $\mathrm{NSOP}_{n}$ for positive integer values of $n$ to the family $\mathrm{NSOP}_{r}$ for real values of $r$, we imposed the constraint that $r \geq 3$. However, there is strong motivation to extend this family to also include properties  $\mathrm{NSOP}_{r}$ for real values $r$ with $2 < r < 3$. This motivation comes from the question of whether $\mathrm{NSOP}_{2}$ is equal to $\mathrm{NSOP}_{3}$. 

Our proposed extension will be the most parsimonious one possible, relative to our existing $\mathrm{NSOP}_{r}$ hierarchy for real values $r \geq 3$. We simply copy our original definition from \cite{ApproxOrder}, Definition \ref{real-valued hierarchy} above, but replace the condition $r \geq 3$ on the real value $r$ with the condition $r > 2$.

\begin{definition}\label{real-valued hierarchy, r > 2}

\textit{At this point, this is a tentative definition, but it will be fully justified in the main result of this section.}

Let $r > 2$ be a real number. A theory $T$ is $\mathrm{NSOP}_{r}$ if there is no definable relation $R(x, y)$ such that:

\begin{itemize}
    \item there exists a sequence $\{a_{i}\}_{i < \omega}$ such that $\models R(a_{i}, a_{j})$ for $i <j$, but

    \item there does not exist a set $\{a_{\theta}\}_{\theta \in S^{1}}$, indexed by the unit circle $S^{1}$, such that $\models R(a_{\theta}, a_{\psi})$ for all $\theta, \psi \in S^{1}$ with $\psi$ lying at most $\frac{2\pi}{r}$ radians counterclockwise from $\theta$.
\end{itemize}

 Otherwise it has $\mathrm{SOP}_{r}$.
    
\end{definition}

\begin{remark}\label{lower limit}
Before giving nontrivial justification for lowering our inclusive lower limit of $3$, for $r$ in the above definition, to a strict lower limit of $2$, we quickly justify terminating our investigation at this lower limit, rather than looking even lower.\footnote{This is all essentially an observation of Nick Ramsey, who pointed out the cases $r=1$ and $r = 2$ in a personal communication with the author.} 

The above condition, where $1 < r \leq 2$, is equivalent to stability. If $T$ is stable, and there exists a sequence $\{a_{i}\}_{i < \omega}$ such that $\models R(a_{i}, a_{j})$ for $i <j$, then we may assume $\{a_{i}\}_{i < \omega}$ is indiscernible. So by stability, $\models R(a_{i}, a_{j})$ for $i \neq j$. But we can then easily find $\{a_{\theta}\}_{\theta \in S^{1}}$ as in the condition. Conversely, suppose $T$ satisfies the condition for $1 < r \leq 2$, but $T$ is unstable. Then there exists $\{a_{i}\}_{i < \omega}$ such that $\models R(a_{i}, a_{j})$ for $i <j$, but $\models \neg R(a_{i}, a_{j})$ for $j < i$. Define $R'(x, y) : = R(x, y) \wedge \neg R(y, x)$. Then $\models R'(a_{i}, a_{j})$ for $i <j$. So there exists $\{a_{\theta}\}_{\theta \in S^{1}}$ for $R'$. But then, because $r \leq 2$, $\models R'(a_{0\mathrm{\: rad}}, a_{\pi\mathrm{\: rad}}) \wedge R'(a_{\pi\mathrm{\: rad}}, a_{0\mathrm{\: rad}})$. This contradicts the definition of $R'(x, y)$.

The above condition, where $0 < r \leq 1$, is equivalent to every model of $T$ being finite. If every model of $T$ is finite, then $T$ is stable, and the proof of the above condition is as in the previous paragraph. Conversely, suppose the above condition is satisfied, but $T$ has an infinite model. Then for $R(x, y):= x\neq y $, there is  $\{a_{i}\}_{i < \omega}$ such that $\models R(a_{i}, a_{j})$. Then $\{a_{\theta}\}_{\theta \in S^{1}}$ exists for $R$. But then, because $r \leq 1$, $\models R(a_{0 \mathrm{\: rad}}, a_{0 \mathrm{\: rad}})$. This contradicts the definition of $R(x, y)$.

\end{remark}

Following Remark 2.8 of \cite{ApproxOrder}, we say why our extension of the definition of $\mathrm{NSOP}_{r}$ to real values $r > 2$ requires further justification. One of the salient features of the original $\mathrm{NSOP}_{n}$ hierarchy, for integers $n \geq 1$, is that it forms an ascending chain:

$$\mathrm{NSOP}_{1} = \mathrm{NSOP}_{2} \subseteq \mathrm{NSOP}_{3} \subsetneq \ldots \subsetneq \mathrm{NSOP}_{n} \subsetneq \mathrm{NSOP}_{n+1} \subsetneq \ldots$$

\:

In extending this hierarchy to obtain properties $\mathrm{NSOP}_{r}$ for non-integer values of $r$, we must preserve this condition, and introducing the properties $\mathrm{NSOP}_{r}$ for $r \geq 3$ poses no problems preserving the condition. As stated in Remark 2.8 of \cite{ApproxOrder}, for real values $r, s$ with $3 \leq r \leq s$, it is immediate from the definitions that $\mathrm{NSOP}_{r} \subseteq \mathrm{NSOP}_{s}$. Moreover, Shelah and Usvyatsov observe the inclusion $\mathrm{NSOP}_{2} \subseteq \mathrm{NSOP}_{3}$ in \cite{SU08}. (This work was referred to earlier in \cite{Sh99} and follows from reformulations of the definition of $\mathrm{NSOP}_{3}$ in terms of consistency and inconsistency, as found in, say, Fact 1.3 of \cite{SU08}.)

The test for whether our tentative definition of $\mathrm{NSOP}_{r}$ for real values $r > 2$ as in Definition \ref{real-valued hierarchy, r > 2} can be justified as an actual definition will then be: does this still preserve the condition that the properties form an ascending chain? Since the inclusion $\mathrm{NSOP}_{r} \subseteq \mathrm{NSOP}_{s}$ for real values $r, s$ with $3 \leq r \leq s$ was just an immediate consequence of the definitions, there will be no problem extending this inclusion to real values $r, s$ with $2 < r \leq s$. So there remains one potential obstruction to showing the properties remain an ascending chain, and thus justifying the introduction of $\mathrm{NSOP}_{r}$ for $2 < r < 3$: whether $\mathrm{NSOP}_{2} \subseteq \mathrm{NSOP}_{r}$ for $2 < r < 3$.

Our main result will be to prove this inclusion, thereby turning the tentative definition of $\mathrm{NSOP}_{r}$ for reals $r > 2$ in Definition \ref{real-valued hierarchy, r > 2} into an actual definition.

To our knowledge, proving that $\mathrm{NSOP}_{2} \subseteq \mathrm{NSOP}_{r}$ for $2 < r < 3$ will require concepts from the theory of Kim-independence in $\mathrm{NSOP}_{1}$ theories, introduced in \cite{KR17}. We review the definition of Kim-independence, as well as the main theorems from \cite{KR17} on its behavior in $\mathrm{NSOP}_{1}$ theories.

\begin{definition}\label{kim-independence}(\cite{KR17})
A formula $\varphi(x, b)$ \emph{Kim-divides} over a model $M$ if there is an invariant Morley sequence $\{b_{i}\}_{i < \omega}$ over $M$ starting with $b$ (said to \emph{witness} the Kim-dividing) such that $\{\varphi(x, b_{i})\}_{i < \omega}$ is inconsistent. A formula  $\varphi(x, b)$ \emph{Kim-forks} over $M$ if it implies a (finite) disjunction of formulas Kim-dividing over $M$. We write $a \ind^{K}_{M} b$, and say that $a$ is \emph{Kim-independent} from $b$ over $M$, if $\mathrm{tp}(a/Mb)$ does not include any formulas Kim-forking over $M$.
\end{definition}

So if $\mathrm{tp}(a/Mb)$ is finitely satisfiable, then $a \ind^{K}_{M} b$. We next state the results on symmetry and the independence theorem for $\ind^{K}$ in $\mathrm{NSOP}_{1}$ theories:

\begin{fact}\label{symmetry and independence theorem}(\cite{KR17})

    Let $T$ be $\mathrm{NSOP}_{1}$. Then:

    \begin{itemize}
        \item Kim-independence is symmetric: if $a \ind^{K}_{M} b$, then $b \ind^{K}_{M} a$.

        \item Kim-independence satisfies the independence theorem: if $a_{1} \ind^{K}_{M} b_{1}$, $a_{2} \ind^{K}_{M} b_{2}$, $b_{1} \ind^{K}_{M} b_{2}$, and $a_{1} \equiv_{M} a_{2}$, there is some $a \ind^{K}_{M} b_{1}b_{2}$ with $a \equiv_{Mb_{i}}a_{i}$ for $i=1, 2$.
    \end{itemize}

\end{fact}

We are now ready to prove our main theorem justifying the definition of $\mathrm{NSOP}_{r}$ for $2 < r < 3$. (A related argument is found in Appendix C of \cite{Patterning}, in joint work of the author with Gabriel Day.)

\begin{theorem}\label{main theorem 1, restated}
    Let $\mathrm{NSOP}_{r}$ for real values $r > 2$ be defined as in Definition \ref{real-valued hierarchy, r > 2}. Then for real values $r > 2$, $\mathrm{NSOP}_{2} \subseteq \mathrm{NSOP}_{r}$.
\end{theorem}

\begin{proof}
    Suppose that $r > 2$ is a real value, and that $T$ is $\mathrm{NSOP}_{2}$. We show that $T$ is $\mathrm{NSOP}_{r}$.

    The following is proven in \cite{NSOP2}:

    \begin{fact}\label{Application of nsop1 = nsop2}
        $\mathrm{NSOP}_{1}=\mathrm{NSOP}_{2}.$
    \end{fact}

    So $T$ is $\mathrm{NSOP}_{1}$.

    We now show that $T$ is $\mathrm{NSOP}_{r}$. Let $R(x, y)$ be a definable relation with a sequence $\{a_{i}\}_{i < \omega}$ such that $\models R(a_{i}, a_{j})$ for $i <j$. It suffices to show that there is a set $\{a_{\theta}\}_{\theta \in S^{1}}$, indexed by the unit circle $S^{1}$, such that $\models R(a_{\theta}, a_{\psi})$ for all $\theta, \psi \in S^{1}$ with $\psi$ lying at most $\frac{2\pi}{r}$ radians counterclockwise from $\theta$, because then we will know that $R(x, y)$ cannot exhibit $\mathrm{SOP}_{r}$. For this, it suffices to show that the relation defined by $R(x, y)$ on $\mathbb{M}^{|x|}$ weakly embeds every finite directed graph, and therefore, by compactness, every infinite directed graph. (Then the directed graph consisting of distinct vertices $\{v_{\theta}\}_{\theta \in S^{1}}$, indexed by the unit circle $S^{1}$, whose edges consist of $v_{\theta}Ev_{\psi}$ for all $\theta, \psi \in S^{1}$ with $\psi$ lying at most $\frac{2\pi}{r}$ radians counterclockwise from $\theta$, will weakly embed in the relation defined by $R(x, y)$, as desired. Note that $r > 2$ is required for $\{v_{\theta}\}_{\theta \in S^{1}}$ to be a directed graph: if $0 < r \leq 2 $, there will be vertices $v, w$ with $vEw$ and $w E v$, so this will not be a directed graph.)

    We may assume $\{a_{i}\}_{i < \omega}$ to be indiscernible. Using the standard argument in, say, \cite{KR17}, we observe that we can assume $\{a_{i}\}_{i < \omega}$ to be a finitely satisfiable Morley sequence over a model $M$.  We may extend $\{a_{i}\}_{i < \omega}$ to an indiscernible sequence $\{a_{i}\}_{i \in \mathbb{Q}}$, and even assume $\{a_{i}\}_{i \in \mathbb{Q}}$ is indiscernible in a Skolemization $T^{\mathrm{Sk}}$. Let $M$ be the Skolem hull of $\{a_{i}\}_{i< 0, i \in \mathbb{Q}}$. Then in the original theory $T$, $\{a_{i}\}_{i < \omega}$ (i.e., where $i$ ranges over the positive integers as a subset of $\mathbb{Q}$) will be a finitely satisfiable Morley sequence over $M \prec \mathbb{M} \models T$ such that $\models R(a_{i}, a_{j})$ for $i <j$, as desired.

    We can conclude that $\mathrm{tp}(a_{0}/Ma_{1})$ is $M$-finitely satisfiable. So $a_0 \ind_{M}^{K} a_{1}$. By the symmetry in Fact \ref{symmetry and independence theorem}, $a_1 \ind_{M}^{K} a_{0}$.\footnote{In fact, invoking symmetry of $\ind^{K}$ is only a convenience. We could even have assumed, by a standard argument, that $\{a_{i}\}_{i < \omega}$ extends to a coheir Morley sequence  $\{a_{i}\}_{i \in \mathbb{Z}}$ over $M$ \textit{when read in either direction}, giving $a_0 \ind_{M}^{K} a_{1}$ and $a_1 \ind_{M}^{K} a_{0}$ without using symmetry of $\ind^{K}$.} So there are $a, b $ with $a  \equiv_{M} b$, $\models R(a, b)$, $a \ind_{M}^{K} b$, $b \ind_{M}^{K} a$. Let $p(x) =: \mathrm{tp}(a/M) = \mathrm{tp}(b/M) $, $p_{\rightarrow}(x, y) = \mathrm{tp}(ab/M) $, $p_{\leftarrow}(x, y) = \mathrm{tp}(ba/M) $.

    Let $G = \{v_i\}^{N}_{i = 1}$ be a finite directed graph with the vertices $v_{i}$ distinct; we show that $G$ weakly embeds into the relation defined by $R(x, y)$. We may assume that $G = \{v_i\}^{N}_{i = 1}$  has an edge in one of the two possible directions between any two distinct vertices (i.e., is a tournament).  We show something even stronger: there are $\{c_{i}\}^{N}_{i = 1}$ such that, for each $i \in \{1 \ldots N \}$, $c_{i} \models p(x)$ and $c_{i} \ind_{M}^{K} c_{1} \ldots c_{i-1}$, and for distinct $i, j \in \{1, \ldots, N\}$, $v_{i} E v_{j} $ implies $\models R(c_{i}, c_{j})$. We prove by induction on $n \leq N$ that there are $\{c_{i}\}^{n}_{i = 1}$ such that, for each $i \in \{1 \ldots n \}$, $c_{i} \models p(x)$ and $c_{i} \ind_{M}^{K} c_{1} \ldots c_{i-1}$, and for distinct $i, j \in \{1, \ldots, n\}$, $v_{i} E v_{j} $ implies $\models R(c_{i}, c_{j})$; in the case $n = N$ we will be done. In the base case just find any $c_{1} \models p(x)$. For the inductive step, suppose that we have $\{c_{i}\}^{n}_{i = 1}$ as desired for $n < N$; we find $c_{n+1}$ such that $\{c_{i}\}^{n+1}_{i = 1}$ will be as desired.
    
    To do this, we perform another induction: for $k \in \{1, \ldots n\}$, we show that there is $b_{k} \models p(x)$ such that $b_{k} \ind_{M}^{K} c_{1}\ldots c_{k}$, and such that for each $i \in \{1, \ldots k\}$, $\models R(c_{i}, b_{k})$ if $v_{i} E v_{n+1}$ and $\models R(b_{k}, c_{i})$ if $v_{n+1} E v_{i}$. Then $c_{n+1} =: b_{n}$ will suffice for us. For the base case, just find $b_{1} \models p_{\leftarrow}(x, c_{1})$ (so $\models R(c_{1}, b_{1})$) if $v_{1}Ev_{n+1}$ or $b_{1} \models p_{\rightarrow}(x, c_{1})$ (so $\models R(b_{1}, c_{1})$) if $v_{n+1}Ev_{1}$; either way, since $b \ind^{K}_{M} a$ and $a \ind^{K}_{M} b$, $b_{1} \ind_{M}^{K} c_{1}$. Now suppose $k < n$, and we have found $b_{k}$; let us find $b_{k+1}$. We first have $b_{k} \ind_{M}^{K} c_{1}\ldots c_{k}$, and $c_{k+1} \ind_{M}^{K} c_{1} \ldots c_{k}$. Now find $b'_{k+1} \models p_{\leftarrow}(x, c_{k+1})$  (so $\models R(c_{k+1},  b'_{k+1})$) if $v_{k+1}Ev_{n+1}$ or $b'_{k+1} \models p_{\rightarrow}(x, c_{k+1})$ (so $\models R(  b'_{k+1}, c_{k+1})$) if $v_{n+1}Ev_{k+1}$. Then $b'_{k+1} \ind_{M}^{K}c_{k+1}$, and because $b'_{k+1}, b_{k} \models p(x)$, $b'_{k+1} \equiv_{M} b_{k}$. So applying the independence theorem in Fact \ref{symmetry and independence theorem}, we may find $b_{k+1}$ with $b_{k+1} \ind_{M}^{K} c_{1}\ldots c_{k+1}$, $b_{k+1} \equiv_{Mc_{1}\ldots c_{k}} b_{k}$, $b_{k+1} \equiv_{Mc_{k+1}} b'_{k+1}$. We have successfully chosen our desired $b_{k+1}$.

\end{proof}

This justifies extending the original definition of the properties $\mathrm{NSOP}_{r}$ for real values $r \geq 3$ to the additional real values $r$ with $2 < r < 3$. (So we can finally declare Definition \ref{real-valued hierarchy, r > 2} to be the actual definition of these properties $\mathrm{NSOP}_{r}$, and not just a tentative definition.)

We have demonstrated our first connection between the properties $\mathrm{NSOP}_{r}$ for non-integer values of $r$, and the open question of whether $\mathrm{NSOP}_{2}$ is equal to $\mathrm{NSOP}_{3}$ within the integer-valued $\mathrm{NSOP}_{n}$ hierarchy. Namely, we begin in \cite{ApproxOrder} by observing that $\mathrm{NSOP}_{r}$ is well-defined for real values $r \geq 3$. We have given a nontrivial argument that the properties $\mathrm{NSOP}_{r}$ for $2 < r < 3$ are also well-defined in the following rigorous sense: if we define $\mathrm{NSOP}_{r}$ for  $2 < r < 3$ just by copying the definition where $r \geq 3$, it is still the case that $\mathrm{NSOP}_{1}$, $\mathrm{NSOP}_{2}$, and $\mathrm{NSOP}_{r}$ for real values $r > 2$ form an ascending chain. And now, having obtained well-defined properties $\mathrm{NSOP}_{r}$ with $\mathrm{NSOP}_{2} \subseteq \mathrm{NSOP}_{r} \subseteq \mathrm{NSOP}_{3}$, if it is the case that every well-defined property $\mathrm{NSOP}_{r}$ for non-integer values of $r$ is a new property, distinct from all of the properties $\mathrm{NSOP}_{n}$ for integers $n \geq 1$, then the answer to the open question of whether $\mathrm{NSOP}_{2}$ is equal to $\mathrm{NSOP}_{3}$ is no.

\begin{remark}
    The properties $\mathrm{NSOP}_{1}$ and $\mathrm{NSOP}_{2}$ were defined by Džamonja and Shelah in \cite{DS04} (as referenced earlier in \cite{Sh99}), after Shelah defined the properties $\mathrm{NSOP}_{n}$ for integers $n \geq 3$ in \cite{She95}. However, the properties $\mathrm{NSOP}_{1}$ and $\mathrm{NSOP}_{2}$ are defined much differently from $\mathrm{NSOP}_{n}$ for $n \geq 3$, prompting the question of why they are called $\mathrm{NSOP}_{1}$ and $\mathrm{NSOP}_{2}$.\footnote{According to personal communications with Mirna Džamonja and Alexander Usvyatsov, these may have been called ``tree properties," rather than ``order properties," had knowledge of $\mathrm{TP}_{2}$ been more widespread at the time these properties were defined.} One plausible etymology comes from the fact, known at the time that $\mathrm{NSOP}_{1}$ and $\mathrm{NSOP}_{2}$ were originally defined, that $\mathrm{simple} \subsetneq \mathrm{NSOP}_{1} \subseteq \mathrm{NSOP}_{2} \subseteq \mathrm{NSOP}_{3}$ (as in \cite{DS04} and \cite{SU08}, referenced earlier in \cite{Sh99}). Thus $\mathrm{NSOP}_{1}$ and $\mathrm{NSOP}_{2}$ were known to extend the ascending chain $\mathrm{simple} \subsetneq \mathrm{NSOP}_{3} \subsetneq \ldots \subsetneq\mathrm{NSOP}_n \subsetneq  \mathrm{NSOP}_{n+1 }\subsetneq \ldots $ defined earlier, by adding exactly two additional classes right below $\mathrm{NSOP}_{3}$.

    The main result of this section gives a nontrivial additional justification for the name $\mathrm{NSOP}_{2}$, as derived from the names of the properties $\mathrm{NSOP}_{n}$ for integers $n \geq 3$. Such a justification will go: we now know that Shelah's original definition of the family $\mathrm{NSOP}_{n}$ for integers $n \geq 3$ can be restated so as to make sense for non-integer values of $n$, and even applying this restated definition to $n = r$ for $2 < r < 3$, we still have the chain of properties  $\mathrm{simple} \subsetneq \mathrm{NSOP}_{1} \subseteq \mathrm{NSOP}_{2} \subseteq \mathrm{NSOP}_{r}$. It is curious that this new, stronger justification is now known well after the property $\mathrm{NSOP}_{2}$ was originally defined.
\end{remark}

\section{An approximate alternative: $\mathrm{NSOP}_{\mathbb{R
}_{\geq 2}} \not\subset \mathrm{NSOP}_{\mathbb{N}}$ vs. $\mathrm{NSOP}_{n+1} \cap \mathrm{NTP}_{2} = \mathrm{NSOP}_{n} \cap \mathrm{NTP}_{2}$}\label{the approximate alternative}

In our previous work in \cite{ApproxOrder}, we concentrated on the question of \textit{whether we actually introduced new classification-theoretic properties} by extending the original integer-valued $\mathrm{NSOP}_{n}$ hierarchy to the real-valued $\mathrm{NSOP}_{r}$ hierarchy. We begin this section by motivating a general way in which this could fail.

We start with the question of whether the $\mathrm{NSOP}_{r}$ hierarchy is distinct from the original integer-valued $\mathrm{NSOP}_{n}$ hierarchy. This question is open not just at the level of full first-order logic, but even at the quantifier-free level. This is the same as asking whether $\mathfrak{o}(\mathcal{H})$\footnote{As in Definition 2.12 of \cite{ApproxOrder}, by the results of the previous section we can also define $\mathrm{NSOP}_{r}$ for hereditary classes for real values $r > 2$, not just real values $r > 3$. In the definition of the quantity $\mathfrak{o}(\mathcal{H})$ for a hereditary class $\mathcal{H}$, the supremum of those real values $r$ for which $\mathcal{H}$ has $\mathrm{SOP}_{r}$, we assumed $r \geq 3$. We could also define the supremum of those real values $r > 2$ for which $\mathcal{H}$ has $\mathrm{SOP}_{r}$, and introduce a new notation for that, say $\mathfrak{O}(\mathcal{H})$. Then the claim that the real-valued $\mathrm{NSOP}_{r}$ hierarchy defined with $r > 2$ is non-distinct from the original integer-valued $\mathrm{NSOP}_{n}$ hierarchy at the quantifier-free level is equivalent to integrality of $\mathfrak{O}(\mathcal{H})$. The proof of Theorem 3.1 of \cite{ApproxOrder} also shows that $\mathfrak{O}(\mathcal{H})$ is an integer when $\mathcal{H}$ is defined by finitely many forbidden weak substructures. On the other hand, to show that $\mathfrak{O}(\mathcal{H})$ can take non-integer values is potentially easier than to show that $\mathfrak{o}(\mathcal{H})$ can take non-integer values, though both problems remain open.} is an integer for every hereditary class $\mathcal{H}$ (Remark 2.13 of \cite{ApproxOrder}). When we speak of the quantifier-free version of $\mathrm{NSOP}_{r}$ for $r > 2$ real (including $\mathrm{NSOP}_{n}$ for $n$ an integer), we mean that no quantifier-free formula exhibits $\mathrm{SOP}_{r}$ (or $\mathrm{SOP}_{n}$). So the question at the quantifier-free level asks: is there some real value $r > 2$ such that, for every integer $n \geq 1$, the quantifier-free version of $\mathrm{NSOP}_{r}$ is not equivalent to the quantifier-free version of $\mathrm{NSOP}_{n}$?

Suppose that, for some fixed integer $n \geq 3$, there is no such $r$ between $n - 1$ and $n$. This means that for every $r$ with $n -1 < r < n$, the quantifier-free version of $\mathrm{NSOP}_{r}$ is equivalent to the quantifier-free version of $\mathrm{NSOP}_{n}$. An elementary way of rephrasing this hypothesis is as follows. For $a, b$ positive integers, define $\mathsf{G}_{a, b}$ to be the directed graph consisting of distinct $\{b_{i}\}^{a-1}_{i = 0}$ such that $b_{i} E b_{j}$ exactly when, for some $k \in \{1, \ldots, b\}$, $i \equiv j+k \mathrm{\:mod\:} a$. So as in \cite{ApproxOrder}, $\mathsf{G}_{a, b}$ consists of $b$ many vertices arranged in a circle, so that each vertex has an edge to the next $a$ many vertices counterclockwise. Then non-distinctness of the quantifier-free version of $\mathrm{NSOP}_{r}$ from the quantifier-free version of $\mathrm{NSOP}_{n}$ for every $r$ with $n -1 < r < n$ is equivalent to the following statement:

\textit{Restatement of non-distinctness for $n -1 < r < n$:} For $G$ any (directed) graph with an infinite chain of vertices $\{v_{i}\}_{i < \omega}$ with $v_{i} E v_{j}$ for $i < j$, if no quantifier-free formula in the graph language in $G$ exhibits $\mathrm{SOP}_{n}$, then $G$ has a (weakly embedded) copy of $\mathsf{G}_{a, b}$ for all positive integers $a, b$ with $n -1 < \frac{a}{b} < n$.\footnote{More literally, this is an elementary way of rephrasing the non-distinctness condition when the quantifier-free version of $\mathrm{NSOP}_{r}$ is stated in the following way: for every quantifier-free definable relation $R(x_{1}, x_{2})$ with a sequence $\{a_{i}\}_{i < \omega}$ such that $R(a_{i}, a_{j})$ for all $i < j$, for all integers $N$ there are \textit{distinct} $b_{0}, ... b_{N-1}$ such that $\models R(b_{i}, b_{i+j \mathrm{\:mod\:} N })$ for $0 \leq i < N$, $1 \leq j \leq \frac{N}{r}$. We need distinctness to get a weak embedding, rather than just a graph homomorphism. Of course, this is an equivalent statement of the quantifier-free version of $\mathrm{NSOP}_{r}$. To get distinctness of the $b_{0}, ... b_{N-1}$, extend $\{a_{i}\}_{i< \omega}$ to $\{a_{i}\}_{i< \omega \times \omega}$. Then apply the original quantifier-free $\mathrm{NSOP}_{r}$ without distinctness to $\mathrm{qftp}(\overline{a}_{0}\overline{a}_{1})$ in $ \{\overline{a}_{n}\}_{n < \omega} =: \{a_{n \omega}a_{n\omega + 1}  \ldots a_{n \omega + i} \ldots \}_{n < \omega}$, and select distinct $b_{0}, ... b_{N-1}$ using the pigeonhole principle. This is exactly as in the proof of Claim 2.19 of \cite{ApproxOrder}.

Generalizing this argument, in Hypothesis \ref{nondistinctness on sufficiently general grounds} and Definition \ref{o-maximality} below, replacing weak embeddings with graph homomorphisms will make no difference. }

Notice that the graphs $\mathsf{G}_{a, b}$ for  positive integers $a, b$ with $n -1 < \frac{a}{b} < n$ are all free of (directed) $n-1$-cycles. (This is in Claim 3.14 of \cite{ApproxOrder}.) So the question of whether the quantifier-free version of the real-valued $\mathrm{NSOP}_{r}$ hierarchy for $n-1 < r < n$ is distinct from the quantifier-free version of $\mathrm{NSOP}_{n}$ is just: for every $\mathsf{G}$ belonging to the specific family of $\leq n-1$-cycle free graphs  $\{\mathsf{G}_{a, b}: a, b \in \mathbb{N}, n -1 < \frac{a}{b} < n\}$, does every graph with an infinite chain for the graph relation and no quantifier-free formulas exhibiting $\mathrm{SOP}_{n}$ have a copy of $\mathsf{G}$?

However, we can even ask this question at a greater level of generality. It is not only open whether in this case we can always obtain a copy of each of the $\leq n-1$-cycle-free graphs $\mathsf{G}$ \textit{from this specific family} $\{\mathsf{G}_{a, b}\}$. In fact, there is not even a known counterexample to being able to obtain a copy of \textit{any} $\leq n-1$-cycle-free graph $\mathsf{G}$ in a graph with an infinite chain and no quantifier-free formulas exhibiting $\mathrm{SOP}_{n}$.

This observation gives us a rigorous sense in which the real-valued $\mathrm{NSOP}_{r}$ hierarchy may fail to be distinct from the integer-valued $\mathrm{NSOP}_{n}$ hierarchy on especially general grounds.

\begin{hypothesis}\label{nondistinctness on sufficiently general grounds}

The real-valued $\mathrm{NSOP}_{r}$ hierarchy is non-distinct from the integer-valued $\mathrm{NSOP}_{n}$ hierarchy \textit{on sufficiently general grounds} if, for all integers $n \geq 3$ and finite (directed) graphs $\mathsf{G}$ without (directed) $\leq n-1$-cycles, for every (infinite, directed) graph $G$ with some sequence of vertices $\{v_{i}\}_{i < \omega}$ with $v_{i} E v_{j}$ for $i < j$ and with no quantifier-free formula in the graph language exhibiting $\mathrm{SOP}_{n}$, there is a (weak) embedding $\mathsf{G} \hookrightarrow G$.
\end{hypothesis}

For $N \geq 3$ a fixed integer, we also say that the real-valued $\mathrm{NSOP}_{r}$ hierarchy is non-distinct from the integer-valued $\mathrm{NSOP}_{n}$ hierarchy  \textit{between $\mathrm{NSOP}_{N-1}$ and $\mathrm{NSOP}_{N}$ on sufficiently general grounds} if this statement holds for this specific value $n= N$, rather than for all $n$. For example, the real-valued $\mathrm{NSOP}_{r}$ hierarchy is non-distinct from the integer-valued $\mathrm{NSOP}_{n}$ hierarchy between $\mathrm{NSOP}_{3}$ and $\mathrm{NSOP}_{4}$ on sufficiently general grounds if every graph $G$ with an infinite chain for the graph relation and no quantifier-free formulas exhibiting $\mathrm{SOP}_{4}$ embeds every graph without directed triangles. This is the first case where the main result of this section will apply.

\begin{remark} \label{no n-1-cycles}

Note that if $\mathsf{G}$ \textit{does} have an $\leq n-1$-cycle, it is not the case that for every graph $G$ with some sequence of vertices $\{v_{i}\}_{i < \omega}$ with $v_{i} E v_{j}$ for $i < j$ and with no quantifier-free formula in the graph language exhibiting $\mathrm{SOP}_{n}$, there is an embedding $\mathsf{G} \hookrightarrow G$. For example, we can let $G$ be an existentially closed member of the class of directed graphs without directed $\leq n-1$-cycles; as shown in \cite{She95}, the theory of $G$ is $\mathrm{NSOP}_{n}$. The hypothesis that the real-valued and integer-valued hierarchies are non-distinct on sufficiently general grounds between $\mathrm{NSOP}_{n-1}$ and $\mathrm{NSOP}_{n}$ says that having an $\leq n-1$-cycle is the only such obstruction.

Thus another way of stating the hypothesis is as follows. Let $\mathcal{G}_{n}$ be the class of graphs $\mathsf{G}$ such that, for every graph $G$ with some sequence of vertices $\{v_{i}\}_{i < \omega}$ with $v_{i} E v_{j}$ for $i < j$ and with no quantifier-free formula in the graph language exhibiting $\mathrm{SOP}_{n}$, there is a weak embedding $\mathsf{G} \hookrightarrow G$. Then the real-valued and integer-valued hierarchies are non-distinct on sufficiently general grounds between $\mathrm{NSOP}_{n-1}$ and $\mathrm{NSOP}_{n}$ if and only if $\mathcal{G}_{n}$ is characterized as exactly those graphs with no $\leq n-1$-cycles.

\end{remark}

    \begin{remark}\label{comparison with existing results}

We remark on how our hypothesis fits into the existing results on the open question of whether the real-valued $\mathrm{NSOP}_{r}$ hierarchy is distinct from the integer-valued $\mathrm{NSOP}_{n}$ hierarchy. First, the content of Theorem 3.1 of \cite{ApproxOrder} for directed graphs is that, if $G$ is a graph with some sequence of vertices $\{v_{i}\}_{i < \omega}$ with $v_{i} E v_{j}$ for $i < j$ and with no quantifier-free formula in the graph language exhibiting $\mathrm{SOP}_{n}$, \textit{and the class of induced subgraphs of $G$ is defined by finitely many forbidden weak subgraphs}, then each $\mathsf{G}_{a, b}$ for $n -1 < \frac{a}{b} < n$ weakly embeds into $G$. However, there is nothing special here about the graphs $\mathsf{G}_{a, b}$, other than the fact that they have no $n-1$-cycles: in this case, the proof shows that \textit{every} $n-1$-cycle-free graph weakly embeds into $G$. This can be taken as additional motivation for replacing the family of graphs $\mathsf{G}_{a, b}$ with the family of \textit{all} $n-1$-cycle-free graphs, the main move of the hypothesis.

The other main theorem of \cite{ApproxOrder}, Theorem 4.6 there, gives an indication of how we might begin to obtain graphs $\mathsf{G} \notin \mathcal{G}_{n}$ that cannot be formed from the infinite chain graph by repeated quantifier-free applications of $\mathrm{NSOP}_{n}$, outside of the obvious case where the graph $\mathsf{G}$ has a $\leq n -1$-cycle. Specifically we see that all \textit{cyclically $n$-decomposable} graphs for integers $n \geq 4$, including the graphs $\mathsf{G}_{a, b}$ for $n -1 < \frac{a}{b} < n$, cannot be formed out of an infinite chain by using the special class of quantifier-free applications of $\mathrm{NSOP}_{n}$ given by \textit{interval helix maps}. On the other hand, in section 4.1 of \cite{ApproxOrder} it is observed that $\mathcal{G}_{n}$ contains not just all cyclically $n$-decomposable graphs but even some cyclically $n$-indecomposable graphs. So  we know that there are some graphs $\mathsf{G} \in \mathcal{G}_{n}$ that are relatively difficult to form from the infinite chain graph by quantifier-free applications of $\mathrm{NSOP}_{n}$, in the sense that we cannot just use interval helix maps to obtain $\mathsf{G}$, but such that we can still form $\mathsf{G}$ anyway using more general quantifier-free applications of $\mathrm{NSOP}_{n}$.

\end{remark}

Having described a general way in which the real-valued $\mathrm{NSOP}_{r}$ hierarchy may fail to introduce new classification-theoretic properties, we are now ready to investigate its consequences. We will show that assuming this hypothesis, we can resolve one of the main questions on the \textit{integer-valued} $\mathrm{NSOP}_{n}$ hierarchy: whether $\mathrm{NSOP}_{n+1} \cap \mathrm{NTP}_{2} = \mathrm{NSOP}_{n} \cap \mathrm{NTP}_{2}$ for integers $n \geq 3$. Informally, we will show that it is \textit{approximately} the case that one of the following two possibilities hold: either (1) the real-valued $\mathrm{NSOP}_{r}$ hierarchy does introduce new properties, or (2) $\mathrm{NSOP}_{n} \cap \mathrm{NTP}_{2} = \mathrm{NSOP}_{n+1} \cap \mathrm{NTP}_{2}$, resolving this open question. Specifically, we will prove Theorem \ref{main theorem 2} from the introduction that if (1) fails on sufficiently general grounds, (2) must be true. 

\begin{theorem}\label{main theorem 2, restated}

Suppose that real-valued $\mathrm{NSOP}_{r}$ hierarchy is non-distinct from the integer-valued $\mathrm{NSOP}_{n}$ hierarchy on sufficiently general grounds. Then for each integer $n \geq 3$, $\mathrm{NSOP}_{n} \cap \mathrm{NTP}_{2} = \mathrm{NSOP}_{n+1} \cap \mathrm{NTP}_{2}$.

More precisely,  for any integer $N \geq 3$, if the real-valued $\mathrm{NSOP}_{r}$ hierarchy is non-distinct from the integer-valued $\mathrm{NSOP}_{n}$ hierarchy between $\mathrm{NSOP}_{N}$ and $\mathrm{NSOP}_{N+1}$ on sufficiently general grounds, then $\mathrm{NSOP}_{N} \cap \mathrm{NTP}_{2} = \mathrm{NSOP}_{N+1} \cap \mathrm{NTP}_{2}$.

\end{theorem}

We will prove this by simulating the standard proof that the model companion of the theory of $\leq N$-cycle-free directed graphs has $\mathrm{TP}_{2}$.

\begin{proof}
    Assume that the real-valued $\mathrm{NSOP}_{r}$ hierarchy is non-distinct from the integer-valued $\mathrm{NSOP}_{n}$ hierarchy between $\mathrm{NSOP}_{N}$ and $\mathrm{NSOP}_{N+1}$ on sufficiently general grounds. Suppose that $T$ has $\mathrm{SOP}_{N}$ but is $\mathrm{NSOP}_{N+1}$; to prove this theorem, our goal will be to show that $T$ has $\mathrm{TP}_{2}$. Let $R(x, y)$ be a definable relation exhibiting $\mathrm{SOP}_{N}$, so there is a sequence $\{a_{i}\}_{i < \omega}$ such that $\models R(a_i, a_j)$ for $i < j$, but $R(x, y)$ has no directed $N$-cycles. (We may assume $R(x, y)$ to be a directed graph relation, so antisymmetric and irreflexive.)

    Consider the directed graph $\mathsf{G}$ consisting of distinct points $\{b_{ij}, c_{ij}\}_{i, j < \omega}$, $\{d^{k}_{ijj'}: i<\omega, j<j' < \omega, k\in \{1, \ldots, N-3\}\}$ (so ignore this for $N = 3$), $\{a_{\sigma}\}_{\sigma \in \omega^{\omega}}$, with exactly the following edges:

    (a) For $j < \omega$, $j < j' < \omega$, edges $b_{ij} E d^{1}_{ijj'}, d^{1}_{ijj'}Ed_{ijj'}^{2}, \ldots, d^{k}_{ijj'}Ed^{k+1}_{ijj'}, \ldots, d^{N-4}_{ijj'}Ed^{N-3}_{ijj'},d^{N-3}_{ijj'} E c_{ij'} $ (so for $N = 3$, $b_{ij} E c_{ij'}$)

    (b) For $\sigma \in \omega^{\omega}$ and $i < \omega$, edges $c_{i \sigma(i)} E a_{\sigma}$ and $a_{\sigma} E b_{i \sigma(i)}$.

We show that $\mathsf{G}$ is $\leq N$-cycle-free. First, within the $b_{ij}, c_{ij}, d^{k}_{ii'j}$ there are no cycles at all: no cycle within $b_{ij}, c_{ij}, d^{k}_{ijj'}$ can contain any $b_{ij}$ or $c_{ij}$, because these are respectively targets of no edges or sources of no edges within $b_{ij}, c_{ij}, d^{k}_{ijj'}$, and the $d^{k}_{ii'j}$ just consist of a disjoint union of paths with no additional edges, so there are no cycles within just the $d^{k}_{ijj'}$. So, if there is an $\leq N$-cycle in $\mathsf{G}$, one of the points must be one of the $a_{\sigma}$; the next point with an edge from $a_{\sigma}$ in the $\leq N$-cycle must then be $b_{i\sigma(i)}$ for some $i < \omega$. Then the next point with an edge from $b_{i\sigma(i)}$ in the $\leq N$-cycle (if $N > 3$) must be $d_{i\sigma(i)j'}^{1}$ for some $j'$ with $\sigma(i)< j' < \omega$; the next point with an edge from $d_{i\sigma(i)j'}^{1}$ in the $\leq N$-cycle must be $d_{i\sigma(i)j'}^{2}$ (if $N > 4$); then, after proceeding in this way for a total number of steps one less than the length of the $\leq N$-cycle, the next point with an edge from one of the $d_{i\sigma(i)j'}^{k}$ or from $c_{ij'}$ in the cycle will be $a_{\sigma}$, where $\sigma(i)< j' < \omega$. Since $\sigma(i) \neq j'$, this contradicts the specification of the edges from $a_{\sigma}$ in (b).

So we may apply this hypothesis to the $\mathrm{NSOP}_{N}$ directed graph defined on $|x|$-tuples by $R(x,y)$, which we assumed has $\{a_{i}\}_{i < \omega}$ such that $\models R(a_i, a_j)$ for $i < j$. By compactness, this shows that $\mathsf{G}$ will embed into the relation defined by $R(x, y)$. We identify $b_{ij}, c_{ij}, d^{k}_{ii'j}, a_{\sigma}$ with their images. We show that $\varphi(x, yz) = : R(x, y) \wedge R(z, x)$ exhibits $\mathrm{TP}_{2}$ with the array $\{b_{ij}c_{ij}\}_{i, j < \omega}$. First, for $\sigma \in \omega^{\omega}$, $\{\varphi(x, b_{i \sigma(i)}c_{i \sigma(i)})\}_{i < \omega}$ is consistent, realized by $a_{\sigma}$. Second, for $i < \omega$, $j < j' < \omega$, $\{\varphi(x, b_{ij}c_{ij}), \varphi(x, b_{ij'}c_{ij'}) \}$ will be inconsistent, and particularly $\{R(x, b_{ij}), R(c_{ij'}, x) \}$ will be inconsistent. Otherwise, for $a$ a realization, $a, b_{ij}, d^{1}_{ijj'}, \ldots, d^{k}_{ijj'}, \ldots, d^{N-3}_{ijj'}, c_{ij'}$ will form an $N$-cycle for $R(x, y)$, contradicting the assumption that $R(x, y)$ exhibits $\mathrm{SOP}_{N}$.

\end{proof}

Therefore, one potential method of answering the open question of whether $\mathrm{NSOP}_{n+1} \cap \mathrm{NTP}_{2}$ is equal to $\mathrm{NSOP}_{n} \cap \mathrm{NTP}_{2}$ for integers $n \geq 3$ is to attempt to give a positive answer to the following open question.\footnote{One of the edge cases of this problem is whether $\mathrm{NSOP}_{3} \cap \mathrm{NTP}_{2}$ is equal to the class of simple theories, equivalently (as in \cite{Sh90}) $\mathrm{NSOP}_{2} \cap \mathrm{NTP}_{2}$. Naïvely, given Theorem \ref{main theorem 3, restated}, one might think this corresponds to the hypothesis that the real-valued $\mathrm{NSOP}_{r}$ hierarchy is non-distinct from the integer-valued $\mathrm{NSOP}_{n}$ hierarchy between $\mathrm{NSOP}_{2}$ and $\mathrm{NSOP}_{3}$ on sufficiently general grounds. However, because we do not have an instance of $\mathrm{SOP}_{n}$ for an integer $n \geq 3$ to start with as in the proof of Theorem \ref{main theorem 2, restated}, it is not clear that the methods of this section apply to that edge case. } Since our hypothesis is stated in quantifier-free terms, this will be a purely combinatorial question:

\begin{question}\label{general non-distinctness question}
    Are the real-valued $\mathrm{NSOP}_{r}$ hierarchy and integer-valued $\mathrm{NSOP}_{n}$ hierarchy (between $\mathrm{NSOP}_{N}$ and $\mathrm{NSOP}_{N+1}$, for an integer $N\geq 3$) non-distinct on sufficiently general grounds?
\end{question}

\begin{remark}\label{intermediate possibility}

We characterized our main result as an approximate alternative between (1) the real-valued $\mathrm{NSOP}_{r}$
 hierarchy being distinct from the integer-valued $\mathrm{NSOP}_{n}$ hierarchy and (2) $\mathrm{NSOP}_{n+1} \cap \mathrm{NTP}_{2}$ being equal to $\mathrm{NSOP}_{n} \cap \mathrm{NTP}_{2}$. Since this is only an approximate alternative, we make explicit what happens in the intermediate case where neither is true.
 
 \: In the proof of the theorem, to apply the hypothesis of non-distinctess on sufficiently general grounds, we only needed that the graph defined by the relation $R(x, y)$ was $\mathrm{NSOP}_{N}$ at the quantifier-free level. Of course, in the setting of this theorem, this graph was $\mathrm{NSOP}_{N}$, in full first-order logic.
 
 \: So if neither (1) nor (2) true, the following must be the case:

 \begin{itemize}
     \item  for $n \geq 3$ an integer, a definable relation $R(x, y)$ in an $\mathrm{NSOP}_{n}$ theory with an infinite chain must have a copy of all of the specific $\leq n-1$-cycle free graphs $\mathsf{G}_{a, b}$, for positive integers $a, b$ with $n-1 < \frac{a}{b} < n$, but 
 \item there is some $n \geq 3$ such that there is a definable relation $R(x, y)$ in an $\mathrm{NSOP}_{n}$ theory with an infinite chain that omits some $\leq n-1$-cycle free graph $\mathsf{G}$ not of the form $\mathsf{G}_{a, b}$. 
 \end{itemize}

 So the real-valued and integer-valued hierarchies will coincide, but for reasons specific to the graphs $\mathsf{G}_{a, b}$ used to define $\mathrm{NSOP}_{r}$ for $r > 2$ a real number, and not just for the more general reason that $\mathsf{G}_{a, b}$ happens to be $\leq n$-cycle-free.
    
\end{remark}

We conclude the main part of this section by extracting the conclusion of our hypothesis as a property of first-order theories. We will remark on two significant edge cases, both the most basic case corresponding to $\mathrm{NSOP}_{3}$, and the asymptotic case of being able to obtain every graph without \textit{small} cycles from an infinite chain.

We isolate the conclusion of our hypothesis as the following family of definitions:

\begin{definition}\label{o-maximality}
    Let $n\geq 3$ be an integer. A theory is \textit{$n$-o-maximal}\footnote{Here the ``o" stands for ``order," as in the $n$-strict order property.} if there is no $\leq n$-cycle-free graph $\mathsf{G}$ and definable relation $R(x, y)$, with $\{a_{i}\}_{i < \omega}$ such that $\models R(a_i, a_j)$ for $i < j$, such that $\mathsf{G}$ does not weakly embed in the relation on $\mathbb{M}^{|x|}$ defined by $R(x, y)$. A theory is \textit{o-maximal} if it is $3$-o-maximal, and is \textit{asymptotically o-maximal} if it is $n$-o-maximal for some integer $n \geq 3$.
\end{definition}

\begin{remark}
    One advantage of this terminology is that a theory cannot be both o-minimal and o-maximal.
\end{remark}

If a theory is $n$-o-maximal, it is $\mathrm{NSOP}_{n}$. However, it is open whether every $\mathrm{NSOP}_{n}$ theory is $n$-o-maximal. The non-distinctness of the real-valued and integer-valued hierarchies between $\mathrm{NSOP}_{n-1}$ and $\mathrm{NSOP}_{n}$ on sufficiently general grounds implies that every $\mathrm{NSOP}_{n}$ theory is $n$-o-maximal, and is equivalent to the claim that every theory which is $\mathrm{NSOP}_{n}$  \textit{at the quantifier-free level} is $n$-o-maximal \textit{at the quantifier-free level}. Moreover, for integers $n \geq 3$, \textit{every $n$-o-maximal theory is $\mathrm{NSOP}_{r}$ for $n -1 < r < n$}. For integers $n \geq 4$, the proof of Theorem \ref{main theorem 2, restated} will imply that \textit{every $n$-o-maximal theory either is $\mathrm{NSOP}_{n-1}$, or has $\mathrm{TP}_{2}$}. 

As suggested by our terminology, the case of $n$-o-maximality for $n=3$ is special. Note that no directed graph has $\leq 2$-cycles (i.e., every directed graph relation is antisymmetric and irreflexive). So a theory is o-maximal exactly when, for every definable relation $R(x, y)$, with $\{a_{i}\}_{i < \omega}$ such that $\models R(a_i, a_j)$ for $i < j$, \textit{every} directed graph weakly embeds into the relation on $\mathbb{M}^{|x|}$ defined by $R(x, y)$. (This is the reason for the ``maximality" in the terminology). It is therefore worth isolating the following question as significant in its own right (even though only the case of $n$-o-maximality for $n \geq 4$ is directly relevant to our main result, Theorem \ref{main theorem 2, restated}).

\begin{question}
    Is every $\mathrm{NSOP}_{3}$ theory o-maximal?
\end{question}

This is a special case of the problem of whether $\mathrm{NSOP}_{2}$ is equal to $\mathrm{NSOP}_{3}$. Specifically, we obtain the following corollary from the proof of Theorem \ref{main theorem 1, restated} from the previous section:

\begin{cor}\label{nsop2 implies o-maximality}
    Every $\mathrm{NSOP}_{2}$ theory is o-maximal.
\end{cor}

\begin{remark}
    There has been recent interest in updating the terminology for $\mathrm{NSOP}_{1}$ to reflect that it is equal to $\mathrm{NSOP}_{2}$. While we do not have an unconditional proposal for this new terminology, \textit{if it is shown to be the case} in the future that $\mathrm{NSOP}_{2}$ is equal to $\mathrm{NSOP}_{3}$, we propose \textit{o-maximality} as the new standard terminology for $\mathrm{NSOP}_{n}$ for $1 \leq n \leq 3$: in that case, $\mathrm{NSOP}_{1}$, $\mathrm{NSOP}_{2}$ and $\mathrm{NSOP}_{3}$ would all be equivalent to o-maximality.
\end{remark}

At the other edge, we can extract general information about asymptotic o-maximality from the proof of Theorem 3.1 of \cite{ApproxOrder}:

\begin{fact}\label{n-cycle-free to small cycle-free}

For an integer $n \geq 3$, let $G$ be a graph in which no quantifier-free formula exhibits $\mathrm{NSOP}_{n}$. If there is some $\leq n-1$-cycle-free graph that does not weakly embed into $G$, then for arbitrarily large $N$, there is some $\leq N$-cycle-free graph that does not weakly embed into $G$.
    
\end{fact}

(To prove this, for $\mathcal{H}$ the hereditary class of graphs weakly embedding into $G$, the proof of Proposition 2.18 of \cite{ApproxOrder} will show that $\mathcal{H}$ is $\mathrm{NSOP}_{n}$. But we can then apply Claim 3.15 of \cite{ApproxOrder} to $\mathcal{H}$.)

The previous fact has the following consequence: in our hypothesis that the real-valued and integer-valued hierarchies are non-distinct on sufficiently general grounds between $\mathrm{NSOP}_{n-1}$ and $\mathrm{NSOP}_{n}$, for arbitrarily large $N$ we can let $\mathsf{G}$ range only over $\leq N$-cycle-free graphs, rather than $n$-cycle-free graphs, and get an equivalent statement. At the level of full first-order logic, this means that a $\mathrm{NSOP}_{n}$ theory is $n$-o-maximal if and only if it is asymptotically o-maximal.

So we can now ask a question, stated in absolute terms, which will have implications for the question of whether the $\mathrm{NSOP}_{n}$ hierarchy is strict within $\mathrm{NTP}_{2}$. (In Problem 3.16 of \cite{Che14}, the questions of whether $\mathrm{NSOP}_{n+1} \cap \mathrm{NTP}_{2} = \mathrm{NSOP}_{n} \cap \mathrm{NTP}_{2}$ were grouped together into the question of whether the $\mathrm{NSOP}_{n}$ hierarchy is strict within $\mathrm{NTP}_{2}$.) Recall that a theory is defined to have the \textit{strong order property} (Definition 2.2 of \cite{She95}) if it is $\mathrm{NSOP}_{n}$ for some $n$. So asymptotic o-maximality implies the strong order property. We can then ask:

\begin{question}
    Is the strong order property equivalent to asymptotic o-maximality?
\end{question}

Because an asymptotically o-maximal theory $\mathrm{NSOP}_{n}$ theory is $n$-o-maximal, our conclusion from the proof of Theorem \ref{main theorem 2, restated} implies that, if the answer to this question is yes, the $\mathrm{NSOP}_{n}$ hierarchy is \textit{not} strict within $\mathrm{NTP}_{2}$.

\addcontentsline{toc}{subsection}{\hspace{1.5em}Sidebar 1: $\mathrm{NTP}_{2}$ graph theory}
\begin{tcolorbox}[
    enhanced,
    breakable,
    oversize,                 
    colback=blue!5!white,     
    colframe=blue!75!black,   
    skin first=enhanced,      
    skin middle=enhanced,     
    skin last=enhanced        
]

\textsf{\textbf{Sidebar 1: $\mathrm{NTP}_{2}$ graph theory}}

\:

\: There is a rich theory of forking-independence in $\mathrm{NTP}_{2}$ theories: see, e.g., \cite{CK12}, \cite{BYC14}, \cite{Che14}, \cite{KS17}, \cite{Sim20}. But we can also develop our understanding of $\mathrm{NTP}_{2}$ theories in general from a different, combinatorial point of view. We will apply the main construction of Theorem \ref{main theorem 2, restated} from this section, as well as the cycle-removal techniques from our earlier results on the real-valued $\mathrm{NSOP}_{r}$ hierarchy in \cite{ApproxOrder}.

\: The main result of this sidebar will relate in a strange way to the existing literature: in Sections 3 and 5 of \cite{INDNSOP3} we showed that $\mathrm{NSOP}_{3}$ theories (possibly under the assumption of symmetric Conant-independence) exhibited independence phenomena more traditionally associated with $\mathrm{NTP}_{2}$ theories. In this sidebar we will move in reverse, showing that $\mathrm{NTP}_{2}$ theories exhibit a graph-theoretic phenomenon essentially established earlier for, and more intuitively suggested by the definition of, $\mathrm{NSOP}_{3}$ theories.

\: By Fact \ref{n-cycle-free to small cycle-free}, which we saw was part of the proof of Theorem 3.1 of \cite{ApproxOrder}, $\mathrm{NSOP}_{3}$ theories exhibit the following graph-theoretic phenomenon: if $G$ is a definable (directed) graph in an $\mathrm{NSOP}_{3}$ theory, and $G$ does not contain the random tournament as an induced subgraph (so equivalently, there is some finite directed graph which does not weakly embed into $G$), then for arbitrarily large $N$, there is some $\leq N$-cycle-free graph that does not weakly embed into $G$. Relative to our result in this sidebar, this may be regarded as more intuitive, because the original definition of $\mathrm{NSOP}_{3}$ as well as the other properties $\mathrm{NSOP}_{n}$ for integers $n \geq 3$  were explicitly stated in terms of directed cycles. However, the definition of $\mathrm{NTP}_{2}$ is not stated in terms of cycles, yet the same holds for $\mathrm{NTP}_{2}$:

\begin{theorem}\label{ntp2 graph theory}

 If $G$ is a definable (directed) graph in an $\mathrm{NTP}_{2}$ theory, and $G$ not contain the random tournament as an induced subgraph, then for arbitrarily large $N$, there is some $\leq N$-cycle-free graph that does not weakly embed into $G$.

\end{theorem}

 \begin{proof} (sketch)

     \: This is exactly as in the argument for Fact \ref{n-cycle-free to small cycle-free} within the proof of Theorem 3.1 of \cite{ApproxOrder}, but with a new cycle-removal proposition.  There, we used an abstract cycle-removal property for helix maps between graphs omitted by a $\mathrm{NSOP}_{3}$ hereditary class, the $n=3$ case of Proposition 3.11 of \cite{ApproxOrder}. Here, we will use the same abstract cycle-removal property, but for maps between graphs not weakly embedding into a fixed $\mathrm{NTP}_{2}$ graph:

     \begin{prop}\label{ntp_2 cycle removal}
    
Let $G^{+}$ be a directed graph definable in an $\mathrm{NTP}_{2}$ theory, and let $\mathcal{H}$ be the hereditary class of finite graphs that weakly embed into $G^{+}$. Let $g: H \to G$ be a graph homomorphism between finite graphs with $H, G \notin \mathcal{H}$. Let $\gamma \hookrightarrow G$ be a directed cycle in $G$ of minimal length, and let $k \geq 3 $ be this length. Then there is a graph homomorphism $h: \tilde{H} \twoheadrightarrow H$ for a finite graph $\tilde{H}$, such that $\tilde{H} \notin \mathcal{H}$ and such that there is no $k$-cycle $\gamma' \hookrightarrow \tilde{H}$ with $\gamma = g\circ h\circ \gamma' $.

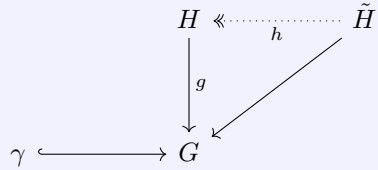
$$\begin{tikzcd}
                        &  & H \arrow[dd, "g"] &  & \tilde{H} \arrow[ll, two heads, "h", dotted] \arrow[lldd] \\
                        &  &              &  &                                      \\
\gamma \arrow[rr, hook] &  & G            &  &                                     
\end{tikzcd}$$

\end{prop}

\: We prove this proposition, which will be enough to prove the theorem.

\begin{proof}
    As in Claim 3.12 of \cite{ApproxOrder}, the cycle $\gamma$, because it is of minimal length in $G$, is an induced cycle in $G$. Let us write this cycle as $u, v, \omega_{1}, \ldots, \omega_{k-3}, w $, compatibly with the edge relation on the cycle. Let $U=\{u^{i}\}_{i \in I_{U}} =: g^{-1}(u)$, $V=\{v^{i}\}_{i \in I_{V}} =: g^{-1}(v)$, $\Omega_{\ell} = \{\omega_{\ell}^{i}\}_{i \in I_{\Omega_{\ell}}} =: g^{-1}(\omega_{\ell}) $, $W=\{w^{i}\}_{i \in I_{W}} =: g^{-1}(w)$ and $D := H \backslash ( U \sqcup V \sqcup \bigsqcup^{k-3}_{\ell =1} \Omega_{\ell} \sqcup W ) $. Then define $\widetilde{H}$ to consist of the disjoint union of $D$ together with the distinct vertices $\{b_{ij}(v^{t})\}^{t \in I_{V}}_{i, j < \omega}$, $\{c_{ij}(w^{t})\}^{t \in I_{W}}_{i, j < \omega}$, $\{d^{\ell}_{ijj'}(\omega^{t}_{\ell}): i<\omega, j<j' < \omega, \ell\in \{1, \ldots, k-3\}, t \in I_{\Omega_{\ell}}\}$, $\{a_{\sigma}(u^{t})\}^{t \in I_{U}}_{\sigma \in \omega^{\omega}}$, with exactly the following edges:

\:

($\alpha$) The original edges on $D$, along with:
\begin{itemize}
    \item $d E a_{\sigma}(u^{t})$ whenever $d E u^{t}$ and $ a_{\sigma}(u^{t}) E d$ whenever $ u^{t} E d$ for $\sigma \in \omega^{\omega}$, $d \in D$, $t \in I_{U}$
    \item $d E b_{ij}(v^{t})$ whenever $d E v^{t}$ and $ b_{ij}(v^{t}) E d$ whenever $v^{t} E d$ for $i, j < \omega$, $d \in D$, $t \in I_{V}$

    \item $d E d^{\ell}_{ijj'}(\omega_{\ell}^{t})$ whenever $d E \omega_{\ell}^{t}$ and $ d^{\ell}_{ijj'}(\omega_{\ell}^{t}) E d$ whenever $\omega_{\ell}^{t} E d$ for $i < \omega$, $j< j' < \omega$, $d \in D$, $t \in I_{\omega_{\ell}}$, $\ell \in \{1, \ldots, k-3\} $,

    \item $d E c_{ij}(w^{t})$ whenever $d E w^{t}$ and $ c_{ij}(w^{t}) E d$ whenever $w^{t} E d$ for $i, j < \omega$, $d \in D$, $t \in I_{W}$
\end{itemize}

    (a) For $j < \omega$, $j < j' < \omega$, edges $b_{ij}(v^{t}) E d^{1}_{ijj'}(\omega^{t'}_{1})$ for $t \in I_{V}$, $t' \in I_{\Omega_{1}}$ whenever $v^{t} E \omega^{t'}_{1}$, $d^{\ell}_{ijj'}(\omega^{t}_{\ell})Ed_{ijj'}^{\ell+1}(\omega^{t'}_{\ell+1})$ whenever $\omega^{t}_{\ell}E\omega^{t'}_{\ell+1}$ for $ \ell < k -3$,  $t \in I_{\Omega_{\ell}}$, $t' \in I_{\Omega_{\ell+1}}$, $d^{k-3}_{ijj'}(\omega^{t}_{k-3}) E c_{ij'}(w^{t'})$ whenever $\omega^{t}_{k-3}Ew^{t'}$ for $t \in I_{\Omega_{k-3}}$, $t' \in I_{W}$ (so for $N = 3$, $b_{ij}(v^{t}) E c_{ij'}(w^{t'})$ for $t \in I_{V}$, $t' \in I_{W}$ whenever $v^{t}Ew^{t'}$.)

(b) For $\sigma \in \omega^{\omega}$, $i < \omega$, edges $c_{i \sigma(i)} (w^{t}) E a_{\sigma}(u^{t'})$ whenever $w^{t}Eu^{t'}$ for $t \in I_{W}$, $t' \in I_{U}$, and $a_{\sigma}(u^{t}) R b_{i \sigma(i)}(v^{t'})$ whenever $u^{t} E v^{t'}$ for $t \in I_{U}$, $t' \in I_{V}$.

\: Now $\widetilde{H}$ does not weakly embed in $G^{+}$. Otherwise, just as in the proof of Theorem \ref{main theorem 2, restated}, $\widetilde{H}$ weakly embedding in $G^{+}$, combined with $H$ not weakly embedding in $G^{+}$, would imply that any theory in which $G^{+}$ is $\mathrm{NTP}_{2}$. So for some finite  $\tilde{H} \subset \widetilde{H}$, $\tilde{H} \notin \mathcal{H}$. Then define the map $h: \tilde{H} \twoheadrightarrow H$ by $d \mapsto d$ for $d \in D$, $a_{\sigma}(u^{t}) \mapsto u^{t}$, $b_{ij}(v^{t}) \mapsto v^{t}$, $d^{k}_{ijj'}(\omega^{t}_{\ell}) \mapsto \omega^{t}_{\ell}$, $c_{ij}(w^{t}) \mapsto w^{t}$. Then $h$ will be a graph homomorphism. Moreover, any  $k$-cycle $\gamma' \hookrightarrow \tilde{H}$ with $\gamma = g\circ h\circ \gamma' $ would lie entirely within $\tilde{H} \backslash D$. But this is a contradiction because $\tilde{H} \backslash D$ has no $\leq k$-cycle, just as $\mathsf{G}$ is shown to have no $\leq N$-cycle in the proof in Theorem \ref{main theorem 2, restated}.

\end{proof}

 \end{proof}

We conclude this sidebar by remarking on a generalization of this theorem that covers both $\mathrm{NSOP}_{3}$ theories and $\mathrm{NTP}_{2}$ theories.

 \begin{remark}

    In \cite{Bailetti25}, Bailetti defines that class of $\mathrm{NPM}^{(2)}$ theories, which includes all $\mathrm{NTP}_{2}$ theories and $\mathrm{NSOP}_{3}$ theories.  While we will not give the rigorous definition, $\mathrm{NPM}^{(2)}$ is the negation of $\mathrm{PM}^{(2)}$ (``$2$-positive maximality"), meaning that there is a formula $\varphi(x,y)$ exhibiting every pattern of consistency and inconsistency among instances of $\varphi(x,y)$, where the instances of inconsistency are all between pairs of formulas. Thus, $\mathrm{PM}^{(2)}$ implies $\mathrm{TP}_{2}$ and $\mathrm{SOP}_{3}$, which can be defined in terms of such patterns. There is precedent for common phenomena in $\mathrm{NSOP}_{3}$ theories and $\mathrm{NTP}_{2}$ theories actually being explained by the fact that both classes of theories are contained in the class $\mathrm{NPM}^{(2)}$. Particularly, in appendix C of \cite{Patterning} we show with Gabriel Day that the result that all internally $\mathrm{NSOP}_{1}$ types are co-$\mathrm{NSOP}_{1}$, previously shown for $\mathrm{NTP}_{2}$ theories (essentially in Theorem 6.17 of \cite{Che14}) and $\mathrm{NSOP}_{3}$ theories (Theorem 1.1 of \cite{INDNSOP3}), is actually true for all $\mathrm{NPM}^{(2)}$ theories. In fact, by straightforwardly adapting the construction in the proof of the previous theorem from $\mathrm{NTP}_{2}$ theories to $\mathrm{NPM}^{(2)}$ theories, we have another example of a common phenomenon in $\mathrm{NSOP}_{3}$ theories and $\mathrm{NTP}_{2}$ theories that is explained by $\mathrm{NPM}^{(2)}$: every graph definable in $\mathrm{NPM}^{(2)}$ theories, that omits the random tournament, likewise omits some $\leq N$-cycle free graph for arbitrarily large $N$. We forego the details of the proof.
\end{remark}

\end{tcolorbox}

\section{A new $\mathrm{NSOP}_{2}$-$\mathrm{SOP}_{3}$ dichotomy}\label{nsop2 vs nsop3}

We now return to the question of whether $\mathrm{NSOP}_{2}$ is equal to $\mathrm{NSOP}_{3}$, whose interactions with the real-valued $\mathrm{NSOP}_{r}$ hierarchy we previously explored in Section \ref{between nsop2 and nsop3}. We concluded in that section that a potential barrier to $\mathrm{NSOP}_{r}$ being well-defined for $2< r < 3$ vanishes for nontrivial reasons, so if all of the well-defined properties $\mathrm{NSOP}_{r}$ for non-integer $r$ are distinct, $\mathrm{NSOP}_{2}$ must not be equal to $\mathrm{NSOP}_{3}$. Our goal in this section will be to apply techniques from the theory of the real-valued $\mathrm{NSOP}_{r}$ hierarchy to make more unconditional progress on the $\mathrm{NSOP}_{2}$-$\mathrm{NSOP}_{3}$ problem. We will prove a \textit{sharp} dichotomy between $\mathrm{NSOP}_{2}$ and $\mathrm{SOP}_{3}$ under a basic combinatorial hypothesis from hereditary class theory.

The real-valued $\mathrm{NSOP}_{r}$ context, specifically within Theorem 3.1 of \cite{ApproxOrder}, serves as the canonical setting for our cycle-removal arguments: there, the helix maps coming from the properties $\mathrm{NSOP}_{n}$ behave especially seamlessly with respect to the cycles we want to remove. And our ultimate result offers insight into the intrinsic properties of hereditary classes, showing that the purely combinatorial quantity $\mathfrak{o}(\mathcal{H})$ motivated by the real-valued $\mathrm{NSOP}_{r}$ hierarchy, known only to be a real value when $\mathcal{H}$ is any hereditary class, is in fact an integer when $\mathcal{H}$ is defined by finitely many forbidden weakly embedded substructures. However, with additional difficulty, we can tailor these cycle-removal arguments to the question of whether $\mathrm{NSOP}_{2}$ is equal to $\mathrm{NSOP}_{3}$. Our setting for investigating the $\mathrm{NSOP}_{2}$-$\mathrm{NSOP}_{3}$ question will be the one where our earlier investigations of the real-valued $\mathrm{NSOP}_{r}$ hierarchy showed the integers played a special combinatorial role: hereditary classes defined by finitely many forbidden weakly embedded substructures, Definition \ref{defined by finitely many forbidden weakly embedded substructures}.

Unlike Theorem 3.1 of \cite{ApproxOrder}, where we are essentially doing pure combinatorics inspired by classification theory, our progress on whether $\mathrm{NSOP}_{2}$ is equal to $\mathrm{NSOP}_{3}$ will be properly a theorem of logic, rather than a theorem of combinatorics. This is because, while straightforward translation from the language of logic to the language of hereditary classes allows us to speak of a hereditary class being $\mathrm{NSOP}_{n}$ for integers $n \geq 3$, or $\mathrm{NSOP}_{r}$ for reals $r > 2$ (\cite{ApproxOrder}, Definition 2.12), it is not clear what it should mean for a hereditary class to be $\mathrm{NSOP}_{2}$. Yet there is still a way to specialize the question of whether $\mathrm{NSOP}_{2}$ is equal to $\mathrm{NSOP}_{3}$ by adding assumptions on hereditary classes.

This is to say, given the apparent difficulty of proving that every $\mathrm{NSOP}_{3}$ theory is $\mathrm{NSOP}_{2}$, we may next attempt to show that every $\mathrm{NSOP}_{3}$ theory satisfying additional assumptions is $\mathrm{NSOP}_{2}$. For example, we may want to show that every countably categorical $\mathrm{NSOP}_{3}$ theory is $\mathrm{NSOP}_{2}$, or that every countably categorical $n$-ary $\mathrm{NSOP}_{3}$ theory is $\mathrm{NSOP}_{2}$.\footnote{Kaplan, Ramsey and Simon (\cite{KRS23}) prove a non-sharp dichotomy, showing that every binary $\mathrm{NSOP}_{3}$ theory is simple.} But we can also easily restate the question of whether $\mathrm{NSOP}_{3}$ is equal to $\mathrm{NSOP}_{2}$ so that the assumptions that we can add are about hereditary classes, rather than about theories.

\begin{fact}\label{hereditary class restatement of sop2-sop3 problem}

The following are equivalent:

(1) The class $\mathrm{NSOP}_{2}$ is equal to the class $\mathrm{NSOP}_{3}$.

(2) Let $\mathcal{H}$ be a hereditary class such that every theory whose models have age $\mathcal{H}$ has $\mathrm{SOP}_{2}$. Then every theory whose models have age $\mathcal{H}$ has $\mathrm{SOP}_{3}$.

\end{fact}

\begin{proof}(sketch)

It is immediate that (1) implies (2). To see that (2) implies (1), suppose that $T$ were a strictly $\mathrm{NSOP}_{3}$ theory. Let $\varphi(x, y)$ exhibit $\mathrm{SOP}_{2}$ in $T$. Expand $T$ by definable predicates $R(x, y)$ for $\varphi(x, y)$ and $I(y_{1}, y_{2})$ for $\neg \exists x \varphi(x, y_{1}) \wedge \varphi(x, y_{2}) $, and let $T_{0}$ be the reduct to the language consisting of the symbols $R(x, y)$ and $I(y_{1}, y_{2})$. Let $\mathcal{H}$ be the age of models of $T_{0}$. Then every theory whose models have age $\mathcal{H}$ has $\mathrm{SOP}_{2}$ (this is essentially Claim 1 of \cite{Patterning}, related to Fact \ref{Bodirsky, Bodor and Marimon theorem} below.) But not every theory whose models have age $\mathcal{H}$ has $\mathrm{SOP}_{3}$: $T_{0}$ is $\mathrm{NSOP}_{3}$. So, given our assumption that (1) is false, (2) is false.

\end{proof}

Rephrasing the assertion that $\mathrm{NSOP}_{2}$ is equal to $\mathrm{NSOP}_{3}$ to quantify over hereditary classes $\mathcal{H}$, rather than theories $T$, opens up a completely new space of possible assumptions. The main result of this section, Theorem \ref{main theorem 3, restated}, will be to prove the assertion of the equivalence of $\mathrm{NSOP}_{2}$ and $\mathrm{NSOP}_{3}$ in terms of hereditary classes $\mathcal{H}$, under the assumption that $\mathcal{H}$ is defined by finitely many forbidden weakly embedded substructures, and to observe that this equivalence is sharp.

Incidentally, there is a sense in which our result finitely approximates a statement implying that every countably categorical $\mathrm{NSOP}_{3}$ theory is $\mathrm{NSOP}_{2}$. Specifically, as a limiting case of the assumption that a hereditary class $\mathcal{H}$ is defined by a finite family of forbidden weakly embedded substructures, we may consider hereditary classes $\mathcal{H}$ defined by a \textit{potentially infinite} family of forbidden weakly embedded substructures. This is equivalent to $\mathcal{H}$ being closed under weak embeddings. In this limiting case, we conclude the following:

\begin{fact}
    Suppose that, for every hereditary class $\mathcal{H}$ defined by a family of forbidden weakly embedded substructures (i.e., which is closed under weak embeddings),  if every theory whose models have age $\mathcal{H}$ has $\mathrm{SOP}_{2}$, every theory whose models have age $\mathcal{H}$ has $\mathrm{SOP}_{3}$. Then every countably categorical $\mathrm{NSOP}_{3}$ theory is $\mathrm{NSOP}_{2}$.
\end{fact}

\begin{proof}(sketch)

    Suppose that there is a countably categorical $\mathrm{NSOP}_{3}$ theory which has $\mathrm{SOP}_{2}$. Let $T_{0}$ be the reduct to the definable predicates $R(x, y)$ and $I(y_{1}, y_{2})$ as in the sketch of the proof of Fact \ref{hereditary class restatement of sop2-sop3 problem}.

Now define a theory $T'$ to be the \textit{weak embedding-model companion} of an $\mathcal{L}$-theory $T$ if it satisfies the following:

    \begin{itemize}
    \item every $\mathcal{L}$-structure that weakly embeds into a model of $T$ weakly embeds into a model of $T'$

    \item every $\mathcal{L}$-structure that weakly embeds into a model of $T'$ weakly embeds into a model of $T$

    \item the theory $T'$ is model complete: if $M \models T'$, $M' \models T'$ and $M \subset M'$ as an induced substructure, then $M \prec M'$.
\end{itemize}

(This will generalize the model companion of a potentially complete first-order theory (as opposed to the more commonly considered model companion of a universal theory) which we will discuss after the statement of Proposition \ref{reduction to generic structures} below.)

By straightforwardly modifying the proof of  Saracino's theorem (\cite{Saracino73}), we can show that every countably categorical theory has a weak embedding-model companion.\footnote{Bodirsky, Bodor and Marimon (\cite{BBM25}) remark that Saracino's theorem can also be modified to apply to their notion of a \textit{core companion}.} Let $T'$ be the weak embedding-model companion of $T_{0}$. Let $\mathcal{H}$ be the age of models of $T'$; then $\mathcal{H}$ will be the hereditary class consisting of finite structures that weakly embed into a model of $T_{0}$, so $\mathcal{H}$ will be closed under weak embedding. Moreover, note that $T_{0}$ is countably categorical and $\mathrm{NSOP}_{3}$.

We first observe that every theory whose models have age $\mathcal{H}$ has $\mathrm{SOP}_{2}$. Because $\mathcal{H}$ is the hereditary class consisting of finite structures that weakly embed into a model of $T_{0}$, every theory whose models have age $\mathcal{H}$ has models that weakly embed exactly the same finite structures that models of $T_{0}$ weakly embed. Now recall that, as we showed in the sketch of the proof of Fact \ref{hereditary class restatement of sop2-sop3 problem}, every theory whose models have the same age as models of $T_{0}$ has $\mathrm{SOP}_{2}$. Similarly, for every theory $T$ whose models weakly embed exactly the same finite structures that models of $T_{0}$ weakly embed, $T$ has $\mathrm{SOP}_{2}$. We conclude that every theory whose models have age $\mathcal{H}$ has $\mathrm{SOP}_{2}$.

It remains to observe that not every theory whose models have age $\mathcal{H}$ has $\mathrm{SOP}_{3}$. Modifying the proof of Fact \ref{Bodirsky, Bodor and Marimon theorem} below, one can show that $\mathrm{NSOP}_{3}$ is preserved under weak-embedding model companions, just as well as under model companions. But then, since $T_{0}$ is $\mathrm{NSOP}_{3}$, its weak embedding-model companion $T'$ is an $\mathrm{NSOP}_{3}$ theory with age $\mathcal{H}$, as desired.
    
\end{proof}

Some preliminary evidence for this section's main result on the question of whether $\mathrm{NSOP}_{2}$ is equal to $\mathrm{NSOP}_{3}$ comes from the analogous result on the question of whether $\mathrm{NSOP}_{n+1}\cap \mathrm{NTP}_{2} = \mathrm{NSOP}_{n} \cap \mathrm{NTP}_{2}$ for integers $n \geq 3$. The proof of this fact, while nontrivial, is easier than the proof of this section's main result, Theorem \ref{main theorem 3, restated}. While previewing some of the background arguments for Theorem \ref{main theorem 3, restated},  we will use the techniques of the previous section. (Like the proof in \cite{ApproxOrder} that $\mathfrak{o}(\mathcal{H})$ is an integer for $\mathcal{H}$ a hereditary class defined by finitely many forbidden weak substructures, the proof of the fact below takes place in the especially well-behaved category of graphs with graph homomorphisms. Yet unlike the proof of the integrality theorem, this fact requires an additional layer of argument: translating $\mathrm{TP}_{2}$ into a $n$-cycle-free graph within this category, as in the previous section. So the proof of our integrality theorem in the real-valued $\mathrm{NSOP}_{r}$ hierarchy really is still the basic foundation, which the arguments of this section are all built on top of.)

\begin{fact}\label{sop_n vs tp_2 in hereditary classes defined by finitely many forbidden weakly embedded substructures}
    Let $\mathcal{H}$ be a hereditary class defined by a finite family of forbidden weakly embedded substructures. Then for $n \geq 3$, if every theory whose models have age $\mathcal{H}$ has $\mathrm{SOP}_{n}$, then either every theory whose models have age $\mathcal{H}$ has $\mathrm{SOP}_{n+1}$, or every theory whose models have age $\mathcal{H}$ has $\mathrm{TP}_{2}$.
\end{fact}

\begin{proof} (sketch)

    Otherwise, similarly to Proposition \ref{reduction to generic structures} below, one of the theories of generic structures $T^{\mathcal{H}}$ of Cherlin, Shelah and Shi (\cite{CSS99}), the complete model companion of the theory of all structures all of whose finite substructures belong to $\mathcal{H}$, is $\mathrm{NSOP}_{n+1}$ and $\mathrm{NTP}_{2}$ but has $\mathrm{SOP}_{n}$. Because $T^{\mathcal{H}}$ has $\mathrm{SOP}_{n}$, there is an indiscernible sequence $\{a_{i}\}_{i < \omega}$ where $\mathrm{tp}(a_{0}, a_{1})$ has no $n$-cycle. Within a Skolemization $(T^{\mathcal{H}})^{\mathrm{Sk}}$ of $T^{\mathcal{H}}$, we may in fact obtain an $\mathrm{acl}^{\mathrm{Sk}}$-closed set $M$, as well as an $M$-indiscernible sequence $\{a_{i}\}_{i \in \mathbb{Z}}$ consisting of $\mathrm{acl}^{\mathrm{Sk}}$-closed sets $a_{i}$ containing $M$, such that $\{a_{i}\}_{i \in \mathbb{Z}}$ is a coheir Morley sequence over $M$ \textit{in both directions}, and $\mathrm{tp}(a_{0}, a_{1}/M)$ has no $n$-cycle. Finding this is standard: we extend $\{a_{i}\}_{i < \omega }$ to $\{\alpha_{i}\}_{i \in \mathbb{Z}} \smallfrown \{a_{i}\}_{i \in \mathbb{Z}} \smallfrown \{\beta_{i}\}_{i \in \mathbb{Z}}$, then take $\mathcal{M} := \mathrm{acl}^{\mathrm{Sk}}(\{\alpha_i\}_{i \in \mathbb{Z}}\{\beta_i\}_{i \in \mathbb{Z}})$. Let $a_{ij} := \mathrm{acl}^{\mathrm{Sk}}(a_{i}a_{j})\backslash (\mathrm{acl}^{\mathrm{Sk}}(a_{i}) \cup \mathrm{acl^{\mathrm{Sk}}}(a_{j}))$ for $i < j < \omega$. Then, as in Lemma \ref{algebraic closures in normal coheir morely sequence}, the $M$, $a_{i}\backslash M$, $a_{ij}$ will all be disjoint. (We perform this same standard argument with the both-directions coheir Morley sequence in the proof of Lemma \ref{standard hereditary class}.)

    We form a hereditary class of directed graphs from this data, similarly to how $\mathcal{H}^{\mathrm{std}}$ will be constructed in Lemma \ref{standard hereditary class}. For a directed graph $G$, let $\tilde{G}$ consist of the disjoint union of:

\begin{itemize}
    \item The model $M \models T^{\mathcal{H}}$ as an induced substructure

    \item For each vertex $v \in G$, a set $\tilde{v}$ satisfying $\mathrm{qftp}(a_{0}\backslash M / M)$

    \item For each directed edge $e=(v, w)$ of $G$, a set $\tilde{e}$ satisfying $\mathrm{qftp}(a_{ij} /a_{i}a_{j}M)$ (so in particular, if $(v, w)$ forms an edge, $\tilde{v}\tilde{w}$ satisfy $\mathrm{qftp}(a_{0}a_{1}/M)$)
\end{itemize}

Moreover, let $\tilde{G}$ be such that there are no more relation instances of the symbols than those necessary to satisfy the above requirements. We let $\mathcal{H}'$ be the hereditary class of directed graphs $G$ such that $\tilde{G} \in \mathcal{H}$. 

Now for $H$ the infinite chain graph consisting of $\omega$ with the order relation, $\tilde{H} \in \mathcal{H}$, as exhibited by the disjoint $M$, $a_{i}\backslash M$, $a_{ij}$. So $H \in \mathcal{H}'$. Moreover, as in the proof of Lemma \ref{standard hereditary class}, $\mathcal{H}'$ is also defined by finitely many forbidden weak substructures. Again as in the proof of Lemma \ref{standard hereditary class}, because $T^{\mathcal{H}}$ is $\mathrm{NSOP}_{n+1}$, $\mathcal{H}'$ is an $\mathrm{NSOP}_{n+1}$ hereditary class\footnote{in the sense of Definition 2.12 of \cite{ApproxOrder}} (and using the arguments from that proof, we can go even farther: the theory of generic structures $T^{\mathcal{H}'}$ is $\mathrm{NSOP}_{n+1}$.)

Now in $T^{\mathcal{H}}$, let $p(x, y) =: \mathrm{tp}(a_{0}, a_{1}/M)$. Then $p(x, y)$ is a type with no $n$-cycles. As in the proof of Theorem \ref{main theorem 2, restated} from the previous section, form the $n$-cycle-free graph $G$ used to show that the model companion of the theory of $n$-cycle-free directed graphs is $\mathrm{TP}_{2}$. Then because $T^{\mathcal{H}}$ is $\mathrm{NTP}_{2}$, the type-definable relation over $M$ given by $p(x, y)$ does not weakly embed the graph $G$. (As in the previous section, an instance of $G$ for a relation with no $n$-cycles would give an instance of $\mathrm{TP}_{2}$.) Now if $\tilde{G} \in \mathcal{H}$, by model completeness of $T^{\mathcal{H}}$, we would get a weakly embedded copy of $G$ in the type-definable relation $p(x, y)$. So $\tilde{G} \notin \mathcal{H}$, and $G \notin \mathcal{H}'$. By Fact \ref{n-cycle-free to small cycle-free}, since $G \notin \mathcal{H}'$, $G$ is $n$-cycle free and $\mathcal{H}$ is $\mathrm{NSOP}_{n}$, $\mathcal{H}'$ omits some $N$-cycle-free directed graph for arbitrarily large $N$. Then just as in Claim 3.16 of \cite{ApproxOrder} and the preceding discussion there, $H \notin \mathcal{H}$, a contradiction.

\end{proof}

We get an analogous statement for the problem of whether the real-valued $\mathrm{NSOP}_{r}$ hierarchy is distinct from the integer-valued $\mathrm{NSOP}_{n}$ hierarchy. This statement resembles Theorem 3.1 of \cite{ApproxOrder} on the combinatorial quantity $\mathfrak{o}(\mathcal{H})$, though it is now a model-theoretic result rather than a statement about intrinsic combinatorial properties of hereditary classes.

\begin{fact}
    Let $\mathcal{H}$ be a hereditary class defined by a finite family of forbidden weakly embedded substructures. Then for any non-integer values $r > 2$ and $n = \lceil r \rceil$ the next integer, if every theory whose models have age $\mathcal{H}$ has $\mathrm{SOP}_{r}$, every theory whose models have age $\mathcal{H}$ has $\mathrm{SOP}_{n}$
\end{fact}

\begin{proof}
   (sketch)
    
    The new setup compared to Theorem 3.1 of \cite{ApproxOrder} is analogous to the proof of the previous fact. Let $\{a_{i}\}_{i < \omega}$ be an indiscernible sequence for which $\mathrm{tp}(a_{0}a_{1})$ exhibits $\mathrm{SOP}_{r}$ for $T^{\mathcal{H}}$, where $T^{\mathcal{H}}$ is assumed to be $\mathrm{NSOP}_{n}$ and $\mathcal{H}$ is defined by finitely many forbidden weak substructures. Then we form the hereditary class $\mathcal{H}'$, consisting of directed graphs, just as before. The hereditary class $\mathcal{H}'$ will be defined by finitely many forbidden weak substructures, and will be $\mathrm{NSOP}_{n}$, but the graph relation will exhibit $\mathrm{SOP}_{r}$. Proceeding as in the proof of Theorem 3.1 of \cite{ApproxOrder}, we get a contradiction.
\end{proof}

These two previous facts suggest that we may be able to obtain the corresponding result for the $\mathrm{NSOP}_{2}$-$\mathrm{NSOP}_{3}$ problem. Our goal for the rest of the main part of this section\footnote{besides a final remark, Remark \ref{potential simplified proof}, sketching a hypothetical simplification of the proof of our main result under some conjectural conditions on algebraic closures within $\mathrm{SOP}_{2}$.} will be to give the proof of this dichotomy between $\mathrm{SOP}_{2}$ and $\mathrm{NSOP}_{3}$, which will be harder than the proofs sketched for the two previous facts. This proof will generalize the techniques used to show that $\mathfrak{o}(\mathcal{H})$ is an integer when $\mathcal{H}$ is a hereditary class defined by finitely many forbidden weak substructures, Theorem 3.1 of \cite{ApproxOrder}, within the context of the properties $\mathrm{NSOP}_{r}$ theories for $r$ a real number. There, we used an abstract cycle-removal property of helix maps within the category of directed graphs with graph homomorphisms, Proposition 3.11 of \cite{ApproxOrder}, to produce graphs omitted from $\mathcal{H}$ without any small directed cycles. We then ultimately produced graphs omitted from $\mathcal{H}$ with no directed cycles. Though, informally speaking, the category we will be working in here will be much less well-behaved, we will still be able to obtain more ad hoc versions of these abstract cycle-removal properties (Lemma \ref{split cycle removal}, Lemma \ref{consanguineous pair removal} and Lemma \ref{potentially pinched alternating cycle removal} below). As with our proof that $\mathfrak{o}(\mathcal{H})$ is an integer, this will allow us to obtain graphs omitted from a strategically chosen hereditary class ($\mathcal{H}^{\mathrm{std}}$ of Lemma \ref{standard hereditary class}) without certain kinds of cycles: the \textit{split cycles} of the below Definition \ref{split cycle}, and the \textit{potentially pinched alternating cycles} of the below Definition \ref{potentially pinched alternating cycle, flat}. Proving the third and last main theorem of this paper, Theorem \ref{main theorem 3} from the introduction, we will show:

\begin{theorem}
\label{main theorem 3, restated}   Let $\mathcal{H}$ be a hereditary class defined by a finite family of forbidden weakly embedded substructures. Then if every theory whose models have age $\mathcal{H}$ has $\mathrm{SOP}_{2}$, every theory whose models have age $\mathcal{H}$ has $\mathrm{SOP}_{3}$.
\end{theorem}

Because we already know that the class of simple theories does not coincide with the class of $\mathrm{SOP}_{3}$ theories, in order to motivate the dichotomy in Theorem \ref{main theorem 3, restated} by showing that it is a sharp dichotomy, we must also prove Observation \ref{main observation} from the introduction as a companion observation to this theorem. This observation states that the statement of Theorem \ref{main theorem 3, restated} is false when replacing the properties $\mathrm{SOP}_{2}$ and $\mathrm{SOP}_{3}$ with the properties $\mathrm{TP}$ and $\mathrm{SOP}_{3}$, the tree property $\mathrm{TP}$ being a failure of simplicity.

\begin{observation}
 \label{main observation, restated}   It is \emph{not} the case that, for $\mathcal{H}$ a hereditary class defined by a finite family of forbidden weakly embedded substructures, if every theory whose models have age $\mathcal{H}$ is non-simple, every theory whose models have age $\mathcal{H}$ has $\mathrm{SOP}_{3}$.
\end{observation}

We will start by proving this observation. This will be a consequence of work of Bodirsky, Bodor and Marimon (\cite{BBM25}, Theorem 3.22 and paragraph before Subsection 3.1)--that simplicity and $\mathrm{NSOP}_{n}$ are preserved under taking model companions--as well as examples due to Conant and Kruckman (\cite{CoK19}), or Kruckman and Ramsey (\cite{KR18}). The result of Bodirsky, Bodor and Marimon will be discussed below, in Fact \ref{Bodirsky, Bodor and Marimon theorem}. The example of Conant and Kruckman (\cite{CoK19}) gives us (upon obtaining the age of a model of their theory) a two-sorted hereditary class $\mathcal{H}$ exemplifying Observation \ref{main observation, restated}. This will be the hereditary class in sorts $P$ and $L$ with incidence relation $I$, defined by omitting the complete bipartite graph $K_{n,m}$ as a forbidden weak substructure. A contrived variation on this example of Conant and Kruckman will give us a one-sorted example. However, for an informally more natural one-sorted example, we consider the example given by Kruckman and Ramsey (\cite{KR17}) (again, upon obtaining the age of a model of one of the main theories discussed there, the theory of the generic binary function). This will be the hereditary class of binary partial functions, as understood in a relational language; i.e., it will be the hereditary class of graphs of binary partial functions.

\begin{proof}
    (of Observation \ref{main observation, restated}): The proof of this observation is related to Proposition \ref{reduction to generic structures} below, which states that Theorem \ref{main theorem 3, restated} will hold true as long as none of the theories of generic structures of Cherlin, Shelah and Shi (\cite{CSS99})--model companions of $T^{\mathcal{H}}_{0}$, the theory of structures all of whose finite substructures belong to $\mathcal{H}$, where $\mathcal{H}$ is a hereditary class defined by a finite family of forbidden weakly embedded substructures--can be $\mathrm{NSOP}_{3}$ but have $\mathrm{SOP}_{2}$. Conversely, our proof will essentially proceed by showing that Observation \ref{main observation, restated} follows from the fact that there \textit{is} a theory of generic structures as in Cherlin, Shelah and Shi that is non-simple but $\mathrm{NSOP}_{3}$, and exhibiting such a theory.

    Let $\mathcal{L}$ be the one-sorted language with one ternary relation symbol $R(x, y, z)$, and let $\mathcal{F}= \{F_1, F_2\}$, where $F_{1}$ is the structure consisting of distinct $a, b, c_1, c_{2}$ whose instances of $R(x, y, z)$ are $R(a, b, c_1)$ and $R(a, b, c_2)$, and $F_{2}$ is the structure consisting of distinct $a, c_1, c_2$ whose instances of $R(x, y, z)$ are $R(a, a, c_1)$ and $R(a, a, c_2)$. Define $\mathcal{H}=:\mathcal{H}(\mathcal{F})$; this will consist of finite sets with a graph for a binary partial function from that set to itself. Then $T^{\mathcal{H}}_{0}$, the theory of $\mathcal{L}$-structures all of whose finite substructures belong to $\mathcal{H}$ as in the below Fact \ref{Cherlin Shelah and Shi theorem}, is the theory of sets with a graph for a binary partial function from that set to itself. Kruckman and Ramsey (\cite{KR17}, Corollary 3.10) show that the theory $T^{f}_{0}$ of binary functions has a model companion $T^{f}$; we may as well view $T^{f}_{0}$ as a theory in the language $\mathcal{L} = \{ R(x, y, z) \}$, where $R(x, y, z)$ represents the graph of the function. (As a theory in a relational language, $T^{f}_{0}$ will not be a universal theory; see Definition \ref{general model companion} below for the definition of the model companion of theory that is not necessarily universal.)  Since any model of $T^{\mathcal{H}}_{0}$ embeds into a model of $T^{f}_{0}$ (any structure consisting of a set with a binary partial function extends to a structure consisting of a set with a binary function), while every model of $T_{0}^{f}$ is a model of $T^{\mathcal{H}}_{0}$, the model companion $T^{f}$ of $T_{0}^{f}$ will be the model companion $\mathcal{T}^{\mathcal{H}}$ of $T^{\mathcal{H}}_{0}$. Kruckman and Ramsey show, however, that $T^{f} = T^{\mathcal{H}}$ is a complete $\mathrm{NSOP}_{1}$ theory, but has $\mathrm{TP}$. So there will be a complete $\mathrm{NSOP}_{3}$ theory, $T^{\mathcal{H}}$, whose models will have age $\mathcal{H}$. However, for any complete theory $T$ with models whose age is $\mathcal{H}$, the model companion of $T$ will be $T^{\mathcal{H}}$, and because $T^{\mathcal{H}}$ has $\mathrm{TP}$, by the below Fact \ref{Bodirsky, Bodor and Marimon theorem} that simplicity is preserved under model companions, $T$ will have $\mathrm{TP}$. So $\mathcal{H}$ will in fact exemplify Observation \ref{main observation, restated}.
\end{proof}

Having motivated Theorem \ref{main theorem 3, restated} by proving Observation \ref{main observation, restated} (subject to Fact \ref{Bodirsky, Bodor and Marimon theorem} below), we will now begin the setup for the proof of this theorem. We will begin, as described in the introduction, by again applying the result essentially due to \cite{BBM25} on preservation of $\mathrm{NSOP}_{n}$ under model companions. Using this preservation theorem, we will reducing Theorem 5.5 to a verification for a concrete family of theories--particularly, the theories of the generic structures forbidding finitely many forbidden weakly embedded substructures, defined by Cherlin, Shelah and Shi (\cite{CSS99}). Since Cherlin, Shelah and Shi explicitly prove that these theories are well-defined only for graphs, we recall the proof of the well-known fact that these theories are well-defined for relational structures more generally, though the argument will be the same.

\begin{fact}
  \label{Cherlin Shelah and Shi theorem} \emph{(Corollary to Theorem 1 of \cite{CSS99})} Let $\mathcal{L}$ be a finite relational language, and $\mathcal{H}$ be a hereditary class of $\mathcal{L}$-structures defined by a finite family of forbidden weakly embedded substructures. Let $T_{0}^{\mathcal{H}}$ be the (universal) theory of $\mathcal{L}$-structures all of whose finite substructures are members of $\mathcal{H}$. Then $T^{\mathcal{H}}_{0}$ has a model companion, $T^{\mathcal{H}}$.
\end{fact}

\begin{proof}
Let $\mathcal{H} = \mathcal{H}(\mathcal{F})$ for $\mathcal{F}$ a finite family of forbidden weakly embedded substructures. Then, let $A \subset B \in \mathcal{H}$, with $B$ finite. We begin by showing that there is a finite set $F_{B/A}$ of finite $\mathcal{L}$-structures containing $A$ such that, for any $\mathcal{L}$-structure $C \supset A$ and $D= B\sqcup_{A} C $ the free amalgam of $B$ and $C$ over $A$, $D \notin \mathcal{H}$ if and only if $C$ contains a member of $F_{B/A}$.\footnote{More technically, we show that there is a finite set $F_{B/A}$ of \textit{isomorphism classes over $A$} of finite $\mathcal{L}$-structures containing $A$ satisfying this condition, where by a member of $F_{B/A}$ we mean a structure belong to one of the isomorphism classes in $F_{B/A}$; this is also what we will mean by a member of $F_{B/A}$ throughout what follows.} Let $n$ be larger than the cardinality of any structure in $\mathcal{F}$. Then we define $F_{B/A}$ as follows: $F_{B/A}$ is the set of $\mathcal{L}$-structures $C' \supset A$ with $C'\leq |A|+ n$ such that, for $D = B\sqcup_{A} C' $ the free amalgam of $B$ and $C'$ over $A$, $D \notin \mathcal{H}$. We show $F_{B/A}$ is as desired.

For the ``if" direction, if $C$ contains some $C' \in F_{B/A}$, then $D= B\sqcup_{A} C $ will contain $D'= B\sqcup_{A} C' \notin \mathcal{H} $. So $D \notin \mathcal{H}$, proving this direction.

For the ``only if" direction, if $D= B\sqcup_{A} C \notin \mathcal{H}$, then there is some $F \subset B\sqcup_{A} C$, with $F$ an image of a member of $\mathcal{F}$ under a weak embedding, such that $|F| < n$ and $F \notin \mathcal{H}$. But then, for $C' = (F \cap C) \cup A$, $C' \in F_{B/A}$: first, $F \subset B\sqcup_{A} C'$, where $B\sqcup_{A} C' \subset B\sqcup_{A} C$ is the free amalgam of $B$ and $C'$ over $A$, and $F \notin \mathcal{H}$, so $B\sqcup_{A} C' \notin \mathcal{H}$. But $A \subset C' $, and $|C'| \leq |A| + n$ because $|F| < n$, so $C' \in F_{B/B}$. Since $C' \subset C$, this proves the ``only if" direction.

Then the following condition on $M \models T^{\mathcal{H}}_{0}$ is axiomatizable by a set of sentences in $\mathcal{L}$, since $F_{B/A}$ is always finite:

(*) Let $A \subset M$ be a finite, and let $B \supset A$, $B \in \mathcal{H}$ also be finite. Then $B$ embeds into $M$ over $A$ as an induced substructure, if no member of $F_{B/A}$ embeds into $M$ over $A$ as an induced substructure.

It therefore suffices to show that a model of $T^{\mathcal{H}}_{0}$ is existentially closed if and only if it satisfies condition (*).

First, suppose $M \models T^{\mathcal{H}}_{0}$ is existentially closed; we show that $M$ satisfies condition (*). Let $D=B \sqcup_{A} M$ be the free amalgam of (an $A$-isomorphic copy of) $B$ with $M$ over $A$. If no member of $F_{B/A}$ embeds into $M$ over $A$ as an induced substructure, then $D \models T^{\mathcal{H}}_{0}$ by choice of $F_{B/A}$, and $M \subset D$. So $B$ embeds into $M$ over $A$ by existential closedness of $M$.

Now suppose that $M\models T^{\mathcal{H}}_{0}$ satsifies condition (*); we show that $M$ is existentially closed. It suffices to show that if $M \subset M' \models T_{0}^{\mathcal{H}}$, $A \subset M$ is finite, and $A \subset B \subset M'$ where $B$ is finite, then $B$ (note $B \in \mathcal{H}$) embeds into $M$ over $A$. This will hold if $B$ embeds into $M$ over $B \cap M \supset A$, so we may assume that $B\cap M = A$.  Now, $B \cup M \models T_{0}^{\mathcal{H}} $ because $B \cup M \subset M'$. So, because $B \cap M = A$, the free amalgam $B \sqcup_{A} M$ of $B$ and $M$ over $A$ weakly embeds into $B \cup M $. By choice of $F_{B/A}$, this implies that no member of $F_{B/A}$ embeds into $M$ over $A$ as an induced substructure. Therefore, applying (*), $B$ embeds into $M$ over $A$.

We have thus proven both directions of the equivalence between (*) and existential closedness of $M$; therefore, that the class of existentially closed models of $T_{0}^{\mathcal{H}}$ is axiomatizable by a theory $T^{\mathcal{H}}$, as desired.

\end{proof}

Having defined the theories of generic structures of Cherlin, Shelah and Shi (\cite{CSS99}), our next goal will be to justify (again, essentially as a corollary of \cite{BBM25}) the reduction of Theorem \ref{main theorem 3, restated} to a concrete verification: that none of these theories of (\cite{CSS99}) are strictly $\mathrm{NSOP}_{3}$. We will observe:

\begin{prop}

\label{reduction to generic structures}

Suppose that none of the theories of Cherlin, Shelah and Shi, the model companions $T^{\mathcal{H}}$ for $\mathcal{H}$ defined by a finite family of forbidden weakly embedded substructures as in Fact \ref{Cherlin Shelah and Shi theorem}, are complete strictly $\mathrm{NSOP}_{3}$ theories (i.e., $\mathrm{NSOP}_{3}$, but with $\mathrm{SOP}_{2}$). Then for every hereditary class $\mathcal{H}$ defined by a finite family of forbidden weakly embedded substructures, if the theory of every structure whose age is $\mathcal{H}$ has $\mathrm{SOP}_{2}$, then the theory of every structure whose age is $\mathcal{H}$ has $\mathrm{SOP}_{3}$.

\end{prop}

The argument for this will involve model companions for complete theories. Because the reader may be accustomed to only considering model companions of universal theories, rather than complete theories, we review the definition of model companions in the greatest generality.

\begin{definition}
\label{general model companion}

Let $T$ be any theory; then $T'$ is the \textit{model companion} of $T$ if:

\begin{itemize}
    \item every model of $T$ embeds as an induced substructure into a model of $T'$

    \item every model of $T'$ embeds as an induced substructure into a model of $T$

    \item the theory $T'$ is model complete (i.e., if $M \models T'$, $M' \models T'$ and $M \subset M'$ as an induced substructure, then $M \prec M'$).
\end{itemize}

\end{definition}

To show Proposition \ref{reduction to generic structures}, we apply the following fact due to Bodor, Bodirsky and Marimon (\cite{BBM25}). To make the exposition more self-contained, we present a direct version of their proof.

\begin{fact}\emph{(Bodirsky, Bodor and Marimon, \cite{BBM25}, Theorem 3.22 and paragraph before Subsection 3.1.)}

\label{Bodirsky, Bodor and Marimon theorem}

Suppose $T$ is a (resp.) simple, $\mathrm{NSOP}_{2}$, or $\mathrm{NSOP}_{3}$ complete theory, and that its model companion $T'$ is well-defined. Then $T'$ is a simple, $\mathrm{NSOP}_{2}$, or $\mathrm{NSOP}_{3}$ complete theory.

\end{fact}

\begin{proof}
   Fix any one of the three properties, between the tree property, $\mathrm{SOP}_{2}$, and $\mathrm{SOP}_{3}$. Then there is some set $S$, family $\mathcal{P} \subset 2^{S}$ of subsets of $S$, and family $\mathcal{N} \subset 2^{S}$ of subsets of $S$ of cardinalty $2$, satisfying the following: a theory has our chosen property if and only if there is a formula $\varphi(x, y)$, and $\{b_{i}\}_{i \in S}$, such that $\{\varphi(x, b_{i})\}_{i \in C}$ is consistent for each $C \in \mathcal{P}$, but $\{\varphi(x, b_{i}), \varphi(x, b_{j}) \}$ is inconsistent for $\{i, j\} \in \mathcal{N}$.  (This is a special case of the \textit{positive straight definability} of Bailetti, \cite{Bailetti25}.) Now suppose $T'$ has our chosen property, exhibited by $\varphi(x, y)$; our aim is to show that $T$ has this property. Since $T'$ is model-complete, so, as is well-known, every formula in $T'$ is equivalent in $T'$ to an existential formula, we may assume $\varphi(x, y) := \exists z \varphi_{0}(x, y, z)$ for $\varphi_{0}(x, y, z)$ a quantifier-free formula. We show that the same formula $\varphi(x, y) := \exists z \varphi_{0}(x, y, z)$ exhibits our chosen property in $T$.

   Again using the same well-known consequence of model completeness of $T'$, find some quantifier-free formula $\psi_{0}(w, y_{1}, y_{2})$ such that $\psi(y_{1}, y_{2}) := \neg \exists x \varphi(x, y_{1}) \wedge \varphi(x, y_{2})$ is equivalent to  $\exists w \psi_{0}(w, y_{1}, y_{2})$ in $T'$. We claim that, in $T$, for $\mathbb{M}$ the sufficiently saturated ambient model of $T$ and $b, c \in \mathbb{M}$, if $b, c \models \exists w \psi_{0}(w, y_{1}, y_{2})$, then $\{\varphi(x, b), \varphi(x, c) \}$ is inconsistent. Otherwise, the quantifier-free formula $\varphi_{0}(x, b, z) \wedge \varphi_{0}(x, c, z) \wedge \psi_{0}(b, c, w)$ will be realized in $\mathbb{M}$, so $\varphi_{0}(x, y_{1}, z) \wedge \varphi_{0}(x, y_{2}, z) \wedge \psi_{0}(y_{1}, y_{2}, w)$ will be realized in $\mathbb{M}$. But then, in $T'$, for $\mathbb{M}' \models T'$ the sufficiently saturated ambient model of $T'$\footnote{We may in fact speak of the sufficiently saturated ambient model $\mathbb{M}' \models T'$: note that $T'$ is complete (as desired in the statement of this fact), as is any model companion of a complete theory. We review the proof of this well-known fact. It suffices to show that for $M_{1}, M_{2}$ models of two completions of $T'$, $M_{1}$ and $M_{2}$ elementarily embed into some common model of $T'$, so those completions agree. Since $T'$ is model complete, it even suffices to show that $M_{1}$, $M_{2}$ embed as induced substructures into some common model of $T'$. But $M_{1}$ and $M_{2}$ both embed into models of $T$, so because $T$ is complete they both embed into any $|M_{1}|+|M_{2}|$-saturated $M \models T$. But $M$ embeds into a model of $T'$, so $M_{1}$ and $M_{2}$ both embed into a model of $T'$, as desired.}, $\varphi_{0}(x, y_{1}, z) \wedge \varphi_{0}(x, y_{2}, z) \wedge \psi_{0}(y_{1}, y_{2}, w)$ will be realized in $\mathbb{M}'$, because $\mathbb{M'} \models T'$ realizes all of the same quantifier-free types as $\mathbb{M} \models T$. (Both $T$ and $T'$ have the same universal theory, because any model of $T$ embeds into a model of $T'$ and vice versa. And models of two complete theories with the same universal theory realize all of the same quantifier-free types (over $\emptyset$).) So $\exists z \varphi_{0}(x, y_{1}, z) \wedge \exists z' \varphi_{0}(x, y_{2}, z') \wedge \exists w \psi_{0}(y_{1}, y_{2}, w)$, and thus $\varphi(x, y_{1}) \wedge \varphi(x, y_{2}) \wedge \psi(y_{1}, y_{2})$, is realized in $\mathbb{M}'$. This contradicts that $\psi(y_{1}, y_{2}):=\neg \exists x \varphi(x, y_{1}) \wedge \varphi(x, y_{2})$. So we have proven that, in $T$,  $b, c \models \exists w \psi_{0}(w, y_{1}, y_{2})$ implies that $\{\varphi(x, b), \varphi(x, c) \}$ is inconsistent for $b, c \in \mathbb{M} \models T$.

Now, in $T'$, let $\{b'_{i}\}_{i \in S} \in \mathbb{M}' \models T'$ exhibit the chosen property for $\varphi(x, y) := \exists z \varphi_{0}(x, y, z)$.   Then for $\{i, j\} \in \mathcal{N}$, $(b'_{i},b'_{j}) \models \exists w \psi_{0}(y_1, y_2, w)$, because $\psi(y_{1}, y_{2}):=\neg \exists x \varphi(x, y_{1}) \wedge \varphi(x, y_{2})$ and $\exists w \psi_{0}(w, y_{1}, y_{2})$ is equivalent to $\psi(y_{1}, y_{2})$. And for $C \in \mathcal{P}$, we may find $a'_{c} \models \{\exists z \varphi_{0}(x, b'_i, z)\}_{i \in C}$.

Then the quantifier-free type

$$\bigcup_{C \in \mathcal{P}}  \{ \varphi_{0}(x^{C}, y_i, z^{C, i})\}_{i \in C} \cup \bigcup_{\{i, j\} \in \mathcal{N}} \{\psi_{0}(y_i, y_j, w^{ij})\}$$

is realized in $\mathbb{M}' \models T'$. Therefore, because $\mathbb{M}$, again, realizes the same quantifier-free types as $\mathbb{M}'$, this type is realized in $\mathbb{M}$. Let $\{b_{i}\}_{i \in S} \subset \mathbb{M}$ instantiate the $\{y_{i}\}_{i \in S}$, and we show that $\{b_{i}\}_{i \in S} \subset \mathbb{M}$ exhibits our chosen property with $\varphi(x, y) := \exists z \varphi_{0}(x, y, z)$. The set $\{\varphi(x, b_{i})\}_{i \in C}$ is consistent for each $C \in \mathcal{P}$: it is realized by the parameter instantiating $x^{C}$ in this realization. And for $\{i, j\} \in \mathcal{N}$, since $b_{i}, b_{j} \models \exists w \psi_{0}(w, y_{1}, y_{2})$, $\{\varphi(x, b_{i}), \varphi(x, b_{j}) \}$ will, as we have shown, be inconsistent. So $\varphi(x, y)$, $\{b_{i}\}_{i \in S}$ exhibit our chosen property in $T$, as we wanted.
\end{proof}

As a corollary of this fact, we now prove the central reduction, Proposition \ref{reduction to generic structures} above.

\begin{proof}
    (of Proposition \ref{reduction to generic structures}) We may assume that $\mathcal{H}$ is the age of some structure. Then $T^{\mathcal{H}}$, which is well-defined by Fact \ref{Cherlin Shelah and Shi theorem}, will be complete, and its models will have age $\mathcal{H}$. (This is well-known: since $\mathcal{H}$ is the age of some structure, it has the joint embedding property: any two members of $\mathcal{H}$ embed as induced substructures into some common member of $\mathcal{H}$. So by compactness, any two models of $T^{\mathcal{H}}$ embed into some common model of $T_{0}^{\mathcal{H}}$, and then into some common model of $T^{\mathcal{H}}$. By model completeness, any two models of $T^{\mathcal{H}}$ then elementarily embed into some common model, so have the same theory (as in the previous footnote), and $T^{\mathcal{H}}$ is complete. And, for any $M \models T^{\mathcal{H}}$, a finite structure will embed as an induced substructure of $M$ if and only if it belongs to $\mathcal{H}$, so the age of any model of $T^{\mathcal{H}}$, by completeness of $T^{\mathcal{H}}$, is $\mathcal{H}$. ) Then suppose that the theory of every structure for which $\mathcal{H}$ is the age has $\mathrm{SOP}_{2}$; in that case, $T^{\mathcal{H}}$ will have $\mathrm{SOP}_{2}$. By the supposition that none of the theories of Cherlin, Shelah and Shi (\cite{CSS99}) are strictly $\mathrm{NSOP}_{3}$, $T^{\mathcal{H}}$ must have $\mathrm{SOP}_{3}$. Now let $T$ be a complete theory whose models have age $\mathcal{H}$. By compactness, for any two complete theories $T_{1}$, $T_{2}$ whose models have the same age, any model of $T_{1}$ embeds into a model of $T_{2}$ and vice versa. So applying this, in particular, to $T$ and $T^{\mathcal{H}}$, $T^{\mathcal{H}}$ must be the model companion of $T$ because $T^{\mathcal{H}}$ is model complete. So by Fact \ref{Bodirsky, Bodor and Marimon theorem}, $T$ has $\mathrm{SOP}_{3}$ because $T^{\mathcal{H}}$ does. We have thus shown that every complete theory whose models have age $\mathcal{H}$ has $\mathrm{SOP}_{3}$, as desired.
\end{proof}

\begin{remark}
  \label{Bodirsky, Bodor and Marimon schema} In the two paragraphs before Question 1.1 of  \cite{BBM25}, Bodirsky, Bodor and Marimon propose a schema for analyzing the \textit{constraint satisfaction problem} $\mathrm{CSP}(\mathfrak{B})$, for structures $\mathfrak{B}$ from some class of structures $\mathcal{C}$, by taking the \textit{core companion} of a structure in this class and reducing the analysis to the case where $\mathfrak{B}$ is a \textit{model complete core}. They also mention doing this with model companions instead of core companions. We observe how Proposition \ref{reduction to generic structures} fits into Bodirsky, Bodor and Marimon's schema, where we want to analyze the \textit{age} $\mathrm{Age}(\mathfrak{B})$ of a structure $\mathfrak{B}$ from some class $\mathcal{C}$, rather than analyzing $\mathrm{CSP}(\mathfrak{B})$. With model companions rather than core companions, Bodirsky, Bodor and Marimon's schema requires that:

  \begin{itemize}
      \item every structure in $\mathcal{C}$ has a model companion, and
      \item the model companion of every structure in $\mathcal{C}$ is again in $\mathcal{C}$.
  \end{itemize}

Now let the class $\mathcal{C}$ be the class of structures with $\mathrm{NSOP}_{3}$ theory whose age is a hereditary class defined by a finite family of forbidden weakly embedded substructures. To prove the main theorem of this section, Theorem \ref{main theorem 3, restated}, what we want to know about $\mathrm{Age}(\mathfrak{B})$ for $\mathfrak{B} \in \mathcal{C}$ is that it cannot be the case that every complete theory whose models have $\mathrm{Age}(\mathfrak{B})$ has $\mathrm{SOP}_{2}$. The first bullet point, the existence of a model companion for any structure whose age is a hereditary class defined by a finite family of forbidden weakly embedded substructures, is Fact \ref{Cherlin Shelah and Shi theorem}. The second bullet point, the preservation of $\mathrm{NSOP}_{3}$ under model companions, is Fact \ref{Bodirsky, Bodor and Marimon theorem}. So we have now reduced to showing our conclusion about $\mathrm{Age}(\mathfrak{B})$ when $\mathfrak{B}$ is a model complete member of $\mathcal{C}$, so a model of a complete $\mathrm{NSOP}_{3}$ example of one of the theories of Cherlin, Shelah and Shi (\cite{CSS99}). This will follow from showing that none of the complete theories of Cherlin, Shelah and Shi can be strictly $\mathrm{NSOP}_{3}$.

In these same paragraphs of \cite{BBM25}, Bodirsky, Bodor and Marimon also propose to use core companions, or model companions, to analyze the existential theory of a structure. The role of existential formulas in our approach to proving Theorem \ref{main theorem 3, restated} is different. In the proposal from \cite{BBM25}, one is interested in some problem which is explicitly about the existential theory of a structure $\mathfrak{B}$, and one investigates the model companion of $\mathfrak{B}$ because it has the same existential theory as $\mathfrak{B}$. However, our assumption in Theorem \ref{main observation, restated} that every theory whose models have age $\mathcal{H}$ has $\mathrm{SOP}_{2}$, and the conclusion that every such theory has $\mathrm{SOP}_{3}$, refers to $\mathrm{SOP}_{2}$ and $\mathrm{SOP}_{3}$ in full first-order logic and not just the existential formulas. Rather, in reducing to showing that the complete theories of Cherlin, Shelah and Shi are not strictly $\mathrm{NSOP}_{3}$, we are reducing to the case where we are \textit{now} only concerned with the existential formulas--because every formula in a theory in this concrete family is equivalent to an existential one. 

\end{remark}

So our goal, for showing Theorem \ref{main theorem 3, restated}, has now become to show that (the complete examples of) the theories of Cherlin, Shelah and Shi (\cite{CSS99}), the complete theories given by model companions $T^{\mathcal{H}}$ where $\mathcal{H}$ is defined by a finite family of forbidden weakly embedded substructures (as in Fact \ref{Cherlin Shelah and Shi theorem}), are never strictly $\mathrm{NSOP}_{3}$. Our initial goal will be to obtain, from a hypothetical strictly $\mathrm{NSOP}_{3}$ theory of generic structures as in \cite{CSS99}, a more combinatorially tractable object whose existence we will ultimately prove impossible using our cycle-removal methods. We will give this object in Lemma \ref{standard hereditary class}, our bridge from the model theory into the combinatorics.  To show this lemma, we will first prove some subtle but crucial technical lemmas about $\mathrm{SOP}_{2}$ in the general setting. These lemmas will ultimately state that, given $\mathrm{SOP}_{2}$, an instance of the \textit{tree property} $\mathrm{TP}$ and its consistency witnesses can be chosen so that certain algebraic closures do not overlap nontrivially (Lemma \ref{noninterference of algebraic closures} below).\footnote{From here to the end of Remark \ref{noninterference in tree property}, we will be working in the setting of an abstract first-order theory $T$, so we will adopt the standard model-theoretic notation for this setting where the sets we refer to are small subsets of an ambient sufficiently saturated model $\mathbb{M} \models T$, and models $M$ refer to small elementary submodels $M \prec \mathbb{M}$. We will necessarily speak more explicitly about ambient models when we again allow ourselves to refer to more than one theory over the course of the same proof, particularly in the proof of theorem \ref{standard hereditary class}.} (See Lemma 5.14 of \cite{AKLL25} for an analogous statement whose proof includes some similar techniques, allowing an instance of $\mathrm{ATP}$ with consistency witnesses to be chosen such that certain algebraic closures do not intersect nontrivially.)

An instance of $\mathrm{SOP}_{2}$ will in particular give us an instance of $\mathrm{SOP}_{1}$, so as in Proposition 3.14 of \cite{KR17} will give us a formula $\varphi(x, b)$, and coheir Morley sequences $\{b_{i}\}_{i< \omega}$, $\{b^{i}\}_{i< \omega}$ over a model $M$ starting with $b$, such that $\{\varphi(x, b_i)\}_{i < \omega}$ is consistent, but $\{\varphi(x, b^i)\}_{i < \omega}$ is inconsistent. (In fact, the proof of that proposition in \cite{KR17} says more: $\{\varphi(x, b^i)\}_{i < \omega}$ is in fact $2$-inconsistent, meaning that any pair of distinct formulas in the set is inconsistent; see Fact \ref{sop1 coheirs} above.)  So our next lemma, our first in producing an instance of $\mathrm{SOP}_{2}$ with the desired nonoverlapping algebraic closures (in Lemma \ref{noninterference of algebraic closures}), will be relevant to this scenario: it will state that, given a pair of coheir Morley sequences over a model beginning with the same term, we can enlarge the terms of both of these coheir Morley sequences such that their initial terms continue to agree, and such that each of the two individual coheir Morley sequences now satisfy desirable conditions on algebraic closures. Recall that given a global $M$-coheir $p(x)$, a \textit{coheir Morley sequence} over $M$ in $p(x)$ is a sequence $\{b_i\}_{i < \omega}$ such that for $i < \omega$, $b_{i} \models p(x)|_{Mb_{< i}}$.

\begin{lemma}
\label{pair of normal coheirs}    Let $M$ be a model and $p(x), q(x)$ global $M$-coheirs with $p|_{M} = q|_M$. Then, for (potentially infinite) variables $y$ extending $x$, we can find global $M$-coheirs $\tilde{p}(y), \tilde{q}(y)$ with $\tilde{p}|_{M} = \tilde{q}|_{M}$, extending $p(x)$ and $q(x)$ respectively, such that for $\{b_{i}\}_{i < \omega}$ a coheir Morley sequence over $M$ in $p(y)$ or $q(y)$ and $n < \omega$, $\mathrm{acl}_{M}(b_{0} \ldots b_{n-1} \cup b_{2n}) \cap \mathrm{acl}_{M}(b_{n} \ldots b_{2n-1} \cup b_{2n})=\mathrm{acl}_{M}(b_{2n})$.
\end{lemma}

\begin{proof}
    First of all, we may assume that the realizations of $p|_{M}(x)=q|_{M}(x)$ contain an $|M|^{+}$-saturated elementary extension of $M$. This just uses the fact (``left extension for finite satisfiability") that, for $r(x)$ a (global) $M$-coheir, $x \subset x'$ and $r'_{0}(x') \in S(M)$ extending $r(x)|_{M}$, $r(x)$ extends to a global $M$-coheir $r'(x')$ extending $r'_{0}(x')$. (Specifically, we apply this fact setting $r(x)$ to $p(x)$, and setting $r(x)$ to $q(x)$, while setting $r'_{0}(x')$ in both cases to some fixed arbitrary choice of extension of $p|_{M}(x)=q|_{M}(x)$ whose realizations contain a $|M|^{+}$-saturated elementary extension of $M$.)
    
    So, having made this assumption about $p|_{M}(x)=q|_{M}(x)$, we obtain the following background assumption as a consequence: if $x \subset x_{1}$, $r(x_{1})$ is an $M$-coheir extending $p(x)$ or $q(x)$, and $\{b_{i}\}$ is a coheir Morley sequence over $M$ in $r(x_{1})$, $r(x_{1})$ is the \textit{unique} coheir over $M$ for which $\{b_{i}\}$ is a coheir Morley sequence. The reason for this is that any other such $M$-coheir $r'(x_{1})$ agrees with $r(x_{1})$ on $b_{0}$, so $r'(x_{1})$ agrees with $r(x_{1})$ on an $|M|^{+}$-saturated elementary extension $M' \succ M$. But then, $r'(x_{1})$ must agree with $r(x_{1})$ globally, because $r(x_{1})$ is identified among global $M$-coheirs by its restriction to $M'$. The argument that the global $M$-coheir $r(x_{1})$ is determined by its restriction to $M'$ is well-known (see, e.g., \cite{sim12}): by invariance of $r(x_{1})$, and $|M|^{+}$-saturation of $M$,  a formula $\varphi(x, b)$ is in $r(x_{1})$ if and only if, for all $b' \equiv_{M} b$ in $M'$, $\varphi(x, b')$ is in $r(x_1)$.

      Say a global coheir Morley sequence $\{b_{i}\}_{i < \omega}$ over $M$ satisfies condition (*) if for $n < \omega$, $\mathrm{acl}_{M}(b_{0} \ldots b_{n-1} \cup b_{2n}) \cap \mathrm{acl}_{M}(b_{n} \ldots b_{2n-1} \cup b_{2n})=\mathrm{acl}_{M}(b_{2n})$. Our central claim is the following:

    \begin{claim}
\label{extending to sequence satisfying condition *}        Let $\{b_{i}\}_{i< \omega}$ be a global coheir Morley sequence over $M$. Then we may find a global coheir Morley sequence $\{c_{i}\}_{i < \omega}$ over $M$, with $b_{i} \subset c_{i}$, satisfying condition (*).
    \end{claim}

\begin{proof}
    (of claim) 
Extend $\{b_{i}\}_{i< \omega}$ to an $I$-indexed $M$-indiscernible sequence $\{b_{i}\}_{i \in I}$, where $I$ is a sufficiently saturated model of $\mathrm{DLO}$ (specifically, saturated enough to guarantee that it embeds an ordinal large enough to apply the Erdős–Rado theorem, so in particular, $\aleph_{0}^{+}$-saturated). Let $i \in I$ and $S_{1}, S_{2} \subset I$  be countably infinite sets with $S_{1} < S_{2} < \{i\}$  (where the notation $A <  B$ for $A, B \subset I$ means that every element of $A$ is $<$-less in $I$ than every element of $B$; such $S_{1}, S_{2}$ exist for every $i \in I$ by $\aleph^{+}_{0}$-saturation of the index set $I$ as an ordered set). We claim that, holding $i \in I$ fixed, the set $\mathrm{acl}_{M}(\{b_j\}_{j \in S_{1}}\cup b_{i}) \cap \mathrm{acl}_{M}(\{b_j\}_{j \in S_{2}}\cup b_{i})$ does not depend on our choice of $S_{1}$, $S_{2}$ satisfying these conditions. Note that the set $\{b_j\}_{\{j\} < S_{2}}$ is indiscernible over $\{b_j\}_{j \in S_{2}}\cup b_{i}$, so over $\mathrm{acl}_{M}(\{b_j\}_{j \in S_{2}}\cup b_{i})$. So we may always replace $S_{1}$ with any countably infinite $S'_{1} \subset I$ with $S'_{1} < S_{2}$, without changing $\mathrm{acl}_{M}(\{b_j\}_{j \in S_{1}}\cup b_{i}) \cap \mathrm{acl}_{M}(\{b_j\}_{j \in S_{2}}\cup b_{i})$. By similar reasoning, we may always replace $S_{2}$ with any countable $S'_{2} \subset I$ with $S_{1} < S'_{2} < \{i\}$, again without changing $\mathrm{acl}_{M}(\{b_j\}_{j \in S_{1}}\cup b_{i}) \cap \mathrm{acl}_{M}(\{b_j\}_{j \in S_{2}}\cup b_{i})$. So, again without changing this set, we may always replace any $S_{1} < S_{2} < \{i\}$ with any $S'_{1} < S'_{2} < \{i\}$ whatsoever for $S'_{1}, S'_{2} \subset I$ countably infinite: we first find some countable $S''_{1} < S''_{2}$ with $S''_{2} < S_{1}$ and $S''_{2} < S'_{1}$ (here we use $\aleph^{+}_{0}$-saturation as an ordered set of the index set $I$). Then we see, by the preceding, that we can replace $S_{1}, S_{2}$ with $S''_{1}, S''_{2}$ and then $S''_{1}, S''_{2}$ with $S'_{1}, S'_{2}$ so that $\mathrm{acl}_{M}(\{b_j\}_{j \in S_{1}}\cup b_{i}) \cap \mathrm{acl}_{M}(\{b_j\}_{j \in S_{2}}\cup b_{i})$ always remains unchanged.

Now, for $i \in I$, let $c_{i} : = \mathrm{acl}_{M}(\{b_j\}_{j \in S_{1}}\cup b_{i}) \cap \mathrm{acl}_{M}(\{b_j\}_{j \in S_{2}}\cup b_{i})$ as above, knowing now that the specific choice of countably infinite $S_{1}, S_{2} \subset I$ with $S_{1} < S_{2} < \{i\}$ does not matter.  (Here, to view $c_{i}$ as an infinite tuple, we just choose some arbitrary enumeration of each $c_{i}$ extending the one on $b_{i}$; in lieu of trying to choose enumerations so that $\{c_{i}\}_{i \in I}$ remains indiscernible over $M$, we will apply the Erdős–Rado theorem at the end to obtain the desired indiscernibility.) We show that, for $i \in I$ and $S_{1} < S_{2} < \{i\}$ with $S_{1}, S_{2} \subset I$ countable, $\mathrm{acl}_{M}(\{c_j\}_{j \in S_{1}}\cup c_{i}) \cap \mathrm{acl}_{M}(\{c_j\}_{j \in S_{2}}\cup c_{i}) = c_{i}$. Immediately, $\mathrm{acl}_{M}(\{c_j\}_{j \in S_{1}}\cup c_{i}) \cap \mathrm{acl}_{M}(\{c_j\}_{j \in S_{2}}\cup c_{i}) \supseteq c_{i}$. We show the reverse inclusion, $\mathrm{acl}_{M}(\{c_j\}_{j \in S_{1}}\cup c_{i}) \cap \mathrm{acl}_{M}(\{c_j\}_{j \in S_{2}}\cup c_{i}) \subseteq c_{i}$. Extending $S_{1}$ and $S_{2}$, we may assume $S_{1}$, $S_{2}$ infinite with no least element (again using saturation properties of $I$). If we can show, for $\ell = 1, 2$, $\mathrm{acl}_{M}(\{c_j\}_{j \in S_{\ell}}\cup
c_{i}) \subseteq \mathrm{acl}_{M}(\{b_j\}_{j \in S_{\ell}}\cup b_{i})$,  the inclusion will follow from the definition $c_{i}  = \mathrm{acl}_{M}(\{b_j\}_{j \in S_{1}}\cup b_{i}) \cap \mathrm{acl}_{M}(\{b_j\}_{j \in S_{2}}\cup b_{i})$. So we show $\mathrm{acl}_{M}(\{c_j\}_{j \in S_{\ell}}\cup
c_{i}) \subseteq \mathrm{acl}_{M}(\{b_j\}_{j \in S_{\ell}}\cup b_{i})$. For this it suffices to show $c_{i} \subset \mathrm{acl}_{M}(\{b_j\}_{j \in S_{\ell}}\cup b_{i})$, and $c_{j'} \subseteq \mathrm{acl}_{M}(\{b_j\}_{j \in S_{\ell}}\cup b_{i})$ for $j' \in S_{\ell}$. That  $c_{i} \subset \mathrm{acl}_{M}(\{b_j\}_{j \in S_{\ell}}\cup b_{i})$ follows from the definition  $c_{i}  = \mathrm{acl}_{M}(\{b_j\}_{j \in S_{1}}\cup b_{i}) \cap \mathrm{acl}_{M}(\{b_j\}_{j \in S_{2}}\cup b_{i})$. And $c_{j'} \subseteq \mathrm{acl}_{M}(\{b_j\}_{j \in S_{\ell}}\cup b_{i})$ follows from $c_{j'} \subseteq \mathrm{acl}_{M}(\{b_j\}_{j \in S_{\ell}, j < j' }\cup b_{j'})$, using that $S_{\ell}$ is assumed to have no least element, so the index set of $\{b_{j}\}_{j \in S_{\ell}, j < j' }$ is a countably infinite set whose terms lie below $j'$. (For any countably infinite set $S'_{1} < \{j'\}$ with $S'_{1} \subset I$, $c_{j'} \subset \mathrm{acl}_{M}(\{b_j\}_{j \in S'_{1}}\cup b_{j'})$; we may find countably infinite $S'_{2} < S'_{1} < \{j'\}$ by saturation properties of $I$, and $c_{j'}  = \mathrm{acl}_{M}(\{b_j\}_{j \in S'_{1}}\cup b_{j'}) \cap \mathrm{acl}_{M}(\{b_j\}_{j \in S'_{2}}\cup b_{j'})$ by the previous paragraph).

We next show that, for $\ind^{u}$ the relation defined so that $a \ind_{M}^{u} b$ if $\mathrm{tp}(a/Mb)$ is finitely satisfiable, $c_{j} \ind^{u}_{M} c_{< j}$ for all $j \in I$. We extend $\{b_{i}\}_{i \in I}$ to $\{b_{i}\}_{i \in I'}$ for an $|I|^{+}$-saturated model $I'$ of $\mathrm{DLO}$. Now the first paragraph of the proof of this claim shows that, because $c_{i}= \mathrm{acl}_{M}(\{b_j\}_{j \in S_{1}}\cup b_{i}) \cap \mathrm{acl}_{M}(\{b_j\}_{j \in S_{2}}\cup b_{i})$ for $S_{1}, S_{2} \subset I$ with $S_{1} < S_{2} < \{i\}$ by definition, $c_{i}= \mathrm{acl}_{M}(\{b_j\}_{j \in S'_{1}}\cup b_{i}) \cap \mathrm{acl}_{M}(\{b_j\}_{j \in S'_{2}}\cup b_{i})$ for  any countably infinite $S'_{1}, S'_{2} \subset I'$ $S'_{1} < S'_{2} < \{i\}$. Thus, by the analogous claim that we showed for $I$, for any countably infinite set (and thus any infinite set) $S'_{1} < \{j\}$ with $S'_{1} \subset I'$, $c_{j} \subset \mathrm{acl}_{M}(\{b_i\}_{j' \in S'_{1}}\cup b_{j})$. Thus, $c_{j} \subseteq \mathrm{acl}_{M}(\{b_{j'}\}_{j' \in I', j > j' > \{j'' \in I: j'' < j\} } \cup \{b_{j}\})$: the index set of $\{b_{j'}\}_{j' \in I', j > j' > \{j'' \in I: j'' < j\} }$ is an infinite subset of $I'$ whose terms lie below $j$, by $|I|^{+}$-saturation of $I'$. Moreover, for $j'' < j $ with $j'' \in I$, $c_{j''} \subset \mathrm{acl}_{M}(\{b_{j'}\}_{j' \in I, j' < j''} \cup \{b_{j''}\}) \subset \mathrm{acl}_{M}(\{b_{j'}\}_{j' \in I, j' < j})$, the first inclusion also obtained by the conclusion $c_{j''} \subset \mathrm{acl}_{M}(\{b_i\}_{j' \in S_{1}}\cup b_{j''})$ for $S_{1} \subset I$, $S_{1} < \{j''\}$ infinite, and the fact that $I$ has no lower bound. Now note that $\{b_{j'}\}_{j' \in I', j > j' > \{j'' \in I: j'' < j\} } \cup \{b_{j}\} \ind_{M}^{u} \{b_{j'}\}_{j' \in I, j' < j}$, because $\{b_{i}\}_{i \in I'}$ is a coheir Morley sequence over $M$. (This is as stated at the beginning of the preliminaries section; we repeatedly apply that $a_{1} \ind^{u}_{M} a_{0}$, $a_{2} \ind^{u}_{M} a_{0}a_{1}$ implies $a_{2}a_{1} \ind^{u}_{M} a_{0}$--particularly, that being a coheir Morley sequence over $M$ is preserved under concatenations of adjacent terms). We now dispense with $I'$ and again assume all indices are in $I$; at this point, we have shown that there are $a, b$ with  $a \ind_{M}^{u} b$ and $c_{j} \subset \mathrm{acl}_{M}(a)$, $c_{< j} \subset \mathrm{acl}_{M}(b)$. Thus, to prove that  $c_{j} \ind^{u}_{M} c_{< j}$, it suffices to show that, for all $a$, $b$ with $a \ind_{M}^{u} b$, $a' \ind_{M}^{u} b$ for $a' \in \mathrm{acl}_{M}(a)$, and $a \ind_{M}^{u} b'$ for $b' \in \mathrm{acl}_{M}(b)$. (While this is likely well-known, we review the proof.)

We first show that, if $a \ind_{M}^{u} b$, then $a' \ind_{M}^{u} b$ for $a' \in \mathrm{acl}_{M}(a)$. Let $a'$ satisfy $\varphi(x, b)$ (where $\varphi(x, y)$ has parameters in $M$), and we show that $\varphi(x, b)$ is satisfied in $M$. Let $\chi(x, z)$ be a formula with parameters in $M$ such that $a'$ satisfies $\chi(x, a)$, and such that the only realizations of $\chi(x, \tilde{a})$ are algebraic over $M\tilde{a}$ for any $\tilde{a}$. The formula $\exists x \chi(x, z) \wedge \varphi(x, b)$ is satisfied by $a$, so because $a \ind_{M}^{u} b$, by some $m \in M$. Therefore, $\varphi(x, b)$ is satisfied by some realization of $\chi(x, m)$. But that realization of $\varphi(x, b)$ will be algebraic over $m\in M$, and will thus belong to $M$, as desired. 

We next show that,  if $a \ind_{M}^{u} b$, then $a \ind_{M}^{u} b'$ for $b' \in \mathrm{acl}_{M}(b)$. Let $a$ satisfy $\varphi(x, b')$ (again, where $\varphi(x, y)$ has parameters in $M$); we show $\varphi(x, b')$ is satisfied in $M$.  Let $\chi(y, b)$, where $\chi$ has parameters in $M$, isolate $b'$ over $Mb$. We claim that, for all $b''$ satisfying $\chi(y, b)$, $a$ satisfies $\varphi(x, b'')$. Suppose otherwise. Then the formula $\exists y, y' (\chi(y, b) \wedge \chi(y', b) \wedge \varphi(x, y) \wedge \neg\varphi(x, y')) $ is satisfied by $a$. So because $a \ind_{M}^{u} b$,  this formula is satisfied by some $m \in M$. Then there are two realizations of $\chi(y, b)$, one of which satisfied $\varphi(m, y)$ and the other of which satisfies $\neg\varphi(m, y)$. This will contradict that $\chi(y, b)$ isolates a type over $Mb$, so we know that, for all $b''$ satisfying $\chi(y, b)$, $a$ satisfies $\varphi(x, b'')$. So $a$ satisfies $\forall y (\chi(y, b) \to \varphi(x, y))$. Again because $a \ind_{M}^{u} b$, some $m \in M$ satisfies this same formula. So because $b'$ satisfies $\chi(y, b)$, $m$ satisfies $\varphi(x, b')$, as desired.

This completes the proof that $c_{j} \ind^{u}_{M} c_{< j}$ for all $j \in I$. We now apply the Erdős–Rado theorem to $\{c_{i}\}_{i \in I}$ (more precisely, to a large enough ordinal-indexed subsequence). This gives us an $M$-indiscernible sequence $\{c_i\}_{i < \omega}$, for which each finite initial segment has the same type over $M$ as $c_{\bar{i}}$ for $\bar{i}$ some finite increasing tuple in $I$. Mapping the resulting conjugate of the sequence $\{b_{i}\}_{i < \omega}$ back to the original $\{b_{i}\}_{i < \omega}$, we may assume that $c_{i} \supset b_{i}$ for each $i < \omega$. Then $\{c_{i}\}_{i < \omega}$ satisfies the conclusion of condition (*), and is also an $M$-indiscernible sequence satisfying $c_{i} \ind_{M}^{u} c_{<i}$ for $i < \omega$. Let $r(x)$ be a global coheir extending the limit type $\mathrm{tp}(c_{\omega}/ M c_{< \omega})$, where we extend $\{c_{i}\}_{i < \omega} $ to $\{c_{i}\}_{i \leq \omega}$; then $\{c_{i}\}_{i < \omega}$ will be a coheir Morley sequence in $q(x)$ over $M$. This together with the conclusion of condition (*) proves the claim.

\end{proof}

It follows from the claim that if $r(x')$ is an $M$-coheir with $x \subset x'$ extending $p(x)$ or $q(x)$, there is an $M$-coheir $r'(x'')$, $x' \subset x''$ extending $r(x')$ whose coheir Morley sequences satisfy condition (*). To show this, for $\{b_{i}\}_{i< \omega}$ a coheir Morley sequence over $M$ in $r(x')$, $\{c_{i}\}_{i < \omega}$ as in the claim will be a coheir Morley sequence over $M$ in some $M$-coheir $r'(x'')$ whose coheir Morley sequences satisfy condition (*). But  $r'(x'')|_{x'}$ will have $\{b_{i}\}_{i< \omega}$ as a coheir Morley sequence over $M$. Recall, by our background assumption that any coheir Morley sequence over $M$ in $r(x')$ is not a coheir Morley sequence over $M$ in any other coheir over $M$ besides $r(x')$; thus $r'(x'')|_{x'} = r(x')$, and $r'(x'')$ is the desired extension of $r(x')$ whose  coheir Morley sequences satisfy condition (*).

Now define, inductively, the following $M$-coheirs $p_{n}(x_{n})$ and $q_{n}(x_{n})$ with $p_{n}(x_{n})|_{M} = q_{n}(x_{n})|_{M}$. Let $p_{0}(x_{0}) : = p(x) $ and $q_{0}(x_{0}) := q(x)$. For odd $n> 0$, find $x_{n} \supseteq x_{n-1}$ and an $M$-coheir $p_{n}(x_{n}) \supset p_{n-1}(x_{n-1})$ whose coheir Morley sequences over $M$ satisfy condition (*). Using the same fact on  ``left extension for finite satisfiability" used at the very beginning of the proof, and the inductive assumption that $p_{n-1}(x_{n-1})|_{M} = q_{n-1}(x_{n-1})|_{M}$ (so $p_{n}(x_{n})|_{M}$ extends $q_{n-1}(x_{n-1})|_{M}$), find an $M$-coheir $q_{n}(x_{n}) \supset q_{n-1}(x_{n-1})$ extending $p_{n}(x_{n})|_{M}$. So $p_{n}(x_{n})|_{M} = q_{n}(x_{n})|_{M}$. Similarly, for even $n > 0$,  find $x_{n} \supseteq x_{n-1}$ and an $M$-coheir $q_{n}(x_{n}) \supset q_{n-1}(x_{n-1})$ whose coheir Morley sequences over $M$ satisfy condition (*), and an $M$-coheir $p_{n}(x_{n}) \supset p_{n-1}(x_{n-1})$ extending $q_{n}(x_{n})|_{M}$; again $p_{n}(x_{n})|_{M} = q_{n}(x_{n})|_{M}$.

Now define $y = \bigcup^{\infty}_{n=0} x_{n}$,  $\tilde{p}(y) = \bigcup^{\infty}_{n=0} p_{n}(x_{n})$, $\tilde{q}(y) = \bigcup^{\infty}_{n=0} q_{n}(x_{n})$. Note that $\tilde{p}(y)|_{M} = \tilde{q}(y)|_{M}$, and each of $\tilde{p}(y)$, $\tilde{q}(y)$ is the union of an ascending chain of  $M$-coheirs whose coheir Morley sequences over $M$ satisfy condition (*). (This is by restricting attention to the odd and even stages of the union, respectively.) So to show that the coheir Morley sequences over $M$ in $\tilde{p}(y)$ and $\tilde{q}(y)$ satisfy condition (*), as desired, it suffices to show the following claim:

\begin{claim}
    Let $r_{n}(z_{n})$ be $M$-coheirs whose coheir Morley sequences over $M$ satisfy condition (*), and which form an ascending chain:  $r_{n-1}(z_{n-1}) \subset r_{n}(z_{n})$ for $n < \omega$. Let $z = \bigcup^{\infty}_{n=0} z_{n}$ and $r(z)= \bigcup^{\infty}_{n=0} r_{n}(z_{n})$ (so $r(z)$ is a global $M$-coheir.) Then coheir Morley sequences over $M$ in $r(z)$ satisfy condition (*).
\end{claim}

\begin{proof}
    (of claim) 
    
    Let $\{b_{i}\}_{i < \omega}$ be a coheir Morley sequence over $M$ in $r(z)$; we must show  $\mathrm{acl}_{M}(b_{0} \ldots b_{i-1} \cup b_{2i}) \cap \mathrm{acl}_{M}(b_{i} \ldots b_{2i-1} \cup b_{2i})\subseteq \mathrm{acl}_{M}(b_{2i})$. Let $b \in \mathrm{acl}_{M}(b_{0} \ldots b_{i-1} \cup b_{2i}) \cap \mathrm{acl}_{M}(b_{i} \ldots b_{2i-1} \cup b_{2i})$ be a singleton in $\mathbb{M}$; our goal is to show $b \in \mathrm{acl}_{M}(b_{2i})$. Let $b^{n}_{j}$ consist of the $z_{n}$-indexed coordinates of $b_{j}$;  then  $b \in \mathrm{acl}_{M}(b^{n}_{0} \ldots b^{n}_{i-1} \cup b^{n}_{2i}) \cap \mathrm{acl}_{M}(b^{n}_{i} \ldots b^{n}_{2i-1} \cup b^{n}_{2i})$ for some $n < \omega$. But $\{b_{i}^{n}\}_{i < \omega}$ forms a coheir Morley sequence over $M$ in $r_{n}(z_{n})$, so satisfies condition (*): $\mathrm{acl}_{M}(b^{n}_{0} \ldots b^{n}_{i-1} \cup b^{n}_{2i}) \cap \mathrm{acl}_{M}(b^{n}_{i} \ldots b^{n}_{2i-1} \cup b^{n}_{2i}) = \mathrm{acl}_{M}(b^{n}_{2i})$. Then  $b \in \mathrm{acl}_{M}(b^{n}_{2i}) \subset \mathrm{acl}_{M}(b_{2i})$, as desired.
\end{proof}
This completes the proof of the lemma.
    
\end{proof}

So $\tilde{p}(y)$ and $\tilde{q}(y)$ as in this lemma belong to a class of (global) $M$ coheirs with desirable properties involving intersections of algebraic closures in coheir Morley sequences, which we codify in the following definition:

\begin{definition}
\label{normal coheir}    An $M$-coheir $p(x)$ is $\textit{normal}$ if, for $\{b_{i}\}_{i< \omega}$ a coheir Morley sequence over $M$ in $p(x)$, for all $n < \omega$, $\mathrm{acl}_{M}(b_{0} \ldots b_{n-1} \cup b_{2n}) \cap \mathrm{acl}_{M}(b_{n} \ldots b_{2n-1} \cup b_{2n})=\mathrm{acl}_{M}(b_{2n})$.
\end{definition}

As a preliminary step to proving Lemma \ref{noninterference of algebraic closures} and obtaining, from $\mathrm{SOP}_{2}$, our instance of $\mathrm{TP}$ with consistency witnesses satisfying the desired properties for algebraic closures, we prove the following two facts about normal coheirs.

\begin{lemma}
\label{algebraic closures in normal coheir morely sequence}    Let $\{b_{i}\}_{i < \omega}$ be a coheir Morley sequence over $M$ in a normal coheir, with the $b_{i}$ algebraically closed sets containing $M$. Let $B_{i} = b_{i} \backslash M$, $B_{ij} = \mathrm{acl}(b_ib_j) \backslash MB_{i}B_{j}  $ for $i < j$. Then all of the $M$, $B_{i}$, $B_{ij}$ are disjoint.  
\end{lemma}

\begin{proof}
    Through this and later proofs, a \textit{nonoverlapping indiscernible sequence} $\{a_{i}\}_{i < \omega}$ over a set $A$ will be an $A$-indiscernible sequence such that for $i \neq j$, $\mathrm{acl}_{A}(a_{i}) \cap \mathrm{acl}_{A}(a_{j}) = \mathrm{acl}(A)$. Note for this and later proofs that, for any nonoverlapping indiscernible sequence $\{a_{i}\}_{i < \omega}$ over $A$ that is also indiscernible over $Ab$, $\mathrm{acl}_{A}(a_{i}) \cap \mathrm{acl}_{A}(b) = A$. 
    
    First, the $B_{i}$ are clearly disjoint from each other: $\{b_{i}\}_{i< \omega}$ is a coheir Morley sequence over $M$, so nonoverlapping over $M$.
    
    Moreover, $B_{ij}$ is disjoint from $B_{ij'}$ for  $i < j, j'$ and $j \neq j'$: $\{b_{\ell}\}_{\ell > i}$ is an invariant Morley sequence over $b_{i}$, so nonoverlapping over $b_{i}$. 
    
Next, $B_{ij}$ is disjoint from $B_{i'j}$ for $i, i' < j$ and $i \neq i'$, by normality. 

We next show disjointness of $B_{ij}$ and $B_{i'j'}$, with  $i \neq i'$, $j \neq j'$, and $i < j$, $i' < j'$, so we prove these instances of disjointness. We may assume without loss of generality that $j > j'$, so that $j$ does not coincide with $i$, $j'$ or $i'$. We may, say, extend $\{ b_{i} \}_{i < \omega}$ to $\{ b_{i} \}_{i \in \mathbb{Q}}$. By our assumption, there will exist an interval $I \subset \mathbb{Q}$ containing $j$ and none of $i$, $j'$ or $i'$. We already showed that $\{b_{\ell}\}_{\ell > i}$ is a nonoverlapping indiscernible sequence over $b_{i}$; therefore, so is $\{b_{\ell}\}_{\ell \in I}$. Moreover, $\{b_{\ell}\}_{\ell \in I}$ is also indiscernible over $b_{i}b_{i'}b_{j'}$. This justifies the last equality in $\mathrm{acl}(b_ib_{j}) \cap \mathrm{acl}(b_i'b_j')\subset \mathrm{acl}(b_ib_{j}) \cap \mathrm{acl}(b_ib_i'b_j') =\mathrm{acl}(b_{i})$, showing disjointness of $B_{ij}$ and $B_{i'j'}$..

Then the only remaining instances of disjointness to show, which are not immediate from the definitions, are those between $B_{i}$ and $B_{i'j'}$ when $i$ is not one of the $i', j'$. Let $I$ be an interval of $\mathbb{Q}$ containing $i$ but not $i', j'$. As we already noted, $\{b_{i}\}_{i < \omega}$ is a nonoverlapping indiscernible sequence over $M$, so $\{b_{i}\}_{i \in I}$ is a nonoverlapping indiscernible sequence over $M$. Also, $\{b_{i}\}_{i \in I}$ is indiscernible over $b_{i'}b_{j'}$. So $\mathrm{acl}(b_{i}) \cap \mathrm{acl}(b_{i'}b_{j'}) = M$, as desired.
\end{proof}

\begin{lemma}
  \label{normal coheir sequence indiscernible over a set}  Let $\{a_{i}\}_{i < \omega}$ be a coheir Morley sequence over $M$ in a normal coheir, with the $a_{i}$ algebraically closed sets containing $M$. Suppose $a$ is an algebraically closed set containing $M$ over which $\{a_i\}_{i < \omega}$ is indiscernible. Let $A_{i}:= \mathrm{acl}(aa_i)$ and $A := \cap_{i < \omega} A_{i}$. Then:

    (1) $A_{i} \cap A_{j} = A$ for $i \neq j$.

    (2) $\{a_{i}\}_{i < \omega}$ is indiscernible over $A$, and $A \cap a_{i} = M$ for $i < \omega$.

    (3) $A_{i} \cap \mathrm{acl}(\{a_{i}\}_{i < \omega}) = a_{i}$ for $ i < \omega$.
\end{lemma}

\begin{proof}
     To show (1), since $\{a_{i}\}_{i < \omega}$ is indiscernible over $a$, it is a standard fact that $A_{i} \cap A_{j}= \mathrm{acl}(aa_i) \cap \mathrm{acl}(aa_j)$ for $i \neq j$ does not depend on $i, j$. (This is proven by a simpler version of the argument from the second paragraph of the proof of Claim \ref{extending to sequence satisfying condition *}.) So $A_{i} \cap A_{j}$ is equal to $A = \cap_{\ell < \omega} A_{\ell}$. 

      To show (2), let $\{a_{i}\}_{i \in \mathbb{Z}}$ be an $a$-indiscernible sequence extending $\{a_i\}_{i < \omega}$. Then by the same reasoning as in (1), $A \subset \mathrm{acl}(aa_{-1})$.  But $\{a_i\}_{i< \omega}$ is $aa_{-1}$-indiscernible, so $\{a_i\}_{i< \omega}$ is $\mathrm{acl}(aa_{-1})$-indiscernible, and thus $A$-indiscernible. And because $\{a_i\}_{i< \omega}$ is a nonoverlapping indiscernible sequence over $M$, which, as we now know, remains indiscernible over $A$, $A \cap a_{i} = M$, as desired.

      For (3), we first prove the claim that $A_{j} \cap \mathrm{acl}(\{a_{i}\}_{i < \omega}) \subset \mathrm{acl}(\{a_{i}\}_{i \leq j})$ for all $j < \omega$. It suffices to show that for any $n$, $A_{j} \cap \mathrm{acl}(\{a_{i}\}_{i \leq j+n}) \subset \mathrm{acl}(\{a_{i}\}_{i \leq j})$. But $\{a_{ (j+1) + kn+  0 }\ldots a_{(j+1)+ kn + (n-1) }\}_{k < \omega}$ is an invariant Morley sequence starting with $\mathrm{acl}(\{a_{i}\}_{j < i \leq j+n})$ over $\mathrm{acl}(\{a_{i}\}_{i \leq j})$. So $\{a_{ (j+1) + kn+  0 }\ldots a_{(j+1)+ kn + (n-1) }\}_{k < \omega}$ is a nonoverlapping over indiscernible sequence over $\mathrm{acl}(\{a_{i}\}_{i \leq j})$. But $\{a_{ (j+1) + kn+  0 }\ldots a_{(j+1)+ kn + (n-1) }\}_{k < \omega}$ is also indiscernible over $A_{j} \mathrm{acl}(\{a_{i}\}_{i \leq j}) \subset \mathrm{acl}(a\{a_{i}\}_{i \leq j})$. It follows that $\mathrm{acl}(A_{j}\{a_{i}\}_{i \leq j}) \cap \mathrm{acl}(\{a_{i}\}_{i \leq j+n}) \subset \mathrm{acl}(\{a_{i}\}_{i \leq j})$. Thus, $A_{j} \cap \mathrm{acl}(\{a_{i}\}_{i \leq j+n }) \subset \mathrm{acl}(\{a_{i}\}_{i \leq j})$, as desired.

      To complete the proof of (3), because we now know that $A_{j} \cap \mathrm{acl}(\{a_{i}\}_{i < \omega}) \subset \mathrm{acl}(\{a_{i}\}_{i \leq j})$, it remains to show that $A_{j} \cap \mathrm{acl}(\{a_{i}\}_{i \leq j}) \subset a_{j}$. Let us again consider the extended $a$-indiscernible sequence $\{a_{i}\}_{i \in \mathbb{Z}}$.  By normality, the sequence $\{a_{(j-1)-(kj+0)} \ldots a_{(j-1)-(kj+(j-1))}\}_{k < \omega}$, which starts with $\{a_{i}\}_{0 \leq i < j}$, is a nonoverlapping $a_j$-indiscernible sequence. But $\{a_{(j-1)-(kj+0)} \ldots a_{(j-1)-(kj+(j-1))}\}_{k < \omega}$ is also indiscernible over $aa_{j}$, so over $\mathrm{acl}(aa_j)=A_{j}$. Thus, $A_{j} \cap \mathrm{acl}(\{a_{i}\}_{0 \leq i < j}) \subset  A_{j} \cap\mathrm{acl}(\{a_{i}\}_{0 \leq i\leq j}) \subset   a_{j}$, as desired.
\end{proof}

Again, our current goal is to show that every $ \mathrm{SOP}_{2}$ theory will have an instance of the tree property, with its consistency witnesses, where some algebraic closures that interest us do not overlap nontrivially. We are now ready to prove a lemma which will, in particular, give us this desired tree-property configuration in every $\mathrm{SOP}_{2}$ theory. Here, for $\varphi(x,y)$, $\{b_{\eta}\}_{\eta \in \omega^{< \omega}}$  an instance of the tree property with consistency witnesses $a_{\sigma} \models \{ \varphi(x, b_{\sigma|_{n}})\}_{n < \omega}$ for $\sigma \in \omega^{\omega}$, the two families of algebraic closures that will interest us will be as follows. First, we will be interested in algebraic closures of the pairs of immediate successors to a common node that exhibit the ($2$-)inconsistency clause of the tree property: sets of the form $\mathrm{acl}(b_{\eta \smallfrown \langle i \rangle}b_{\eta \smallfrown \langle j \rangle})$ for $\eta \in \omega^{< \omega} $,  $i < j < \omega$. And second, we will be interested in the algebraic closures of the pairs giving the required incidences between $\varphi(x, y)$ and consistency witnesses: sets of the form  $\mathrm{acl}(a_{\sigma}b_{\sigma|_{n}})$ for $\sigma \in \omega^{<\omega}$, $n < \omega$. The main part of this lemma that we will use will be that these algebraic closures do not overlap nontrivially, though we will state the lemma at a greater level of detail. This detailed statement will be suggestive of how our tree will be constructed from the coheir Morley sequences given by $\mathrm{NSOP}_{1}$ (coheir Morley sequences $\{b_i\}_{i < \omega} $, $\{b^i\}_{i < \omega}$ starting with $b$ over $M$ with $\{\varphi(x, b_{i})\}_{i < \omega}$ consistent, $\{\varphi(x, b^{i})\}_{i < \omega}$ $2$-inconsistent, as in the proof of Proposition 3.14 of \cite{KR17} and Fact \ref{sop1 coheirs} above) and the extensions to normal coheirs given by Lemma \ref{pair of normal coheirs}.

\begin{lemma}
\label{noninterference of algebraic closures}
Let $M$ be a model, and let $\varphi(x,b)$ be a formula and $\{b_i\}_{i < \omega} $, $\{b^i\}_{i < \omega}$  coheir Morley sequences that exhibit $\mathrm{SOP}_{1}$ as in Fact \ref{sop1 coheirs}: $b_0 = b^0 = b$, $\{\varphi(x, b_{i})\}_{i < \omega}$ is $2$-inconsistent, and $\{\varphi(x, b^{i})\}_{i < \omega}$ is consistent. (For these to exist is the same as the theory being $\mathrm{SOP}_{2}$.)

Then we may find the following:

\begin{itemize}
    \item an algebraically closed set $c$ extending $Mb$, and coheir Morley sequences $J_{1} = \{c_{i}\}_{i < \omega}$, $J_{2}=\{c^{i}\}_{i < \omega}$ over $M$, that still exhibit $\mathrm{SOP}_{1}$\footnote{In the rest of the statement and proof of this lemma, and in the proof of Lemma \ref{standard hereditary class}, we use the following notation. For $\{c_i\}_{i \in I} = c \supset b= \{b_i\}_{i \in J}$ the explicit enumerations of $ c, b$ with $ c_i, b_i$ singletons, and for a set of variables $\{z_i\}_{i \in I} = z \supset y= \{y_i\}_{i \in J}$ with $z_i, y_i$ singletons, and $w \supset x$ another set of variables:

--the formula $\varphi(w, z)$ is just $\varphi(x, y)$ considered as a formula in the larger sets of variables

--when $d = \{d_{i}\}_{i \in I}$, $\varphi(w, d)$ is an instance of $\varphi(w, z)$

So informally speaking, when parameters and variables occurring in $\varphi$ and its instances are enumerated by larger index sets than those assigned to their positions in $\varphi(x, y)$, the additional parameters and variables are just ignored.} for $\varphi(x, c)$: $c_0 = c^0 = c$, $\{\varphi(x, c_{i})\}_{i < \omega}$ is $2$-inconsistent, and $\{\varphi(x, c^{i})\}_{i < \omega}$ is consistent.

\item a tree $\{c_{\eta}\}_{\eta \in \omega^{< \omega}}$, and a $\omega^{\omega}$-indexed set $\{a_{\sigma}\}_{\sigma \in \omega^{\omega}}$, both consisting of algebraically closed sets $c_{\eta}$ containing $M$

\end{itemize}

satisfying the following conditions:

(1) For any $\eta \in \omega^{< \omega}$ $\{c_{\eta \smallfrown \langle i \rangle}\}_{i < \omega} \equiv_{M} J_{1}$, so $\{\varphi(x, c_{\eta \smallfrown \langle i \rangle})\}_{i < \omega}$ is $2$-inconsistent and $\mathrm{tp}(b_{\eta \smallfrown \langle i \rangle}b_{\eta \smallfrown \langle j \rangle})$ does not depend on $\eta \in \omega^{< \omega} $,  $i < j < \omega$.

(2) For any $\sigma \in \omega^{\omega}$, $k$, $c_{\sigma|_{0}}, \ldots  c_{\sigma|_{k-1}}\equiv_{M} c^{k-1}\ldots c^{0}$--i.e., $\{c_{\sigma|_{n}}\}_{n < \omega}$ is a reversed coheir Morley sequence over $M$ in $q(z)$. Moreover, for variables $w \supset x$, $|w| = |a_{\sigma}|$,  $a_{\sigma} \models  \{\varphi(w, c_{\sigma|_{n}})\}_{n < \omega}$, $\{c_{\sigma|_{n}}\}_{n < \omega}$ is indiscernible over $a_{\sigma}$ for each $\sigma \in \omega^{\omega}$, and $\mathrm{tp}(a_{\sigma}\{c_{\sigma|_{n}}\}_{n < \omega})$ does not depend on $\sigma$.

(3) For $\sigma \in \omega^{\omega}$ $n < \omega$, let $A_{\sigma n}:=\mathrm{acl}(a_{\sigma}c_{\sigma|_{n}})\backslash(a_{\sigma} \cup c_{\sigma|_{n}})$. For $\eta \in \omega^{< \omega}$, $i < j < \omega$ let $C_{\eta ij}:= \mathrm{acl}(c_{\eta \smallfrown \langle i \rangle } c_{\eta \smallfrown \langle j \rangle }) \backslash (c_{\eta \smallfrown \langle i \rangle } \cup c_{\eta \smallfrown \langle j \rangle }) $. Then all of the $M, a_{\sigma} \backslash M,c_{\eta} \backslash M, A_{\sigma n}, C_{\eta ij}$ are disjoint.

\end{lemma}

\begin{proof}
    Let $p_{0}(y)$ be a coheir in which $\{b_{i}\}_{i < \omega}$ is a coheir Morley sequence over $M$, and $q_{0}(y)$ a coheir in which $\{b^{i}\}_{i < \omega}$ is a coheir Morley sequence over $M$. Then by Lemma \ref{pair of normal coheirs}, there are $p(z) \supset p_{0}(z)$ , $q(z) \supset p_{0}(z)$ normal coheirs over $M$ which restrict to the same type $r(z) \in S(M)$. But we can even find extensions of each of $p(z), q(z)$ to new variables that are coheirs over $M$ and restrict to the same type over $M$, where that type over $M$ is realized by $\mathrm{acl}_{M}(a)$ for some $a \models r(z)$. (Just find, say, some type $r'(z')$ over $M$ extending $r(z)$ which is realized by $\mathrm{acl}_{M}(a)$. Then using ``left extension for finite satisfiability," find a coheir over $M$ extending $r'(z')$ and $p(z)$, and do the same with $r'(z')$ and $q(z)$.) These extensions will be normal, because $p(z)$ and $q(z)$ are normal, and the conditions $\mathrm{acl}_{M}(a_{0} \ldots a_{n-1} \cup a_{2n}) \cap \mathrm{acl}_{M}(a_{n} \ldots a_{2n-1} \cup a_{2n})=\mathrm{acl}_{M}(a_{2n})$ are always preserved replacing $a_{i}$ with $\mathrm{acl}_{M}(a_{i})$. So we may assume that our original $p(z)$ and $q(z)$ as above restrict, over $M$, to the type of an algebraically closed set containing $M$. By assumption on $\{b_{i}\}_{i < \omega}$, $\{b^{i}\}_{i < \omega}$, for $J_{1} = \{c_{i}\}_{i < \omega}$, $J_{2}=\{c^{i}\}_{i < \omega}$ coheir Morley sequences over $M$ in $p(z)$, $q(z)$ respectively, $J_{1}$ and $J_{2}$ exhibit $\mathrm{SOP}_{1}$ for $\varphi(x, c)$ as in the statement of the lemma.

    Now, similarly to the proof of Lemma 4.5 of \cite{NSOP2} (but much more easily), we may find a tree $\{c_{\eta}\}_{\eta \in \omega^{< \omega}}$ such that:

(a) for each $\eta \in \omega^{<\omega}$, $\{c_{\unrhd \eta \smallfrown \langle i \rangle}\}_{i < \omega}$ is a coheir Morley sequence over $M$, and $\{c_{ \eta \smallfrown \langle i \rangle}\}_{i < \omega} \equiv_{M} J_{1}$.

(b) for each $\eta \in \omega^{<\omega}$, $c_{\eta} \models q(x)|_{c_{\rhd \eta}}$.

We review the form of the argument from the proof of Lemma 4.5 of \cite{NSOP2} that we need. Suppose a tree $A^{n}$ of height $n$ satisfying these conditions is already constructed, and we get a tree of height $n+1$ satisfying these conditions. We may find a coheir Morley sequence $\{A^{n}_{i}\}_{i < \omega}$ over $M$, consisting of realizations of $\mathrm{tp}(A^{n}/M)$, such that, for $a^{n}_{i}$ the term of $A_{i}^{n}$ indexed by the root node, $\{a^{n}_{i}\}_{i< \omega} \models \mathrm{tp}(J_{1}/M)$. To do this, first extend the coheir $p(z)$ to a coheir $P(Z)$ restricting, over $M$, to $\mathrm{tp}(A^{n}/M)$ (using ``left extension for finite satisfiability"). Then take a coheir Morley sequence over $M$ in $P(Z)$. This gives us $\{A^{n}_{i}\}_{i < \omega}$. But then, we can find a realization of $q(x)$ over $\{A^{n}_{i}\}_{i < \omega}$, and reindexing, we get a tree of height $n+1$ satisfying (a) and (b).

Already, we see that (1) is satisfied, by (a). We also see that, for distinct $\eta, \eta' \in \omega^{< \omega}$, $c_{\eta} \backslash M$  is disjoint from $c_{\eta'} \backslash M$. When $\eta, \eta'$ are incomparable, this is (a) and the fact that coheir Morley sequences over $M$ are nonoverlapping indiscernible sequences over $M$. When $\eta, \eta'$ are comparable, this is (b) with the same reasoning.

We also see that, for fixed $\eta$, the $M, c_{ \eta \smallfrown \langle i \rangle}\backslash M, C_{\eta ij}$ are all disjoint. This is by Lemma \ref{algebraic closures in normal coheir morely sequence}, and normality of $p(z)$.

We next prove the following two claims together: that any two $C_{\eta ij}, C_{\eta' i'j'} $ are disjoint when $\eta \neq \eta'$, and that $C_{\eta ij}$ is disjoint from $c_{\eta'} \backslash M$ when $\eta'$ is not of the form $\eta \smallfrown \langle \ell \rangle$. This, together with the preceding, will show that any two of the $M,  c_{\eta} \backslash M , C_{\eta ij}$ are disjoint. 

It suffices for both of these cases to show that 
$\mathrm{acl}(c_{ \mu\smallfrown \langle i_{1} \rangle} c_{ \mu\smallfrown \langle i_{2} \rangle}) \backslash M$ is disjoint from $\mathrm{acl}(\{c_{\nu}\}_{\nu \in S})\backslash M$, when $i_{1} <i_{2}$ and $S \subset \omega^{< \omega}$ contains no $\nu \rhd \mu$. But by (a) above, $\{c_{ \mu \smallfrown \langle k(i_{2}+1) + 0 \rangle} \ldots c_{ \mu \smallfrown \langle k(i_{2}+1) + ((i_{2}+1)-1) \rangle}\}_{k < \omega}$ is a coheir Morley sequence over $M$. So $\{c_{ \mu \smallfrown \langle k(i_{2}+1) + 0 \rangle} \ldots c_{ \mu \smallfrown \langle k(i_{2}+1) + ((i_{2}+1)-1) \rangle}\}_{k < \omega}$ is a nonoverlapping indiscernible sequence over $M$. Also notice that $\{c_{ \mu \smallfrown \langle k(i_{2}+1) + 0 \rangle} \ldots c_{ \mu \smallfrown \langle k(i_{2}+1) + ((i_{2}+1)-1) \rangle}\}_{k < \omega}$ is indiscernible over $\mathrm{acl}(\{c_{\nu}\}_{\nu \in S})$. So the algebraic closure of the first term, $\mathrm{acl}(c_{ \mu \smallfrown \langle  0 \rangle} \ldots c_{ \mu \smallfrown \langle i_{2} \rangle}) \supset \mathrm{acl}(c_{ \mu\smallfrown \langle i_{1} \rangle} c_{ \mu\smallfrown \langle i_{2} \rangle})$, does not meet $\mathrm{acl}(\{c_{\nu}\}_{\nu \in S})$ outside of $M$. Thus $\mathrm{acl}(c_{ \mu\smallfrown \langle i_{1} \rangle} c_{ \mu\smallfrown \langle i_{2} \rangle}) \backslash M$ is disjoint from $\mathrm{acl}(\{c_{\nu}\}_{\nu \in S})\backslash M$, as desired.

Since $\{\varphi(x, c^{i})\}_{i < \omega}$ is consistent, we may find an algebraically closed set $c'$ containing $M$ and variables $w' \supset x$ with $|w'| = |c'|$ such that $c' \models \{\varphi(w', c^{i})\}_{i < \omega}$ and $J_{2} = \{c^{i}\}_{i < \omega}$ is indiscernible over $c'$. Let $c=\mathrm{acl}(c'c^{i}) \cap \mathrm{acl}(c'c^{j})$ for some $i, j$ distinct, and $w \supset w' \supset x$ be variables with $|w| = |c|$ . Then $c \models \{\varphi(w, c^{i})\}_{i < \omega}$, and by (2) of Lemma \ref{normal coheir sequence indiscernible over a set}, $J_{2} = \{c^{i}\}_{i < \omega}$ is indiscernible over $c$. By Lemma \ref{normal coheir sequence indiscernible over a set}, and the observation that $\mathrm{acl}(c'c^{i}) =\mathrm{acl}(cc^{i})$, we have

(1$'$) $\mathrm{acl}(cc_{i}) \cap \mathrm{acl}(cc_{j}) = c$ for $i \neq j$.

(2$'$) $c \cap c_i = M$ for each $i < \omega$.

(3$'$) $\mathrm{acl}(cc_{i}) \cap \mathrm{acl}(\{c^{i}\}_{i < \omega}) = c_{i}$ for each $i < \omega$.

By (b), $c_{\sigma|_{0}}, \ldots  c_{\sigma|_{k-1}}\equiv_{M} c^{k-1}\ldots c^{0}$ for each $k$ by (b)--i.e., $\{c_{\sigma|_{n}}\}_{n < \omega}$ is a reversed coheir Morley sequence over $M$ in $q(z)$. So find, for each $\sigma \in \omega^{\omega}$, $a_{\sigma}$ with $a_{\sigma}c_{\sigma|_{0}}, \ldots  c_{\sigma|_{k-1}}\equiv_{M} cc^{k-1}\ldots c^{0}$ for each $k$. We already see that (2) will be satisfied for any such choice of $a_{\sigma}$. But we can also choose the $a_{\sigma}$ such that $a_{\sigma} \ind_{\{c_{\sigma|_{n}}\}_{n < \omega}}^{a} \{c_{\eta}\}_{\eta \in \omega^{< \omega}}$ for each $\sigma$, and $a_{\sigma} \ind_{ \{c_{\eta}\}_{\eta \in \omega^{< \omega}}}^{a} a_{\sigma'} $ for distinct $\sigma, \sigma'$, where $A\ind_{C}^{a}B$ means that $\mathrm{acl}(AC) \cap \mathrm{acl}(BC) = C$.  We show that these are as desired. All that remains is the rest of the instances of disjointness in (3), having already shown the disjointness of all of the $M, c_{\eta} \backslash M, C_{\eta ij}$ .

Now, by (3$'$), (2$'$) and choice of $a_{\sigma}$, $a_{\sigma} \cap \mathrm{acl}(\{c_{\sigma|_{n}}\}_{n < \omega}) = M$. So by this and $a_{\sigma} \ind_{\{c_{\sigma|_{n}}\}_{n < \omega}}^{a} \{c_{\eta}\}_{\eta \in \omega^{< \omega}}$,  $a_{\sigma} \cap \mathrm{acl}(\{c_{\eta}\}_{\eta \in \omega^{< \omega}}) = M$. Thus $a_{\sigma} \backslash M$ is disjoint from each of the $M, M\backslash c_{\eta}, C_{\eta ij}$ for all $\sigma \in \omega^{\omega}$.

Moreover, for $\sigma \neq \sigma'$, $a_{\sigma}\backslash M \cup \bigcup^{\infty}_{i=0} A_{\sigma n}$ is disjoint from $a_{\sigma'} \backslash M \cup \bigcup^{\infty}_{i=0} A_{\sigma' n}$: their intersection, by $a_{\sigma} \ind_{ \{c_{\eta}\}_{\eta \in \omega^{< \omega}}}^{a} a_{\sigma'} $ for distinct $\sigma, \sigma'$, will be in $\mathrm{acl}( \{c_{\eta}\}_{\eta \in \omega^{< \omega}})$. So it remains to show that both of $a_{\sigma}\backslash M \cup \bigcup^{\infty}_{n=0} A_{\sigma n}$, $a_{\sigma'} \backslash M \cup \bigcup^{\infty}_{n=0} A_{\sigma' n}$ is disjoint from $\mathrm{acl}( \{c_{\eta}\}_{\eta \in \omega^{< \omega}})$. Without loss of generality we show $a_{\sigma}\backslash M \cup \bigcup^{\infty}_{n=0} A_{\sigma n}$ is disjoint from $\mathrm{acl}( \{c_{\eta}\}_{\eta \in \omega^{< \omega}})$. But we already saw that each $a_{\sigma} \backslash M$ is disjoint  from $\mathrm{acl}(\{c_{\eta}\}_{\eta \in \omega^{< \omega}})$. And by (3$'$) and choice of $a_{\sigma}$, for each $k < \omega$, $A_{\sigma k} \cap \mathrm{acl}(\{c_{\sigma|_{n}}\}_{n < \omega}) \subset c_{\sigma|_{k}}$, So $A_{\sigma k} \cap \mathrm{acl}(\{c_{\sigma|_{n}}\}_{n < \omega}) = \emptyset$. By this and $a_{\sigma} \ind_{\{c_{\sigma|_{n}}\}_{n < \omega}}^{a} \{c_{\eta}\}_{\eta \in \omega^{< \omega}}$, $A_{\sigma k}$ is disjoint from  $\mathrm{acl}( \{c_{\eta}\}_{\eta \in \omega^{< \omega}})$. So , having shown each of $a_{\sigma}$, $A_{\sigma k}$ is disjoint from  $\mathrm{acl}( \{c_{\eta}\}_{\eta \in \omega^{< \omega}})$, we conclude that $a_{\sigma}\backslash M \cup \bigcup^{\infty}_{n=0} A_{\sigma n}$, (and $a_{\sigma'} \backslash M \cup \bigcup^{\infty}_{n=0} A_{\sigma' n}$) is disjoint from $\mathrm{acl}( \{c_{\eta}\}_{\eta \in \omega^{< \omega}})$. This gives us our conclusion that for  $\sigma \neq \sigma'$, $a_{\sigma}\backslash M \cup \bigcup^{\infty}_{n=0} A_{\sigma n}$ is disjoint from $a_{\sigma'} \backslash M \cup \bigcup^{\infty}_{n=0} A_{\sigma' n}$.

Also, any two $A_{\sigma n}, A_{\sigma m}$ for $n \neq m$ are disjoint by (1$'$).

Finally, each $A_{\sigma n}$ is disjoint from $\mathrm{acl}(\{c_{\eta}\}_{\eta \in \omega^{< \omega}})$, as just shown, so disjoint from each of the $c_{\eta} \backslash M, C_{\eta ij}$. We have now checked all of the instances of disjointness among the $M, a_{\sigma} \backslash M, c_{\eta} \backslash M, A_{\sigma n}, C_{\eta ij}$ that are not  immediate from the definitions. So we have shown (3), and we are done.

\end{proof}

\begin{remark}
\label{noninterference in tree property}

Ironically, the existence of an instance of the tree property as in Lemma \ref{noninterference of algebraic closures}, where the specified algebraic closures do not overlap nontrivially as in point (3), does not follow \textit{from the tree property itself}; if it did, so that we could replace the assumption of $\mathrm{SOP}_{2}$ with $\mathrm{TP}$ in Lemma \ref{noninterference of algebraic closures} and Lemma \ref{standard hereditary class},  we would be able to replace $\mathrm{SOP}_{2}$ with $\mathrm{TP}$ in the statement of Theorem \ref{main theorem 3, restated}, contradicting Observation \ref{main observation, restated}. So, as per the examples of generic $K_{n, m}$-free incidence structures and generic binary functions cited as separate justifications for Observation \ref{main observation, restated}, the reader wishing to see concrete examples of how the tree property can occur \textit{even in the absence of an instance of the tree property satisfying these criteria that specific algebraic closures do not nontrivially overlap} can turn to the proofs that these theories are not simple. These are, respectively, the following verifications that some $\mathrm{NSOP}_{1}$ theories are not simple: Corollary 4.17 of \cite{CoK19} (the model companion of the theory $T_{n,m}$ of generic $K_{n, m}$-free incidence structures is not simple) and Proposition 3.14 of \cite{KR18} (the theory $T^{f}$ of the generic binary function is not simple). But historically, the paradigmatic example of this phenomenon is the observation of Shelah (Claim 2.8.6, \cite{She95}) that the theory $T_{feq}^{*}$ of the generic parametrized equivalence relation (introduced in \cite{She94}; shown to be $\mathrm{NSOP}_1 $ in \cite{SU08} as referenced in \cite{Sh99}), also $\mathrm{NSOP}_{1}$ is not simple: a failure of simplicity relies on overlaps between points of the parameter sort that vary between the sets of immediate successors $\{c_{\eta \smallfrown \langle i \rangle} \}_{i < \omega}$ for different $\eta \in \omega^{< \omega}$. This is exposited as part of more general results on generic parametrizations in Corollary 6.18 of \cite{CR15}. And in fact, it appears to be open whether there is \textit{any} (strictly) $\mathrm{NSOP}_{1}$ theory where there is an instance of the failure of simplicity (with consistency witnesses) satisfying the nonoverlapping conditions for algebraic closures from Lemma \ref{noninterference of algebraic closures}: $\varphi(x, y)$ with $\{b_{\eta}\}_{\eta \in \omega^{< \omega}}$ with $\{\varphi(x, b_{\eta \smallfrown \langle i \rangle}) \}_{i < \omega}$ $2$-inconsistent for all $\eta \in \omega^{< \omega}$ and with $\{a_{\sigma}\}_{\sigma \in \omega^{\omega}}$ with $a_{\sigma} \models \{ \varphi(x, b_{\sigma|_{n}})\}_{n < \omega}$ for all $\sigma \in \omega^{\omega}$, chosen such that $\mathrm{tp}(a_{\sigma}b_{\sigma|_{n}})$ does not depend on $\sigma \in \omega^{<\omega}$, $n < \omega$ and $\mathrm{tp}(b_{\eta \smallfrown \langle i \rangle}b_{\eta \smallfrown \langle j \rangle})$ does not depend on $\eta \in \omega^{< \omega} $,  $i < j < \omega$, where (in $T^{\mathrm{eq}}$) all of the sets of the form $\mathrm{acl}(\emptyset), \mathrm{acl}(b_{\eta})\backslash \mathrm{acl}(\emptyset), \mathrm{acl}(a_{\sigma})\backslash \mathrm{acl}(\emptyset),  \mathrm{acl}(a_{\sigma}b_{\sigma|_{n}}) \backslash (\mathrm{acl}(a_{\sigma}) \cup \mathrm{acl}(b_{\sigma|_{n}})), \mathrm{acl}(b_{\eta \smallfrown \langle i \rangle}b_{\eta \smallfrown \langle j \rangle}) \backslash (\mathrm{acl}(b_{\eta \smallfrown \langle i \rangle}) \cup \mathrm{acl}(b_{\eta \smallfrown \langle j \rangle}) )$ (for $i < j$)  are disjoint. We mention this question in the hope that this or similar questions could suggest some general structural phenomenon in $\mathrm{NSOP}_{1}$ theories.

\end{remark}

Having proven our main general lemma on the property $\mathrm{SOP}_{2}$, we now turn to completing the task of translating a hypothetical strictly $\mathrm{NSOP}_{3}$ generic structure into a combinatorially tractable object. As a preliminary step, we prove the following lemma, which generalizes the fact that the terms of a $n$-cycle obtained from $\mathrm{NSOP}_{n}$ can be chosen not to overlap over a base (see Claim 2.19 of \cite{ApproxOrder}). (While we will need, and will prove, only the $n=3$ case of this generalization, the proof will generalize to large values of $n$.)   We will apply this lemma again in the more combinatorial phase of the proof. There we will use it to give generalizations of the \textit{helix map} characterization of $\mathrm{NSOP}_{n}$ (specifically for $n=3$), due to Shelah (\cite{She95}) and Malliaris (\cite{Mal10b}), that was explicated in Fact 3.8 of \cite{ApproxOrder}.

\begin{lemma}
\label{non-overlapping pairs}    Let $\mathcal{H}$ be any hereditary class in a relational language, and let $A \sqcup \bigsqcup_{i < \omega} A_i \sqcup \bigsqcup_{i < j < \omega} A_{ij}  \in \mathcal{H}$ with $A_iA_jA_{ij} \equiv_A A_{i'}A_{j'}A_{i'j'}$ for all $i < j < \omega$, $i' < j' < \omega$.  Suppose $\mathcal{H}$ is the age of a structure whose theory is $\mathrm{NSOP}_{3}$. Then there exists a structure of the form $B \sqcup \bigsqcup^{2}_{i = 0} B_{i} \sqcup \bigsqcup^{2}_{i = 0} B_{i(i+1\mathrm{\:mod\:} 3)} \in \mathcal{H}$ where for $i =0, 1, 2$,  $B B_{i} B_{i+1 \mathrm{\: mod \:} 3}B_{i(i+1\mathrm{\:mod\:} 3)}  \equiv AA_0A_1A_{01}$.
\end{lemma}

\begin{proof}
    Let $T$ be a complete theory whose models have $\mathcal{H}$ as its age, and $\mathbb{M} \models T$ the ambient model (so all of its substructures belong to $\mathcal{H}$). Then there is $a \in \mathbb{M}$, a sequence $\{a_{i}\}_{i< \omega} \in \mathbb{M}$ and a family $\{a_{ij}\}_{i < j< \omega} \in \mathbb{M}$ such that any two sets from among the $a, a_i, a_{ij}$ are disjoint, and such that $aa_ia_ja_{ij} \equiv^{\mathrm{qf}} A A_{0}A_{1}A_{01}$ for $i < j < \omega$, where $\equiv^{\mathrm{qf}}$ denotes isomorphism as enumerated $\mathcal{L}$-structures. Since the existence of such $\{a_{ij}\}_{i < j< \omega}$ for $\{a_{i}\}_{i< \omega}$ is a type-definable condition on $\{a_{i}\}_{i< \omega}$, these can additionally be chosen such that $\{a_{i}\}_{i< \omega}$ is $a$-indiscernible. It will be enough to find $ \{b_i\}^{2}_{i =0}, \{b_{i(i+1\mathrm{\:mod\:} 3)}\}^{2}_{i=0}$ such that any two sets from among the $a, b_i, b_{ij}$ are disjoint, and for $i =0, 1, 2$,  $ab_{i} b_{i+1 \mathrm{\: mod \:} 3}b_{i(i+1\mathrm{\:mod\:} 3)}  \equiv^{\mathrm{qf}} AA_0A_1A_{01}$. Let $\{a_i\}_{i < \omega\times\omega}$ be an $a$-indiscernible sequence extending $\{a_{i}\}_{i< \omega}$, and let $p(\overline{x}, \overline{y})=\mathrm{tp}(\{a_i\}_{i< \omega}, \{a_i\}_{\omega\leq i < \omega+\omega}/a)$. Then the $a$-indiscernible sequence $\{\overline{a}_{i}\}_{i < \omega} : = \{a_{i \omega} \ldots a_{i\omega + n}, \ldots\}_{i < \omega}$ satisfies $\overline{a}_{0}, \overline{a}_1\models p(\overline{x}, \overline{y})$. So because $T$ is $\mathrm{NSOP}_{3}$, we may find (just as we did in the proof of Claim 2.19 of \cite{ApproxOrder}) $\{a^\ell_i\}_{i < \omega, 0 \leq \ell \leq 2}$ such that, for $0 \leq \ell \leq 2$, $\{a^\ell_i\}_{i < \omega} , \{a^{\ell+1 \mathrm{\: mod \:}3}_i\}_{i < \omega} \models p(\overline{x}, \overline{y})$. So for each fixed $0 \leq \ell \leq 2$, we may find $\{a^{\ell (\ell+1 \mathrm{\: mod\: } 3)}_{ij}\}_{i,j < \omega}$  such that $a a^{\ell}_i a_{j}^{\ell+1 \mathrm{\:mod\:}3} a^{\ell (\ell+1 \mathrm{\: mod\: } 3)}_{ij}\equiv^{\mathrm{qf}} AA_0A_1A_{01}$ for all $i < j < \omega$, and any two sets from among the $a, a^{\ell}_i, a^{\ell+1 \mathrm{\: mod \:}3}_{i},  a^{\ell (\ell+1 \mathrm{\: mod\: } 3)}_{ij}$ are disjoint. Moreover, for $\ell \neq \ell'$ and $i, i' < \omega$, $a^{\ell}_{i}$ is disjoint from $a^{\ell'}_{i'}$: without loss of generality we may assume that $\ell' = \ell + 1 \mathrm{\: mod \:} 3$, so $a^{\ell}_{i}a^{\ell'}_{i'} \equiv^{\mathrm{qf}}a_{0}a_{1}$, and $a^{\ell}_{i}$ is disjoint from $a^{\ell'}_{i'}$ because $a_{0}$, $a_{1}$ are disjoint. Thus all of the $a^{\ell}_{i}$ are disjoint.

    Now to find $ b_i, b_{ij}$ as desired, we may use a more direct combinatorial argument (involving extracting a long indiscernible sequence) to ensure that the sets $b_{ij}$ do not overlap. However, if we want to argue entirely by extracting indiscernible sequences, we will need to use an extended language in a way that will appear more artificial.  We may assume that the sets $A_i$ in the statement we want to prove are nonempty--otherwise, the statement is obvious. So the sets $a^{\ell}_i$ are all nonempty. Write the $a_i^{\ell}, $$a^{\ell (\ell+1 \mathrm{\: mod\: } 3)}_{ij}$ more explicitly as $\{(a_i^{\ell})_{\mathbf{i}}\}_{\mathbf{i} \in I}$ and $\{(a^{\ell (\ell+1 \mathrm{\: mod\: } 3)}_{ij})_{\mathbf{j}}\}_{\mathbf{j} \in J}$, according to their (previously implicit) enumerations. Then choose some $\mathbf{i}^{*} \in I$. Because the $a^{\ell}_i$ are all disjoint, $\ell$, $i$ will depend only on $(a_{i}^{\ell})_{\mathbf{i}^{*}}$. Therefore, we may find a family of functions $\{f_{\mathbf{j}}: \mathbb{M}^{2} \to \mathbb{M}\}_{\mathbf{j} \in J}$ such that for all $i, j, \ell, \mathbf{j}$, $f_{\mathbf{j}}((a^{\ell}_i)_{\mathbf{i}^{*}}, (a^{\ell+1 \mathrm{\: mod \:} 3}_{j})_{\mathbf{i}^{*}})= (a^{\ell (\ell+1 \mathrm{\: mod\: } 3)}_{ij})_{\mathbf{j}}$. For any $a=\{a_{\mathbf{i}}\}_{\mathbf{i} \in I}$, $b=\{b_{\mathbf{i}}\}_{\mathbf{i} \in I}$, denote $\overline{f}(a, b):= \{f_{\mathbf{j}}((a)_{\mathbf{i}^{*}}, (b)_{\mathbf{i}^{*}})\}_{\mathbf{j} \in J}$. Let $\mathbb{M}_{f}$ be the structure, in the original relational language expanded by a family $\{\mathbf{f}_{\mathbf{j}}\}_{\mathbf{j} \in J}$ of new binary function symbols, consisting of $\mathbb{M}$ equipped with the functions $\{f_{\mathbf{j}}: \mathbb{M}^{2} \to \mathbb{M}\}_{\mathbf{j} \in J}$  with $\mathbf{f}_{\mathbf{j}}$ interpreted as $f_{\mathbf{j}}$. Let $T_{f}$ be the theory of $\mathbb{M}_{f}$ Now define, as in, say, \cite{Che14}, a family of sequences $\{\{c_{i}^{j}\}_{i < \omega} : j < n \}$ to be \textit{mutually indiscernible} over a set $C$ if, for each $j < n$, $\{c_{i}^{j}\}_{i < \omega}$ is indiscernible over the $C\overline{c}^{\neq j}$  ($C$ together with all of the other sequences.) By applying Lemma 1.2 of \cite{Che14} within $T_{f}$ to the sequences  $\{a^0_i\}_{i < \omega}$, $\{a^1_i\}_{i < \omega}$,  $\{a^2_i\}_{i < \omega}$, we may then find $\overline{\mathbb{M}}_{f} \succ \mathbb{M}_{f}$ containing mutually $a$-indiscernible sequences $\{b^0_i\}_{i < \omega}$, $\{b^1_i\}_{i < \omega}$,  $\{b^2_i\}_{i < \omega}$ such that, for $\ell < 3$, $\{b^{\ell}_i\}_{i < \omega} \{b^{\ell+1 \mathrm{\: mod\:} 3}_i\}_{i < \omega} \equiv_{a} \{a^{\ell}_i\}_{i < \omega} \{a^{\ell+1 \mathrm{\: mod\:} 3}_i\}_{i < \omega}$.\footnote{More technically, the actual statement of this lemma from \cite{Che14} would apply to indiscernible sequences $\{a^0_i\}_{i < \lambda}$, $\{a^1_i\}_{i < \lambda}$,  $\{a^2_i\}_{i < \lambda}$ for some $\lambda$,. While the condition $\{b^{\ell}_i\}_{i < \omega} \{b^{\ell+1 \mathrm{\: mod\:} 3}_i\}_{i < \omega} \equiv_{a} \{a^{\ell}_i\}_{i < \omega} \{a^{\ell+1 \mathrm{\: mod\:} 3}_i\}_{i < \omega}$ is type-definable, so it is not really necessary to start with longer sequences to show our statement that the  $\{b^0_i\}_{i < \omega}$, $\{b^1_i\}_{i < \omega}$,  $\{b^2_i\}_{i < \omega}$ exist, to apply the lemma as stated, we can just first use compactness to obtain $\{a^0_i\}_{i < \lambda}$, $\{a^1_i\}_{i < \lambda}$,  $\{a^2_i\}_{i < \lambda}$ for which the EM-type of the sequence $\{a_{i}^{0}a_{i}^{1}a_{i}^{2}\}_{i< \lambda}$ contains the EM-type of $\{a_{i}^{0}a_{i}^{1}a_{i}^{2}\}_{i< \omega}$, then apply the lemma from \cite{Che14} to $\{a^0_i\}_{i < \lambda}$, $\{a^1_i\}_{i < \lambda}$,  $\{a^2_i\}_{i < \lambda}$. } It follows that $a b^{\ell}_i b_{j}^{\ell+1 \mathrm{\:mod\:}3} \overline{f}(b^{\ell}_i ,b_{j}^{\ell+1 \mathrm{\:mod\:}3})\equiv^{\mathrm{qf}} AA_0A_1A_{01}$ for all $i < j < \omega$ and $\ell < 3$. It also follows that, for each \textit{fixed} $\ell < 3$, any two sets from among the $a, b^{\ell}_i,  \overline{f}(b^{\ell}_i ,b_{j}^{\ell+1 \mathrm{\:mod\:}3})$ are disjoint. We show $b_{i} : =b^{0}_{i}$, $b_{ij} : = \overline{f}(b^{i}_0 ,b^{j}_{0})$ are as desired for $i < 3$, $j = i+1 \mathrm{\: mod\:} n$. (These will be in $\overline{\mathbb{M}}_{f}$, but of course we can then find them as desired in $\mathbb{M}$.)  By $a b^{\ell}_i b_{j}^{\ell+1 \mathrm{\:mod\:}3} \overline{f}(b^{\ell}_i ,b_{j}^{\ell+1 \mathrm{\:mod\:}3})\equiv^{\mathrm{qf}} AA_0A_1A_{01}$, all that remains is to show that $b_{01}, b_{12}, b_{20}$ are disjoint. Without loss of generality, it suffices to show that $b_{01} := \overline{f}(b^{0}_0 ,b_{0}^{1})$ is disjoint from $b_{12} :=\overline{f}(b^{1}_0 ,b_{0}^{2})$. But for $i \neq j$, any two of $\overline{f}(b^{0}_i ,b_{0}^{1})$, $\overline{f}(b^{0}_j ,b_{0}^{1})$ are disjoint.  Now suppose $\overline{f}(b^{0}_0 ,b_{0}^{1})$ were not disjoint from $\overline{f}(b^{1}_0 ,b_{0}^{2})$, so there is some $c \in \overline{f}(b^{1}_0 ,b_{0}^{2})$ which is also in $\overline{f}(b^{0}_0 ,b_{0}^{1})$.  Then, by mutual $a$-indiscernibility of $\{b^0_i\}_{i < \omega}$, $\{b^1_i\}_{i < \omega}$,  $\{b^2_i\}_{i < \omega}$ in $T_{f}$, each $b^{0}_{i}$ satisfies the same type in $T_{f}$ over $b^{1}_{0} b^{2}_{0}$ (i.e., not depending on $i$). So because $c \in \overline{f}(b^{0}_0 ,b_{0}^{1})$, it is also the case that $c \in \overline{f}(b^{0}_1 ,b_{0}^{1})$. And this contradicts that $\overline{f}(b^{0}_0 ,b_{0}^{1})$ and $\overline{f}(b^{0}_1 ,b_{0}^{1})$ are disjoint.
\end{proof}

We are now ready to prove our lemma giving the interface between our model-theoretic hypothesis--the existence of a strictly $\mathrm{NSOP}_{3}$ example among the constructions of Cherlin, Shelah and Shi (\cite{CSS99})--and the combinatorial arguments that will follow. This will be analogous to Lemma 3.3 of \cite{ApproxOrder} in the proof that $\mathfrak{o}(\mathcal{H})$ is an integer for $\mathcal{H}$ a hereditary class defined by finitely many forbidden weakly embedded substructures. There, we started with a relation $R$ exhibiting $\mathrm{SOP}_{r}$ within an $\mathrm{NSOP}_{n}$ hereditary class defined by finitely many forbidden weakly embedded substructures, and obtained a $\mathrm{NSOP}_{n}$ hereditary class, which we saw was \textit{still} defined by finitely many forbidden weakly embedded substructures, in a language with just a edge relation symbol representing this relation $R$ exhibiting $\mathrm{SOP}_{r}$. Analogously, here we see that if $T^{\mathcal{H}}$, for $\mathcal{H}$ a hereditary class defined by finitely many forbidden weakly embedded substructures, is $\mathrm{NSOP}_{3}$ but has $\mathrm{SOP}_{2}$, then $\mathrm{SOP}_{2}$, but in particular $\mathrm{TP}$, can be viewed as being represented by a relation $R(x, y)$ that exhibits the tree property, and a relation $I(y_{1}, y_{2})$ that exhibits the instances of inconsistency in the tree property. Our lemma will give us a hereditary class that is also the age of a structure whose theory is $\mathrm{NSOP}_{3}$, in the language with symbols for $R$ and $I$ where these symbols for $R$ and $I$ also represent the tree property, that is \textit{still} defined by finitely many forbidden weakly embedded substructures in this language. Within an instance of the tree property exhibited by $\varphi(x,y)$, $\{b_{\eta}\}_{\eta \in \omega^{< \omega}}$  with consistency witnesses $a_{\sigma} \models \{ \varphi(x, b_{\sigma|_{n}})\}_{n < \omega}$ for $\sigma \in \omega^{\omega}$, the relation $R(x, y)$ will correspond to algebraic closures of the form $\mathrm{acl}(a_{\sigma}b_{\sigma|_{n}})$ that give the incidences of $\varphi(x, y)$, and the relation $I(y_{1}, y_{2})$ will correspond to algebraic closures of the form $\mathrm{acl}(b_{\eta \smallfrown \langle i \rangle}b_{\eta \smallfrown \langle j \rangle})$ that exhibit $2$-inconsistency. It will be essential that our instance of $\mathrm{TP}$ will be of the special kind where these algebraic closures do not overlap nontrivially, as in the above Lemma \ref{noninterference of algebraic closures}.

Before stating our lemma, we must give some definitions. We first give a standard language in which we can express in a purely formal way an instance of the tree property.

\begin{definition}
  \label{standard tp-structure}  Let $\mathcal{L}_{\mathrm{std}}$ be the language in sorts $O, P$ consisting of the relation symbols $R$ and $I$, where $R$ is a symbol for a binary relation between sorts $O$ and $P$, and $I$ is a symbol for a binary relation on $P$. Then the \textit{standard $\mathrm{TP}$-structure} is the $\mathcal{L}_{\mathrm{std}}$-structure $\mathbb{TP}$ consisting of distinct $b_{\eta} \in P(\mathbb{TP})$ for $\eta \in \omega^{< \omega}$, and $a_{\sigma} \in O(\mathbb{TP})$ for $\sigma \in \omega^{\omega}$, with all and only the following instances of $R$ and $I$: $I(b_{\eta \smallfrown \langle i \rangle},b_{\eta \smallfrown \langle j \rangle})$ for $\eta \in \omega^{< \omega}$, $i < j < \omega$, and $R(a_{\sigma}, b_{\sigma|_{n}} )$ for $\sigma \in \omega^{\omega}$, $n < \omega$.
\end{definition}

Since $I(y_{1}, y_{2})$ is intended to represent inconsistency between two instances of the relation $R(x, y)$, we express this requirement in the following definition:

\begin{definition}
    \label{sound structure} An $\mathcal{L}_{\mathrm{std}}$-structure $A$ is \textit{sound} if:
    
    --there are no distinct $a \in O(A)$, $b_1, b_2 \in P(A)$ with $R(a, b_i)$ for $i = 1, 2$ and $I(b_1,b_2)$, 
    
    --$I$ is irreflexive: there is no $b \in P(A)$ with $I(b, b)$.
\end{definition}
With these definitions in hand, we can now state our lemma:

\begin{lemma}
 \label{standard hereditary class}   Suppose that $\mathcal{H}$ is a hereditary class defined by a finite family of forbidden weakly embedded substructures, and that  $T^{\mathcal{H}}$ is a strictly $\mathrm{NSOP}_{3}$ complete theory (i.e., an $\mathrm{NSOP}_{3}$ theory with $\mathrm{SOP}_{2}$). Then there is a hereditary class $\mathcal{H}^{\mathrm{std}}$ consisting of sound $\mathcal{L}_{\mathrm{std}}$-structures, defined by a finite family of forbidden weakly embedded substructures, which is the age of some structure whose theory is $\mathrm{NSOP}_{3}$, and such that $\mathbb{TP} \in \mathcal{H}^{\mathrm{std}}$.
\end{lemma}

\begin{proof}
    Let the formula $\varphi(x, y)$ exhibit $\mathrm{SOP}_{2}$, so $\mathrm{SOP}_{1}$, in $T^{\mathcal{H}}$. Let  the $\mathcal{L}^{\mathrm{Sk}}$-theory $(T^{\mathcal{H}})^{\mathrm{Sk}}$ be a Skolemization of $T^{\mathcal{H}}$, and unless otherwise noted let $\mathrm{acl}$ denote $\mathrm{acl}^{(T^{\mathcal{H}})^{\mathrm{Sk}}}$. By the coheir characterization of $\mathrm{SOP}_{1}$ given by Fact \ref{sop1 coheirs}, there will exist coheir Morley sequences $\{\varphi(x, b_{i})\}_{i < \omega}$ and $\{\varphi(x, b^{i})\}_{i < \omega}$ over some model $M \prec \mathbb{M} \models (T^{\mathcal{H}})^{\mathrm{Sk}}$ such that $b_0 = b^0 = b$, $\{\varphi(x, b_{i})\}_{i < \omega}$ is $2$-inconsistent, and $\{\varphi(x, b^{i})\}_{i < \omega}$ is consistent. So we can apply Lemma \ref{noninterference of algebraic closures} to $\varphi(x, y)$ within $(T^{\mathcal{H}})^{\mathrm{Sk}}$. This gives us a model $M \prec \mathbb{M} \models (T^{\mathcal{H}})^{\mathrm{Sk}}$ (where $\mathbb{M}$ is a sufficiently saturated ambient model of $(T^{\mathcal{H}})^{\mathrm{Sk}}$), a tree $\{c_{\eta}\}_{\eta \in \omega^{< \omega}} \in \mathbb{M}$ of algebraically closed sets (in $T^{\mathrm{Sk}}$) containing $M$, and an $\omega^{\omega}$-indexed set $\{a_{\sigma}\}_{\sigma \in \omega^{\omega}} \in \mathbb{M}$ of algebraically closed sets (in $T^{\mathrm{Sk}}$) containing $M$, subject to the following conditions\footnote{Again, we may add variables to $\varphi(x, y)$ and incorporate parameters with index sets strictly containing those of the variables $x, y$, using the same notation as in the proof of Lemma \ref{noninterference of algebraic closures}.}:

    (a) For $\eta \in \omega^{< \omega}$, $\{\varphi(x, c_{\eta \smallfrown \langle i \rangle})\}_{i < \omega}$ is $2$-inconsistent.

    (b) For $\sigma \in \omega^{\omega}$,  $a_{\sigma} \models  \{\varphi(w, c_{\sigma|_{n}})\}_{n < \omega}$.

    (c) For $\sigma \in \omega^{\omega}$, let $A_{\sigma n}:=\mathrm{acl}(a_{\sigma}c_{\sigma|_{n}})\backslash(a_{\sigma} \cup c_{\sigma|_{n}})$. For $\eta \in \omega^{< \omega}$, $i < j < \omega$ let $C_{\eta ij}:= \mathrm{acl}(c_{\eta \smallfrown \langle i \rangle } c_{\eta \smallfrown \langle j \rangle }) \backslash (c_{\eta \smallfrown \langle i \rangle } \cup c_{\eta \smallfrown \langle j \rangle }) $. Then all of the $M, a_{\sigma} \backslash M,c_{\eta} \backslash M, A_{\sigma n}, C_{\eta ij}$ are disjoint.

    (d) For $C_{\eta ij}$, $A_{\sigma n}$ as defined in (c), $\mathrm{qftp}^{\mathcal{L}}(a_{\sigma} A_{\sigma n}c_{\sigma|_{n}})$ does not depend on $\sigma, n$. Moreover, $ \mathrm{qftp}^{\mathcal{L}}(c_{\eta \smallfrown \langle i \rangle }C_{\eta ij} c_{\eta \smallfrown \langle j \rangle }/M)$ does not depend on $\eta$, $i, j$ with $i < j$. 

Here (d) comes from the fact that, in applying Lemma \ref{noninterference of algebraic closures}, the statement of this lemma says in $T^{\mathrm{Sk}}$ that $\mathrm{tp}^{\mathcal{L}^{\mathrm{Sk}}}(a_{\sigma}c_{\sigma|_{n}})$ does not depend on $\sigma, n$ and  $ \mathrm{tp}^{\mathcal{L}^{\mathrm{Sk}}}(c_{\eta \smallfrown \langle i \rangle } c_{\eta \smallfrown \langle j \rangle }/M)$ does not depend on $\eta$, $i, j$ with $i < j$. We then choose enumerations of $A_{\sigma n}:=\mathrm{acl}^{(T^{\mathcal{H}})^{\mathrm{Sk}}}(a_{\sigma}c_{\sigma|_{n}})\backslash(a_{\sigma} \cup c_{\sigma|_{n}})$, $C_{\eta ij}:= \mathrm{acl}^{(T^{\mathcal{H}})^{\mathrm{Sk}}}(c_{\eta \smallfrown \langle i \rangle } c_{\eta \smallfrown \langle j \rangle }) \backslash (c_{\eta \smallfrown \langle i \rangle } \cup c_{\eta \smallfrown \langle j \rangle }) $ satisfying (d).

In (d), let $p(w, W, z)  :=\mathrm{qftp}^{\mathcal{L}}(a_{\sigma} A_{\sigma n}c_{\sigma|_{n}}/M)$ and $q(z_{1}, Z, z_{2}):= \mathrm{qftp}^{\mathcal{L}}(c_{\eta \smallfrown \langle i \rangle }C_{\eta ij} c_{\eta \smallfrown \langle j \rangle }/M)$. (Note that these quantifier-free types are in $\mathcal{L}$, not $\mathcal{L}^{\mathrm{Sk}}$.) Now let $\mathcal{H}^{\mathrm{std}}$ be the hereditary class of $\mathcal{L}^{\mathrm{std}}$-structures defined as follows. The hereditary class $\mathcal{H}^{\mathrm{std}}$ will consist only of $\mathcal{L}^{\mathrm{std}}$-structures for which $I$ is antisymmetric and anti-reflexive: there are no $A \in \mathcal{H}^{\mathrm{std}}$ and $b_1, b_2 \in P(A)$ with $I(b_1, b_2)$ and $I(b_2, b_1)$. And for $A$ such a $\mathcal{L}^{\mathrm{std}}$-structure, $A \in \mathcal{H}^{\mathrm{std}}$ if and only if  $\tilde{A} \in \mathcal{H}$, where $\tilde{A}$ is the $\mathcal{L}$-structure defined as follows. The $\mathcal{L}$-structure $\tilde{A}$ will consist of the disjoint union of:

  (1) the model $M \models T^{\mathcal{H}}$,

    (2) for each $a \in O(A)$, a set $\tilde{a}$ satisfying $\mathrm{qftp}(a_{\sigma}\backslash M/M)$,

    (3) for each $b \in P(A)$, a set $\tilde{b}$ satisfying $\mathrm{qftp}(b_{\eta}\backslash M/M)$,

    (4) for each $a \in O(a), b \in P(A)$ with $R(a, b)$, a set $R^{ab}$ realizing $p(\tilde{a}M, W, \tilde{b}M)$, and

     (5) for each $ b_1, b_{2} \in P(A)$ with $I(b_{1}, b_{2})$, a set $I^{ab}$ realizing $q(\tilde{b_{1}}M, Z, \tilde{b_{2}}M)$

     and  the set of instances of the relation symbols of $\mathcal{L}$ on $\tilde{A}$ will be \textit{minimal} such that $M$ is an induced substructure of $\tilde{A}$ and the requirements of (2)-(5) are satisfied. (In particular, when $a$ and $b$ are related by $R$, the points $\tilde{a}\tilde{b}M$ themselves must satisfy the instances of the relation symbols imposed by $p(\tilde{a}M, W, \tilde{b}M)$. A similar consideration holds for any $b_{1}, b_{2}$ related by $I$, and $q(\tilde{b_{1}}M, Z, \tilde{b_{2}}M)$. Note that $\tilde{A}$ will be well-defined, because $I$ is antireflexive and asymmetric while $R$ is between two different sorts, so the sets of instances of relation symbols imposed on the $\tilde{a}\tilde{b}M$, $\tilde{b}_{1}\tilde{b}_{2}M$ as in (4) and (5) are well-defined.)

     We first show that all of the structures in $\mathcal{H}^{\mathrm{std}}$ are sound. Otherwise (by definition, $I$ is irreflexive) there will be $\{a, b_1, b_2\} \in \mathcal{H}^{\mathrm{std}}$ with $R(a, b_i)$, $I(b_1, b_2)$. Then $\widetilde{\{a, b_1, b_2\}} \in \mathcal{H}$ so by model completeness of $T^{\mathcal{H}}$ we can assume $\widetilde{\{a, b_1, b_2\}} \subset \mathbb{M}$.  So there will be $\tilde{a}, \tilde{b_{1}}, \tilde{b_{2}} \subset \mathbb{M}$ and, for $i= 1, 2$, $R^{i} \models p(\tilde{a}M, W, \tilde{b_{i}}M)$, as well as $I \models q(\tilde{b_{1}}M, Z, \tilde{b_{2}}M)$, both within $\mathbb{M}$. But $Mb_{1}b_{2}I$ will be a model of $T^{\mathcal{H}}$ containing $Mb_{1}b_{2}$, since we specifically applied Lemma \ref{noninterference of algebraic closures} \textit{within the Skolemization} $(T^{\mathcal{H}})^{\mathrm{Sk}}$. By $I \models q(\tilde{b_{1}}M, Z, \tilde{b_{2}}M)$, $Mb_{1}b_{2}I\equiv^{\mathrm{qftp}} M (c_{\langle0 \rangle}\backslash M) (c_{\langle 1 \rangle}\backslash M)C_{\emptyset 01}$.  So by model completeness, $Mb_{1}b_{2}I\equiv M (c_{\langle0 \rangle}\backslash M) (c_{\langle 1 \rangle}\backslash M)C_{\emptyset 01}$. Particularly, $Mb_{1}b_{2}\equiv M (c_{\langle0 \rangle}\backslash M) (c_{\langle 1 \rangle}\backslash M)$. Similarly, by $R^{i} \models p(\tilde{a}M, W, \tilde{b_{i}}M)$, $Mab_{i}\equiv M (a_{\langle 0 \rangle^{\omega}}\backslash M) (c_{\langle 0 \rangle}\backslash M)$ for $i = 1, 2$. But this contradicts (a), (b).

     Moreover, $\mathbb{TP}\in \mathcal{H}^{\mathrm{std}}$, because $\widetilde{\mathbb{TP}} \in \mathcal{H}$: by (c) and the definition of $p(w, W, z)$ $q(z_{1}, Z, z_{2})$, the partitioned set consisting of the $M, a_{\sigma} \backslash M,c_{\eta} \backslash M, A_{\sigma n}, C_{\eta ij}$ obtained at the beginning of this proof defines a weak embedding from $\widetilde{\mathbb{TP}}$ into $\mathbb{M}$. So $\widetilde{\mathbb{TP}}$ weakly embeds into some $A \in \mathcal{H}$, and therefore  $\widetilde{\mathbb{TP}} \in \mathcal{H}$.

We next show that $\mathcal{H}$ is defined by a finite family of forbidden weakly embedded substructures. Let $\mathcal{H} = \mathcal{H}(\mathcal{F})$ for $\mathcal{F}$ finite, and let $n \geq 1$ be at least as large as the largest cardinality of any $F \in \mathcal{F}$. Then define $\mathcal{F}'$ to be the set of $\mathcal{L}^{\mathrm{std}}$-structures $A \notin \mathcal{H}^{\mathrm{std}}$ with $|A| \leq 2n$. We show that $\mathcal{H}^{\mathrm{std}} = \mathcal{H}(\mathcal{F}')$. For $\mathcal{H}^{\mathrm{std}} \supseteq \mathcal{H}(\mathcal{F}')$, suppose $A$ weakly embeds $F \in \mathcal{F}'$; we show $A \notin \mathcal{H}^{\mathrm{std}}$. We may assume $I$ is irreflexive and antisymmetric on $A$, so $\tilde{A}$, and $\tilde{F}$, is well-defined.  But $\tilde{F} \notin \mathcal{H}$, and the weak embedding $F\hookrightarrow A$ gives us a weak embedding $\tilde{F} \hookrightarrow \tilde{A}$, so $\tilde{A} \notin \mathcal{H}$ and $A \notin \mathcal{H}^{\mathrm{std}}$. 

To show $\mathcal{H}^{\mathrm{std}} \subseteq \mathcal{H} (\mathcal{F}')$, suppose $A \notin \mathcal{H}^{\mathrm{std}}$. We may assume $I$ is irreflexive and antisymmetric on $A$, because otherwise $A$ weakly embeds some $A_{0} \notin \mathcal{H}^{\mathrm{std}}$ with $|A_{0}| \leq 2$; then $A_{0} \in \mathcal{F}'$, and $A \notin \mathcal{H} (\mathcal{F}')$. We want to find $A_{0} \in \mathcal{F}'$ with $A_{0} \subset A$. We have  $\tilde{A} \notin \mathcal{H}$, so for some induced substructure $F \subseteq \tilde{A}$, $|F| \leq n$ and $F \notin\mathcal{H}$. Let the induced substructure $A_{0} \subset A$ consist of each of those $a \in O(A)$ for which $\tilde{a}$ contains a point of $F$, each of those $b \in O(A)$ for which $\tilde{b}$ contains a point of $F$, both of $a \in O(A)$, $b \in P(A)$ whenever $R^{ab}$ contains a point of $F$, and both of $b_1, b_2 \in P(A)$ whenever $I^{b_{1}b_{2}}$ contains a point of $F$. Since $|F| \leq n$ and the $\tilde{a}$, $\tilde{b}$, $R^{ab}$, $I^{b_{1}b_{2}}$ are all disjoint, $|A_{0}| \leq 2n$. Moreover, $\tilde{A_{0}} \supset F$ with $F$ as an induced substructure, so $\tilde{A_{0}} \notin \mathcal{H}$ because $F \notin\mathcal{H}$.  So $A_{0} \notin \mathcal{H}^{\mathrm{std}}$, and $A_{0} \in \mathcal{F}'$ because $|A_0| \leq 2n$, with $A_{0} \subset A$ as desired. 

Thus, by Fact \ref{Cherlin Shelah and Shi theorem}, $T^{\mathcal{H}^{\mathrm{std}}}$ is well-defined. We show that $T^{\mathcal{H}^{\mathrm{std}}}$ is complete; then $\mathcal{H}^{\mathrm{std}}$ will be the age of any model of $T^{\mathcal{H}^{\mathrm{std}}}$, as in the proof of Proposition \ref{reduction to generic structures}. As in the proof of that proposition it suffices to show that any two members of $\mathcal{H}^{\mathrm{std}}$ embed as induced substructures into a common member of $\mathcal{H}^{\mathrm{std}}$. Let $A, B \in \mathcal{H}^{\mathrm{std}}$; we show that $A \sqcup B \in \mathcal{H}^{\mathrm{std}}$ for $A \sqcup B $ the free amalgam of $A$ and $B$ over $\emptyset$, which will give us this common extension. Now, $\tilde{A} , \tilde{B} \in \mathcal{H}$, so by model completeness of $T^{\mathcal{H}}$, we may assume that $\tilde{A}, \tilde{B} \subset \mathbb{M} \models T^{\mathcal{H}}$. Since $M$ is algebraically closed in $\mathbb{M}$ (in the sense of $T^{\mathcal{H}}$), we may even, finding $M$-isomorphic copies $\tilde{A}', \tilde{B}'$ of $\tilde{A}, \tilde{B}$ with $\tilde{A}' \ind_{M}^{a} \tilde{B}'$, assume that $\tilde{A} \cap \tilde{B} = M$.  Then the free amalgam $\tilde{A} \sqcup_{M} \tilde{B}$ of $\tilde{A}$ and $\tilde{B}$ over $M$ will weakly embed in $\tilde{A} \cup \tilde{B} \subset \mathbb{M}$. But $\tilde{A} \cup \tilde{B}  \in \mathcal{H}$, so $\widetilde{A \sqcup B}= \tilde{A} \sqcup_{M} \tilde{B} \in \mathcal{H}$ and $A \sqcup B \in \mathcal{H}$.

     It remains to show $T^{\mathcal{H}^{\mathrm{std}}}$  is $\mathrm{NSOP}_{3}$, so that $\mathcal{H}^{\mathrm{std}}$ is the age of a structure with $\mathrm{NSOP}_{3}$ theory. It suffices to show no $\mathcal{L}^{\mathrm{std}}$-formula exhibits $\mathrm{SOP}_{3}$ in the Skolemization $(T^{\mathcal{H}^{\mathrm{std}}})^{\mathrm{Sk}}$ .  Let $(\mathbb{M}^{\mathrm{std}})^{\mathrm{Sk}} \models (T^{\mathcal{H}^{\mathrm{std}}})^{\mathrm{Sk}}$ be a sufficiently saturated ambient model of $(T^{\mathcal{H}^{\mathrm{std}}})^{\mathrm{Sk}}$. Let $\varphi(x, y)$ be a $\mathcal{L}^{\mathrm{std}}$-formula, and suppose that there is $\{a_{i}\}_{i \in \mathbb{Z}} \in (\mathbb{M}^{\mathrm{std}})^{\mathrm{Sk}}$ indiscernible in the sense of $(T^{\mathcal{H}^{\mathrm{std}}})^{\mathrm{Sk}}$ such that $(\mathbb{M}^{\mathrm{std}})^{\mathrm{Sk}} \models \varphi(a_{i}, a_{j}) $ for $i < j$. We must show that there are $b_0, b_1, b_2 \in (\mathbb{M}^{\mathrm{std}})^{\mathrm{Sk}}$ with $\varphi(b_0, b_1),\varphi(b_1, b_2), \varphi(b_2, b_0) $. We may assume that for some $\mathcal{M} \prec (\mathbb{M}^{\mathrm{std}})^{\mathrm{Sk}}$, $\{a_{i}\}_{i \in \mathbb{Z}}$ is indiscernible over $\mathcal{M}$ with $\mathrm{tp}(a_{0}/a_{\neq 0} \mathcal{M})$ finitely satisfiable in $\mathcal{M}$. (This is standard: extending $\{a_{i}\}_{i \in \mathbb{Z}}$ to $\{\alpha_{i}\}_{i \in \mathbb{Z}} \smallfrown \{a_{i}\}_{i \in \mathbb{Z}} \smallfrown \{\beta_{i}\}_{i \in \mathbb{Z}}$ , take $\mathcal{M} := \mathrm{acl}^{\mathrm{Sk}}(\{\alpha_i\}_{i \in \mathbb{Z}}\{\beta_i\}_{i \in \mathbb{Z}})$.) So $\{a_{i}\}_{i \in \mathbb{Z}}$ is a Morley sequence over $\mathcal{M}$ in a normal $\mathcal{M}$-coheir: $\{a_{i}\}_{i \in \mathbb{Z}}$ is a coheir Morley sequence over $\mathcal{M}$ when read in reverse, in the sense of $(T^{\mathcal{H}^{\mathrm{std}}})^{\mathrm{Sk}}$. Therefore, $\mathrm{tp}(a_{0} \ldots a_{n-1} /\mathcal{M} a_{n} \ldots a_{2n-1} \cup a_{2n} )$ extends to a $\mathcal{M}$-invariant type, so a $\mathcal{M} a_{2n}$-invariant type. So letting $\mathrm{acl}_{\mathcal{M}}, \mathrm{acl}$ denote $\mathrm{acl}_{\mathcal{M}}^{\mathrm{Sk}}, \mathrm{acl}^{\mathrm{Sk}}$, $\mathrm{acl}_{\mathcal{M}}(a_{0} \ldots a_{n-1} \cup a_{2n}) \cap \mathrm{acl}_{\mathcal{M}}(a_{n} \ldots a_{2n-1} \cup a_{2n})=\mathrm{acl}_{\mathcal{M}}(a_{2n})$. (We really only need this for $n =1$.) Let $M_{i}:= \mathrm{acl}(\mathcal{M}a_{i})\backslash \mathcal{M}$ and $M_{ij}= \mathrm{acl}(\mathcal{M}a_{i}a_{j}) \backslash (\mathcal{M} \sqcup M_{i} \sqcup M_{j} )$ for $i < j < \omega$. Then by Lemma \ref{algebraic closures in normal coheir morely sequence}, the $\mathcal{M}, M_i, M_{ij}$ are all disjoint.

     Thus, just in the original language $\mathcal{L}^{\mathrm{std}}$, we have found $\mathcal{M}, M_{i}, M_{ij} \in \mathbb{M}^{\mathrm{std}} :=(\mathbb{M}^{\mathrm{std}})^{\mathrm{Sk}}|_{\mathcal{L}^{\mathrm{std}}}$ such that:

     (i) the $\tilde{M}$, $\tilde{M}M_{i}$, $\tilde{M}M_{i}M_{j}M_{ij}$ are all models of $T^{\mathcal{H}^{\mathrm{std}}}$ (because they are Skolem closures).

     (ii) the $\mathcal{M}, M_i, M_{ij}$ are all disjoint

     (iii)  For all $i < j$, $\varphi(\mathcal{M}M_{i}, \mathcal{M}M_{j})$\footnote{(again adopting a similar notation to before for extended sets of parameters: $\varphi(A, B)$ will mean that the  parameters assigned to the variables $x, y$ will satisfy $\varphi(x, y)$.}.

     Now let the $\mathcal{L}^{\mathrm{std}}$-structure $B:= \mathcal{N} \sqcup \bigsqcup^{2}_{i=0} N_{i} \sqcup N_{01} \sqcup N_{12} \sqcup N_{20}$ such that:

     \begin{itemize}
         \item for $i \leq 2$,  $\mathcal{N} N_{i}N_{i+1 \mathrm{\: mod \:} 3} N_{i(i+1 \mathrm{\: mod \:} 3)}\equiv^{\mathrm{qftp}} \mathcal{M}M_{0}M_{1}M_{01}$ in $\mathcal{L}^{\mathrm{std}}$
         \item the sets of instances of the relations $R, I$ on $B$ are minimal subject to the previous condition.  
     \end{itemize}

Then, to get  $b_0, b_1, b_2 \in \mathbb{M}^{\mathrm{std}}$ with $\varphi(b_0, b_1),\varphi(b_1, b_2), \varphi(b_2, b_0) $ it suffices to show $B \in \mathcal{H}^{\mathrm{std}}$. To see this, if $B \in \mathcal{H}^{\mathrm{std}}$ then $B$ will embed in $\mathbb{M}^{\mathrm{std}}$. By (i) and model completeness of $T^{\mathcal{H}^{\mathrm{std}}}$, the image of each $\mathcal{N} N_{i}N_{i+1 \mathrm{\: mod \:} 3} N_{i(i+1 \mathrm{\: mod \:} 3)}$ will satisfy the same type as $\mathcal{M}M_{0}M_{1}M_{01}$. So by (iii), the images  $b_0, b_1, b_2 \in \mathbb{M}^{\mathrm{std}}$ of $\mathcal{M}N_{0}, \mathcal{M}N_{1}, \mathcal{M}N_{2}$ will satisfy $\varphi(b_0, b_1),\varphi(b_1, b_2), \varphi(b_2, b_0) $. These $b_i$ (or really the coordinates of the $b_i$ actually assigned to the variables $x$, $y$ of $\varphi(x, y)$) will be as desired. 

     So it remains to show that $B \in \mathcal{H}^{\mathrm{std}}$. Because $\mathcal{M}, M_{i}, M_{ij} \in \mathbb{M}^{\mathrm{std}} :=(\mathbb{M}^{\mathrm{std}})^{\mathrm{Sk}}|_{\mathcal{L}^{\mathrm{std}}}$, $\mathcal{M} \sqcup \bigsqcup_{i < \omega} M_{i}  \sqcup \bigsqcup_{i < \omega} M_{ij} \in \mathcal{H}^{\mathrm{std}}$, so $\widetilde{\left(\mathcal{M} \sqcup \bigsqcup_{i < \omega} M_{i}  \sqcup \bigsqcup_{i < \omega} M_{ij}\right)} \in \mathcal{H}$. We apply Lemma \ref{non-overlapping pairs} to $\tilde{\mathcal{M}} \sqcup \bigsqcup_{i < \omega} \tilde{M_{i}}  \sqcup \bigsqcup_{i < \omega} \tilde{M_{ij}} = \widetilde{\left(\mathcal{M} \sqcup \bigsqcup_{i < \omega} M_{i}  \sqcup \bigsqcup_{i < \omega} M_{ij}\right)} \in \mathcal{H}$ (where the instances of the relation symbols on $\tilde{\mathcal{M}} \sqcup \bigsqcup_{i < \omega} \tilde{M_{i}}  \sqcup \bigsqcup_{i < \omega} \tilde{M_{ij}}$ are specified by this equation). Doing so, we get a structure $\overline{B} = \overline{\mathcal{N}} \sqcup \bigsqcup^{2}_{i=0} \overline{N}_{i} \sqcup \overline{N}_{01} \sqcup \overline{N}_{12} \sqcup \overline{N}_{20} \in \mathcal{H}$, satisfying the conditions that for $i \leq 2$, $\overline{\mathcal{N}} \overline{N}_{i}\overline{N}_{i+1 \mathrm{\: mod \:} 3} \overline{N}_{i(i+1 \mathrm{\: mod \:} 3)} \equiv^{\mathrm{qftp}} \tilde{\mathcal{M}}\tilde{M_{0}}\tilde{M_{1}}\tilde{M_{01}} = \widetilde{\left(\mathcal{M}M_{0}M_{1}M_{01}\right)}$. But $\overline{B}$ will weakly embed $\tilde{B}$, since $\tilde{B}= \tilde{\mathcal{N}} \sqcup \bigsqcup^{2}_{i=0} \tilde{N_{i}} \sqcup \tilde{N_{01}} \sqcup \tilde{N_{12}} \sqcup \tilde{N_{20}}$, with $\mathcal{N} N_{i}N_{i+1 \mathrm{\: mod \:} 3} N_{i(i+1 \mathrm{\: mod \:} 3)}\equiv^{\mathrm{qftp}} \mathcal{M}M_{0}M_{1}M_{01}$ implying $\widetilde{\left(\mathcal{N} N_{i}N_{i+1 \mathrm{\: mod \:} 3} N_{i(i+1 \mathrm{\: mod \:} 3)}\right)}=\tilde{\mathcal{N}} \tilde{N_{i}}\widetilde{N_{i+1 \mathrm{\: mod \:} 3}}\widetilde{ N_{i(i+1 \mathrm{\: mod \:} 3)}} \equiv^{\mathrm{qftp}} \tilde{\mathcal{M}}\tilde{M_{0}}\tilde{M_{1}}\tilde{M_{01}} = \widetilde{\left(\mathcal{M}M_{0}M_{1}M_{01}\right)}$ for each $i < 3$. So $\overline{B} \in \mathcal{H}$, implies $\tilde{B} \in \mathcal{H}$ and $B \in \mathcal{H}^{\mathrm{std}}$, as desired.

\end{proof}

We have now reached the combinatorial phase of the proof, where we aim to show that the complete theories $T^{\mathcal{H}}$ of Cherlin, Shelah and Shi (\cite{CSS99}) can never be strictly $\mathrm{NSOP}_{3}$ by showing that the hereditary class $\mathcal{H}^{\mathrm{std}}$ from Lemma \ref{standard hereditary class} cannot exist. This will be Corollary \ref{main combinatorial lemma} below. And we really will look at $\mathcal{H}^{\mathrm{std}}$ from the purely combinatorial point of view: the only part of the conditions on $\mathcal{H}^{\mathrm{std}}$ from Lemma \ref{standard hereditary class} that involves model-theoretic vocabulary is that $\mathcal{H}^{\mathrm{std}}$ is the age of a structure whose theory is $\mathrm{NSOP}_{3}$. Yet what we use about this is the above Lemma \ref{non-overlapping pairs}, an entirely combinatorial statement allowing us to obtain cycles from sequences within an hereditary class. (And Lemma \ref{non-overlapping pairs} really only uses the assumption that a hereditary class is the age of the models of a theory whose \textit{existential} formulas do not exhibit $\mathrm{NSOP}_{3}$.)

\begin{remark}
    \label{potential noninterference in SOP_2} 

    We are therefore in a very unusual position, because in proving $\mathcal{H}^{\mathrm{std}}$ from Lemma \ref{standard hereditary class} cannot exist, we will be proving a statement that appears at the formal level like a dichotomy between \textit{simplicity} and $\mathrm{SOP}_{3}$, rather than between $\mathrm{NSOP}_{2}$ and $\mathrm{SOP}_{3}$. Specifically, we will be showing (Corollary \ref{main combinatorial lemma}) that if $\mathcal{H}$ is a hereditary class of $\mathcal{L}^{\mathrm{std}}$-structures defined by a finite family of forbidden weakly embedded substructures, then if the relations $R(x, y)$, $I(y_1, y_2)$ of this language give a formal instance of the tree property in the sense that the structures of $\mathcal{H}$ are sound and $\mathbb{TP} \in \mathcal{H}$, the theory of every structure with $\mathcal{H}$ as the age must have $\mathrm{SOP}_{3}$. Naïvely, one might make the incorrect inference from this that there is no hereditary class $\mathcal{H}$ defined by finitely many forbidden weak substructures such that the theory $T^{\mathcal{H}}$ of Cherlin, Shelah and Shi (\cite{CSS99}) is a non-simple $\mathrm{NSOP}_{3}$ theory, or equivalently (by the reasoning in the proof of Proposition \ref{reduction to generic structures} and Observation \ref{main observation, restated}) a hereditary class $\mathcal{H}$ for which every theory whose models have age $\mathcal{H}$ is non-simple, but not every theory whose models have age $\mathcal{H}$ has $\mathrm{SOP}_{3}$. This incorrect inference would contradict Observation \ref{main observation, restated} and falsely imply that the assumptions of the main theorem of this section, Theorem \ref{main theorem 3, restated}, are overly powerful. What, fortunately for us, prevents this from happening is discussed in Remark \ref{noninterference in tree property}: if $T^{\mathcal{H}}$ is $\mathrm{NSOP}_{3}$ but has the tree property, the tree property itself will not be enough to obtain an instance of the tree property with consistency witnesses \textit{where the algebraic closures that interest us do not overlap}, and thus not enough to get $\mathcal{H}^{\mathrm{std}}$ as in Lemma \ref{standard hereditary class}. It is $\mathrm{SOP}_{2}$ (really $\mathrm{SOP}_{1}$), not just the tree property in general, which gives us that special instance of the tree property.

    So this apparent formal dichotomy between simplicity and $\mathrm{SOP}_{3}$ (Corollary \ref{main combinatorial lemma} below) is quite powerful, though, as we have just seen, it is not too powerful. Its proof, as we are about to see, is also hard enough that one may wish to make it easier. This prompts the question of whether we can prove that the theories of Cherlin, Shelah and Shi (\cite{CSS99}) cannot be strictly $\mathrm{NSOP}_{3}$--and thus, by Proposition \ref{reduction to generic structures}, prove the main theorem of this section--through an approach that appears more like the obvious one: proving an analogous statement to Corollary \ref{main combinatorial lemma}, but one that appears more like a $\mathrm{NSOP}_{2}$-$\mathrm{SOP}_{3}$ dichotomy instead of a simplicity-$\mathrm{SOP}_{3}$ dichotomy. Specifically, let the ``standard $\mathrm{SOP}_{2}$-structure" be the $\mathcal{L}_{\mathrm{std}}$-structure $\mathbb{SOP}_{2}$ consisting of distinct $b_{\eta} \in P(\mathbb{TP})$ for $\eta \in 2^{< \omega}$, and $a_{\sigma} \in O(\mathbb{TP})$ for $\sigma \in 2^{\omega}$, with all and only the following instances of $R$ and $I$: $I(b_{\eta_1},b_{\eta_2})$ for incomparable $\eta_1, \eta_2 \in 2^{<\omega}$  with $\eta_{2}$ lexicographically greater than $\eta_{1}$, and $R(a_{\sigma}, b_{\sigma|_{n}} )$ for $\sigma \in \omega^{\omega}$, $n < \omega$. Thus, we can ask whether proving the following statement suffices for proving the main theorem of this section:

 (*): If $\mathcal{H}$ is a hereditary class of $\mathcal{L}^{\mathrm{std}}$-structures defined by a finite family of forbidden weakly embedded substructures, then if the structures of $\mathcal{H}$ are sound and $\mathbb{SOP}_2 \in \mathcal{H}$, every theory whose models have age $\mathcal{H}$ must have $\mathrm{SOP}_{3}$.

In fact, this statement has a quicker proof than its counterpart for $\mathrm{TP}$, Corollary \ref{main combinatorial lemma}; we sketch this simpler proof in Remark \ref{potential simplified proof} below. So if we can show that the main theorem of this section follows from this statement (*), just as we now know by Lemma \ref{standard hereditary class} that the main theorem follows from Corollary \ref{main combinatorial lemma}, we will have expedited the proof of the main theorem.

We will be able to do so if the answer to the following general question on $\mathrm{SOP}_{2}$ is yes:

(**): Is it true that, in any theory with $\mathrm{SOP}_{2}$, there is a set $C$ and instance $\varphi(x, y)$, $\{b_{\eta}\}_{\eta \in 2^{< \omega}}$ of $\mathrm{SOP}_{2}$--i.e. with $\{\varphi(x, b_{\eta_1}), \varphi(x, b_{\eta_{2}})\}$ inconsistent for $\eta_1, \eta_2 \in 2^{<\omega}$ incomparable, and $\{\varphi(x, b_{\sigma|_{n}})\}_{n < \omega}$ for $\sigma \in 2^{\omega}$--together with consistency witnesses $a_{\sigma} \models \{\varphi(x, b_{\sigma|_{n}})\}_{n < \omega}$ for $\sigma \in 2^{\omega}$, chosen such that $\mathrm{tp}_{C}(b_{\eta_1}b_{\eta_2})$ for $\eta_1, \eta_2 \in 2^{<\omega}$ incomparable does not depend on the choice of $\eta_{1}, \eta_{2}$ with $\eta_{2}$ lexicographically greater than $\eta_{1}$ and $\mathrm{tp}_{C}(a_{\sigma}b_{\sigma|_{n}})$ does not depend on the choice of $\sigma \in 2^{ \omega}$, $n < \omega$, such that the sets $\mathrm{acl}(C)$ $\mathrm{acl}_{C}(a_{\sigma}) \backslash \mathrm{acl}(C)$ for $\sigma \in 2^{\omega}$, $\mathrm{acl}_{C}(b_{\eta}) \backslash \mathrm{acl}(C)$ for $\eta \in 2^{< \omega}$, $\mathrm{acl}_{C}(a_{\sigma} b_{\sigma|_{n}})\backslash (\mathrm{acl}_{C}(a_{\sigma} ) \cup \mathrm{acl}_{C}(b_{\sigma|_{n}}))$ for $\sigma \in \omega^{\omega}$, and $\mathrm{acl}_{C}(b_{\eta_{1}}b_{\eta_{2}}) \backslash (\mathrm{acl}_{C}(b_{\eta_{1}})\cup \mathrm{acl}_{C}(b_{\eta_{2}}))$ for $\eta_1, \eta_2 \in 2^{<\omega}$ incomparable are all disjoint?

Here, having already shown (Lemma \ref{noninterference of algebraic closures}) that any instance of $\mathrm{SOP}_{2}$ will give us a special instance of $\mathrm{TP}$ where the algebraic closures that interest us do not overlap nontrivially, we are now asking whether any instance of $\mathrm{SOP}_2$ will give us a special instance of $\mathrm{SOP}_{2}$ of the same kind. Then, by the proof of Lemma \ref{standard hereditary class}, a positive answer to (**) will imply that, if one of the complete theories of Cherlin, Shelah and Shi is strictly $\mathrm{NSOP}_{3}$, statement (*) is false--so only the quicker task of proving statement (*) would be necessary to complete the proof of the main theorem of this section, rather than the more technically difficult proof of Corollary \ref{main combinatorial lemma} that we currently need.

While we were not able to answer question (**), we were able to show that a statement analagous to a positive answer to (**) is true for obtaining a special instance of $\mathrm{SOP}_{1}$, rather than a special instance of $\mathrm{SOP}_{2}$. However, because the version of (*) with the analogously defined ``standard $\mathrm{SOP}_{1}$-structure" $\mathbb{SOP}_{1}$ instead of the ``standard $\mathrm{SOP}_{2}$-structure" $\mathbb{SOP}_{2}$ does not appear much easier to prove than Corollary \ref{main combinatorial lemma}, our statement analogous to a positive answer to (**) for a special instance of $\mathrm{SOP}_{1}$ will not likely help expedite our proof of the main theorem of this section, and we will omit the proof.

 \end{remark}

In this combinatorial part of the proof, we will first obtain criteria for a $\mathcal{L}^{\mathrm{std}}$-structure $A$ to (weakly) embed into the standard $\mathrm{TP}$-structure, $\mathbb{TP}$, in the form of requirements that $A$ must not contain certain kinds of cycles involving the relations $R(x, y)$ and $I(y_{1}, y_{2})$. Then we will show--for $\mathcal{H}$ a hereditary class of sound $\mathcal{L}^{\mathrm{std}}$-structures defined by finitely many forbidden weakly embedded substructures which is the age of a structure whose theory is $\mathrm{NSOP}_{3}$--that there is some $A \notin \mathcal{H}$ such that $A$ satisfies the embedding criterion. Thus we will obtain $\mathbb{TP} \notin \mathcal{H}$, so the conditions on $\mathcal{H}^{\mathrm{std}}$ in Lemma \ref{standard hereditary class} are never satisfied.

Our cycle-freeness conditions for a structure to embed into the standard  $\mathrm{TP}$-structure will be analogous to the far more basic observation used in the proof of Theorem 3.1 of \cite{ApproxOrder}, that directed graphs without directed cycles embed into the infinite chain graph $(\omega, <)$ (Claim 3.16, \cite{ApproxOrder}).  Our conditions that will guarantee that a $\mathcal{L}^{\mathrm{std}}$-structure $A$ embeds into $\mathbb{TP}$ are twofold:

First, $A$ will be a \textit{presidium} (Definition \ref{alternating cycle}.1), a sound $\mathcal{L}^{\mathrm{std}}$-structure where $P(A)$ can be decomposed into a disjoint union of linear orders for $I$. Embedding into a presidium itself follows from a cycle-freeness condition, the absence of any \textit{split cycles} (Definition \ref{split cycle}) below.

Second, $A$ will omit any \textit{alternating cycles} (Definition \ref{alternating cycle}.1). In order to obtain a presidium omitting alternating cycles, we will actually define the larger, more robust class of \textit{potentially pinched alternating cycles} (Definition \ref{potentially pinched alternating cycle, flat}.1 below) and find a presidium omitted from $A$ without any potentially pinched alternating cycles (step three and Lemma \ref{eliminating potentially pinched alternating cycles} below).

We will give the necessary definitions for these conditions, then show that these conditions are sufficient to embed into the standard $\mathrm{TP}$-structure.

\begin{definition}
\label{alternating cycle}    
    (1) A \textit{presidium} is a sound $\mathcal{L}_{\mathrm{std}}$-structure $A$ such that $P(A) = \bigsqcup_{i \in I} C_i$ for some $C_i$, where $I(b, c)$ only if $b, c $ belong to some common $C_i$, and the $C_i$ are linearly ordered by $I$. (This means that $I$ is transitive and antisymmetric on $C_i$, as well as linear on $C_i$: for each distinct $b, c \in C_{i}$, either $I(b,c)$ or $I(c, b)$.) We call the $C_{i}$ the \textit{convivia} of $A$.

    (2) For $A$ an $\mathcal{L}_{\mathrm{std}}$-structure and $a, b \in P(A)$, let $I^{*}(a, b)$ denote that either $I(a, b)$ or $I(b, a)$. Then an \textit{alternating cycle} in $A$ consists of distinct $b_{1}, \ldots, b_{2n} \in P(A)$, $a_1, \ldots, a_n \in O(A)$ for $n \geq 1$ with $I^{*}(b_i, b_{i+1} )$ for odd $i \leq 2n$, $R(a_i, b_{2i})$ and $R(a_i, b_{2i+1})$ for $i= 1, \ldots, n-1$, and $R(a_n, b_{2n})$ and $R(a_n, b_{1})$. 
\end{definition}

\begin{lemma}
    \label{embedding into the standard tp-structure} Let $A$ be a finite presidium with no alternating cycles. Then $A$ (weakly) embeds into the standard $\mathrm{TP}$-structure.
\end{lemma}

\begin{proof}
    We first prove the following:

        \begin{claim}
       \label{order construction 1} Let $A$ be a presidium with no alternating cycles. Let $a, b$ be distinct elements of $P(A)$ belonging to a common convivium, and let $c, d$ be distinct elements of $P(A)$ belonging to a different common convivium. Suppose that there is $e \in O(A) $ such that $R(e, b)$ and $R(e, c)$. Then either the presidium $A' : = A \sqcup \{f\}$, $f \in O(A')$ whose instances of $R, I$ are exactly those of $A$ together with $R(f, a)$ and $R(f, c)$ has no alternating cycles, or the presidium $A' : = A \sqcup \{f\}$, $f \in O(A')$ whose instances of $R, I$ are exactly those of $A$ together with $R(f, b)$ and $R(f, d)$ has no alternating cycles.
    \end{claim}
\begin{proof}
    (of claim)
    
    Suppose that the first of the two possibilities stated in the claim fails: the presidium $A' : = A \sqcup \{f\}$, $f \in O(A')$ whose instances of $R, I$ are exactly those of $A$ together with $R(f, a)$ and $R(f, c)$ does have an alternating cycle. Then because $A$ itself does not already have any alternating cycles, there is a sequence 
   
   $$a_1, a_2, a^{1}, a_{3}, a_{4}, a^{2}, \ldots a_{2k-1}, a_{2k}, a^{k}, a_{2k+1}, a_{2k+2}, \ldots a_{2n-1}, a_{2n}, a^{n}, a_{2n+1}, a_{2n+2} $$

for $n \geq 1$ consisting of distinct $a_i \in P(A)$, $a^i \in O(A)$ with $I^{*}(a_i, a_{i+1})$ for $i \leq 2n+2$ odd, $R(a^i, a_{2i})$ and $R(a^i, a_{2i+1})$ for $i \leq n$,  and $a_1 = a$ and $a_{2n+2} = c$.  We claim $a_2 = b$. Suppose otherwise, and suppose first that one of the $a_i$ (but not $a_1= a$ or $a_2$) is equal to $b$; say this is $a_k$.  In the case that $k$ is odd, we have $I^{*}(a_k, a_2)$ because $I^{*}(a, b) $ and $I^{*}(a_{1}, a_{2})$ for $a_{1} = a$, so $a_{2}$ and $a_{k} = b$ are distinct members of the same convivium. So the interval of the $a_i, a^{i}$, as ordered according to the above enumeration, between $a_{2}$ and $a_{k}$ inclusive will comprise an alternating cycle, a contradiction.  In the case that $k$ is even, we have $I^{*}(a_k, a_2)$ as before, but also $I^{*}(a_{k-1}, a_k)$, so $a_{2}$ and $a_{k-1}$ belong to the same convivium and $I^{*}(a_{k-1}, a_{2})$. Thus, the interval of the $a_i, a^{i}$ between $a_{2}$ and $a_{k-1}$ inclusive will comprise an alternating cycle, again a contradiction.  So given our supposition that $a_2 \neq b$, we see that none of the $a_i$ are equal to $b$. Now suppose one of the $a^i$, say $a^{k}$, is equal to $e$.  Then as before, $I^{*}(b, a_{2})$, while also $R(a^{k}, b)$. By this and the fact that none of the $a_i$ are equal to $b$, the interval between $a_{2}$ and $a^{k}$ inclusive together with $b$ will comprise an alternating cycle, giving us a contradiction yet again. So because none of the $a_i$ are equal to $b$ or $a_i$ to $e$, while $R(e, a_{2n+2})$ (because $a_{2n+2}=c$), $R(e, b)$ and $I^{*}(b, a_1)$, the $a_i$ and $a^i$ together with $e$ and $b$ will comprise an alternating cycle, another contradiction.

We have shown that $a_2 = b$; before proceeding, we will also show that none of the $a^{i}$ are equal to $e$. Suppose, say, that $a^{k} = e$; then $R(a^{k}, a_{2n+2})$. So we get another contradiction: the interval between $a^k$ and $a_{2n+2}$ inclusive comprises an alternating cycle.

In summary, taking the interval between $a_2$ and $a_{2n+2}$ inclusive in the above enumeration of the $a_i$ and $a^{i}$, we get a sequence

$$h_{2}, h^{1}, h_{3}, h_{4}, h^{2}, \ldots h_{2k-1}, h_{2k},  h^{k}, h_{2k+1}, h_{2k+2}, \ldots h_{2n-1}, h_{2n},  h^{n}, h_{2n+1}, h_{2n+2}$$

for $n \geq 1$ consisting of distinct $h_i \in P(A)$ and $h^i \in O(A)$ with $I^{*}(h_i, h_{i+1})$ for odd $i$ with $2 \leq i \leq  2n+2$, $R(h^{1}, h_2)$ and $R(h^1, h_3)$, $R(h^i, h_{2i})$ and $R(h^i, h_{2i+1})$ for $i = 2, \ldots n $, $h_{2}= b $ and $h_{2n+2} = c$, and all of the $h^{i}$ distinct from $e$.

Now suppose the second of the two possibilities stated in the claim also fails. Then similarly, we get a sequence

$$g_{2}, g^{1}, g_{3}, g_{4}, g^{2}, \ldots g_{2k-1}, g_{2k},  g^{k}, g_{2k+1}, g_{2k+2}, \ldots g_{2m-1}, g_{2m},  g^{m}, g_{2m+1}, g_{2m+2}$$

for $m \geq 1$ consisting of distinct $g_i \in P(A)$ and $g^i \in O(A)$ with $I^{*}(g_i, g_{i+1})$ for odd $i$ with $2 \leq i \leq  2m+2$, $R(g^{1}, g_2)$ and $R(g^1, g_3)$, $R(g^i, g_{2i})$ and $R(g^i, g_{2i+1})$ for $i = 2, \ldots m $, $g_{2}= c $ and $g_{2m+2} = b$, and all of the $h^{i}$ distinct from $e$. (So all that is different about these conditions is that $b$ and $c$ are reversed.)

Now suppose that the intersection of these two sequences is more than just $\{b, c\}$. Suppose first that the first of the $g_i, g^i$ besides $g_2=b$ in this intersection, per the above enumeration, is of the form $g^i$, say $g^k$, and that it coincides with $h^{\ell}$. Then the interval between $g_2$ and $g^k$ inclusive (again per that enumeration), together with the interval between $h^{\ell} = g^k$ and $h_{2n+2} = g_2$ (per the above enumeration on the $h_i, h^i$), will comprise an  alternating cycle, a contradiction. Now suppose that the first of the $g_i, g^i$ besides $g_2$ in the intersection is of the form $g_i$, say $g_k$, and it coincides with $h_\ell$. First consider the cases where $k$ is even and $\ell$ is even, or $k$ is odd and $\ell$ is odd. Then the interval between $g_2$ and $g_k$ inclusive, together with the interval between $h_{\ell} = g_k$ and $h_{2n+2} = g_2$, will comprise an alternating cycle, again a contradiction. Next, consider the case where $k$ is even and $\ell$ is odd. Then $I^{*}(g_{k-1}, g_{k})$ and $I^{*}(h_{\ell}, h_{\ell+1})$, so $g_{k-1}$ and $h_{\ell+1}$ share a convivium and $I^{*}(g_{k-1}, h_{\ell+1})$. Then the interval between $g_2$ and $g_{k-1}$ inclusive, together with the interval between $h_{\ell+1}$ inclusive and  $h_{2n+2} = g_2$, will comprise an alternating cycle, again a contradiction. Lastly, consider the case where $k$ is odd and $\ell$ is even. Then, using that all of the $g^i$, $h^i$ are distinct from $e$, the interval between $g_{2}=c$ and $g_{k}$, inclusive, together with the interval between $h_{2}=b$ inclusive and $h_{\ell} = g_{k}$ (which will be read in reverse), together with $e$, will comprise an alternating cycle, yet again a contradiction.

So the intersection of the sequence of $g_i, g^i$ and the sequence of $h_i, h^i$ is just $\{b, c\}$. Therefore, the former sequence together with the latter (read in reverse) comprise an alternating cycle. This final contradiction proves the claim.

\end{proof}
Now let $A$ be a finite presidium with no alternating cycles; we wish to show that $A$ embeds into the standard $\mathrm{TP}$-structure. First, suppose that $A$ has a convivium which is a singleton, $\{a\}$. Then $A \sqcup \{b\}$ with $b$ of sort $P$, where the instances of $R, I$ are exactly those of $A$ together with $I(b, a)$, will be a presidium with no alternating cycles. It follows from this that we may assume that each convivium of $A$ has size at least $2$. Similarly, we may assume that each convivium of $A$ has size at least $3$, even at least $4$.

Moreover, by the previous claim, we may now assume the following: if $a, b$ are distinct elements of $P(A)$ belonging to a common convivium, and $c, d$ are distinct elements of $P(A)$ belonging to a different common convivium, and if there is $e \in O(A) $ such that $R(e, b)$ and $R(e, c)$, then either there is $f \in O(A)$ (distinct from $e$) such that $R(f, a)$ and $R(f, c)$, or there is $f\in O(A)$ (again, distinct from $e$) such that $R(f, b)$ and $R(f, d)$.

With these assumptions in place, let $C_i$ and $C_j$ be two distinct convivia of $A$, and suppose that there are some $a \in C_i$, $b \in C_j$ and  $c \in O(A)$ such that $R(c, a)$ and $R(c, b)$.  Since $A$ has no alternating cycles (particularly, none with $1$ or $2$ points of sort $O$), there can be no distinct $a_1, a_2 \in C_i$, $b_1, b_2 \in C_j$ such that $a_1$ and $b_1$ have a common $R$-neighbor, and $a_2$ and $b_2$ have a common $R$-neighbor. Therefore, by the above assumption, it is either the case that for $C_i=\{a_\ell\}_{\ell \in I}$ with $a_i$ distinct, there are $\{c_\ell\}_{\ell \in I}$ with $c_\ell \in O(A)$ distinct such that $R(c_\ell, a_\ell)$ and $R(c_\ell, b)$, or that for $C_j=\{b_\ell\}_{\ell \in I}$ with $b_\ell$ distinct, there are $\{c_\ell\}_{\ell \in I}$ with $c_\ell \in O(A)$ distinct such that $R(c_\ell, b_\ell)$ and $R(c_\ell, a)$. (That the $c_{\ell}$ are distinct follows, say, from soundness of $A$.)

For $a, b \in P(A)$ belonging to distinct convivia, we now define $a < b$ if, for $\{b_i\}_{i \in I}$ the convivium to which $b$ belongs with $b_i$ distinct, there are $\{c_i \}_{i \in I}$ with $c_\ell \in O(A)$ (which must be distinct, by soundness) such that $R(c_i, b_i)$ and $R(c_i, a)$.  By the previous paragraph we now know that, if $a$ and $b$ have a common $R$-neighbor (and thus, by soundness, must belong to different convivia), either $a < b$ or $b < a$. Moreover, since $A$ has no alternating cycles (with $1$ or $2$ $O$-points) and each convivium has size at least $2$, the relation $<$ is antisymmetric: it is impossible that $a < b$ and $b < a$. We will also need that $<$ has no $3$-cycles: there are no $a < b$, $b < c$, $c < a$. To see this, because the convivium to which $c$ belongs has size at least $2$, there is $c'$ with $I^{*}(c, c')$. Because $b < c$, $c'$ and $b$ will have a common $R$-neighbor, $d$. Since the convivium to which $b$ belongs has size at least $3$, and $a < b$, we may find another member $b'$ (so $I^{*}(b, b')$) for which $b'$ and $a$ have a common $R$-neighbor $e$, distinct from $d$. Finally, since the convivium to which $a$ belongs has size at least $4$, and $c < a$, we may find another member $a'$ (so $I^{*}(a, a')$) for which $a'$ and $c$ have a common $R$-neighbor, $f$, distinct from $d$ and $e$. Then $b, b', d, c, c', e, a, a', f$ form an alternating cycle, a contradiction.

Our next goal will be to show that we can assume that $<$ is a tree ordering: $<$ is transitive and for $b < a$, $c< a$, either $b < c$ or $c < b$. To this end we show the following two claims:

\begin{claim}
    \label{order construction 2} Suppose $a < b$, $b < c$, and there is not already some $f\in O(A)$ with $R(f, a)$ and $R(f, c)$. Then the presidium $A' : = A \sqcup \{f\}$, $f \in O(A')$, whose instances of $R, I$ are exactly those of $A$ together with $R(f, a)$ and $R(f, c)$, has no alternating cycles.
\end{claim}

\begin{proof}
(of claim) 

(We do not really need ``there is not already some $f\in O(A)$ with $R(f, a)$ and $R(f, c)$" here, but include it to simplify the proof.) Because $a < b$ there are $b' \in P(A)$ with $I^{*}(b', b)$ (so $b'$ is distinct from $b$, and $b'$ is distinct from $a$ and $c$ because $b'$ is in the same convivium as $b$), and distinct $d, d' \in O(A)$ with $R(d, a)$, $R(d', a)$, $R(d, b)$, $R(d', b')$. Because $b < c$, there is also $e \in O(A)$  such that $R(e, b)$ and $R(e, c)$; $e$ is distinct from $d $ and $d'$ by the assumption that there is not already some $f\in O(A)$ with $R(f, a)$ and $R(f, c)$. Now suppose the presidium $A' : = A \sqcup \{f\}$ described in the claim has an alternating cycle. Then since $A$ does not already have an alternating cycle, there is a sequence 

  $$a_1, a_2, a^{1}, a_{3}, a_{4}, a^{2}, \ldots a_{2k-1}, a_{2k}, a^{k}, a_{2k+1}, a_{2k+2}, \ldots a_{2n-1}, a_{2n}, a^{n}, a_{2n+1}, a_{2n+2} $$

for $n \geq 1$ consisting of distinct $a_i \in P(A)$, $a^i \in O(A)$ with $I^{*}(a_i, a_{i+1})$ for $i \leq 2n+2$ odd, $R(a^i, a_{2i})$ and $R(a^i, a_{2i+1})$ for $i \leq n$,  and $a_1 = a$ and $a_{2n+1} = c$. Now suppose that the intersection of this sequence with $\{b, b', d, d', e\}$ is nonempty. First, suppose in particular that the first member of this intersection in the sequence (ordered as per the above enumeration), $a^k$, is either $d$ or $d'$. Then the interval between $a_1$ and $a^k$ inclusive will comprise an alternating cycle in $A$, a contradiction. Now suppose the same intersection is nonempty and its first member in the sequence, $a_k$, is $b$. If $k$ is odd, then the interval between $a_1$ and $a_k$ inclusive together with $b'$ and $d'$ comprises an alternating cycle, a contradiction. If $k$ is even, then $I^{*}(a_{k-1}, a_{k})$ and $I^{*}(a_k, b')$ implies $a_{k-1}$ and $b'$ are in the same convivium, so $I^{*}(a_{k-1}, b')$. Then the interval between $a_{1}$ and $a_{k-1}$ inclusive together with $b'$ and $d'$ comprises an alternating cycle, also a contradiction. The case where this same intersection is nonempty, and its first member in the sequence is $b'$, is handled similarly, with $b$ and $d$ in the role of $b'$ and $d'$, again getting a contradiction. If the intersection is nonempty and its first member in the sequence, $a^{k}$, is $e$, then the interval between $a_1$ and $a^{k}$ inclusive together with, say, $b$, $b'$ and $d'$ will comprise an alternating cycle, yet again a contradiction. So the intersection between this sequence and $\{b, b', d, d', e\}$, so in particular with $\{b,b',  d, e\}$, is nonempty. Therefore,  $a_i, a^{i}$ together with $e$, $b$, $b'$ and $d$ will comprise an alternating cycle; this final contradiction gives us the claim.

\end{proof}

\begin{claim}
    \label{order construction 3} Suppose $a < b$, $c < b$ with $a, c$ distinct, and there is not already some $f\in O(A)$ with $R(f, a)$ and $R(f, c)$. Then the presidium $A' : = A \sqcup \{f\}$, $f \in O(A')$, whose instances of $R, I$ are exactly those of $A$ together with $R(f, a)$ and $R(f, c)$, has no alternating cycles.
\end{claim}

\begin{proof}
    (of claim)

    (Again, we do not really need ``there is not already some $f\in O(A)$ with $R(f, a)$ and $R(f, c)$" here, but include it to simplify the proof.) 

    There is $b'$ distinct from $b$ in the same convivium as $b$, as well as $d, d', e, e' \in O(A)$, such that $R(d, a)$, $R(d', a)$, $R(d, b)$, $R(d', b')$ and $R(e, c)$, $R(e', c)$, $R(e, b)$, $R(e', b')$; any of the $d, d'$ must be distinct from any of the $e, e'$ by the assumption that  there is not already some $f\in O(A)$ with $R(f, a)$ and $R(f, c)$, while $c$ must be distinct from $c'$ and $d$ from $d'$ by soundness. Suppose that the presidium $A' : = A \sqcup \{f\}$ described in the claim has an alternating cycle. As  in the previous two claims, there is a sequence 

  $$a_1, a_2, a^{1}, a_{3}, a_{4}, a^{2}, \ldots a_{2k-1}, a_{2k}, a^{k}, a_{2k+1}, a_{2k+2}, \ldots a_{2n-1}, a_{2n}, a^{n}, a_{2n+1}, a_{2n+2} $$

for $n \geq 1$ consisting of distinct $a_i \in P(A)$, $a^i \in O(A)$ with $I^{*}(a_i, a_{i+1})$ for $i \leq 2n+2$ odd, $R(a^i, a_{2i})$ and $R(a^i, a_{2i+1})$ for $i \leq n$,  and $a_1 = a$ and $a_{2n+1} = c$. Suppose that the intersection of this sequence with $\{b, b', d, d', e, e'\}$ is nonempty. First, suppose in particular that the first member of this intersection in the sequence (once again, ordered according to the above enumeration), $a^{k}$ is one of $d, d', e, e'$. If $a^{k} $ is $d$ or $d'$, then the interval from $a_1$ to $a^k$ inclusive in this sequence will comprise an alternating cycle, a contradiction. If $a^{k}$ is $e$, the interval from $a_1$ to $a^k$ inclusive together with $b, b', d'$ will comprise an alternating cycle, a contradiction, while if $a^{k}$ is $e'$, the interval from $a_1$ to $a^k$ inclusive together with $b', b, d$ will comprise an alternating cycle, also a contradiction. Now suppose this first member, $a_k$, is $b$; if $k$ is odd then the interval from $a_1$ to $a_k$ inclusive together with $b'$ and $d'$ will comprise an alternating cycle, a contradiction. If $k$ is even then $I^{*}(a_{k-1}, a_{k})$ and $I^{*}(a_{k}, b')$ will imply $a_{k-1}$ and $b'$ are in the same convivium, so $I^{*}(a_{k-1}, b')$. Then the interval from $a_1$ to $a_{k-1}$ inclusive together with $b'$ and  $d'$ will comprise an alternating cycle, again a contradiction. (This is exactly as in the case where $a_{k}$ is equal to $b$ in the previous claim.) The case where this first member is $b'$ is handled similarly, with $b$ and $d$ in the role of $b'$ and $d'$, again getting a contradiction. So the intersection of the sequence with $\{b, b', d, d', e, e'\}$, and particularly with, say, $\{b, b', d', e\}$, must be empty. The sequence of $a_i, a^i$ together with $e$, $b$, $b'$, and $d'$ must comprise an alternating cycle, a contradiction that proves our claim. (So we did not need, say, $e'$ here at all, but include it anyway for reasons of symmetry.)
\end{proof}

Then, applying our three claims in this proof, we may assume that $A$ satisfies the assumptions from the paragraph after the proof of Claim \ref{order construction 1}, and also that any distinct $a, b$, such that there is a $c$ with $a < c < b$ or $a < c$, $b < c$, must have a common $R$-neighbor. (Assuming that $A$, and thus the number of pairs of points of $P(A)$, is finite, we only need to add finitely many points of $O(A)$, so $A$ stays finite). So, by the discussion preceding Claim \ref{order construction 2}, $<$ is, in addition to being antisymmetric, transitive and a tree order. (In particular, the absence of $3$-cycles for $<$ is used to prove transitivity: if $a< b$, $b< c$, then either $a < c$ or $c < a$, because $a$, $c$ must have a common $R$-neighbor. But $c <a$ is ruled out by the absence of $3$-cycles.)

Now define an injective map $\iota$ of $P(A)$ into the sort-$P$ points $\{b_{\sigma}\}_{\sigma \in \eta}$ of the standard $\mathrm{TP}$-structure, inductively on the $<$-height of an element of $P(A)$, as follows. Note that by definition of $<$, all elements of the same convivia have the same $<$-downsets, so have the same $<$-height. First, let $C_0, \ldots, C_\ell$ enumerate the convivia of $A$ all of whose elements have height $0$.  Then define $\iota(b)$, for $b$ the $k$-th element of $C_{i}$ in the linear order on $C_i$ defined by $I$, to be equal to $b_{\langle ik \rangle}$. Now suppose $\iota$ has been defined on elements of $P(A)$ of $<$-height at most $n-1$. For each element $b'$ of $<$-height $n -1$, let $C^{b'}_{0}, \ldots, C^{b'}_{\ell^{b'}}$ be a fixed enumeration of the convivia whose elements are all immediate $<$-successors of $b$. Let $b$ be an element of $P(A)$ of $<$-height $n > 0$; we define $\iota$ on $b$. Since $<$ is a tree order, there is a unique element $b'$ of $P(A)$ of height $n-1$ with $b' < b$.  Suppose we have already defined $\iota(b') = b_{\eta}$. Then define $\iota(b) := b_{\eta \smallfrown \langle ik \rangle}$, where $b$ is in $C^{b'}_i$ and is the $k$-th element of that convivium in the linear order on it defined by $I$. Having defined this map, we see it has the following two features: first, $I(b_1, b_2)$ implies $I(\iota(b_1), \iota(b_2))$ for $b_1, b_2 \in P(A)$, and second, if $b_1 < b_2$ for $b_1, b_2 \in P(A)$, and $\iota(b_1)=b_{\eta_1}$ while $\iota(b_2) = b_{\eta_2}$, then $\eta_{1} \lhd \eta_2$.

It remains to extend this map $\iota$ on $P(A)$ to an embedding of all of $A$ into the standard $\mathrm{TP}$-structure. We can find a family of distinct points $a^i_{S}$ of sort $O$ of the standard $\mathrm{TP}$-structure, where $i < \omega$ and $S$ ranges over those finite subsets of $\omega^{< \omega}$ linearly ordered by $\lhd$, such that $R(a_{S}^{i}, b_{\eta})$ for $\eta \in S$. Now, for any $A_{0} \subset P(A)$, let $\{a_{A_{0}}^{i}\}^{n_{A_{0}}}_{i =0}$ be an enumeration, with $a^{i}_{A_{0}} \neq a_{A_{0}}^{i'}$ for $i \neq i'$, of the set of points in $O(A)$ consisting of those $a$ for which $A_0$ is equal to the set of all $b \in P(A)$ with $R(a, b)$. By definition, for distinct $A_{0}, A'_{0}$, the set $\{a_{A_{0}}^{i}\}^{n_{A_{0}}}_{i =0}$ and $\{a_{A'_{0}}^{i}\}^{n_{A'_{0}}}_{i =0}$ are disjoint.  Then $\{a_{A_{0}}^{i}\}^{n_{A_{0}}}_{i =0}$ can only be nonempty if $A_{0}$ is linearly ordered by $<$, by the discussion before Claim \ref{order construction 2}.  For any linearly ordered $A_{0} \subset P(A)$, let $S_{A_0}$ be the set of $\eta \in \omega^{<\omega}$ such that there is $b \in A_{0}$ with $\iota(b) = b_{\eta}$; by the second condition satisfied by the map $\iota$ on $P(A)$, $S_{A_0}$ is linearly ordered by $\lhd$. So we define $\iota(a^i_{A_{0}}):=a_{S_{A_{0}}}^{i}$ for $A_0$ any subset of $P(A)$. Then $\iota$, an injective map defined on all of $A$, satisfies the condition that $R(a, b)$ implies $R(\iota(a), \iota(b))$ for $a \in O(A)$, $b\in P(A)$. With the condition that $\iota$ was noted to satisfy before that $I(b_1, b_2)$ implies $I(\iota(b_1), \iota(b_2))$, we have shown that $\iota$ is a weak embedding of $\mathcal{L}_{\mathrm{std}}$-structures, completing the proof of the lemma.

\end{proof}

Having shown that being a presidium without alternating cycles suffices for embedding into $\mathbb{TP}$, we now want to find $A \notin \mathcal{H}$--again, where $\mathcal{H}$ is a hereditary class of sound $\mathcal{L}^{\mathrm{std}}$-structures defined by finitely many forbidden weakly embedded substructures which is the age of a structure whose theory is $\mathrm{NSOP}_{3}$--for which $A$ satisfies these embedding criterion. To find this $A \notin \mathcal{H}$, we will need to extend the facts about \textit{helix maps} explicated in \cite{ApproxOrder} to make use of the assumption of $\mathrm{NSOP}_{3}$ as it applies to existential formulas, rather than just to quantifier-free formulas. These helix-shaped covering maps, essentially developed in \cite{She95} and \cite{Mal10b}, have become implicitly well-established as the correct framework for reasoning about $\mathrm{NSOP}_{n}$ at a combinatorial level; the technique of viewing helix maps as composable maps and applying their abstract cycle-removal properties was only introduced more recently, in the setting of the real-valued $\mathrm{NSOP}_{r}$ hierarchy in \cite{ApproxOrder}. 

For $\mathcal{H}$ a hereditary class that is the age of a structure whose theory is $\mathrm{NSOP}_{3}$, we now define our general class of helix maps under which this hereditary class must be closed, using $\mathrm{NSOP}_{3}$ for existential formulas; while this will straightforwardly generalize to $\mathrm{NSOP}_{n}$ for any $n \geq 3$, we will only need it for $\mathrm{NSOP}_{3}$.

\begin{definition}
    \label{jughandle decomposition} Let $A$ be a structure in a relational language. Then a $\textit{cyclic $3$-jughandle decomposition}$ of $A$ is a partition $A= D \sqcup \bigsqcup^{2}_{i = 0} A_{i} \sqcup \bigsqcup^{2}_{i = 0} A_{i(i+1\mathrm{\:mod\:} 3)}$ where every instance of a relation symbol in $A$ has all coordinates in a set of the form $DA_{i}A_{i+1 \mathrm{\:mod\:}3}A_{i(i+1\mathrm{\:mod\:} 3)}$,
\end{definition}

As in Definition 3.7 of \cite{ApproxOrder}, for clarity, in defining the helix map associated to a cyclic $3$-jughandle decomposition of a structure, we will make the enumeration of each set in the decomposition explicit. However, we may refer to these sets using standard model-theoretic notation later on.

\begin{definition}
  \label{jughandle helix map}  Let $A= D \sqcup \bigsqcup^{2}_{i = 0} A_{i} \sqcup \bigsqcup^{2}_{i = 0} A_{i(i+1\mathrm{\:mod\:} 3)}$ be a cyclic $3$-jughandle decomposition of a structure $A$. Enumerate $D$ as $\{d(0), \ldots, d(n_{d}))\}$,  $A_{i}$ as $\{a_{i}(0), \ldots, a_{i}(n_{i})\}$, and $A_{i(i+1\mathrm{\:mod\:} 3)}$ as $\{a_{i(i+1\mathrm{\:mod\:} 3)}(0), \ldots, a_{i(i+1\mathrm{\:mod\:} 3)}(n_{i(i+1\mathrm{\:mod\:} 3)})\}$. Let $B$ be the structure with distinct vertices $b^{jj'}_{i(i+1\mathrm{\:mod\:} 3)}(k)$ for $j < j' < \omega$, $i \leq 2$ and $k <n_{i(i+1\mathrm{\:mod\:} 3)}$,  $b^{j}_{i}(k)$, for $j < \omega$, $i \leq 2$ and $k \leq n_{i}$, along with the vertices $d(0) \ldots, d(n_{d})$ of $D$ forming an induced substructure of $B$ (with $b^{jj'}_{i(i+1\mathrm{\:mod\:} 3)}(k), b^{j}_{i}(k) \notin D$), with the unique interpretations of the relation symbols subject to the following conditions:

    (1) For $j < j'$, and $i' = i + 1 \mathrm{\: mod \:} 3$, the map from $\{d(0), \ldots, d(n_{d}))\} \sqcup \{b^{j}_{i}(0)\ldots b^{j}_{i}(n_{i})\} \sqcup \{b^{j'}_{i'}(0)\ldots b^{j'}_{i'}(n_{i'})\} \sqcup  \{b^{jj'}_{ii'}(0)\ldots b^{jj'}_{ii'}(n_{ii'})\}$ to $ D  \sqcup A_{i} \sqcup A_{i'} \sqcup A_{ii'}$ given by $b^{j}_{i}(k) \mapsto a_{i'}(k)$, $b^{j'}_{i'}(k) \mapsto a_{i}(k)$ $b^{jj'}_{ii'}(k) \mapsto a_{ii'}(k)$, $d(k) \mapsto d(k)$ is an isomorphism.

    (2) There are no instances of the relation symbols whose coordinates are not in some fixed set of the form $\{d(0), \ldots, d(n_{d}))\} \sqcup \{b^{j}_{i}(0)\ldots b^{j}_{i}(n_{i})\} \sqcup \{b^{j'}_{i'}(0)\ldots b^{j'}_{i'}(n_{i'})\} \sqcup  \{b^{jj'}_{ii'}(0)\ldots b^{jj'}_{ii'}(n_{ii'})\}$ for $j < j'$, and $i' = i + 1 \mathrm{\: mod \:} 3$.

    Then the \textit{$3$-helix map} associated with the cyclic $3$-jughandle decomposition $A= D \sqcup \bigsqcup^{2}_{i = 0} A_{i} \sqcup \bigsqcup^{2}_{i = 0} A_{i(i+1\mathrm{\:mod\:} 3)}$ is the map $h: B \twoheadrightarrow A$ given by $d(k) \mapsto d(k)$, $b^{j}_{i}(k) \mapsto a_{i}(k)$, $b_{i(i \mathrm{\: mod\:} 3)}^{jj'}(k) \mapsto a_{i(i \mathrm{\: mod\:} 3)}(k) $. For $\ell > 1$, the \emph{$3$-helix map of length $\ell$} associated with this cyclic $3$-jughandle decomposition is the map $h^{\ell}: B^{\ell} \twoheadrightarrow A$ that is defined similarly, where $B^{\ell}$ is defined similarly to $B$ but for $j, j' < \ell$ instead of $j, j' < \omega$.
\end{definition}

In Fact 3.8 of \cite{ApproxOrder}, the argument for which is contained within the proof of Fact 7.7 of \cite{Mal10b}, we made the longstanding interpretation of $\mathrm{NSOP}_{n}$ in terms of helix maps explicit, by presenting the proof of the equivalence between $\mathrm{NSOP}_{n}$ and closure under  helix maps associated with cyclic decompositions. Here we will do the same for our extended class of helix maps, associated with cyclic $3$-jughandle decompositions. This will give a general schema for standard combinatorial arguments involving $\mathrm{NSOP}_{3}$ for \textit{existential} formulas, just as the earlier presentation in Fact 3.8 of \cite{ApproxOrder} gave a general schema for standard combinatorial arguments involving $\mathrm{NSOP}_{n}$ for quantifier-free formulas. We will need to apply Lemma \ref{non-overlapping pairs} from earlier in this section, which states the combinatorial content of $\mathcal{H}$ being the age of a structure whose theory is $\mathrm{NSOP}_{3}$.

\begin{fact}
   \label{closure under jughandle helix maps} \emph{(Malliaris, \cite{Mal10b})} Let $\mathcal{H}$ be a hereditary class of structures in a relational language that is closed under weak embeddings, and suppose $\mathcal{H}$ is the age of a structure whose theory is $\mathrm{NSOP}_{3}$. Then $\mathcal{H}$ is closed under $3$-helix maps associated with cyclic $3$-jughandle decompositions.
\end{fact}

\begin{proof}
    Let $A= D \sqcup \bigsqcup^{2}_{i = 0} A_{i} \sqcup \bigsqcup^{2}_{i = 0} A_{i(i+1\mathrm{\:mod\:} 3)}$ be a cyclic $3$-jughandle decomposition of a structure $A$, and let $h: B \twoheadrightarrow A$ be the associated helix map. Suppose $B \in \mathcal{H}$; we show $A \in \mathcal{H}$. For $i \leq 2$, $i' = i + 1 \mathrm{\: mod \:} 3$, $j < j' < \omega$, let $B^{jj'}_{ii'}$ denote $\{b^{jj'}_{ii'}(0)\ldots b^{jj'}_{ii'}(n_{ii'})\}$ and let $B^{j}_{i}$ denote $\{b^{j}_{i}(0)\ldots b^{j}_{i}(n_{i})\}$ as in Definition \ref{jughandle helix map}. For $j < j' < \omega$, define $B^{j} :=\{B^{j}_{1}B^{j}_{2}B^{j}_{3}\}$, $B^{jj'}: = \{B_{01}^{jj'}B^{jj'}_{12}B_{23}^{jj'}\}$.  Then we may apply Lemma \ref{non-overlapping pairs} to $B= D \sqcup \bigsqcup_{i < \omega} B^i \sqcup \bigsqcup_{i < j < \omega} B^{ij}  \in \mathcal{H}$. This gives us a structure $C= D \sqcup \bigsqcup^{2}_{i, j = 0} C_{i}^{j} \sqcup \bigsqcup^{2}_{i,j = 0} C^{j(j+1\mathrm{\:mod\:} 3)}_{i(i+1\mathrm{\:mod\:} 3)} \in \mathcal{H}$ such that, for $j \leq 2$, $j' = j+1 \mathrm{\: mod \:}3$

    $$C^{j}_{0}C^{j}_{1}C^{j}_{2}C^{j'}_{0}C^{j'}_{1}C^{j'}_{2}C_{01}^{jj'}C^{jj'}_{12}C_{23}^{jj'} \equiv_{D} B^{0}_{0}B^{0}_{1}B^{0}_{2}B^{1}_{0}B^{1}_{1}B^{1}_{2}B_{01}^{01}B^{01}_{12}B_{23}^{01}$$

In particular, for $i \leq 2$, 

 $$C^{i}_{i}C^{i + 1 \mathrm{\:mod\:}3}_{i + 1 \mathrm{\:mod\:}3} C^{i(i + 1 \mathrm{\:mod\:}3)}_{i(i + 1 \mathrm{\:mod\:}3)} \equiv_{D} B^{0}_{i}B^{1}_{i+1 \mathrm{\:mod \:} 3 }B^{01}_{i(i + 1 \mathrm{\:mod\:}3)}\equiv_{D} A_{i}A_{i + 1 \mathrm{\:mod\:}3}A_{i(i + 1 \mathrm{\:mod\:}3)}$$
Here the second equivalence is given by the construction of the $3$-helix map.  So because $A= D \sqcup \bigsqcup^{2}_{i = 0} A_{i} \sqcup \bigsqcup^{2}_{i = 0} A_{i(i+1\mathrm{\:mod\:} 3)}$ is a cyclic $3$-jughandle decomposition, $A$ weakly embeds into $C$ by the map that restricts to the identity on $D$, restricts on $A_i$ to the isomorphism exhibiting $A_{i}\equiv_{D} C^{i}_{i}$, and restricts on $A_{i(i+1\mathrm{\:mod\:} 3)}$ to the isomorphism exhibiting $A_{i(i+1\mathrm{\:mod\:} 3)} \equiv_{D} C^{i(i + 1 \mathrm{\:mod\:}3)}_{i(i + 1 \mathrm{\:mod\:}3)}$. (This is again the ``diagonal argument" used in Theorem 7.7 of \cite{Mal10b}.) Therefore, because $C \in \mathcal{H}$ and $\mathcal{H}$ is closed under weak embeddings, $A \in \mathcal{H}$.
    
\end{proof}

Having defined our target conditions for $A \notin \mathcal{H}$ to embed into the standard $\mathrm{TP}$-structure (being a presidium and omitting alternating cycles), and explicated the core combinatorial substrate of our arguments within the helix map interpretation of $\mathrm{NSOP}_{3}$, we are ready to refine the cycle-removal argument we introduced in the context of the real-valued $\mathrm{NSOP}_{r}$ hierarchy (\cite{ApproxOrder}) to handle our current, more strenuous context. Our refinements of this argument will include three steps:

\begin{itemize}
    \item \textbf{Step one:} Remove \textit{split cycles} from some initial $A_{0} \notin \mathcal{H}$ to obtain a presidium $A \notin \mathcal{H}$.
    \item \textbf{Step two:} From our presidium, remove the smallest potentially pinched alternating cycles (with three or four points of sort $P$), which requires different helix maps than the larger potentially pinched alternating cycles.
    \item \textbf{Step three:} Remove the rest of the potentially pinched alternating cycles--in particular, obtaining a presidium with no alternating cycles, as desired.
\end{itemize}

We start by dealing with the split cycles.

\:

\textbf{Step one: Omitting a presidium}

In this first step, we want to show:

\begin{lemma}
  \label{omitting a presidium}  Let $\mathcal{H}$ be a hereditary class consisting of sound $\mathcal{L}_{\mathrm{std}}$-structures, defined by a finite family of forbidden weakly embedded substructures, which is the age of some structure whose theory is $\mathrm{NSOP}_{3}$. Then there is some presidium $A$ such that $A \notin \mathcal{H}$.
\end{lemma}

In this first step, as in the two later steps, we employ the same strategy of removing unwanted cycles by obtaining iterated helix covers as in the proof of Theorem 3.1 of \cite{ApproxOrder}, in the real-valued $\mathrm{NSOP}_{r}$ setting. While we refer to the concept of ``cycles" informally, the split cycles, alternating cycles and potentially pinched alternating cycles we will remove here, and the directed cycles removed in \cite{ApproxOrder}, give several different incarnations of this common pattern. However, we unfortunately can no longer work in any natural category of morphisms that is as well-behaved with respect to the cycles we are trying to eliminate as the category of directed graph morphisms is with respect to the class of directed cycles. We define the kind of cycles we must eliminate at this stage:

\begin{definition}
   \label{split cycle} Let $A$ be a $\mathcal{L}_{\mathrm{std}}$-structure. A \textit{split cycle} in $A$ is a set consisting of distinct $\alpha \in O(A)$, $a_0, \ldots, a_{n} \in P(A)$ for $n \geq 1$ with $R(\alpha, a_0)$, $R(\alpha, a_{n})$ and $a_0, \ldots, a_{n}$ a path for the relation $I^{*}$.
\end{definition}

We will call this an \textit{$n$-split cycle} when we need to specify the value of $n$. In lieu of a natural class of morphisms valid for our purposes, we give the following more ad hoc family of properties applying to our choices of helix maps. Similarly to a morphism of graphs, in what follows a \textit{map} $f: A \to B$ of $\mathcal{L}_{\mathrm{std}}$-structures will be a sort-preserving function for which $I(a, b)$ implies $I(f(a), f(b))$ for all $a, b \in P(A)$ (so in particular, as long as $I$ is irreflexive on $B$, that $f(a) \neq f(b)$), and $R(a, b)$ implies $R(f(a), f(b))$ for $a \in O(A)$, $b \in P(A)$.

\begin{definition}
\label{n-satisfactory}    Let $A$ and $B$ be $\mathcal{L}_{\mathrm{std}}$-structures. A map $f: A \to B$ of $\mathcal{L}_{\mathrm{std}}$-structures is $\textit{n}$-satisfactory if the image of any $k$-split cycle under $f$, for $k \leq n$, contains a split cycle.
\end{definition}

The central lemma of the proof of this step will be the following,, which will play the same role as Proposition 3.11 of \cite{ApproxOrder}. That proposition gave our abstract cycle-removal property for the helix maps obtained from $\mathrm{NSOP}_{n}$ within the category of directed graphs. Because of the more technical nature of our next statement in comparison to Proposition 3.11 of \cite{ApproxOrder}, we will refer to it only as a lemma, rather than as a proposition of independent interest. (For similar reasons, the analogous lemmas in steps two and three, Lemma \ref{consanguineous pair removal} and Lemma \ref{potentially pinched alternating cycle removal}, will be just that--lemmas--rather than full-blown propositions.) 

Before giving our central lemma for this step, let us first make the following notational remark. We may alternatively view an $n$-split cycle in the $\mathcal{L}_{\mathrm{std}}$-structure $A$ as an injective map $\gamma \hookrightarrow A$, with $\gamma = \{\alpha, a_0, \ldots, a_{n}\}$, $\alpha \in O(\gamma)$, $a_i \in P(\gamma)$ where the instances of $R$ in $\gamma$ are exactly $R(\alpha, a_0)$, $R(\alpha, a_n)$ and the instances of $I$ in $\gamma$ are only between $a_i$ and $a_{i+1}$ in either direction, with $I^{*}(a_i, a_{i+1})$ for all $i < n$. By abuse of notation, we may use $\gamma$ to refer to an $n$-split cycle in an $\mathcal{L}_{\mathrm{std}}$-structure considered as set, to this same $n$-split cycle considered as a map, and to the domain of that map. Which usage we will be employing will be clear from context.

\begin{lemma}
    \label{split cycle removal} Let $\mathcal{H}$ be a hereditary class consisting of sound $\mathcal{L}_{\mathrm{std}}$-structures which is the age of some structure whose theory is $\mathrm{NSOP}_{3}$. Let $g: B \to A$ be a map between finite $\mathcal{L}_{\mathrm{std}}$-structures with $A, B\notin \mathcal{H}$. Suppose, for $n \geq 1$, that $A$ has no $k$-split cycles for $k < n$, and let $\gamma \hookrightarrow A$ be an $n$-split cycle in $A$.  Then there is an $n$-satisfactory $3$-helix map $h: \tilde{B} \twoheadrightarrow B$ of some length $\ell$ associated with a cyclic $3$-jughandle decomposition of $B$, such that $\tilde{B} \notin \mathcal{H}$ and such that there is no $n$-split cycle $\gamma' \hookrightarrow \tilde{B}$ with $\gamma = g\circ h\circ \gamma'$.

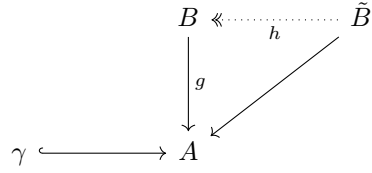
$$\begin{tikzcd}
                        &  & B \arrow[dd, "g"] &  & \tilde{B} \arrow[ll, two heads, "h", dotted] \arrow[lldd] \\
                        &  &              &  &                                      \\
\gamma \arrow[rr, hook] &  & A            &  &                                     
\end{tikzcd}$$

\end{lemma}

In practice, $g$ will be $n$-satisfactory. Before proving this lemma, we first confirm that this lemma really does play the same role in removing all small split cycles that our original abstract directed-cycle removal properties played in removing all small directed cycles. Specifically, we recapitulate  the argument from Claim 3.15 of \cite{ApproxOrder}, with Lemma \ref{split cycle removal} playing the role of Proposition 3.11 of \cite{ApproxOrder} and split cycles in a $\mathcal{L}_{\mathrm{std}}$-structure playing the role of directed cycles in a directed graph.

\begin{lemma}
 \label{arbitrarily large minimal split cycles}   Let $\mathcal{H}$ be a hereditary class consisting of sound $\mathcal{L}_{\mathrm{std}}$-structures which is the age of some structure whose theory is $\mathrm{NSOP}_{3}$. Then for any $n$, there is some $A \notin \mathcal{H}$ with no $k$-split cycles for $k \leq n$, and in which $I$ has no directed cycles of any length. (Particularly, $I$ has no $1$- or $2$-cycles, i.e., is irreflexive and antisymmetric.)
\end{lemma}

\begin{proof}
    We assume Lemma \ref{split cycle removal}, which we will prove after concluding this proof.  We start with the following claim, asserting that adding new morphisms to our composition preserves the absence of lifts of split cycles. This claim plays the exact same role in this proof as Lemma 3.13 of \cite{ApproxOrder} plays in Claim 3.15 there.

\begin{claim}\label{pulling back does not add split cycles}

Let $g: B \to A$ be a map of $\mathcal{L}_{\mathrm{std}}$-structures, and let $\gamma \hookrightarrow A$ be an $n$-split cycle in $A$. Suppose there is no $n$-split cycle $\gamma' \hookrightarrow A$ such that $g \circ \gamma' = \gamma$. Let $h: B' \to B$ be another map of $\mathcal{L}_{\mathrm{std}}$-structures. Then there is no $n$-split cycle $\gamma'' \hookrightarrow  B$ such that $(g \circ h) \circ \gamma'' = \gamma$.

\end{claim}

\begin{proof}
    Exactly as in the proof of Lemma 3.13 of \cite{ApproxOrder} (which is itself essentially immediate). If $\gamma'' \hookrightarrow B$ did exist, then for $\gamma' := h \circ \gamma''$, $\gamma' \hookrightarrow H$ would satisfy $g \circ \gamma'= g \circ ( h \circ \gamma'')  = (g \circ h) \circ \gamma'' = \gamma$. Also, $\gamma'$ will be injective because $g \circ \gamma' =\gamma$ is injective. So $\gamma'$ will be an $n$-split cycle with $g \circ \gamma' =\gamma$, a contradiction.
\end{proof}

We argue by induction on $n$. We first show the base case, where we just want to find $A \notin \mathcal{H}$ where $I$ has no directed cycles. But it is immediate from the assumption that $\mathcal{H}$ consists of sound $\mathcal{L}_{\mathrm{std}}$-structures that $A = \{\alpha, a_0, a_1\}$, $\alpha \in O(A)$, $a_i \in P(A)$, where the instances of $R$ and $I$ are just $R(\alpha, a_0)$, $R(\alpha, a_1)$, $I(a_0, a_1)$, will be as desired.

Now assume that $A \notin \mathcal{H}$ for some $A$ without any $k$-split cycles for $k < n$ or directed cycles for $I$; our induction step will be to find $B \notin \mathcal{H}$ for $B$ without any $k$-split cycles for $k \leq n$ or directed cycles for $I$. As in the proof of Lemma 3.13 of \cite{ApproxOrder}, we enumerate the $n$-split cycles of $A$ as $\gamma_0 \hookrightarrow A, \ldots, \gamma_k \hookrightarrow A$. Then we show, by induction on $i \leq k+1$, that there is $B_i \notin \mathcal{H}$ and $n$-satisfactory $g_i: B_i \to A$ such that, for $j < i$, there is no $n$-split cycle $\gamma' \hookrightarrow B_i$ with $g_i \circ \gamma' = \gamma_j$. As the base case, take $g_0$ to be the identity on $A$. Now assume that we have found $g_{i}: B_i \to A$ for $i < k+1$, and we show how to find $g_{i+1}: B_{i+1} \to A$.  By Lemma \ref{split cycle removal}, we may find an $n$-satisfactory map $h: \tilde{B}_{i} \twoheadrightarrow B_i$ such that there is no $n$-split cycle $\gamma' \hookrightarrow \tilde{B}_{i}$ with $g_i \circ h \circ \gamma' = \gamma_i$. (Here we use that $A$ has no $\ell$-split cycles for $\ell < n$.)  Moreover, for $j < i$, there is no $n$-split cycle $\gamma' \hookrightarrow \tilde{B}_{i}$  $  g_j \circ h  \circ \gamma' = \gamma_j$, by the induction hypothesis and the previous claim. So we may choose $B_{i+1} : =  \tilde{B}_{i}$ and $g_{i+1} := h \circ g_i $, which will be $n$-satisfactory, noting that a composition of $n$-satisfactory maps is an $n$-satisfactory map.

Then $B : = B_{}$ will be as desired: first, it will have no directed cycles for $I$. This is because the image of such a cycle under $g_{k}$ would contain a directed cycle for $I$ in $A$. But this cannot exist, because, by the induction hypothesis, $I$ has no directed cycles in $A$. It remains to show that $B$ has no $\ell$-split cycles for $\ell \leq n$. Suppose there is a $\ell$-split cycle $\gamma \hookrightarrow B$ for $\ell \leq n$. Then the image of $g_{k} \circ \gamma$ must contain an $\ell'$-split cycle in $A$ for $\ell' \leq \ell$, because $g_{k}$ is $n$-satisfactory. But $A$ does not have any $\ell'$-split cycles for $\ell' < n$, so $\ell' = \ell = n$. So $\gamma$ is an $n$-split cycle with $g_{k} \circ \gamma$ one of the $n$-split cycles $\gamma_0, \ldots, \gamma_k$, contradicting our hypothesis on $g_k$. This completes the inductive step and the proof of the lemma.

\end{proof}

Now that we have proven that we can remove all small split cycles (Lemma \ref{arbitrarily large minimal split cycles}) as long as we assume our split cycle-removal property for $n$ -satisfactory helix maps (Lemma \ref{split cycle removal}), we give the proof of this helix map removal property.

\begin{proof}
 (of Lemma \ref{split cycle removal})  

 We first define a cyclic $3$-jughandle decomposition of $A$. We will pull it back along $g$ to a cyclic $3$-jughandle decomposition of $B$. Let $\gamma = \{\alpha, a_0, \ldots, a_{n}\}$ as in the discussion before Lemma \ref{split cycle removal}. First, define $A_0$ to be the singleton consisting of the image of $\alpha$ under the map $\gamma \hookrightarrow A$. Second, define $A_1$ to be the singleton consisting of the image of $a_0$. Third, define $A_2$ to be the set consisting of all other $R$-neighbors of the image of $\alpha$, except the image of $a_0$. Finally, define $A_{12} := A \backslash (A_0 \sqcup A_1 \sqcup A_2)$. Then the image of $\alpha$, the sole element of $A_{0}$, has no $R$-neighbors in $A_{12}$. So $A = \bigsqcup^{2}_{i=0} A_i \sqcup A_{12}$\footnote{Throughout this proof, we suppress any references to a subset indexed in a cyclic $3$-jughandle decomposition, or the domain of its associated helix map, whenever that subset is empty. For example, here we do not explicitly make reference to $A_{01}$ or $A_{20}$ in the cyclic $3$-jughandle decomposition  $A = \bigsqcup^{2}_{i=0} A_i \sqcup A_{12}$, because these are empty.} is a cyclic $3$-jughandle decomposition of $A$. Moreover, there is no $I^{*}$-path in $A$ between the element of $A_1$ and any element of $A_{2}$ with less than $n$ edges: otherwise, $A$ would have a $k$-split cycle for $k < n$, contradicting our assumption on $A$.

  Define $B_i := g^{-1}(A_i)$, $B_{12} := g^{-1}(A_{12})$. Then because $g$ is a map of $\mathcal{L}_{\mathrm{std}}$-structures, we preserve the properties from before: $B = \bigsqcup^{2}_{i=0} B_i \sqcup B_{12}$ is a cyclic $3$-jughandle decomposition of $B$, and there is no $I^{*}$-path in $B$ with less than $n$ edges between a point of $B_1$ and a point of $B_2$. Thus, by Fact \ref{closure under jughandle helix maps}, for some $\ell < \omega$ and $\tilde{B} : = \tilde{B}^{\ell}$ where $h: \tilde{B}^{\ell} \twoheadrightarrow B $ is the $3$-helix map of length $\ell$ associated with this cyclic $3$-jughandle decomposition, $\tilde{B}^{\ell} \notin \mathcal{H}$. Of course $h$ is a map of $\mathcal{L}^{\mathrm{std}}$-structures. Write $\tilde{B}^{\ell} = \bigsqcup_{i \leq 2, j < \ell} B_{i}^{j} 
\sqcup \bigsqcup_{j < j' < \ell} B_{12}^{jj'}$ as in the proof of Fact \ref{closure under jughandle helix maps}.

 To show that $\tilde{B}, h$ are as desired, we first show that there is no $n$-split cycle $\gamma' \hookrightarrow \tilde{B} $ with $g \circ h \circ \gamma' = \gamma$.  Suppose otherwise, and write the domain of $\gamma'$ as  $\{\alpha, a_0, \ldots, a_{n}\}$ as we wrote the domain of $\gamma$ above. Then $\gamma'(\alpha) \in B^{j}_{0}$ for some $j$. Moreover, $R(\gamma'(\alpha), \gamma'(a_{0}))$; therefore, $\gamma'(a_0) \in B^{j'}_{1} $ for some $j' > j$ (by construction of $3$-helix maps). Also, $R(\gamma'(\alpha), \gamma'(a_{n}))$, so $\gamma'(a_n) \in B^{j''}_{2} $ for some $j'' < j$.  In the case that $n=1$, by $j'' < j'$, there are no instances of $I^{*}$ between a point of $B^{j'}_{1}$ and a point of $B^{j''}_{2}$. In particular, it is not the case that $I^{*}(\gamma'(a_{0}), \gamma'(a_{1}))$, a contradiction. If $n > 1$, by $I^{*}(\gamma'(a_0), \gamma'(a_1))$, and the fact that there are no instances of $I^{*}$ between a point of $B^{j'}_{1}$ and a point of $B_{12}^{\ell\ell'}$ unless $\ell=j'$ (in which case $\ell' > j'$), $\gamma'(a_1) \in B^{j'j'''}_{12}$ for some $j'<j'''$.  We next show by induction on $k$ that, for $ 1 \leq k \leq n-1$, $\gamma'(a_k) \in B^{j'j'''}_{12}$. If $\gamma'(a_{k-1}) \in B^{j'j'''}_{12}$, then since $I^{*}(\gamma'(a_{k-1}), \gamma'(a_k))$ and there are no instances of $I^{*}$ between points of two distinct $B^{\ell\ell'}_{12}$, $\gamma'(a_k) \in B^{j'j'''}_{12}$.  We have proven that, for $ 1 \leq k \leq n-1$, $\gamma'(a_k) \in B^{j'j'''}_{12}$, and in particular that  $\gamma'(a_{n-1}) \in B^{j'j'''}_{12}$. Finally, since $I^{*}(\gamma'(a_{n-1}), \gamma'(a_n))$, $\gamma'(a_{n-1}) \in B^{j'j'''}_{12}$ and there are no instances of $I^{*}$ between $\gamma'(a_{n-1}) \in B^{j'j'''}_{12}$ and $B_{2}^{\ell}$ unless $\ell=j'''$, $\gamma'(a_n) \in B^{j'''}_{2} $. But we showed it is also the case that $\gamma'(a_n) \in B^{j''}_{2} $, while $j'' < j$ and $j''' > j$. This contradicts that $B^{j'''}_{2}$ is disjoint from $B^{j''}_{2}$ for $j'' \neq j'''$.

 It remains to show that $h$ is $n$-satisfactory. Suppose that $\alpha, a_{0}, \ldots, a_k$ is a $k$-split cycle in $\tilde{B}$ for $k \leq n$, and that its image under $h$ contains no split cycle. The only way this is possible is if $h(a_0)=h(a_k)$. Otherwise, since $h$ is a map of $\mathcal{L}_{\mathrm{std}}$-structures, $h(\alpha), h(a_{0}), h(a_{k})$ will be distinct with $R(h(\alpha), h(a_0))$, $R(h(\alpha), h(a_k))$, while $h(a_{0}), h(a_k)$ are the endpoints of some $I^{*}$-path consisting of distinct vertices in the image of $h$.  This implies that the image contains a split cycle, a contradiction. So $h(a) = h(b)$ for two distinct $a, b \in P(\tilde{B})$ such that $a$ and $b$ admit both a common $R$-neighbor $\alpha \in O(\tilde{B})$ and an $I^{*}$-path with at most $n$ edges. The point $\alpha$ must be in $B^{j}_{0}$ for some $j$: otherwise $\alpha$ is in $B^{jj'}_{12}$ for some $j < j'$. Then the set of $R$-neighbors of $\alpha$, and particularly $\{a, b\}$, will be contained in $B^{j}_{1}  \sqcup B^{jj'}_{12} \sqcup B_{2}^{jj'}$. But $B^{j}_{1}  \sqcup B^{jj'}_{12} \sqcup B_{2}^{jj'}$ contains no two points mapped to the same point by $h$, a contradiction. We have shown that $\alpha \in B^{j}_{0}$ for some $j$. Moreover,  $a, b$ must either belong to two distinct sets of the form $B_{1}^{j'}$, or to two distinct sets of the form $B_{2}^{j'}$: no set of the form $B^{j'j''}_{12}$ has an $R$-neighbor in $B^{j}_{0}$, so because $R(\alpha, a)$, $R(\alpha, b)$, neither $a$ nor $b$ belong to any set of the form $B^{j'j''}_{12}$. Therefore,  one of the two aforementioned possibilities is required for $h$ to map $a$ and $b$ to the same point. Let us assume the case that $a$ and $b$ belong to distinct sets of the form $B_{1}^{j'}$; the rest of the proof in the other case is exactly the same except for exchanging the subscripts $1$ and $2$, and replacing the expression $j''> j'$ with $j'' < j'$. It suffices to show that any point with an $I^{*}$-path from $a$ with at most $n$ edges\footnote{To avoid confusion when using the shorthand term \textit{length} in what follows, we define the length of a path to be its number of edges.}  must belong to $B_{1}^{j'}$, a set of the form $B_{12}^{j'j''}$ for some $j'' >j'$, or a set of the form $B^{j''}_{2}$: then such a point cannot be $b$, a contradiction. We first show by induction on $k \leq n-1$ that a point $a'$ with a path of length $k$ from $a$ must belong either to $B_{1}^{j'}$ or to one of the $B_{12}^{j'j''}$ as above. The point $a'$ will have an $I^{*}$-neighbor with an $I^{*}$-path of length $n-1$ from $a$, and this $I^{*}$-neighbor must be in $B_{1}^{j'}$ or one of the $B_{12}^{j'j''}$ by induction. But an $I^{*}$-neighbor of a point of $B_{1}^{j'}$ or one of the $B_{12}^{j'j''}$ must belong to $B_{1}^{j'}$, one of the $B_{12}^{j'j''}$, or one of the $B^{j''}_{2}$. So $a'$ belongs to $B_{1}^{j'}$, one of the $B_{12}^{j'j''}$, or one of the $B^{j''}_{2}$. But $a'$ cannot belong to one of the $B^{j''}_{2}$, because $a'$ has an $I^{*}$-path of length less than $n$ to $a \in B^{j'}_{1}$, and there is no $I^{*}$-path of length less than $n$ between a point of $B_{1}$ and a point of $B_{2}$. We have shown, for  $k \leq n-1$, that a point $a'$ with a path of length $k$ from $a$ must belong either to $B_{1}^{j'}$ or to one of the $B_{12}^{j'j''}$. It remains to show that, if $a'$ is a point with an $I^{*}$-path of length $n$ from $a$, then $a'$ must belong to $B_{1}^{j'}$, a set of the form $B_{12}^{j'j''}$ for some $j'' >j'$, or a set of the form $B^{j''}_{2}$. But $a'$ must have an $I^{*}$-neighbor with an $I^{*}$-path of length $n-1$ from $a$. By the previous claim for $k = n-1$, this $I^{*}$-neighbor belongs to $B_{1}^{j'}$ or one of the $B_{12}^{j'j''}$. But, again by the fact an $I^{*}$-neighbor of a point of $B_{1}^{j'}$ or one of the $B_{12}^{j'j''}$ must belong to $B_{1}^{j'}$, one of the $B_{12}^{j'j''}$, or one of the $B^{j''}_{2}$, $a'$ must belong to $B_{1}^{j'}$, one of the $B_{12}^{j'j''}$, or one of the $B^{j''}_{2}$. This completes our argument that any $a'$ with an $I^{*}$-path of length at most $n$ from $a$ must belong to  $B_{1}^{j'}$, one of the $B_{12}^{j'j''}$, or one of the $B^{j''}_{2}$, and thus completes our proof that $h$ is $n$-satisfactory.
    
\end{proof}

We now complete our proof of Lemma \ref{omitting a presidium}, using the same strategy as in the conclusion to the proof of Theorem 3.1 of \cite{ApproxOrder}. Here, to infer (among $\mathcal{L}^{\mathrm{std}}$-structures without directed cycles for $I$) the existence of $A \notin \mathcal{H}$ without split cycles from the existence of $A' \in \mathcal{H}$ without \textit{small} split cycles, we apply the assumption that $\mathcal{H}$ is defined by a finite family of forbidden weakly embedded substructures; we will do the same for potentially pinched alternating cycles in step three.

\begin{proof}
We first show, using the same argument from \cite{LZ17} used in the proof of Theorem 3.1 of \cite{ApproxOrder}, that there is some $A \notin \mathcal{H}$ with no split cycles at all and no directed cycles for $I$. Let $\mathcal{H} = \mathcal{H}(\mathcal{F})$ for some finite family $\mathcal{F}$ of finite $\mathcal{L}_{\mathrm{std}}$-structures. Choose $n$ greater than the cardinality of every structure in $\mathcal{F}$. By Lemma \ref{arbitrarily large minimal split cycles}, there is some $A_0 \notin \mathcal{H}$ with no $k$-split cycles for $k \leq n$ and no directed cycles for $I$. The set $A_0$ must weakly embed some $A \in \mathcal{F}$, and $A \notin \mathcal{H}$. But since $A_{0}$ contains no $n$-split cycles, and $A$ has a weak embedding into $A_0$, $A$ contains no $n$-split cycles. And because $|A| < n$ and any $n$-split cycle has cardinality greater than $n$, $A$ has no split cycles at all (and no directed cycles for $I$).

Then $A$ will be as desired, and we show that, by adding instances of $I$ on $A$, we can obtain a presidium: this suffices, because the fact that $\mathcal{H}$ is closed under weak embeddings means that this presidium, into which $A$ weakly embeds, will not be in $\mathcal{H}$. Let $P(A) = \sqcup^{N}_{i =1} C_i$ where $C_i$ are the connected components of $P(A)$ with respect to $I^{*}$; no two members of the same $C_i$ have a common $R$-neighbor in $A$, because $A$ has no split cycles. The $C_{i}$ have no directed cycles for $I$; now recall the well-known fact, Claim 3.16 of \cite{ApproxOrder}, that every relation with no directed cycles on a set can be extended to a linear order on that set. So let us extend the relation $I$ on $A$ so that each of the $C_i$ from before is now linearly ordered by this extended relation, and there are no instances of this extended relation between points of distinct $C_i$. To show that the resulting $\mathcal{L}_{\mathrm{std}}$-structure is a presidium (with the $C_i$ its convivia), it remains to show it is sound. But  any two neighbors of $P(A)$ with respect to this extended relation must belong to the same $C_i$, so have no common $R$-neighbor, as remarked above. This proves soundness of $A$ with $R$ given the original interpretation and $I$ interpreted as this extended relation, as desired.
\end{proof}

This completes step one.

\begin{remark}
  \label{standard tp2-structure}  Before proceeding to step two, we note, from a formal perspective, the presence of $\mathrm{TP}_{2}$ in the proof. Specifically, let the \textit{standard $\mathrm{TP}_{2}$-structure} be the $\mathcal{L}_{\mathrm{std}}$-structure $\mathbb{TP}_{2}$ consisting of distinct $b_{ij} \in P(\mathbb{TP}_{2})$ for $i, j < \omega$, and $a_{\sigma} \in O(\mathbb{TP}_{2})$ for $\sigma \in \omega^{\omega}$, with all and only the following instances of $R$ and $I$: $I(b_{i j_{1}},b_{ij_{2}})$ for $i, j_{1}, j_{2} < \omega$, $j_{1} < j_{2}$, and $R(a_{\sigma}, b_{i\sigma(i)} )$ for $\sigma \in \omega^{\omega}$, $i < \omega$. Then a finite $\mathcal{L}^{\mathrm{std}}$-structure is a presidium if and only if it (weakly) embeds into the standard $\mathrm{TP}_{2}$-structure. So in step one we have shown that there is an $\mathcal{L}^{\mathrm{std}}$-structure $A \notin \mathcal{H}$ that embeds into the standard $\mathrm{TP}_{2}$-structure, and we now want to find $\mathcal{L}^{\mathrm{std}}$-structure $A \notin \mathcal{H}$ that embeds into the standard $\mathrm{TP}$-structure.
\end{remark}

\:

\textbf{Step two: Eliminating consanguineous pairs} 

In steps two and three, we will show, from the conclusion (Lemma \ref{omitting a presidium}) that $\mathcal{H}$ omits a presidium when $\mathcal{H}$ is a hereditary class of $\mathcal{L}_{\mathrm{std}}$-structures defined by finitely many forbidden weak substructures which is the age of a structure whose theory is $\mathrm{NSOP}_{3}$, that $\mathcal{H}$ even omits a presidium with no \textit{potentially pinched alternating cycles}. This is a more robust class of cycles containing all of the alternating cycles, defined in Definition \ref{potentially pinched alternating cycle, flat} in step three below. Because all alternating cycles are potentially pinched alternating cycles, it follows that $\mathcal{H}$ omits a presidium with no alternating cycles. Then applying Lemma \ref{embedding into the standard tp-structure}, we conclude that $\mathcal{H}$ omits the standard $\mathrm{TP}$-structure, proving Corollary \ref{main combinatorial lemma} as needed to conclude the proof of the main theorem of this section. However, the helix maps required to find a presidium not belonging to $\mathcal{H}$ omitting potentially pinched alternating cycles with only $3$ or $4$ points of sort $P$ are defined differently from the helix maps required to find a presidium omitting all potentially pinched alternating cycles smaller than a size larger than $4$. Therefore, eliminating these smallest potentially pinched alternating cycles will constitute its own step, which we pursue here.

We make and justify the following assumption purely as an expositional convenience. Call a finite presidium \textit{balanced} if all convivia have the same cardinality, which is at least $2$. As in the short argument within the proof of Lemma \ref{embedding into the standard tp-structure} that we may assume that the presidium considered in that proof has convivia of cardinality at least four, we may extend any presidium to a balanced presidium. Thus, within the statement of Lemma \ref{omitting a presidium}, the omitted presidium $A \notin \mathcal{H}$ may be chosen to be balanced, and within steps two and three, when we refer to a presidium, \textit{we mean a balanced presidium}.

Our goal in this step will be to show that our hereditary class $\mathcal{H}$ omits a presidium with no potentially pinched alternating cycles whose points of sort $P$ lie in just two convivia:

\begin{definition}
    \label{consanguineous pair} Let $A$ be a presidium. Then a \textit{consanguineous pair} in $A$ is a pair of distinct convivia $C$, $C'$ of $A$, such that there are $c, d \in C$, $c', d' \in C'$ such that $c$ and $c'$ have a common $R$-neighbor, $d$ and $d'$ have a common $R$-neighbor, and it is not the case that both $c = d$  and $c' = d'$.
\end{definition}

So, using soundness, a presidium $A$ has no consanguineous pair if and only if $A$ has no subset consisting of distinct $a_1, a_2 \in O(A)$, $b_1, b_2, \beta \in P(A)$ with $I^{*}(b_1, b_2)$ and $R(a_{i}, b_{i})$ and $R(a_i, \beta)$ for $i = 1, 2$, and no subset consisting of distinct $a_1, a_2 \in O(A)$, $b_1, b_2, c_1, c_2 \in P(A)$ with $I^{*}(b_1, b_2)$ and $I^{*}(c_1, c_2)$ and $R(a_{i}, b_{i})$ and $R(a_i, c_i)$ for $i = 1, 2$.  We call the first of these, a subset of $A$ consisting of distinct $a_1, a_2 \in O(A)$, $b_1, b_2, \beta \in P(A)$ with $I^{*}(b_1, b_2)$ and $R(a_{i}, b_{i})$ and $R(a_i, \beta)$ for $i = 1, 2$, a \textit{potentially pinched alternating cycle} with exactly three points of sort $P$. And we have already stated (Definition \ref{alternating cycle}) that the second kind of subset, a subset of $A$ consisting of distinct $a_1, a_2 \in O(A)$, $b_1, b_2, c_1, c_2 \in P(A)$ with $I^{*}(b_1, b_2)$ and $I^{*}(c_1, c_2)$ and $R(a_{i}, b_{i})$ and $R(a_i, c_i)$ for $i = 1, 2$, is the definition of an \textit{alternating cycle} with exactly four points of sort $P$. While we have not yet given the full definition of a potentially pinched alternating cycle (as we will in step three, in Defintion \ref{potentially pinched alternating cycle, flat} below), the potentially pinched alternating cycles we will refer to in the rest of this step will be those with exactly three points of sort $P$, as well as alternating cycles with exactly $4$ points of sort $P$. So those presidia without consanguineous pairs will be those without the two smallest possible sizes of potentially pinched alternating cycles (there are no alternating cycles with only two points of sort $P$ by soundness), the ones with exactly three or four points of sort $P$.

We want to find such a presidium $A$ with $A \notin \mathcal{H}$. Having defined what a consanguineous pair is, we can now state the goal of this step:

\begin{lemma}
    \label{eliminating consanguineous pairs} Let $\mathcal{H}$ be a hereditary class consisting of sound $\mathcal{L}_{\mathrm{std}}$-structures, defined by a finite family of forbidden weakly embedded substructures, which is the age of some structure whose theory is $\mathrm{NSOP}_{3}$. Then there is some presidium $A$ such that $A \notin \mathcal{H}$ and that $A$ has no consanguineous pairs, or equivalently, does not have any potentially pinched alternating cycles (see Definition \ref{potentially pinched alternating cycle, flat} below) with exactly $3$ or $4$ points of sort $P$.
\end{lemma}

We define a class of maps between presidia, \textit{regular} maps, that respect the structure of presidia. Then analogously to Definition \ref{n-satisfactory} in step one, we define an ad hoc property of maps that preserve potentially pinched alternating cycles with exactly $3$ or $4$ points of sort $P$.

\begin{definition}
    \label{regular, consanguineous}

    (1) A map $f: A \to B$ of presidia is \textit{regular} if the restriction of $f$ to any convivium $C$ of $A$ is a bijection between $C$ and the convivium of $B$ containing $f(C)$. (So for balanced finite presidia, this will be true if the cardinality of all of the convivia in $A$ is equal to the cardinality of all of the convivia in $B$.)

    (2) A regular map $f: A \to B$ is \textit{consanguineous} if, whenever $C, C'$ are a consanguineous pair of convivia in $A$, $f(C)$ and $f(C')$ are not the same convivium of $B$.
\end{definition}

We observe that, for any consanguineous map $f$, the image under $f$ of (resp.) any potentially pinched alternating cycle with $3$ or $4$ points of sort $P$ is an alternating cycle with $3$ or $4$ points of sort $P$.  In the case where the potentially pinched alternating cycle consists of $b_1, b_2, \beta \in P(A)$, $a_1, a_2 \in O(A)$ (as in the discussion before \ref{eliminating consanguineous pairs}). $b_1, b_2$ belong to one convivium and $\beta$ (by soundness) belongs to a distinct convivium. These two convivia form a consanguineous pair.  So $f(b_1), f(b_2), f(\beta)$ are distinct: $f (\beta)$ is distinct from $f(b_{1})$ and $f(b_{2})$ by consanguineousness of $f$, and $f(b_1)$ and $f(b_2)$ are distinct because $I^{*}(b_1, b_2)$, so $I^{*}(f(b_1), f(b_2))$. And $f(a_1), f(a_2)$ are distinct by soundness. So the image of $b_1, b_2, \beta, a_1, a_2 $ is a potentially pinched alternating cycle with $3$ points of sort $P$, as desired. The case of $4$ points of sort $P$ is similar.

The following will play the role played by Lemma \ref{split cycle removal} in step one, and in turn, Proposition 3.11 of \cite{ApproxOrder}. We refer the reader to the commutative diagram presented in the statement of Lemma \ref{split cycle removal}. Just as for split cycles in step one (and allowing the same abuse of notation), we can alternatively view a potentially pinched alternating cycle in a presidium $A$, with $3$ or $4$ points of sort $P$, as an injective map $\gamma \hookrightarrow A$. Here $\gamma$ consists of either distinct $a_1, a_2 \in O(\gamma)$, $b_1, b_2, \beta \in P(\gamma)$ with $I^{*}(b_1, b_2)$ and $R(a_{i}, b_{i})$ and $R(a_i, \beta)$ for $i = 1, 2$ giving all the instances of $I$ and $R$ on $\gamma$, or of distinct $a_1, a_2 \in O(\gamma)$, $b_1, b_2, c_1, c_2 \in P(\gamma)$ with $I^{*}(b_1, b_2)$ and $I^{*}(c_1, c_2)$ and $R(a_{i}, b_{i})$ and $R(a_i, c_i)$ for $i = 1, 2$ giving all of the instances of $I$ and $R$ on $\gamma$.

\begin{lemma}
    \label{consanguineous pair removal} Let $\mathcal{H}$ be a hereditary class consisting of sound $\mathcal{L}_{\mathrm{std}}$-structures which is the age of some structure whose theory is $\mathrm{NSOP}_{3}$. Let $g: B \to A$ be a regular map between finite presidia with $A, B\notin \mathcal{H}$. Let $\gamma \hookrightarrow A$ be a potentially pinched alternating cycle with (resp.) exactly $3$ or $4$ points of sort $P$.  Then there is a consanguineous $3$-helix map $h: \tilde{B} \twoheadrightarrow B$ of some length $\ell$ associated with a cyclic $3$-jughandle decomposition of $B$ such that $\tilde{B} \notin \mathcal{H}$ and such that there is no potentially pinched alternating cycle $\gamma' \hookrightarrow \tilde{B}$ with exactly $3$ or $4$ points of sort $P$ with $\gamma = g\circ h\circ \gamma'$.

\end{lemma}

In practice, $g$ will be consanguineous. Also note that, due to the nature of our construction of the helix map, unlike in Lemma \ref{split cycle removal} in step one above or Lemma \ref{potentially pinched alternating cycle removal} in step three below, we need no minimality assumption on the potentially  pinched alternating cycle $\gamma$ (say, that there must be no potentially pinched alternating cycles with exactly $3$ points of sort $P$ in $A$, if $\gamma$ is a potentially pinched alternating cycle with exactly $4$ points of sort $P$.) We first observe that Lemma \ref{eliminating consanguineous pairs} follows from Lemma  \ref{consanguineous pair removal}, so it will then remain to show Lemma \ref{consanguineous pair removal} to complete this step.

\begin{proof}
    (of Lemma \ref{eliminating consanguineous pairs}, assuming Lemma \ref{consanguineous pair removal}.)

By Lemma \ref{omitting a presidium}, there is a presidium $A  \notin \mathcal{H}$. Now enumerate the potentially pinched alternating cycles with $3$ or $4$ sort-$P$ points in $A$ as $\gamma_1 \hookrightarrow A, \ldots, \gamma_k \hookrightarrow A$. Using Lemma \ref{consanguineous pair removal} (in place of Lemma \ref{split cycle removal}), and the observation that the composition of two consanguineous maps is consanguineous (in place of the observation that the composition of $n$-satisfactory maps is $n$-satisfactory), the following is proven just as in the proof of Lemma \ref{arbitrarily large minimal split cycles}: there is a presidium $B \notin \mathcal{H}$ and a consanguineous map $g_k: B \to A$ such that, for all $i \leq k$, there is no potentially pinched alternating cycle $\gamma' \hookrightarrow B$ with $3$ or $4$ sort-$P$ points with $g_k \circ \gamma' = \gamma_i$. However, this implies that there is no potentially pinched alternating cycle $\gamma' \hookrightarrow B$ with $3$ or $4$ sort-$P$ points at all, because otherwise, by consanguineousness, $g_{k} \circ \gamma'$ would be equal to one of the $\gamma_i$. So $B$ will give us the desired presidium in the statement of Lemma \ref{eliminating consanguineous pairs}.
\end{proof}

We now prove Lemma \ref{consanguineous pair removal}, which will complete the proof of Lemma \ref{eliminating consanguineous pairs} and complete step two.

\begin{proof}
 (of Lemma \ref{consanguineous pair removal})

 As in the proof of Lemma \ref{split cycle removal}, we define a cyclic $3$-jughandle decomposition of $A$, and then pull it back by $g: B \to A$. With the domain of $\gamma$ consisting of distinct $a_1, a_2 \in O(\gamma)$, $b_1, b_2, \beta \in P(\gamma)$ or distinct $a_1, a_2 \in O(\gamma)$, $b_1, b_2, c_1, c_2 \in P(\gamma)$ as in the discussion preceding the statement of this lemma (Lemma \ref{consanguineous pair removal}), let $a : = \gamma(a_1)$, and let $C$ be the convivium to which $b_1$ and $b_2$ belong. Define $A_{0} := C$, $A_{1} := \{a\}$, $A_{2} := P(A) \backslash C$ and $A_{20} = O(A) \backslash \{a\}$. Then because $A_{1}$, $A_{20}$ both consist of points of sort $O$, there are no instances of $R$ or of $I$ between points of $A_{1}$ and points of $A_{20}$, so $A = \sqcup^{2}_{i=0} A_i \sqcup A_{02}$ is a cyclic $3$-jughandle decomposition of $A$. Moreover, by soundness it is the case that for any convivium $C$ in $A_{0}$ (here just the one convivium) or in $A_{2}$, there is a point $c \in C$ such that for every point (here just the one) in $A_{1}$, the only $R$-neighbor of that point in $C$ is $c$.

 Now let $B_i : = g^{-1}(A_i)$, $B_{20} : = g^{-1}(A_{20})$. Then the properties $A = \sqcup^{2}_{i=0} A_i \sqcup A_{02}$ stated above are preserved: $B = \sqcup^{2}_{i=0} B_i \sqcup B_{20}$ is a cyclic $3$-jughandle decomposition of $B$, and it remains the case that any convivium in $B_{0}$ or $B_{2}$ has at most one point that has any $R$-neighbors at all in $B_{1}$. By Fact \ref{closure under jughandle helix maps}, for some $\ell < \omega$ and $\tilde{B} : = \tilde{B}^{\ell}$ where $h: \tilde{B}^{\ell} \twoheadrightarrow B $ is the $3$-helix map of length $\ell$ associated with the cyclic $3$-jughandle decomposition $B = \sqcup^{2}_{i=0} B_i \sqcup B_{20}$, $\tilde{B} \notin \mathcal{H}$. Of course, $h$ will be a regular map of presidia (because $h$ is a map of $\mathcal{L}^{\mathrm{std}}$-structures, if $\tilde{B}$ is not sound then $h(\tilde{B})$ would not be sound, contradicting that $B$ is sound.)

 To show $\tilde{B}$, $h$ are as desired, we first show there is no potentially pinched alternating cycle $\gamma' \hookrightarrow \tilde{B}$ with $3$ or $4$ sort-$P$ points (as many as $\gamma$) such that $h \circ \gamma' = \gamma$. Suppose otherwise. Then $\gamma'(a_1) \in B^{j}_{1}$ for some $j$. Moreover, $\gamma'(b_1) \in B^{j'}_{0}$ for some $j' < j$, as these are the only $j'$ for which $B^{j'}_{0}$ has an $R$-neighbor in $B_{1}^{j}$ (just by the construction of the helix map). And $\gamma'(b_2) \in B^{j'}_{0}$ as well, because $\gamma'(b_1)$ and $\gamma'(b_2)$ must be in the same convivium. Similarly, depending on the kind of potentially pinched alternating cycle, either $\gamma'(c_1), \gamma'(c_{2}) \in B^{j''}_{2}$ or $\gamma'(\beta) \in B^{j''}_{2}$ for some $j'' > j$. But there must be a common $R$-neighbor in $\tilde{B}$ of $\gamma'(b_2)$ and either $\gamma'(\beta)$ or $\gamma'(c_{2})$. (Specifically, this common $R$-neighbor will be $\gamma'(a_{2})$.) Since each covivium of $B_0$ has exactly one point that has some $R$-neighbor in $B_1$, each convivium of $B^{j'}_{0}$ has only one point that has some $R$-neighbor in one of the $B^{\ell}_{1}$. So $\gamma'(b_1)$ is the unique point in its convivium with an $R$-neighbor in one of the $B^{\ell}_{1}$ (one of which will be $\gamma'(a_{1}) \in B_{1}^{j}$). Therefore, $\gamma'(b_2)$ cannot have an $R$-neighbor in one of the $B^{\ell}_{1}$. So the common $R$-neighbor of $\gamma'(b_2)$ and either $\gamma'(\beta)$ or $\gamma'(c_{2})$ must be in one of the $B^{\ell\ell'}_{20}$ for $ \ell < \ell'$.  But a point of $B^{\ell\ell'}_{20}$ can only have $R$-neighbors in $B^{\ell}_{2}$ or $B^{\ell'}_{0}$. So $\ell = j''$, $\ell' = j'$ and $j'' < j'$. This contradicts $j' < j < j''$.

 We now show consanguineousness of $h$. Suppose otherwise, and let $C, C'$ be a consanguineous pair in $\tilde{B}$ with $h(C) = h(C')$. Then $C, C'$ must be in distinct $B_{i}^{j}$ for fixed $i = 0, 2$. No point of one of the $B^{\ell\ell'}_{20}$ can have two $R$-neighbors in distinct $B_{i}^{j}$ for fixed $i = 0, 2$, so any common $R$-neighbor of a point of $C$ and a point of $C'$ must be in one of the $B^{j}_{1}$.  But $h(C) = h(C')$ has at most one point that is an $R$-neighbor of some point of $B_1$. So $C$ and $C'$ each have at most one point that is an $R$-neighbor of some point of some $B^{j}_{1}$;  call these points $c \in C$ and $c' \in C'$ respectively. But then any common $R$-neighbor in $\tilde{B}$ of a point in $C$ and a point in $C'$, being in one of the $B^{j}_{1}$, must have $c$ its one $R$-neighbor in $C$ and $c'$ its one $R$-neighbor in $C'$. So $C$ and $C'$ are not a consanguineous pair, a contradiction.

\end{proof}

This concludes step two.

\:

\textbf{Step three: Eliminating potentially pinched alternating cycles}

We now define in general our more robust class of cycles, the \textit{potentially pinched alternating cycles}, that in particular contains all of the alternating cycles; our goal in this step will be to show, from our conclusion from step two that $\mathcal{H}$ omits a presidium without the smallest potentially pinched alternating cycles, that $\mathcal{H}$ omits a presidium without any potentially pinched alternating cycles at all--and hence without any alternating cycles. The class of potentially pinched alternating cycles will specifically be robust in the sense that it will be preserved under a class of maps of $\mathcal{L}^{\mathrm{std}}$-structures that we will also now define, the \textit{flat} maps, for which we will also have a cycle-removal lemma: Lemma \ref{potentially pinched alternating cycle removal} in what follows, corresponding to Lemma \ref{consanguineous pair removal} for removing the smallest potentially pinched alternating cycles in step two, Lemma \ref{split cycle removal} for removing split cycles, and Proposition 3.11 of \cite{ApproxOrder} for removing directed cycles within the context of the real-valued $\mathrm{NSOP}_{r}$ hierarchy.

\begin{definition}
   \label{potentially pinched alternating cycle, flat} (1) Let $A$ be a presidium. A \textit{potentially pinched alternating cycle} in $A$ consists of either an alternating cycle in $A$, or of distinct $b_{1}, \ldots, b_{2n}, \beta_1 \in P(A)$, $a_1, \ldots, a_{n-1}, \alpha_1, \alpha_2 \in O(A)$ for $n \geq 1$, such that  $I^{*}(b_i, b_{i+1} )$ for odd $i \leq 2n$, $R(a_i, b_{2i})$ and $R(a_i, b_{2i+1})$ for $i \leq n-1$, and $R(\alpha_{1}, b_{2n})$, $R(\alpha_1, \beta_1)$, $R(\alpha_2, \beta_1)$, $R(\alpha_2, b_1)$. (So there are no $a_i$ when $n = 1$, only $b_1$, $b_2$, $\alpha_1$, $\alpha_2$, $\beta_1$ as described).

   (2) A \textit{flat} map of $\mathcal{L}^{\mathrm{std}}$-structures is a map $\flat: A \to B$ between $\mathcal{L}^{\mathrm{std}}$-structures $A$ and $B$ such that, for any $a, b \in P(A)$ with a common $R$-neighbor, $\flat(a) \neq \flat(b)$.
\end{definition}

We show robustness of the class of potentially pinched alternating cycles to flat maps of presidia:

\begin{lemma}
    \label{potentially pinched alternating cycle preserved under flat map} Let $A$ be a $\mathcal{L}^{\mathrm{std}}$-structure and let $B$ be a presidium, and let $\flat: A \to B$ be a flat map of $\mathcal{L}^{\mathrm{std}}$-structures. Then the image under $\flat$ of any potentially pinched alternating cycle contains a potentially pinched alternating cycle.
\end{lemma}

\begin{proof}
    Suppose otherwise, and let $\flat: A \to B$ be a flat map of presidia for which this fails. We first show that we may assume that $\flat$ does not map any two distinct points of sort $O$ in the potentially pinched alternating cycle in $A$ to the same point of $B$.  Suppose $\flat$ is not already injective on the sort-$O$ points of the potentially pinched alternating cycle. Otherwise, in the case of an alternating cycle, let us consider either arc between these two points (inclusive) in that alternating cycle. However, in the other case, let us consider the arc between these two points in that potentially pinched alternating cycle not containing $\beta_1$. In either case, this will give us $c^{0}, \ldots, c^k \in O(A) $, $c_{1}, \ldots, c_{2k} \in P(A)$ ($k \geq 1$)  in the potentially pinched alternating cycle in $A$ with $I^{*}(c_i, c_{i+1} )$ for odd $i \leq 2k$, $R(c^i, c_{2i})$ and $R(c^i, c_{2i+1})$ for $i = 1, \ldots, k-1$, $R(c^0, c_1)$, and $R(c^k, c_{2k})$, with $\flat(c^0)=\flat(c^k)$. Then also, by flatness of $\flat$, $\flat(c_i) \neq \flat(c_{i+1}) $ even for even $i < 2k$ (because $c_{i}, c_{i+1}$ have a common $R$-neighbor.)
    
    Suppose first that $\flat(c_1) \neq \flat(c_{2k})$ (which will include the case $k =1$, since $I^{*}(c_1, c_{2k})$ in this case). Let the $\mathcal{L}_{\mathrm{std}}$-structure $C$ consist of an alternating cycle with exactly $k$ points of sort $O$: distinct $b_{1}, \ldots, b_{2k} \in P(C)$, $a_1, \ldots, a_k \in O(C)$ for $n \geq 1$ with $I^{*}(b_i, b_{i+1} )$ for odd $i \leq 2k$, $R(a_i, b_{2i})$ and $R(a_i, b_{2i+1})$ for $i= 1, \ldots, k-1$, and $R(a_k, b_{2k})$ and $R(a_k, b_{1})$, with no other instances of $R$, $I$.  Define $\overline{\flat}: C \to B$ by $\overline{\flat}(b_i) = \flat(c_i)$ and $\overline{\flat}(a_i)=\flat(c^{i})$. Then $\overline{\flat}(a_k)= \flat(c^k) = \flat(c^0)$. By this last point and the fact that $\flat$ is a map of $\mathcal{L}^{\mathrm{std}}$-structures, $\overline{\flat}$ is a map of $\mathcal{L}^{\mathrm{std}}$-structures (particularly $R(\overline{\flat}(a_k),\overline{\flat}(b_1) )$). Moreover, by $\flat(c_i) \neq \flat(c_{i+1}) $ for even $i < 2k$ and the supposition that $\flat(c_1) \neq \flat(c_{2k})$, $\overline{\flat}$ is flat. So by induction on the size of a potentially pinched alternating cycle in the domain of a flat map of $\mathcal{L}^{\mathrm{std}}$-structures, the image of $\overline{\flat}$ contains a potentially pinched alternating cycle. But by definition, the image of $\overline{\flat}$ contains the image of $\flat$ on the potentially pinched alternating cycle in $A$. So the image of $\flat$ on the potentially pinched alternating cycle in $A$ contains a potentially pinched alternating cycle, a contradiction.

    Suppose next that $\flat(c_1) = \flat(c_{2k})$ (so $k \geq 2$, because if $k =1$ then $I^{*}(c_1, c_{2k})$), but $\flat(c_2) \neq \flat(c_{2k-1})$.  Let the $\mathcal{L}_{\mathrm{std}}$-structure $C$ consist of an alternating cycle with exactly $k-1$ points of sort $O$: distinct $b_{1}, \ldots, b_{2k-2} \in P(C)$, $a_1, \ldots, a_{k-1} \in O(C)$ for $n \geq 1$ with $I^{*}(b_i, b_{i+1} )$ for odd $i \leq 2k-2$, $R(a_i, b_{2i})$ and $R(a_i, b_{2i+1})$ for $i= 1, \ldots, k-2$, and $R(a_{k-1}, b_{2k-2})$ and $R(a_{k-1}, b_{1})$, with no other instances of $R$, $I$.  Define the map $\overline{\flat}: C \to B$ by $\overline{\flat}(b_1) = \flat(c_{2k-1})$, $\overline{\flat}(b_i) =\flat(c_i)$ for $i = 2, \ldots 2k-2$ (so in particular $\overline{\flat}(b_2) = \flat(c_{2})$), and $\overline{\flat}(a_i) =\flat(c^i)$. Now, note that $I^{*}(c_1, c_2)$ implies $I^{*}(\flat(c_1), \flat(c_2))$, and $I^{*}(c_{2k}, c_{2k-1})$ implies $I^{*}(\flat(c_{2k}), \flat(c_{2k-1}))$. So  $I^{*}(\flat(c_{1}), \flat(c_{2k-1}))$, $\flat(c_2)$ and $\flat(c_{2k-1})$ are distinct members of the same convivium, and $I^{*}(\flat(c_2), \flat(c_{2k-1}))$. (This is where we use that $\flat(c_2) \neq \flat(c_{2k-1})$.)  It follows from this and the fact that $\flat$ is a map of $\mathcal{L}_{\mathrm{std}}$-structures that $\overline{\flat}$ is a map of $\mathcal{L}_{\mathrm{std}}$-structures (particularly $I^{*}(\overline{\flat}(b_1), \overline{\flat}(b_2))$. Moreover,  by $\flat(c_i) \neq \flat(c_{i+1}) $ for even $i < 2k$, $\overline{\flat}$ is flat, so we conclude as before, getting a contradiction.

    Finally, suppose that $\flat(c_1) = \flat(c_{2k})$, and $\flat(c_2) = \flat(c_{2k-1})$. In that case, since $\flat$ is flat, $k \geq 3$ (otherwise $c_2$ and $c_{2k-1} = c_{3}$ would have a common $R$-neighbor). Let the $\mathcal{L}$-structure $C$ consist of the potentially pinched alternating cycle with $n = (k-2)$ as in the above Definition \ref{potentially pinched alternating cycle, flat}: distinct  $b_{1}, \ldots, b_{2(k-2)}, \beta_1 \in P(C)$, $a_1, \ldots, a_{(k-2)-1}, \alpha_1, \alpha_2 \in O(C)$, such that  $I^{*}(b_i, b_{i+1} )$ for odd $i \leq 2(k-2)$, $R(a_i, b_{2i})$ and $R(a_i, b_{2i+1})$ for $i \leq (k-2)-1$, and $R(\alpha_{1}, b_{2(k-2)})$, $R(\alpha_1, \beta_1)$, $R(\alpha_2, \beta_1)$, $R(\alpha_2, b_1)$. Define $\overline{\flat}: C \to B$ as follows: $\overline{\flat}(b_i)= \flat(c_{i+2})$, $\overline{\flat}(a_i)=\flat(c^{i+1})$, $\overline{\flat}(\alpha_1)= \flat(c^{k-1})$, $\overline{\flat}(\beta_1)= \flat(c_{2k-1})$, $\overline{\flat}(\alpha_{2})=\flat(c^1)$. Then $\overline{\flat}(\beta_1)=\flat(c_2)$ as well. By this last point and the fact that $\flat$ is a map of $\mathcal{L}^{\mathrm{std}}$-structures, $\overline{\flat}$ is again a map of $\mathcal{L}^{\mathrm{std}}$-structures (particularly $R(\overline{\flat}(\alpha_{2}),\overline{\flat}(\beta_1))$). Moreover, since $\flat(c_i) \neq \flat(c_{i+1}) $ for even $i < 2k$ and $\overline{\flat}(b_{2(k-2)})=\flat(c_{2k-2}) \neq \flat(c_{2k-1}) = \overline{\flat}(\beta_1)=\flat(c_2) \neq \flat(c_{3})=\overline{\flat}(b_1)$, $\overline{\flat}$ is flat. So we again conclude as before, getting a contradiction.

    So we may now assume that $\flat: A \to B$ does not map any two points of sort $O$ in the potentially pinched alternating cycle to the same point of $B$. Because the image of $\flat$ on the potentially pinched alternating cycle is assumed not to contain any potentially pinched alternating cycle, $\flat$ is not injective on this potentially pinched alternating cycle. So there must be two points of sort $P$ within this potentially pinched alternating cycle which are mapped to the same point of $B$. In the case of an alternating cycle, let us consider either arc between these two points in that alternating cycle (inclusive of those two points).  In the other case, let us consider the arc between those points in that potentially pinched alternating cycle not containing $\beta_1$.  There are three possibilities for this arc: first, there are $d_0, \ldots, d_{2k} \in P(A)$, $d^{0}, \ldots, d^{k-1} \in O(A)$ for $k \geq 1$ in the potentially pinched alternating cycle in $A$ such that $I^{*}(d_{i}, d_{i+1})$ for $i \leq 2k$ odd, $R(d^{i}, d_{2i})$, $R(d^{i}, d_{2i+1})$ for $i \leq k-1$, and $\flat(d_0) = \flat(d_{2k})$. (I.e., starting with an instance of $R$ and ending with an instance of $I^{*}$.) Second, there are $d_1, \ldots, d_{2k} \in P(A)$, $d^{1}, \ldots, d^{k-1} \in O(A)$ for $k \geq 1$ in the potentially pinched alternating cycle in $A$ such that $I^{*}(d_{i}, d_{i+1})$ for $i \leq 2k$ odd, $R(d^{i}, d_{2i})$, $R(d^{i}, d_{2i+1})$ for $i \leq k-1$, and $\flat(d_1) = \flat(d_{2k})$. (I.e., starting and ending with instances of $I^{*}$.) Finally, since $\flat$ is a flat map, so this arc cannot consist of two points of sort $P$ with a common $R$-neighbor (or two points of sort $P$ that are $I^{*}$-neighbors), the sole remaining possibility is that there are $d_0, \ldots, d_{2k-1} \in P(A)$, $d^{0}, \ldots, d^{k-1} \in O(A)$ for $k \geq 2$ in the potentially pinched alternating cycle in $A$ such that $I^{*}(d_{i}, d_{i+1})$ for $i \leq 2k-2$ odd, $R(d^{i}, d_{2i})$, $R(d^{i}, d_{2i+1})$ for $i \leq k-1$, and $\flat(d_0) = \flat(d_{2k-1})$. (I.e., starting and ending with instances of $R$.) In all three of these cases, $\flat(d_{i}) \neq \flat(d_{i+1}) $ even for $i < 2k$ even, by flatness of $\flat$.

    We first obtain a contradiction from the first possibility ($R$ on one end and $I^{*}$ on the other). Let the $\mathcal{L}_{\mathrm{std}}$-structure $C$ consist of an alternating cycle with exactly $k$ points of sort $O$: distinct $b_{1}, \ldots, b_{2k} \in P(C)$, $a_1, \ldots, a_k \in O(C)$ for $n \geq 1$ with $I^{*}(b_i, b_{i+1} )$ for odd $i \leq 2k$, $R(a_i, b_{2i})$ and $R(a_i, b_{2i+1})$ for $i= 1, \ldots, k-1$, and $R(a_k, b_{2k})$ and $R(a_k, b_{1})$, with no other instances of $R$, $I$. Define $\overline{\flat}(b_i)=\flat(d_{i})$, $\overline{\flat}(a_{i})=\flat(d^{i})$ for $i < k$, $\flat(a_k)=\flat(d^0)$. Then $\overline{\flat}(b_{2k})=\flat(d_{2k})=\flat(d_{0})$. By this last point and the fact that $\flat$ is a map of $\mathcal{L}_{\mathrm{std}}$-structures, $\overline{\flat}$ is also a map of $\mathcal{L}_{\mathrm{std}}$-structures (and particularly, $R(\overline{\flat}(a_{k}), \overline{\flat}(b_{2k}))$). By $\flat(d_{i}) \neq \flat(d_{i+1}) $ for $i < 2k$ even and $\overline{\flat}(b_{2k})=\flat(d_{k})=\flat(d_{0}) \neq \flat(d_1) = \overline{\flat}(b_1)$, $\overline{\flat}$ is flat. So we conclude as before and reach a contradiction.

Finally, in the second case ($I^{*}$ on both ends), we reach a contradiction as in the cases with $\flat(c_1) = \flat(c_{2k})$ and $\flat(c_{2}) \neq \flat(c_{2k-1}) $ and with $\flat(c_1) = \flat(c_{2k})$ and $\flat(c_{2}) = \flat(c_{2k-1}) $ within the reduction to where $\flat$ is injective on the sort-$O$ points of the potentially pinched alternating cycle (the third and fourth paragraphs of this proof). In the third case ($R$ on both ends), we reach a contradiction as in the case with $\flat(c_1) = \flat(c_{2k})$ and $\flat(c_i) = \flat(c_{i+1}) $ within that same reduction (the fourth paragraph of this proof). This completes the proof of the lemma.
    
\end{proof}

Having defined the class of potentially pinched alternating cycles and shown its robustness under flat maps, we now state the lemma which it will be our goal in this step to prove, on finding a presidium omitted by $\mathcal{H}$ without any potentially pinched alternating cycles.

\begin{lemma}
 \label{eliminating potentially pinched alternating cycles}   Let $\mathcal{H}$ be a hereditary class consisting of sound $\mathcal{L}_{\mathrm{std}}$-structures, defined by a finite family of forbidden weakly embedded substructures, which is the age of some structure whose theory is $\mathrm{NSOP}_{3}$. Then there is some presidium $A \notin \mathcal{H}$ with no potentially pinched alternating cycles.
\end{lemma}

Note that, to apply Lemma \ref{embedding into the standard tp-structure} and show that $A \notin \mathcal{H}$ above embeds into the standard $\mathrm{TP}$-structure, we do not need that $A$ has no potentially pinched alternating cycles--only that it has no alternating cycles. There are in fact potentially pinched alternating cycles that embed into the standard $\mathrm{TP}$-structure. So the statement of Lemma \ref{eliminating potentially pinched alternating cycles} is stronger than is needed, but the more robust nature of the class of potentially pinched alternating cycles compared with the class of alternating cycles, as demonstrated by Lemma \ref{potentially pinched alternating cycle preserved under flat map}, will play an important role in the proof.

To prove Lemma \ref{eliminating potentially pinched alternating cycles}, we first show that we may restrict our attention to potentially pinched alternating cycles respecting the presidium structure of the ambient presidium:

\begin{lemma}
    \label{potentially pinched alternating cycle respects presidium structure}

Let $A$ be a presidium, and let distinct $b_{1}, \ldots, b_{2n} \in P(A)$, $a_1, \ldots, a_n \in O(A)$ be an alternating cycle as in Definition \ref{alternating cycle}, or let distinct $b_{1}, \ldots, b_{2n}, \beta_1 \in P(A)$, $a_1, \ldots, a_{n-1}, \alpha_1, \alpha_2 \in O(A)$ be a potentially pinched alternating cycle as in Definition \ref{potentially pinched alternating cycle, flat}. Suppose that this is a potentially pinched alternating cycle of minimal size in $A$. Then there are distinct convivia $C_{1}, \ldots, C_{n} $, as well as, in the second, non--alternating cycle case, a convivium $\Gamma_{1}$ distinct from the $C_{i}$, such that $C_{i}$ contains $b_{2i-1}$ and $b_{2i}$, and $\beta_{1} \in \Gamma_{1}$.

\end{lemma}

\begin{proof}
    Clearly $b_{2i-1}$ and $b_{2i}$ belong to the same convivium (because $I^{*}(b_{2i-1}, b_{2i})$). So it suffices to show that the convivium containing $b_{2i-1}, b_{2i}$ is distinct from the convivium containing $b_{2j-1}, b_{2j}$ for $i \neq j$, and that the convivium containing $b_{2i-1}, b_{2i}$ and the continuum containing $\beta_1$ are distinct (in the second case). For the first claim,  suppose otherwise, so that $b_{2i-1}, b_{2i}$ are in the same convivium as $b_{2j-1}, b_{2j}$ for $i \neq j$; we may of course assume $i < j$. Then $I^{*}(b_{2i}, b_{2j-1})$. So $b_{2i} , \ldots b_{2j-1}$ together with $a_{i} \ldots, a_{j-1}$ will form an alternating cycle smaller than the one we started with, contradicting minimality.  For the second claim, otherwise $I^{*}(b_{2i-1}, \beta_{1})$. So $\beta_{1}, b_{1}, \ldots, b_{2i-1}$ and $\alpha_{2}$ together with (if $i > 1$) $a_{1} \ldots a_{i-1}$ will form a smaller alternating cycle, again a contradiction.
\end{proof}

We now prepare to state our cycle-removal lemma for this step, Lemma \ref{potentially pinched alternating cycle removal}, which will say that flat helix maps satisfy a similar cycle-removal property for potentially pinched alternating cycles that consanguineous helix maps satisfy for consanguineous pairs (as in Lemma \ref{consanguineous pair removal} in step two) and $n$-satisfactory helix maps satisfy for split cycles (as in Lemma \ref{split cycle removal} in step one). To give the precise statement of Lemma \ref{potentially pinched alternating cycle removal}, we will need to introduce the following definition, which will allow us to handle potentially pinched alternating cycles for which the convivia $C_{i}$, $\Gamma_{1}$ as in the previous lemma may have additional pairs of points with a common $R$-neighbor that are not already accounted for by the potentially pinched alternating cycle.

\begin{definition}
      \label{potentially pinched alternating connectivity pattern}
      
      (1) A \textit{potentially pinched alternating connectivity pattern} consists of:

      (a) a presidium $C$, such that there is

\begin{itemize}
    \item an alternating cycle in $C$ given (resp.) by distinct $b_{1}, \ldots, b_{2n} \in P(C)$, $a_1, \ldots, a_n \in O(C)$ or distinct $b_{1}, \ldots, b_{2n}, \beta_1 \in P(C)$, $a_1, \ldots, a_{n-1}, \alpha_1, \alpha_2 \in O(C)$, and
    \item convivia (resp.) $C_{1}, \ldots C_{n}$ or $C_{1}, \ldots C_{n}, \Gamma_{1}$ as in Lemma \ref{potentially pinched alternating cycle respects presidium structure}, so $b_{2i}, b_{2i+1} \in C_{i}$ and $\beta_{1} \in \Gamma_{1}$
\end{itemize}

      such that

      \begin{itemize}
          \item (resp.) $C = \bigsqcup^{n}_{i=1} C_{i} \sqcup  \{a_{1}, \ldots, a_{n}\} $ or  $C = \bigsqcup^{n}_{i=1} C_{i} \sqcup \Gamma_{1} \sqcup \{a_{1}, \ldots, a_{n-1}, \alpha_{1}, \alpha_{2}\} $ (so the $C_{i}$ or $C_{i}, \Gamma_{1}$ are all of the convivia of $C$), and
          \item (resp.) $R(a_i, b_{2i})$ and $R(a_i, b_{2i+1})$ for $i= 1, \ldots, n-1$, and $R(a_n, b_{2n})$ and $R(a_n, b_{1})$ are the only instances of $R$ on $C$, or $R(a_i, b_{2i})$ and $R(a_i, b_{2i+1})$ for $i \leq n-1$, and $R(\alpha_{1}, b_{2n})$, $R(\alpha_1, \beta_1)$, $R(\alpha_2, \beta_1)$, $R(\alpha_2, b_1)$ are the only instances of $R$ on $C$.
      \end{itemize}

(Note that $C$ will have exactly one potentially pinched alternating cycle: any potentially pinched alternating cycle will lie inside of the $a_{i}, \alpha_{i}$ together with their $R$-neighbors, so within the $a_{i}, \alpha_{i}, b_{i}, \beta_{1}$, but then must be equal to that potentially pinched alternating cycle. So, more succinctly, we can state that $C$ is equivalently the presidium consisting of the union of a potentially pinched alternating cycle that will be the unique one in $C$, together with convivia containing the sort-$P$ points of that potentially pinched alternating cycle, with only the instances of $R$ and $I$ necessary to satisfy these conditions.)

      (b) an undirected graph relation $\mathrm{Conn}$ on $P(C)$, such that whenever $b, c \in P(C)$ have a common $R$-neighbor in  $C$, $\mathrm{Conn}(b,c)$.

(2) For $A$ a presidium, a \textit{realization} of a potentially pinched alternating connectivity pattern $C$ in $A$ is an injective regular map $\gamma: C \hookrightarrow A$ of presidia such that, if $\mathrm{Conn}(b, c)$, $\gamma(b)$ and $\gamma(c)$ have a common $R$-neighbor in $A$. 
      
\end{definition}

Now, when formulating our analogue of Lemma \ref{split cycle removal} (or Proposition 3.11 of \cite{ApproxOrder}) stating that we find flat helix maps allowing us to remove potentially pinched alternating cycles, Definition \ref{potentially pinched alternating connectivity pattern} will have given us a suitable analogue of a split cycle of minimal length in a presidium (or a directed cycle of minimal length in a directed graph): a realization of a potentially pinched alternating connectivity pattern corresponding to a minimal-length potentially pinched alternating cycle in a presidium, that maximizes the number of instances of $\mathrm{Conn}$. In stating this lemma we refer the reader to the commutative diagram in Lemma \ref{split cycle removal}, just with $\gamma: C \hookrightarrow A$ instead of $\gamma \hookrightarrow A$ and $\flat$ instead of $h$.

\begin{lemma}
    \label{potentially pinched alternating cycle removal} Let $\mathcal{H}$ be a hereditary class consisting of sound $\mathcal{L}_{\mathrm{std}}$-structures which is the age of some structure whose theory is $\mathrm{NSOP}_{3}$. Let $g: B \to A$ be a regular map of presidia with $A, B\notin \mathcal{H}$, where $A$ has no consanguineous pairs. Suppose, for $N \geq 1$, that $A$ has no potentially pinched alternating cycle with $k$ points of sort $P$ for $k < N$, and let $\gamma: C \hookrightarrow A$ be a realization of a potentially pinched alternating connectivity pattern $C$ in $A$ such that:

\begin{itemize}
    \item the potentially pinched alternating cycle of $C$ has $N$ points of sort $P$, and
    \item the number of instances of $\mathrm{Conn}$ on $C$ is maximal among potentially pinched alternating connectivity patterns with the same underlying presidium as $C$ which have a realization in $A$
\end{itemize}

    Then there is a regular, flat $3$-helix map $\flat: \tilde{B} \twoheadrightarrow B$ of some length $\ell$ associated with a cyclic $3$-jughandle decomposition of $B$, such that $\tilde{B} \notin \mathcal{H}$, and such that there is no realization  $\gamma': C \hookrightarrow \tilde{B}$ of the potentially pinched alternating connectivity pattern $C$ in $\tilde{B}$ with $\gamma = g\circ \flat\circ \gamma'$.

\end{lemma}

In practice, $g$ will be flat. Before showing Lemma \ref{potentially pinched alternating cycle removal}, we will show that the lemma which is our goal for this step, Lemma \ref{eliminating potentially pinched alternating cycles}, follows from this cycle removal lemma. This will be analogous to Lemma \ref{arbitrarily large minimal split cycles} from step one stating that $\mathcal{H}$ omits a $\mathcal{L}^{\mathrm{std}}$-structure without small split cycles, together with the arguments from the end of step one using the fact that $\mathcal{H}$ is defined by a finite family of forbidden weakly embedded substructures to infer that $\mathcal{H}$ omits a $\mathcal{L}^{\mathrm{std}}$-structure without any split cycles (or, in turn, to Claim 3.15 of \cite{ApproxOrder} giving a directed graph without small directed cycles omitted from a hereditary class, and the argument from the proof of Theorem 3.1 of \cite{ApproxOrder} originating from \cite{LZ17} using the same assumption to get an omitted directed graph without any directed cycles.)

\begin{proof}
     (of Lemma \ref{eliminating potentially pinched alternating cycles} assuming Lemma \ref{potentially pinched alternating cycle removal})

     By Lemma \ref{eliminating consanguineous pairs}, there is a presidium $A \notin \mathcal{H}$ with no consanguineous pairs; let $M$ be the size of its convivia. As a convenience, to the end of step three, whenever we refer to a presidium (and particularly the underlying presidium of a potentially pinched alternating connectivity pattern), we mean this presidium, which we already assumed balanced, has all convivia of size $M$.  This assumption will be harmless for our arguments, because for any regular map of presidia, the convivia of its domain will have size $M$ if and only if the convivia of its codomain have size $M$. So we already know that there is some presidium not belonging to $\mathcal{H}$ without any potentially pinched alternating cycles with at most $4$ points of sort $P$, and we show that, if there is a presidium $A \notin \mathcal{H}$ with no potentially pinched alternating cycles with less then $n$ points of sort $P$ for $n \geq 5$, there is a presidium $B \notin \mathcal{H}$ with no potentially pinched alternating cycles with at most $n$ points of sort $P$. By induction, it will follow that for any $n$, there is a presidium $A \notin \mathcal{H}$ with no potentially pinched alternating cycles with at most $n$ points of sort $P$.
     
     So let $n \geq 5$, and let $A \notin \mathcal{H}$ be a presidium with no potentially pinched alternating cycles with less then $n$ points of sort $P$. Let us enumerate the connectivity patterns (up to isomorphism as $\{R, I, \mathrm{Conn}\}$-structures with sorts $O$, $P$) for which the unique potentially pinched alternating cycle has $n$ points of sort $P$ as $C^{1}, \ldots C^{K}$; additionally, choose this enumeration such that that if $i < j$, $C^{i}$ has at least as many instances of $\mathrm{Conn}$ as $C^{j}$. (The point of the assumption that all convivia have fixed size is to ensure that there are only finitely many $C^{i}$).
     
     Now if a presidium $B$ has a potentially pinched alternating cycle with $n$ points of sort $P$ and no potentially pinched alternating cycles with fewer points of sort $P$, then by Lemma \ref{potentially pinched alternating cycle respects presidium structure} there is a realization of one of these $C^i$ in $B$. To see this, take the underlying presidium $C$ of this potentially pinched alternating connectivity pattern to  to be the union of the $C_{i}$ and possibly $\Gamma_{1}$ given by Lemma \ref{potentially pinched alternating cycle respects presidium structure} for this alternating cycle in $B$, together with (the $O$-points of) the potentially pinched alternating cycle itself. Let $\mathrm{Conn}$ be the relation on $P(C)$ giving the tautological potentially pinched alternating connectivity pattern structure on $C \subset B$, where $\mathrm{Conn}(b, c)$ for $b, c \in P(C)$ if $b, c $ have a common $R$-neighbor in $B$. Then the inclusion $C \hookrightarrow B$ is a realization of $C$, so there is a realization in $B$ of whichever of the $C^{i}$  is isomorphic to $C$.

     So it suffices to show that there is some presidium $B \notin \mathcal{H}$ with no potentially pinched alternating cycles with fewer than $n$ points of sort $P$, which does not have a realization of any of the  $C^{1}, \ldots C^{K}$. By induction on $L \leq K+1$, it then suffices to show that if there is some presidium $A \notin \mathcal{H}$ with no potentially pinched alternating cycles with fewer than $n$ points of sort $P$ and with no realization of $C^{i}$ for $i < L$, then there is some presidium $B \notin \mathcal{H}$ with no potentially pinched alternating cycles with fewer than $n$ points of sort $P$ and which does not have a realization of $C^{i}$ for $i \leq L$.

      We now aim to show this last claim: let $\gamma_{1}: C^{L} \hookrightarrow A, \ldots, \gamma_{k} :{C}^{L} \hookrightarrow A$ enumerate the realizations of $C^{L}$ in $A$. We  apply the following claim, which, like the analogous Claim \ref{pulling back does not add split cycles}, is proven exactly as in Lemma 3.13 of \cite{ApproxOrder}; we do review the argument because we have changed the formal setup.

\begin{claim}
\label{pulling back does not create new realizations of a potentially pinched alternating connectivity pattern}

Let $g: A_{2} \to A_{1}$ be a regular map of presidia, and let $\gamma: C \hookrightarrow A_{2}$ be a realization of a potentially pinched alternating connectivity pattern in $A_{2}$. Suppose that there is no realization $\gamma': C \hookrightarrow A_{2}$  such that $g \circ \gamma' = \gamma$. Let $h: A_{2}' \to A_{2}$ be another regular map of presidia. Then there is no realization $\gamma'': C \hookrightarrow B$ such that $(g \circ h) \circ \gamma'' = \gamma$.

\end{claim}

\begin{proof}
    Suppose otherwise, and let $\gamma''$ be as in the statement of the claim. Since $(g \circ h) \circ \gamma''= \gamma$ is injective, so is $h \circ \gamma''$.  And since $\gamma''$ is a regular map of presidia which respects the $\mathrm{Conn}$-structure as in the definition of a realization of a potentially pinched alternating connectivity pattern (Definition \ref{potentially pinched alternating connectivity pattern}.2), so is $h \circ \gamma''$. So $h \circ \gamma: C \hookrightarrow A_{2}$ is a realization of $C$ in $A_{2}$. But, for $\gamma' : = h \circ \gamma''$, $g \circ \gamma' = \gamma$, so we get a contradiction. 
\end{proof}

We note that that by our choice of enumeration, and the hypothesis that none of the $C^{i}$ have a realization in $A$ for $i < L$, $C^{L}$ has the maximum number of instances of $\mathrm{Conn}$ among potentially pinched alternating connectivity patterns realized in $A$ whose potentially pinched alternating cycle has the same size as the smallest potentially pinched alternating cycle in $A$.
So applying this claim, Lemma \ref{potentially pinched alternating cycle removal}, and the observation that a composition of flat regular maps is a flat regular map, exactly as in Lemma \ref{arbitrarily large minimal split cycles} we get a presidium $B$ and a flat map $\flat_{k+1}: B \to A$ such that there is no realization $\gamma': C^{L} \hookrightarrow B$ such that $\flat_{k+1} \circ \gamma' = \gamma_{i}$ for $i \leq k$. 

We first claim that $B$ does not have any potentially pinched alternating cycles with fewer than $n$ points of sort $P$. Otherwise, the image of such a potentially pinched alternating cycle under $\flat_{k+1}$ would be a potentially pinched alternating cycles with fewer than $n$ points of sort $P$ in $A$, by Lemma \ref{potentially pinched alternating cycle preserved under flat map}, a contradiction. 

We next claim that $B$ does not have a realization $\gamma: C^{i} \hookrightarrow B$ for $i < L$. Otherwise, we first show that $\flat_{k+1} \circ \gamma$ is injective. Suppose otherwise, and first suppose that $\flat_{k+1} \circ \gamma$ is at least injective on the potentially pinched alternating cycle of $C^{i}$. Then the image of that potentially pinched alternating cycle under $\flat_{k+1} \circ \gamma$ would have over two points of sort $P$ in the same convivium of $A$. But then, by Lemma \ref{potentially pinched alternating cycle respects presidium structure}, $A$ must have a potentially pinched alternating cycle with less than $n$ points of sort $P$, contradicting the hypothesis on $A$. Now assume the other case, that  $\flat_{k+1} \circ \gamma$ is not even injective on the potentially pinched alternating cycle of $C^{i}$. By Lemma \ref{potentially pinched alternating cycle preserved under flat map}, the image  of that potentially pinched alternating cycle under $\flat_{k+1} \circ \gamma$ will contain a potentially pinched alternating cycle in $A$, which will also have less than $n$ points of sort $P$. This also contradicts the hypothesis on $A$. So $\flat_{k+1} \circ \gamma$ is injective. So as in the proof of Claim \ref{pulling back does not create new realizations of a potentially pinched alternating connectivity pattern}, $\flat_{k+1} \circ \gamma$ must be a realization of $C^{i}$. This contradicts that $A$ has no realizations of $C^{i}$.

Similarly, there is no realization $\gamma: C^{L} \hookrightarrow B$. Otherwise, as before, $\flat_{k+1} \circ \gamma$ would be a realization of $C^{L}$ in $A$, This contradicts that $\flat_{k+1} \circ \gamma$ cannot be any of the $\gamma_{i}$.

We have shown that for any $n$, there is a presidium $A \notin \mathcal{H}$ with no potentially pinched alternating cycles with at most $n$ points of sort $P$. Exactly as in the argument at the beginning of the proof of Lemma \ref{omitting a presidium} concluding step one, we conclude from the fact that $\mathcal{H}$ is a hereditary class of $\mathcal{L}^{\mathrm{std}}$-structures defined by a finite family of forbidden weakly embedded substructures that there is $A \notin \mathcal{H}$ with no potentially pinched alternating cycles of any size. This is what we need in Lemma \ref{potentially pinched alternating cycle removal}.

\end{proof}

We now prove Lemma \ref{potentially pinched alternating cycle removal}, from which, as we have just shown, Lemma \ref{eliminating potentially pinched alternating cycles} will follow, completing step three.

\begin{proof}
     (of Lemma \ref{potentially pinched alternating cycle removal})

We first analyze potentially pinched alternating connectivity patterns $C$ with a realization $\gamma: C \hookrightarrow A$. Choose $D^{1}$ to be one of the convivia $C_{i}$ of $C$. Then there are distinct $c, d \in D^{1}$, with a sequence of distinct convivia $D^{2}, \ldots, D^{k}$ distinct from $D^{1}$, such that $c$ has a $\mathrm{Conn}$-neighbor in $D^{2}$, if $2 \leq i \leq k$ there is a point of $D^{i}$ that has a $\mathrm{Conn}$-neighbor in $D^{i+1}$, and $d$ has a $\mathrm{Conn}$-neighbor in $D^{k}$. Holding $D^{1}$ fixed, find $c, d, D^{2}, \ldots, D^{k}$ satisfying these hypotheses so that $k$ is as small as possible. Since $A$ has no consanguineous pairs and $\gamma: C \hookrightarrow A$ is a realization of $C$ in $A$, $k \geq 3$. Now first, $c$ does not have a $\mathrm{Conn}$-neighbor in $D^{i}$ for $2 < i < k$, because then $c, d, D^{i}, \ldots D^{k}$ would exhibit a contradiction to minimality of $k$. Similarly, $d$ does not have a $\mathrm{Conn}$-neighbor in $D^{i}$ for $2 < i < k$. Moreover, no point of $D^{1}$ besides $c$ will have a $\mathrm{Conn}$-neighbor in $D^{2}$, nor will $c$ have more than one $\mathrm{Conn}$-neighbor in $D^{2}$: otherwise, since $C$ has a realization in $A$, $A$ would have a consanguineous pair. Similarly, no point of $D^{1}$ besides $d$ will have a $\mathrm{Conn}$-neighbor in $D^{k}$ nor will $d$ have more than one $\mathrm{Conn}$-neighbor in $D^{k}$. Additionally, among the $D^{2}, \ldots, D^{k}$, the only two pairs $D^{i}, D^{j}$ such that a point of $D^{i}$ is a $\mathrm{Conn}$-neighbor of a point of $D^{j}$ are those of the form $D^{i}, D^{i+1}$. Otherwise, assuming $i < j$, $c, d$ together with the $D^{\ell}$ for $2 \leq \ell \leq k$ with $\ell \leq i$ or $\ell \geq j$ will exhibit a contradiction to minimality of $k$.  Finally, for $2\leq i <k$, there is only one member of $D^{i}$ with a $\mathrm{Conn}$-neighbor in $D^{i+1}$ which will be its unique $\mathrm{Conn}$-neighbor in $D^{i+1}$. This is again because $A$ has no consanguineous pairs and $C$ has a realization in $A$. By soundness of $A$ and the fact that $C$ is realized in $A$, there are of course no instances of $\mathrm{Conn}$ within any of the $D^{i}$. In conclusion, we have found convivia $D^{1}, \ldots, D^{k}$ of $C$ such that the instances of $\mathrm{Conn}$ on $\sqcup^{k}_{i =1} D^{i}$ consist of a single instance between a point of $D^{i}$ and a point of $D^{i+1}$ for $1 \leq i < k$, and a single instance between a point of $D^{k}$ and a point of $D^{1}$, and the two points of $D^{1}$ that belong to instances of $\mathrm{Conn}$ on $\sqcup^{k}_{i =1} D^{i}$ are distinct.

Let $C^{i}$ be the convivium $\gamma(D^{i+1})$ of $A$ for $0 \leq i < k$; recall that $k \geq 3$, We claim that the only (unordered) pairs of distinct points in $\sqcup^{n-1}_{i=0} C^{i}$ with a common $R$-neighbor consist of, for each of $1 \leq i \leq k-1$, a single pair with a point from $C^{i}$ and a point from $C^{i+1 \mathrm{\: mod \:} k}$, and that the two points of $C^{0}$ that belong to one of these pairs within $\sqcup^{k}_{i =1} C^{i}$ are distinct. By the previous paragraph, it suffices to show that all unordered pairs of distinct points with a common $R$-neighbor within $\sqcup^{k-1}_{i=0} C^{i}$ are images under $\gamma$ of instances of $\mathrm{Conn}$ on $\sqcup^{k}_{i =1} D^{i}$.  Suppose otherwise, so there are distinct $b_1, b_2 \in \sqcup^{k-1}_{i=0} C^{i}$ with a common $R$-neighbor such that the pair $\{b_1, b_2\}$ is not an image under $\gamma$ of an instance of $\mathrm{Conn}$ on $\sqcup^{k}_{i =1} D^{i}$. Let $C'$ be the potentially pinched alternating connectivity pattern with the same underlying presidium as $C$, where the relation $\mathrm{Conn}$ on $C'$ consists of those instances already on $C$ together with $\{\gamma^{-1}(b_1), \gamma^{-1}(b_2)\}$. Then $C'$ will have more instances of $\mathrm{Conn}$ than $C$. But $\gamma: C' \hookrightarrow C$ will still be a realization of $C'$, because $b_{1}, b_{2}$ have a common $R$-neighbor. This contradicts the maximality of the number of instances of $\mathrm{Conn}$ on $C$.

We now define a cyclic $3$-jughandle decomposition on $A$, which we will again pull back along $g$. First of all, let $D_{0} = P(A) \backslash \sqcup^{k-1}_{i=0} C^{i}$. Then let $C^{01}$ consist of those points of $O(A)$ with only $R$-neighbors in $D_{0} \sqcup C^{0} \sqcup C^{1}$, excluding those with at least one $R$-neighbor in $C^{1}$ but with only $R$-neighbors in $D_{0} \sqcup C^{1}$. In particular, $C^{01}$ will contain points with only $R$-neighbors in $D_{0}$, as well as points with an $R$-neighbor in $C^{0}$ and whose only other $R$-neighbors are in $C_{0} \sqcup D_{0}$. For $1 \leq i \leq k-1$ let $C^{i(i +1 \mathrm{\: mod \:} k)}$ consist of those points of $O(A)$ with at least one $R$-neighbor in $C^{i}$ and only $R$-neighbors in $D_{0} \sqcup C^{i} \sqcup C^{i+1 \mathrm{\:mod \: }k}$. In particular, $C^{i(i +1 \mathrm{\: mod \:} k)}$ will contain points with an $R$-neighbor in $C^{i}$ but only $R$-neighbors in $D_{0} \sqcup C^{i} $. We claim that $A=  D_{0} \sqcup \bigsqcup^{k-1}_{i=0} C^{i} \sqcup C^{i(i +1 \mathrm{\: mod \:} k)} $ gives a partition of $A$. First, by definition, the sets are disjoint; we show that equality holds.  The previous paragraph implies that a point of $O(A)$ can only have $R$-neighbors in two distinct $C^{i}$ and $C^{j}$ if $i$ differs from $j$ by $1$ modulo $k$. If $k \geq 4$, a point with $R$-neighbors in two distinct $C^{i}$ and $C^{j}$ will then, for some $i < k$, only have $R$-neighbors in $C^{i}$ and $C^{i+1  \mathrm{\: mod\: } k }$ and in no other $C^{\ell}$; this proves the equality when $k \geq 4$.  If $k = 3$, we must consider the additional possibility that a point has $R$-neighbors in all of $C^{0}, C^{1}, C^{2}$. But this is ruled out by the fact that the two points of $C^0$ that have a common $R$-neighbor with a point in $C^{1}$ or a point of $C^{2}$ are distinct, with one not having a common $R$-neighbor with any point of $C^{1}$, and the other not having a common $R$-neighbor with any point of $C^{2}$.  This proves the equality in the remaining case when $k = 3$. Now let $A_{0}:=C^{0}$, $A_{1}:= \bigsqcup^{k-2}_{i=1} C^{i}$, $A_{2}:= C^{k-1}$, $A_{01}:=\sqcup^{k-3}_{i=0} C^{i(i+1 \mathrm{\:mod \: k)}}$, $A_{12}:=C^{(k-2)(k-1)}$, $A_{20} := C^{(k-1)0}$. Then  $D_{0} \sqcup \bigsqcup^{2}_{i=0}A_{i} \sqcup  \bigsqcup^{2}_{i=0}A_{i(i +1 \mathrm{\:mod \:} 3)}$ is a cyclic $3$-jughandle decomposition of $A$, by construction.

     Now we again pull back along $g$ to get a cyclic $3$-jughandle decomposition $B= D \sqcup \bigsqcup^{2}_{i = 0} B_{i} \sqcup \bigsqcup^{2}_{i = 0} B_{i(i+1\mathrm{\:mod\:} 3)}$ of $B$: let $D:=g^{-1}(D_{0})$, $B_{i}:=g^{-1}(A_{i})$ and $B_{01}, B_{12}, B_{20} := g^{-1}(A_{01}),g^{-1}(A_{12}), g^{-1}(A_{20})$. Now let $\flat: \tilde{B}^{\ell} \twoheadrightarrow B$ be the helix map of length $\ell$ associated with this cyclic $3$-jughandle decomposition. By Fact \ref{closure under jughandle helix maps} we can find $\ell$ large enough that $\tilde{B}^{\ell} \notin \mathcal{H}$, and define $\tilde{B} : = \tilde{B}^{\ell}$. (Note as in Lemma \ref{consanguineous pair removal} that $\flat$ is again a regular map of presidia.)  Write $\tilde{B}^{\ell}= D \sqcup \bigsqcup_{j < \ell, i < 3} B^{j}_{i} \sqcup \bigsqcup_{j < j' < \ell} B^{jj'}_{01}\sqcup \bigsqcup_{j < j' < \ell} B_{12}^{jj'} \sqcup \bigsqcup_{j < j' < \ell} B^{jj'}_{20}$, again as in the proof of Fact \ref{closure under jughandle helix maps}. Note that $P(\tilde{B})=   \bigsqcup_{j < j' < \ell} B^{jj'}_{01}\sqcup \bigsqcup_{j < j' < \ell} B_{12}^{jj'} \sqcup \bigsqcup_{j < j' < \ell} B^{jj'}_{20}$; also note that, by construction of helix maps, for $j < j' < \ell$, $ii'=01, 12, 20$, $B^{jj'}_{ii'}$ only has $R$-neighbors in $D$, $B^{j}_{i}$, and $B^{j'}_{i'}$. Thus for $ii'=01, 12, 20$, there can only be a point in $B^{j}_{i}$ and a point in $B^{j'}_{i'}$ with a common $R$-neighbor if $j' > j$. Moreover, for $i = 0, 1, 2$, there can only be a point of $B^{j}_{i}$ and $B^{j'}_{i}$ with a common $R$-neighbor if $j' = j$.

     We first show that there is no realization $\gamma': C \hookrightarrow \tilde{B}$ such that $g \circ \flat \circ \gamma' = \gamma$. Suppose otherwise, and define $E^{i} := \gamma'(D^{i+1}) $ for $0 \leq i < k$ (so $g \circ \flat(E^{i})= \gamma(D^{i+1}) =  C^{i}$). Then $E^{0}$ is in some $B^{j}_{0}$.  Because some point of $D^{0}$ has a $\mathrm{Conn}$-neighbor in $D^{1}$ and $\gamma'$ is a realization of $C$, some point of $E^{1}$ must have a common $R$-neighbor with a point of $E^{0}$.  So $E^{1}$ can only be contained in $B^{j'}_{1}$ for some $j' > j$. We next show by induction on $k$ with $1 \leq k' \leq k-2$ that $E^{k'}$ is contained in $B^{j'}_{1}$. If $E^{k'-1} \subset B^{j'}_{1}$, then by similar reasoning to before, some point of $E^{k'-1}$ must have a common $R$-neighbor with some point of $E^{k'}$. So the only $B^{\ell}_{1}$ in which which $E^{k'}$ is contained is $B^{j'}_{1}$. So we have shown for $1 \leq k' \leq k-2$ that $E^{k'}$ is contained in $B^{j'}_{1}$, and particularly that $E^{k
    -2} \subset B^{j'}_{1}$. But, applying the same reasoning a third time, some point of $E^{k-2}$ has a common $R$-neighbor with some point of $E^{k-1}$. Thus $E^{k-1}$ can only be contained in $B^{j''}_{2}$ for some $j''> j'$.  So $E^{k-1} \subset B^{j''}_{2}$, $E^{0} \subset B^{j}_{0}$, and moreover, applying the same reasoning yet again, some point of $E^{k-1}$ and some point of $E^{0}$ have a common $R$-neighbor. But $j'' > j' > j$, so this is impossible, proving by contradiction that there can be no such $\gamma'$.

    It remains to show that $\flat$ is flat. Suppose $b, b' \in P(\tilde{B})$ are distinct with $\flat(b) = \flat(b')$. Then by construction of the helix map $\flat$, for some $i$ and some $j \neq j'$, $b \in B^{j}_{i}$ and $b \in B^{j'}_{i}$. So $b'$, $b$ cannot have a common $R$-neighbor in $\tilde{B}$, and $\flat$ is flat.

     \end{proof}

This completes step three. Having now shown Lemma \ref{eliminating potentially pinched alternating cycles}, we know that in particular, when $\mathcal{H}$ is a hereditary class of sound $\mathcal{L}^{\mathrm{std}}$-structures defined by a finite family of forbidden weakly embedded substructures, which is the age of a structure with $\mathrm{NSOP}_{3}$ theory, there is some presidium $A \notin \mathcal{H}$ with no alternating cycles. But by Lemma \ref{embedding into the standard tp-structure}, $A$ weakly embeds into the standard $\mathrm{TP}$-structure.  So because $\mathcal{H}$ is closed under weak embeddings, it omits an induced substructure of the standard $\mathrm{TP}$-structure.

We have shown

\begin{cor}
\label{main combinatorial lemma}
Let $\mathcal{H}$ be a hereditary class consisting of sound $\mathcal{L}_{\mathrm{std}}$-structures, defined by a finite family of forbidden weakly embedded substructures, which is the age of some structure whose theory is $\mathrm{NSOP}_{3}$. Then for $\mathbb{TP}$ the standard $\mathrm{TP}$-structure, $\mathbb{TP} \notin \mathcal{H}$.

\end{cor}

By this corollary, $\mathcal{H}^{\mathrm{std}}$ as in Lemma \ref{standard hereditary class} cannot exist. So none of the complete theories of Cherlin, Shelah and Shi (\cite{CSS99}), complete $T^{\mathcal{H}}$ for $\mathcal{H}$ any hereditary class defined by a finite family of forbidden weakly embedded substructures, can be strictly $\mathrm{NSOP}_{3}$.  By Proposition \ref{reduction to generic structures}, which reduces the main theorem to the claim that none of these $T^{\mathcal{H}}$ can be strictly $\mathrm{NSOP}_{3}$, \textit{we have now completed the proof of the main theorem of this section, Theorem \ref{main theorem 3, restated}.}

\begin{remark}
    \label{potential simplified proof}

In Remark \ref{potential noninterference in SOP_2}, we introduced the condition that, in any theory with $\mathrm{SOP}_{2}$, there is an instance of $\mathrm{SOP}_{2}$ where certain algebraic closures do not overlap nontrivially (and not just an instance of $\mathrm{TP}$ satisfying non-overlapping conditions, Lemma \ref{noninterference of algebraic closures}.) This was condition (**) of that remark, and we do not know whether it is true. We remarked that if it is true, Theorem \ref{main theorem 3, restated} can be proven by proving the following:

(*): If $\mathcal{H}$ is a hereditary class of $\mathcal{L}^{\mathrm{std}}$-structures defined by a finite family of forbidden weakly embedded substructures, then if the structures of $\mathcal{H}$ are sound and $\mathbb{SOP}_2 \in \mathcal{H}$, every theory whose models have age $\mathcal{H}$ must have $\mathrm{SOP}_{3}$.

We also asserted that this condition (*) was much easier to prove than the version we used here, Corollary \ref{main combinatorial lemma}, where $\mathbb{TP}$ stands in the place of $\mathbb{SOP}_{2}$. In this remark we will give a sketch of this easier proof, which would hypothetically, assuming the algebraic closure condition (**), vastly simplify the proof of the overall theorem.

Suppose $\mathcal{H}$ is a hereditary class defined by finitely many forbidden weak substructures, that it is sound, and that $\mathcal{H}$ is the age of the models of some $\mathrm{NSOP}_{3}$ theory, so in particular an $\mathrm{NSOP}_{3}$ hereditary class (in the sense of Definition 2.12 of \cite{ApproxOrder}.) We want to show that $\mathbb{SOP}_{2} \notin \mathcal{H}$.

Say a finite $\mathcal{L}^{\mathrm{std}}$-structure $A$, where $I$ is antisymmetric and irreflexive on $A$, is an \textit{obstruction} if $A$ does not weakly embed into $\mathbb{SOP}_{2}$.   Observe the following two facts about obstructions:

(1) In any obstruction $A$, the undirected graph with edges between any two points of $O(A)$ or $P(A)$ with an instance of $R(x, y)$ or $I(y_{1}, y_{2})$ has an undirected cycle. (Otherwise, we can index this structure as a tree, where each node only has relations (in either direction for $I$) to its immediate successor and predecessor in the tree. We can then use this tree-indexing to inductively construct a weak embedding of $A$ into $\mathbb{TP}_{2}$.)

(2) Suppose that $A$ is any finite $\mathcal{L}^{\mathrm{std}}$-structure, and $A$ has a homomorphism into $\mathbb{SOP}_{2}$. Then $A$ has a weak embedding into $\mathbb{SOP}_{2}$: suppose there is a homomorphism $h: A \to \mathbb{SOP}_{2}$ with $k < |A|$ points $ p \in A$ such that $p$ is the unique point $p'$ with $h(p')=h(p)$. We find a homomorphism $h': A \to \mathbb{SOP}_{2}$ with $k+1$ such points. This will suffice by induction, because for $k = |A|$ we get an injective map.  First, we may assume that $h$ only maps points of sort $P$ to, say, the even levels of the tree of sort-$P$ points of $\mathbb{SOP}_{2}$. Now suppose first that some points $p, p'$ of sort $P$ are mapped to the same point of $\mathbb{SOP}_{2}$. We  obtain $h'$ by just changing the image of $p'$ so that $p'$ is now mapped to the immediate successor of $h(p)$, rather than to $h(p)$ itself. Then $p'$ is now the unique point of $A$ mapped by $h'$ to $h'(p')$. Moreover, $p'$ has instances of $R$ and $I$ with the exact same points of  $h(A \backslash \{p\})$ that $p$ does (and in the same direction, for $I$), so $h'$ remains a homomorphism. Then $h'$ is a homomorphism with one additional point that's the unique point mapped to its image, as desired. Second, suppose that some points $p, p'$ of sort $O$ are mapped to the same point $a \in \mathbb{SOP}_{2}$. Note that there are infinitely many points in $O(\mathbb{SOP}_{2})$ with the exact same instances of $R$ with points of $P(h(A))$ that $a$ has. We may choose a new such point to map $p'$ to that's not already in $h(A)$, and conclude as before. So, if a finite $\mathcal{L}^{\mathrm{std}}$-structure has a homomorphism into $\mathbb{SOP}_{2}$, it must have a weak embedding into $\mathbb{SOP}_{2}$. It follows that any homomorphic image of an obstruction is also an obstruction.

Enumerate the obstructions as $\{A_i\}_{i< \omega}$ such that, if $A_{i}$ is smaller than $A_{j}$ or is the same size as $A_{j}$ but has more edges, $i < j$. Using (1), we may remove obstructions using helix maps. Specifically, we have, just as in Lemmas \ref{split cycle removal}, \ref{consanguineous pair removal}, \ref{potentially pinched alternating cycle removal} or the original cycle-removal property Proposition 3.11 of \cite{ApproxOrder}, the following obstruction-removal property of helix maps:

Let $\mathcal{H}$ be an $\mathrm{NSOP}_{3}$ hereditary class of $\mathcal{L}_{\mathrm{std}}$-structures. Let $g: B \to A$ be a map between finite $\mathcal{L}_{\mathrm{std}}$-structures with $A, B\notin \mathcal{H}$. Suppose that, for $n < \omega$, $A$ does not weakly embed $A_{k}$ for $k < n$, and let $\gamma: A_{n} \hookrightarrow A$ be an embedding of the obstruction $A_{n}$ into $A$.  Then there is a $3$-helix map $h: \tilde{B} \twoheadrightarrow B$ of some length $\ell$ associated with a cyclic $3$-decomposition of $B$ (i.e. a cyclic $3$-jughandle decomposition such that $B_{01}=B_{12}=B_{20}=\emptyset$) such that $\tilde{B} \notin \mathcal{H}$ and such that there is no embedding $A_{n} \hookrightarrow \tilde{B}$ with $\gamma = g\circ h\circ \gamma'$.

The main point is that, by (1), $A_{n}$ has an undirected cycle in the $R, I$, which can be assumed to be an \textit{induced} cycle in $A_{n}$. By minimality of $n$, $\gamma$ maps this to a \textit{induced} cycle in $A$. (Otherwise some extension $A_{k}$ of $A_{n}$ by additional instances of $R, I$ would have a weak embedding into $A$, and $k < n$.) And just as we could find a helix map with our cycle-removal property with respect to induced split cycles in Lemma \ref{split cycle removal}, or with respect to induced directed cycles in Proposition 3.11 of \cite{ApproxOrder}, we may find a helix map exhibiting the cycle-removal property for induced cycles here as well; the proof does not use any radically different idea.

Then by (2), the image of some obstruction $A_{k}$ under \textit{any} homomorphism contains an induced obstruction $A_{j}$ for $j < k$. That is, we no longer need any extra ad hoc properties of homomorphisms to preserve obstructions. So, we can argue just as in the proof of Theorem 3.1 of \cite{ApproxOrder}: starting with the fact that $\mathcal{H}$ omits some $\mathcal{L}^{\mathrm{std}}$-structure\footnote{(speaking now only of $\mathcal{L}^{\mathrm{std}}$-structures where $I$ is antisymmetric and irreflexive)} because it is sound, we can find a $\mathcal{L}^{\mathrm{std}}$-structure $A \notin \mathcal{H}$ omitting  all $A_{k}$ for $k$ less than some arbitrarily large $N$. So  we can find an $\mathcal{L}^{\mathrm{std}}$-structure $A \notin \mathcal{H}$ omitting all obstructions of size less than some arbitrarily large $N$. Applying the assumption that $\mathcal{H}$ is defined by finitely many forbidden weak substructures, we conclude that $\mathcal{H}$ omits an $\mathcal{L}^{\mathrm{std}}$-structure with no obstructions, so one that embeds in $\mathbb{SOP}_{2}$. So $\mathbb{SOP}_{2} \notin \mathcal{H}$, as desired.

Note the relatively precise analogy with the proof of Theorem 3.1 of \cite{ApproxOrder}. For example, there we have the initial condition that our hereditary class of directed graphs there omits the $(N, k)$-cycle $G$, which has no $n$-cycles. This corresponds here to soundness, which says that our hereditary class omits some initial $\mathcal{L}^{\mathrm{std}}$-structure. In the proof of Theorem 3.1 of \cite{ApproxOrder}, we would then ultimately like to conclude that our hereditary class omits the infinite chain graph, corresponding here to $\mathbb{SOP}_{2}$. The directed graphs that do not weakly embed into the infinite chain graph are the directed cycles--this corresponds to point (1) on obstructions in this remark. And directed cycles are preserved under homomorphisms, just as obstructions are preserved under homomorphisms by point (2). Again, this contrasts with steps one through three of the proof of Theorem \ref{main theorem 3, restated} in this section, where we need to show that our helix maps satisfy special conditions, other than just being homomorphisms, to make sure that they preserve the substructures that we are trying to remove.

\end{remark}

\addcontentsline{toc}{subsection}{\hspace{1.5em}Sidebar 2: Approximate implications between classification-theoretic properties}
\begin{tcolorbox}[
    enhanced,
    breakable,
    oversize,                 
    colback=blue!5!white,     
    colframe=blue!75!black,   
    skin first=enhanced,      
    skin middle=enhanced,     
    skin last=enhanced        
]

\textsf{\textbf{Sidebar 2: Approximate implications between classification-theoretic properties}}

\:

\: The cycle-removal arguments in this section even demonstrate a more general classification-theoretic phenomenon within all $\mathrm{NSOP}_{3}$ theories, and particularly within $\mathrm{NSOP}_1$ theories. We will express this phenomenon within the language of \textit{characteristic sequences} developed by Malliaris in \cite{Mal10}. When we talk about hypergraph sequences, we mean sequences $R_{\infty}=(V,\{R_{n}\}_{n < \omega})$ of hypergraphs on a common set of vertices $V$, where $R_{n}$ is an $n$-ary edge relation consisting of $n$-tuples of distinct elements which is permutation invariant, and each edge of $R_{n}$ is an $R_{m}$-clique for $m \leq n$. Malliaris gives the following definition, expressing the pattern of consistency and inconsistency within the instances of a formula:

\begin{definition}(\cite{Mal10})
   Let $\varphi(x, y)$ be a formula. The \emph{characteristic sequence} of $\varphi(x, y)$ is the sequence of hypergraphs, on the vertices $\mathbb{M}^{|y|}$, with $R_{n}(y_{1}, \ldots y_{n})$ defined by $\bigwedge_{i=1}^{n}\varphi(x, y_{i})$ (for distinct $y_{1}, \ldots, y_{n}$.)
\end{definition}

\: We need a way to talk about patterns of consistency and inconsistency that can be found within the characteristic sequence of a formula, where not all instances of consistency and inconsistency need to be defined. We  get this from the following definitions:

\begin{definition}
(1) A \textit{partial hypergraph sequence} is a sequence $(U, \{P_{n}\}_{n < \omega})$ for some set $V$, where $P_{n}=( R^{+}_{n} ,R_{n}^{-} ,R_{n}^{\sim})$ is a partition of the set of $n$-tuples of distinct elements of $U$ into three permutation-invariant parts, such that each tuple of $R^{+}_{n} $ is a $R_{k}^{+}$-clique for $k < n$, and each $n$-tuple containing a tuple in $R_{k}^{-}$ is in $R_{n}^{-}$ for $k <n$.

(2) An \textit{embedding} $\iota: U \hookrightarrow V$ of a partial characteristic sequence $(U, \{P_{n}\}_{n < \omega})$ into a hypergraph sequence $(V,\{R_{n}\}_{n < \omega})$ is an injective map such that if $(u_{1}, \ldots u_{n}) \in R^{+}_{n} $, $(\iota(u_{1}), \ldots \iota(u_{n})) \in R_{n} $, and if $(u_{1}, \ldots u_{n}) \in R^{-}_{n} $, $(\iota(u_{1}), \ldots \iota(u_{n})) \notin R_{n} $. An \textit{embedding} $\iota: U \hookrightarrow U'$ of partial characteristic sequences $(U, \{( R^{+}_{n} ,R_{n}^{-} ,R_{n}^{\sim})\}_{n < \omega})$ and $(U', \{( (R^{+}_{n})' ,(R_{n}^{-} )',(R_{n}^{\sim})')\}_{n < \omega})$ is an injective map such that if $(u_{1}, \ldots u_{n}) \in R^{+}_{n} $, $(\iota(u_{1}), \ldots \iota(u_{n})) \in (R^{+}_{n})' $, and if $(u_{1}, \ldots u_{n}) \in R^{-}_{n} $, $(\iota(u_{1}), \ldots \iota(u_{n})) \in (R^{-}_{n})' $

\end{definition}

\: As first observed by \cite{Mal10}, many classification-theoretic properties are associated with hypergraph sequences. For example, let an $(\omega, \omega, 1)$-array, as in \cite{Mal10}, be the hypergraph sequence consisting of an array $\{b_{ij}\}_{i, j \in \omega}$ such that, for all $i$, $\{b_{ij}\}_{j \in \omega}$ is an $R_{2}$-anticlique, and for any $\sigma \in \omega^{\omega}$, $\{ b_{i\sigma(i)}\}_{i \in \omega}$ is a clique for each $R_{n}$. Then the hypergraph sequence we associate with $\mathrm{TP}_{2}$ will be the $(\omega, \omega, 1)$-array: as observed in \cite{Mal10}, a formula has $\mathrm{TP}_{2}$ if and only if its characteristic sequence embeds the $(\omega, \omega, 1)$-array (as a sequence of induced sub-hypergraphs). Similarly, let the hypergraph sequence for $\mathrm{MSOP}_{2}$, defined in \cite{INDNSOP3}, consist of vertices $\{b_{\eta}\}_{\eta \in 2^{<\omega}}$ such that, for every $\sigma \in 2^{\omega}$, $\{ b_{\sigma \upharpoonleft n}\}_{n < \omega}$ is a clique for each $R_{n}$, but for any incomparable $\eta_{1}, \eta_{2}$, $\{b_{\eta_{1}},  b_{\eta_{2} }\}$ is an anticlique for $R_{2}$. Then we can associate this characteristic sequence with $\mathrm{SOP}_{2}$: a formula has $\mathrm{SOP}_{2}$ if and only if its characteristic sequence embeds the hypergraph sequence for $\mathrm{MSOP}_{2}$.

\: Within some classification-theoretic properties, such as $\mathrm{TP}$, not all tuples of elements within the configuration are required to satisfy a consistency condition or an inconsistency condition, so we associate them with a partial hypergraph sequence.  We associate the following partial hypergraph sequence with the property $\mathrm{TP}$: the vertex set consists of distinct $\{b_{\eta}\}_{\eta< \omega^{< \omega}}$, the tuples of $R^{+}_{n}$ comprise all of the tuples of $n$ distinct elements of sets of the form $\{b_{\sigma|_{\ell}}\}_{\ell < \omega}$ for $\sigma \in \omega^{\omega}$, the tuples of $R^{-}_{n}$ comprise all of the tuples of $n$ distinct elements containing a pair of the form $\{R_{\eta \smallfrown \langle i\rangle}, R_{\eta \smallfrown \langle j\rangle} \}$ with $i \neq j$, and the tuples of $R^{\sim}_{n}$ are all of the other tuples. A formula will have $\mathrm{TP}$ if and only if its characteristic sequence embeds this partial hypergraph sequence.

\: These observations motivate the following ``approximate implication" relation between classification-theoretic properties. Let $\mathrm{NP}_{1}$ be a classification-theoretic property defined as the negation of some classification-theoretic property $\mathrm{P}_{1}$. Let $\mathrm{NP}_{2}$ be a classification-theoretic property defined as the negation of some classification-theoretic property $\mathrm{P}_{2}$ associated with a hypergraph sequence or partial hypergraph sequence. Then $\mathrm{NP}_{1} \leadsto \mathrm{NP}_{2}$ if, for every formula $\varphi(x, y)$ in a theory satisfying $\mathrm{NP}_{1}$, for arbitrarily large $N$ there is a finite partial hypergraph sequence $(U, \{P_{n}\}_{n < \omega})$, which does not embed into the characteristic sequence of $\varphi(x,y)$, but such that, for every $U_{0} \subset U$ of size at most $N$, the induced sequence $(U_{0}, \{P_{n}|U_{0}\}_{n < \omega})$ embeds into the hypergraph sequence\footnote{again, as a sequence of induced sub-hypergraphs} or partial hypergraph sequence associated with $\mathrm{P}_{2}$.

\: Note that, if $\varphi(x, y)$ is a formula whose characteristic sequence does not embed some finite partial hypergraph sequence $(U, \{P_{n}\}_{n < \omega})$, but $(U, \{P_{n}\}_{n < \omega})$ embeds into the hypergraph sequence or partial hypergraph sequence associated with $\mathrm{P}_{2}$, then $(U, \{P_{n}\}_{n < \omega})$ will be an obstruction to $\varphi(x, y)$ exhibiting $\mathrm{P}_{2}$. So informally speaking, $\mathrm{NP}_{1} \leadsto \mathrm{NP}_{2}$ says that, if a theory $T$ satisfies $\mathrm{NP}_{1}$, then for every formula $\varphi(x, y)$ there is a constraint on the characteristic sequence of $\varphi(x,y)$ that, to an arbitrarily close finite approximation, looks like an obstruction to $\varphi(x, y)$ exhibiting $\mathrm{P}_{2}$.

Using the work of this section, we get the following approximate implications between classification-theoretic properties:

\begin{theorem}\label{the approximate implications}
     $\mathrm{NSOP}_{3} \leadsto \mathrm{NTP}_{2}$, and $\mathrm{NSOP}_{3} \leadsto \mathrm{NSOP}_{2}$.
\end{theorem}

\begin{proof}
    (sketch) 
    
   \: Let $T$ be $\mathrm{NSOP}_{3}$, and $\varphi(x, y)$ be any formula. Let $A$ be an $\mathcal{L}^{\mathrm{std}}$-structure where $I$ forms an undirected graph relation on $P(A)$; say $A$ is \textit{realized} in $T$ if there is an injective map $\iota$, taking points of $O(A)$ to $|x|$-tuples and points of $P(A)$ to $|y|$-tuples in the ambient model $\mathbb{M} \models T$, with $R(a, b) \Rightarrow \varphi(\iota(a), \iota(b))$ and $I(b_1, b_2) \Rightarrow \neg \exists x \varphi(x, b_{1}) \wedge \varphi(x, b_{2})$. Then for arbitrarily large $N < \omega$, the proof of Lemma \ref{omitting a presidium} in step one of this section\footnote{though this is for $I$ forming a directed graph relation} gives us a finite $\mathcal{L}^{\mathrm{std}}$-structure $A$ not realized in $T$ such that every induced substructure of $A$ of size at most $N$ weakly embeds into the standard $\mathrm{TP}_{2}$-structure (Remark \ref{standard tp2-structure}). Let $( \check{U}, \{P_{n}\}_{n < \omega})$ consist of the vertex set $\check{U}=P(A)$, with the $n$-tuples of distinct elements partitioned into $R^{+}_{n}$ the tuples of points with a common $R$-neighbor, $R^{-}_{n}$ the tuples of points containing an instance of $I$, and $R^{\sim}_{n}$ the other tuples.  
   
   \: Then every induced sequence of sub-hypergraphs in  $( \check{U}, \{P_{n}\}_{n < \omega})$ of size at most $N$ embeds into the $(\omega, \omega, 1)$-array, the hypergraph sequence associated with $\mathrm{TP}_{2}$. Now it is not yet the case that $( \check{U}, \{P_{n}\}_{n < \omega})$  does not embed into the characteristic sequence of $\varphi(x, y)$. It is only the case that there is no embedding into $\mathbb{M}^{|y|}$ such that, \textit{for $K= O(A)$}, and for $(b_{1}, \ldots b_{n})$ the image of a tuple of $R^{+}_{n}$, $\models \exists^{\geq K} x \bigwedge^{n}_{i = 1} \varphi(x,b_{i})$, while for $(b_{1}, \ldots b_{n})$ the image of a tuple of $R^{-}_{n}$, $\models \neg \exists x \bigwedge^{n}_{i = 1} \varphi(x,b_{i})$. So we slightly modify $( \check{U}, \{P_{n}\}_{n < \omega})$ to correct this.\footnote{Another solution is to not require that $\iota$ be injective in the above definition of $A$ being realized in $T$. Under this version of the defintion of $A$ being realized in $T$, the proof of Lemma \ref{omitting a presidium} still yields a finite $\mathcal{L}^{\mathrm{std}}$-structure $A$ not realized in $T$ such that every induced substructure of $A$ of size at most $N$ weakly embeds into the standard $\mathrm{TP}_{2}$-structure. Then $( \check{U}, \{P_{n}\}_{n < \omega})$ already does not embed in the characteristic sequence of $\varphi(x, y)$.} 
   
   \: Let $U =: \check{U} \sqcup S$, where $S$ is a set such that ${|S|\choose N} \geq K$, and extend the $P_{n}$ to form a partial hypergraph sequence on $U$ such that every subset of $S$ of size $N$ is an $R^{+}_{N}$-edge, every subset of $S$ of size $N+1$ is an $R^{-}_{N+1}$-edge, and every union of an $R^{+}_{n}$-edge of $\check{U}$ and an $R^{+}_{N}$-edge of $S$ is an $R^{+}_{n+N}$-edge of $U$.  Then every induced sequence of sub-hypergraphs in  $( U, \{P_{n}\}_{n < \omega})$ of size at most $N$ still embeds into the $(\omega, \omega, 1)$-array: this will consist of a subset of $\check{U}$ of size of most $N$, which will embed into the array, where we then extend each $R^{+}_{n}$-edge in $U$ by some common subset $S_{0}  \subset S$ of size at most $N$. We can then extend the embedding to $S_{0}$ by mapping the points of $S_{0}$ to distinct new rows of the array. Moreover, $( U, \{P_{n}\}_{n < \omega})$ cannot embed into the characteristic sequence of $\varphi(x, y)$: if there is an embedding, we can distinguish at least $K$ realizations of the consistency conditions for $\varphi(x, y)$ on the image of $\check{U}$ by which instances of $\varphi(x, y)$ they satisfy on the image of $S$. So the embedding restricts on $\check{U}$ to one such that, for $(b_{1}, \ldots b_{n})$ the image of a tuple of $R^{+}_{n}$, $\models \exists^{\geq K} x \bigwedge^{n}_{i = 1} \varphi(x,b_{i})$, while for $(b_{1}, \ldots b_{n})$ the image of a tuple of $R^{-}_{n}$, $\models \neg \exists x \bigwedge^{n}_{i = 1} \varphi(x,b_{i})$, a contradiction.

  \:  The proof for $\mathrm{SOP}_{2}$ is similar, but with Remark \ref{potential simplified proof} instead of Lemma \ref{omitting a presidium}.
\end{proof}

\: Note that the relation $\mathrm{NP}_{1}\leadsto \mathrm{NP}_{2}$ should not be taken as evidence that $\mathrm{NP}_{1}$ implies $\mathrm{NP}_{2}$: for example, despite our result that $\mathrm{NSOP}_{3} \leadsto \mathrm{NTP}_{2}$, $\mathrm{NSOP}_{3}$ does not imply $\mathrm{NTP}_{2}$. We may even suspect that, just as our proof of Lemma \ref{omitting a presidium} on the standard $\mathrm{TP}_{2}$-structure gives us $\mathrm{NSOP}_{3} \leadsto \mathrm{NTP}_{2}$, our full proof of our main combinatorial lemma, Corollary \ref{main combinatorial lemma} on the standard $\mathrm{TP}$-structure, might give us $\mathrm{NSOP}_{3} \leadsto \mathrm{simple}$. While our arguments do not appear to directly generalize to this case without additional work, we would not be surprised if some version of these arguments could be used to show $\mathrm{NSOP}_{3} \leadsto \mathrm{simple}$.

\: We conclude by anticipating one possible objection to many general results on $\mathrm{NSOP}_{3}$ theories: we don't have a concrete example of $\mathrm{NSOP}_{3}$ being its own distinct classification-theoretic property, because the equality of $\mathrm{NSOP}_{3}$ and $\mathrm{NSOP}_{1}$ remains open. However, we submit that this objection does not apply in the case of $\mathrm{NSOP}_{3} \leadsto \mathrm{NTP}_{2}$ . Even if $\mathrm{NSOP}_{3}$ is equal to $\mathrm{NSOP}_{1}$, $\mathrm{NSOP}_{1} \leadsto \mathrm{NTP}_{2}$ appears to be a new and nontrivial result on $\mathrm{NSOP}_{1}$ theories, and it is not obvious that it follows, say, from properties of Kim-independence.

\end{tcolorbox}

\textbf{Acknowledgements:} The author would like to thank Alex Kruckman for some helpful bibliographic corrections to an earlier version of this paper, as well as James Hanson for drawing attention to the appendix of Hanson's thesis, from before the introduction of the properties $\mathrm{SOP}_{r}$ in \cite{ApproxOrder}, where he defines a real-valued family of quasimetric properties and gives an equivalence argument with Shelah's original integer-valued $\mathrm{SOP}_{n}$ hierarchy for $n \geq 3$, as discussed in the appendix.

\appendix

\section{A historical note on Hanson-$\mathrm{SOP}_{r}$}

Some prior interest in the prospect of extending the original $\mathrm{NSOP}_{n}$ hierarchy to non-integer values of $n$ is demonstrated in a thesis appendix of Hanson (\cite{Hanson2020}). Hanson introduces properties that he describes as rough analogues of ``$\mathrm{SOP}_{\frac{1}{r}}$" for real $r$ within the context of continuous logic. Unlike the real-valued properties $\mathrm{SOP}_{r}$ we discuss in this paper, whose equivalence with the integer-valued properties $\mathrm{SOP}_{n}$ remains an apparently difficult open question, it is known due to Hanson that his properties coincide with the original integer-valued properties $\mathrm{SOP}_{n}$ (at least in classical two-valued logic, rather than continuous logic). In this appendix, we will give a short exposition of Hanson's properties, as well as Hanson's proof of their equivalence with Shelah's original properties $\mathrm{SOP}_{n}$ for integers $n \geq 3$.

Hanson's properties are defined in terms of \textit{quasimetrics}: functions $\rho: X^{2} \to \mathbb{R}_{\geq 0}$ for some set $X$ such that for all $x, y, z \in X$, $\rho(x, x) = 0$ and $\rho(x, z) \leq \rho(x, y) + \rho (y, z)$. For purposes of this exposition, we will say a quasimetric on a definable set $X$ within a first-order theory is \textit{definable} if the conditions $\rho(x, y) \geq n$, $\rho(x, y) > n$, $\rho(x, y) \leq n$, $\rho(x, y) < n$, $\rho(x, y) = n$ are defined by formulas $\varphi(x, y)$.

To motivate Hanson's properties, consider for fixed $n$ the model companion of the theory of $\{0, \ldots, n\}$-valued metric spaces, introduced in \cite{CW04}. This is essentially equivalent to considering the model companion of the theory of $\{0, \frac{1}{n}, \ldots, \frac{k}{n} , \ldots, 1\}$-valued metric spaces. As shown in \cite{CW04}, this theory has $\mathrm{SOP}_{n}$. This motivates the following definition:

\begin{definition}\label{hanson-sopr}(Hanson, Definition B.5.2 of \cite{Hanson2020})

Let $r > 2$ be a real value. A theory has \textit{Hanson-$\mathrm{SOP}_{r}$} if there is a definable $[0, 1]$-valued quasimetric $\rho(x, y)$ on a definable set $X$, together with a sequence $\{a_{i}\}_{i < \omega}$ of elements of $X$, such that $\rho(a_{i}, a_{j})\leq \frac{1}{r}$ for $i \leq j$, but $\rho(a_{i}, a_{j})= 1$ for $i > j$.
    
\end{definition}

Note that for $1 \geq \alpha' >  \alpha > 0 $, the existence of a definable $[0, 1]$-valued quasimetric $\rho(x, y)$ such that $\rho(a_{i}, a_{j}) = \alpha $ for $i \leq j$ and $\rho(a_{i}, a_{j})= 1$ for $i > j$ implies the existence of a definable $[0, 1]$-valued quasimetric $\rho'(x, y)$ such that $\rho'(a_{i}, a_{j}) = \alpha' $ for $i \leq j$ and $\rho'(a_{i}, a_{j})= 1$ for $i > j$. Just take $\rho'(x, y) = : \mathrm{min}(\{1, \frac{\alpha'}{\alpha}\rho(x, y)\})$. So we may replace $\rho(a_{i}, a_{j})\leq \frac{1}{r}$ with $\rho(a_{i}, a_{j}) = \frac{1}{r}$ within the definition of Hanson-$\mathrm{SOP}_{r}$.

In classical logic, Hanson showed that Hanson-$\mathrm{SOP}_{r}$ for reals $r > 2$ does not introduce new properties: for $r > 2$ real and $n=\lceil r \rceil$ the next integer, Hanson-$\mathrm{SOP}_{r}$ is equivalent to $\mathrm{SOP}_{n}$.

One direction of the equivalence is proven explicitly by Hanson within classical logic. (Hanson treats this case separately from his continuous logic arguments, because the equivalence is only true up to an infinitesimal within continuous logic. Namely, in continuous logic, $\mathrm{SOP}_{n}$ implies Hanson-$\mathrm{SOP}_{r}$ for any real $r < n$, but is not known to imply Hanson-$\mathrm{SOP}_{n}$.) We will simply review Hanson's proof.

\begin{fact}(Hanson (\cite{Hanson2020}), Proposition B.5.9)

 Let $n \geq 3$ be an integer. Then any theory with $\mathrm{SOP}_{n}$ has Hanson-$\mathrm{SOP}_{n}$ (so has Hanson-$\mathrm{SOP}_{r}$ for any real value $r < n$, $r > 2$).
    
\end{fact}

\begin{proof}

Let $R(x, y)$ exhibit $\mathrm{SOP}_{n}$, so there is $\{a_{i}\}_{i < \omega}$ with $\models R(a_{i}, a_{j})$ for $i <j$, but $R(x, y)$ has no directed $n$-cycles. Construct the definable quasimetric $\rho(x, y)$ on $\mathbb{M}^{|x|}$ as follows: $\rho(x, y) =: \mathrm{min}(\{1, \frac{d_{R}(x, y)}{n}
\})$, where $d_{R}(x, y)$ is the length (in number of edges) of the shortest directed $R(x, y)$-path from $x$ to $y$. Then $R(a_{i}, a_{j})$ for $i <j$ will imply $\rho(a_{i}, a_{j}) \leq \frac{1}{n}$ for $i \leq j$. But for $i > j$, $\models R(a_{j}, a_{i})$, so there can be no directed $R(x, y)$-path from $a_i$ to $a_j$ of length less than $n$: otherwise $R(x, y)$ would have an $n$-cycle. So $d_{R}(a_{i}, a_{j}) \geq n$ and $\rho(a_{i}, a_{j}) = 1$, giving Hanson-$\mathrm{SOP}_{n}$.

\end{proof}

The other direction was proven by Hanson within the continuous setting, so we rephrase Hanson's argument within the language of classical logic.

\begin{fact}(Hanson (\cite{Hanson2020}), Proposition B.5.5)

Let $r > 2$ be real, and let $n = \lceil r \rceil $ be the next integer. Suppose the theory $T$ has Hanson-$\mathrm{SOP}_{r}$. Then $T$ has $\mathrm{SOP}_{n}$.
    
\end{fact}

\begin{proof}
    Let $\rho(x, y)$ be an $[0, 1]$-valued quasimetric exhibiting Hanson-$\mathrm{SOP}_{r}$, so there is a sequence $\{a_{i}\}_{i < \omega}$ of elements of $X$, such that $\rho(a_{i}, a_{j})\leq \frac{1}{r}$ for $i \leq j$, but $\rho(a_{i}, a_{j})= 1$ for $i > j$. Particularly, $\rho(a_{i}, a_{j}) < \frac{1}{n-1}$ for $i \leq j$. By definability of the quasimetric, we can find some formula $\varphi(x, y)$ with $\models \varphi(a_{i}, a_{j})$ for $i < j$, such that $\varphi(x, y)$ guarantees $\rho(x, y) < \frac{1}{n-1}$, $\rho (y, x) = 1$. It suffices to show that the relation defined by $\varphi(x, y)$ has no directed $n$-cycles. Suppose for a contradiction that $\varphi(x, y)$ does have a directed $n$-cycle. Then there are $b_{0}, \ldots b_{n-1}$ such that $\models \varphi(b_{i}, b_{i+1 \mathrm{\: mod \:} n})$ for all $i \in \{0, \ldots n-1\}$. Then $\rho(b_{0}, b_{1}) < \frac{1}{n-1}, \ldots , \rho(b_{k}, b_{k +1}) < \frac{1}{n-1}, \ldots, \rho(b_{n-2}, b_{n-1}) < \frac{1}{n-1}$ will imply $\rho(b_{0}, b_{n-1}) \leq \sum^{n-2}_{i =0} \rho(b_{i}, b_{i +1})< 1$. However, because  $\varphi(x, y)$ implies $\rho (y, x) = 1$, and $\models \varphi(b_{n-1}, b_{0})$, $\rho(b_{0}, b_{n-1}) = 1$, a contradiction.
\end{proof}

This proves that properties Hanson-$\mathrm{SOP}_{r}$ for real values $r > 2$ coincide (in classical logic) with the original $\mathrm{SOP}_{n}$ hierarchy for $n \geq 3$ an integer.

However, in continuous logic, the properties Hanson-$\mathrm{SOP}_{r}$ may not reduce  to Shelah's original integer-valued $\mathrm{SOP}_{n}$ hierarchy for $n \geq 3$. Hanson uses continuous versions of the two previous facts to show that the collection of reals $r > 2$ for which a theory has Hanson-$\mathrm{SOP}_{r}$ satisfies one of four possibilities. It is empty, equal to $(2, \infty)$, equal to $(2, n]$ for some integer $n > 2$, or equal to $(2, n)$ for some integer $n > 2$. It is open whether the fourth possibility can be ruled out. So in continuous logic, the properties Hanson-$\mathrm{SOP}_{r}$ for real values of $r$ may differ from the integer-valued $\mathrm{NSOP}_{n}$ hierarchy, but only infinitesimally.

Note that the possibility that the properties Hanson-$\mathrm{SOP}_{r}$ for real values of $r$ differ infinitesimally from $\mathrm{SOP}_{n}$ for integer $n$ in continuous logic is not exactly analogous to the possibility, discussed after Question 2.14 of \cite{ApproxOrder}, that the real-valued $\mathrm{SOP}_{r}$ hierarchy differs only infinitesimally from the integer-valued $\mathrm{SOP}_{n}$ hierarchy in discrete logic. In the scenario where the properties Hanson-$\mathrm{SOP}_{r}$ differ infinitesimally from integer-valued $\mathrm{SOP}_{n}$ in continuous logic, $\mathrm{SOP}_{n}$ for integers $n$ will be equivalent to Hanson-$\mathrm{SOP}_{r}$ for all reals $r < n$, as is already known. But for some integer $n$ there will be theories with $\mathrm{SOP}_{n}$ and Hanson-$\mathrm{SOP}_{n}$, as well as theories with $\mathrm{SOP}_{n}$ but without Hanson-$\mathrm{SOP}_{n}$. So the properties Hanson-$\mathrm{SOP}_{r}$ will not actually be an extension of the integer-valued $\mathrm{SOP}_{n}$ hierarchy.\footnote{The properties will be a true extension if, in the definition of Hanson-$\mathrm{SOP}_{r}$ in Definition \ref{hanson-sopr} above,  the inequality $\rho(a_{i}, a_{j})\leq \frac{1}{r}$ is replaced with the strict inequality $\rho(a_{i}, a_{j}) < \frac{1}{r}$. Then, under this modified definition of ``Hanson-$\mathrm{SOP}_{r}$," Hanson's arguments will show that for any real value $r > 2$, ``Hanson-$\mathrm{SOP}_{r}$," will coincide with $\mathrm{SOP}_{n}$ for $n=\lceil r \rceil$ the next integer. This will be the case even in continuous logic. But then, it will be settled that the properties ``Hanson-$\mathrm{SOP}_{r}$" are not new properties, even in continuous logic.} Let us now consider the claim that the real-valued $\mathrm{SOP}_{r}$ hierarchy differs, but only infinitesimally, from integer-valued $\mathrm{SOP}_{n}$ hierarchy in classical logic. If this is true, it will be the case as always that the real-valued $\mathrm{SOP}_{r}$ hierarchy coincides exactly with the original definition of $\mathrm{SOP}_{n}$ when $r = n$ happens to be an integer. The real-valued and integer-valued hierarchies differing only infinitesimally will refer to the statement that the hierarchies are not distinct, but this non-distinctness is only exhibited by theories that are $\mathrm{NSOP}_{n}$ for some integer $n \geq 3$, but have $\mathrm{SOP}_{r}$ for all real values $r < n$.

\bibliographystyle{plain}
\bibliography{refs}

\end{document}